\documentclass[11pt]{amsbook}

\pdfoutput=1

\usepackage{amsmath, amssymb, amsthm, amscd}
\usepackage{graphicx} 
\usepackage{color} 
\usepackage{hyperref}
\usepackage[latin1]{inputenc}
\usepackage{booktabs}
\usepackage[all]{xy}

\definecolor{Red}{rgb}{1.,0.,0.}  \definecolor{Blue}{rgb}{0.,0.,1.}

\newcommand{\id}{\mathop \textnormal{id}\nolimits} \newcommand{\R}{\mathbb{R}}
 
\newcommand{\N}{\mathbb{N}} 
\newcommand{\diam}{\mathop \textnormal{diam}\nolimits}
\newcommand{\tr}{\mathop \textnormal{tr}\nolimits}
\newcommand{\riem}{\mathcal{R}}
\newcommand{\T}{\mathcal{T}} \newcommand{\M}{\mathcal{M}}
\newcommand{\Ms}{\mathcal{M}^s} \newcommand{\Mmu}{\mathcal{M}_\mu}
\newcommand{\Mmus}{\mathcal{M}_\mu^s}
\newcommand{\Mhyp}{\mathcal{M}_{-1}}

\newcommand{\Mhypd}{\mathcal{M}_{-1} / \mathcal{D}_0}
\newcommand{\Mf}{\mathcal{M}_f}
\newcommand{\Mfhat}{\widehat{\mathcal{M}_f}}
\newcommand{\Mm}{\mathcal{M}_m}
\newcommand{\Mmhat}{\widehat{\mathcal{M}_m}}
\newcommand{\Ncal}{\mathcal{N}}
\newcommand{\U}{\mathcal{U}}
\newcommand{\s}{\mathcal{S}} \newcommand{\st}{\mathcal{S}_g^T}
\newcommand{\stt}{\mathcal{S}_g^{TT}}
\newcommand{\sconf}{\mathcal{S}_g^c} 

\newcommand{\D}{\mathcal{D}} 
 \newcommand{\DO}{\mathcal{D}_0}

 \newcommand{\V}{\mathcal{V}}

\newcommand{\pos}{\mathcal{P}} \newcommand{\poss}{\mathcal{P}^s}

 \newcommand{\btop}{\partial \M}

\newcommand{\supp}{\mathop \textnormal{supp}\nolimits}
\newcommand{\carr}{\mathop \textnormal{carr}\nolimits}
\newcommand{\Vol}{\mathop \textnormal{Vol} \nolimits}
\newcommand{\diag}{\mathop \textnormal{diag}}

\newcommand{\overarrow}[1]{\buildrel #1 \over \longrightarrow}
\newcommand{\Matx}{\M_x} \newcommand{\satx}{\s_x}

\newcommand{\integral}[4]{\int_{#1}^{#2} #3 \, #4}
\newcommand{\MV}{\mathcal{M}_V}
\newcommand{\cl}{\textnormal{cl}}
\newcommand{\midmid}{\; \middle| \;}

\newcommand{\cal}[1]{\mathcal{#1}}
\providecommand{\abs}[1]{\lvert #1 \rvert}
\providecommand{\norm}[1]{\lVert #1 \rVert}

\newtheorem{thm}{Theorem} \newtheorem*{thm*}{Theorem}
\newtheorem{prop}[thm]{Proposition} \newtheorem*{prop*}{Proposition}
\newtheorem{cor}[thm]{Corollary} \newtheorem{lem}[thm]{Lemma}
\theoremstyle{definition} \newtheorem{dfn}[thm]{Definition}
\theoremstyle{definition} \newtheorem{eg}[thm]{Example}
\theoremstyle{remark} \newtheorem{rmk}[thm]{Remark}
\theoremstyle{remark} \newtheorem{cvt}[thm]{Convention}

\numberwithin{equation}{chapter}
\numberwithin{thm}{chapter}

\usepackage[paper=a4paper,margin=3cm]{geometry}

\begin{document}

\author{Brian Clarke \\[2ex] \normalsize{Ph.D. Thesis, University of
    Leipzig} \\ \normalsize{Corrected Version} \\ \normalsize{\today}}
\date{April 1, 2009} \title{The Completion of the Manifold of Riemannian
  Metrics with Respect to its $L^2$ Metric}
\maketitle

\tableofcontents

\chapter*{Acknowledgments}

First and foremost, thanks go to my advisor Jürgen Jost for his years
of encouragement, advice and patient nudges in the right direction.  I
am indebted to him above anyone else for introducing me to this topic
and helping me to reach this point.  I am also indebted to the
International Max Planck Research School \emph{Mathematics in the
  Sciences} and the Max Planck Institute for Mathematics in the
Sciences, as well as the University of Leipzig Mathematical Institute
and the Graduate College \emph{Geometry, Analysis, and their
  Interaction with the Natural Sciences} for financial support and for
providing an excellent work environment.

Thanks also to the Geometry and Physics group---Guy Buss, Alexei
Lebedev, Christoph Sachse and Miaomiao Zhu---for feedback in our
seminar.  My special gratitude goes to Guy, Christoph and Nadine Große
for many interesting discussions over the years and for their tireless
efforts to exterminate any potential mistakes from this thesis.  Guy
in particular provided the one thing that any author should treasure
above all else---a critical eye that is as sharp-sighted as its owner
is merciless.  Of course, any and all responsibility for the continued
presence of mistakes rests solely on my own shoulders.

I would like to thank Rafe Mazzeo for welcoming me for two pleasant
and productive months at Stanford University, where Chapter
\ref{cha:almost-everywh-conv} was written.  I am obliged to Larry
Guth, who provided valuable comments on my work in several nice
discussions.

My gratitude goes to Yurii Savchuk for discussions related to Subsection
\ref{sec:pointwise-estimates}, which was the key to unlocking many
useful results.

I am thankful to my family for their love and support, and for never
having to wonder whether it might be there for me when I needed it.  I
am also indebted to the many teachers and mentors I have had in
Leipzig and in Ann Arbor, who taught me to love mathematics and gave
me the tools to follow through.  Thanks to Alex, Dave, David, Jens,
Sandra and Will for holding the rope.  Thanks to Anne and Julia for
always being there, for always putting up with the little monster, and
for making Leipzig feel like a home away from home.

And most importantly, thanks to Marie and Chaya, whose insanity was
the only thing keeping me sane.


\chapter{Introduction}\label{cha:introduction}

\section{Summary of results}\label{sec:summary-results}

Let $M$ be a smooth, closed, finite-dimensional, oriented manifold,
and denote by $\M$ the Fréchet manifold of smooth Riemannian metrics
on $M$.  There is a natural Riemannian metric on $\M$ called the $L^2$
metric and denoted by $(\cdot, \cdot)$.  It is the primary goal of
this thesis to give a description of the completion of $(\M, (\cdot,
\cdot))$.  There are three main results we will summarize here.  The
first is the following:

\begin{thm*}
  With the Riemannian distance function $d$ induced from $(\cdot,
  \cdot)$, $(\M, d)$ is a metric space.
\end{thm*}

This is indeed a theorem that needs proving, as the $L^2$ metric on
$\M$ is an example of a so-called weak Riemannian metric.  This means
that the tangent spaces of $\M$ are not complete with respect to
$(\cdot, \cdot)$, and so many general theorems from the usual theory
of Riemannian Hilbert manifolds do not hold. In this theory, typically
only so-called strong metrics are considered, with respect to which
the tangent spaces are complete.  Of course, for a strong Riemannian
Hilbert manifold, the Riemannian metric induces a metric space
structure on the manifold. However, for weak Riemannian manifolds,
there are examples
(cf.~\cite{michor05:_vanis_geodes_distan_spaces_of},
\cite{michor06:_rieman_geomet_spaces_of_plane_curves}) where this does
not hold, as the distance between some points may be zero.  Therefore,
one must explicitly prove that a given weak Riemannian manifold is a
metric space.

Given a metric space structure on $\M$, we know that it has a
completion, and the second result---which can be seen as \emph{the}
main result of the thesis---gives a concrete description of this.  Let
$\M_f$ denote the set of semimetrics on $M$ (i.e., sections of the
bundle $S^2 T^* M$ that induce a positive semidefinite scalar product
on each tangent space of $M$) that have measurable coefficients and
finite volume.  Define an equivalence relation on $\M_f$ by saying
$g_0 \sim g_1$ if the following statement holds for almost every $x
\in M$: if $g_0(x)$ and $g_1(x)$ differ, then both $g_0(x)$ and
$g_1(x)$ fail to be positive definite.  If we let $\overline{\M}$
denote the completion of $\M$ with respect to $(\cdot, \cdot)$, then:

\begin{thm*}
  There is a natural bijection $\Omega : \overline{\M} \rightarrow
  \M_f / {\sim}$ that is the identity when restricted to $\M \subset
  \overline{\M}$.
\end{thm*}

This completion fits in with the general philosophy that in order to
complete a space of objects, one must allow objects of a somewhat more
general type.  Note that we start with smooth metrics, yet in order to
complete $\M$, we must add in points corresponding to metrics with far
worse properties.  This essentially arises from the fact that the
$L^2$ metric---as its name implies---induces the $L^2$ topology on the
tangent spaces of $\M$, which themselves only consist of smooth
objects.  Thus, the extreme incompleteness of the tangent spaces is
reflected in the incompleteness of the space $\M$ itself.

The final result we describe here is an application of the completion
of $\M$ to Teichmüller theory.  If the base manifold $M$ is
additionally assumed to be a Riemann surface of genus larger than one,
then the Teichmüller space $\T$ of $M$ can be identified with the
space of conformal classes of metrics on $M$ modulo $\DO$, by which we
denote the diffeomorphisms of $M$ that are homotopic to the identity.
Let $\Ncal$ be a smooth submanifold of $\M$ which is invariant under
the action (by pull-back) of the diffeomorphism group, and which
contains exactly one representative from each conformal class.  Then
we have a diffeomorphism $\T \cong \Ncal / \DO$, and the $L^2$ metric
restricted to $\Ncal$ induces a Riemannian metric on $\T$.  As a
corollary of the last theorem, we have:

\begin{thm*}
  Each point in the completion of $\T$ with respect to the Riemannian
  metric described above can be identified with an element of $\M_f /
  {\sim}$.  This identification is not unique.
\end{thm*}

The metrics on $\T$ we have just constructed generalize the
Weil-Petersson metric on Teichmüller space, and this theorem
generalizes what is already known about the completion of Teichmüller
space with respect to the Weil-Petersson metric.

\section{Motivation}\label{sec:motivation}

The original motivation for studying this problem comes from
Teichmüller theory, and that is why this application in particular is
given.  Let us describe how our considerations arose from similar ones
in Teichmüller theory.

As above, let the base manifold $M$ be a Riemann surface of genus
greater than one.  Consider the group $\pos$ of positive functions on
$M$; it acts on $\M$ by pointwise multiplication.  The quotient space
$\M / \pos$ is a smooth manifold, called the manifold of conformal
classes on $M$.  Furthermore, the pull-back action of a diffeomorphism
on $\M$ descends to an action on $\M / \pos$.

Fischer and Tromba \cite{tromba-teichmueller} have given a description
of Teichmüller space $\T$ in this context.  They show that there
exists a diffeomorphism
\begin{equation*}
  \T \cong (\M / \pos) / \DO.
\end{equation*}
Thus, Teichmüller theory can be considered, in their words, in a
``purely Riemannian'' way.

A crucial step in Fischer and Tromba's approach is using the Poincaré
uniformization theorem to show a diffeomorphism between $\M / \pos$
and the space $\Mhyp$ of hyperbolic metrics (those with constant
scalar curvature $-1$) on $M$.  By the Poincaré uniformization
theorem, there exists exactly one hyperbolic metric in each conformal
class on $M$.  Thus, Teichmüller space can just as well be described
as
\begin{equation*}
  \T \cong \Mhypd.
\end{equation*}

The advantage of using $\Mhyp$ is that the submanifold $\Mhyp \subset
\M$ is easier to work with than the quotient space $\M / \pos$.
Furthermore, the above construction allows us to define a metric on
Teichmüller space by first restricting the $L^2$ metric to $\Mhyp$ and
then looking at the metric it induces on the quotient $\Mhypd$, and of
course also on $\T$.  The metric thus defined coincides, up to a
constant scalar factor, with the well-known Weil-Petersson metric on
Teichmüller space.

The Weil-Petersson metric has been the object of much study, and its
completion has very interesting properties.  Wolpert
\cite{wolpert75:_noncom_of_weil_peter_metric_for_space} and Chu
\cite{chu76:_weil_peter_metric_in_modul_space} independently proved
that it is incomplete, as there are geodesics that cannot be
indefinitely extended---in finite time, they hit a singular limit
surface.

Masur \cite{masur-extension} computed the asymptotics of the
Weil-Petersson metric as one approaches the boundary of Teichmüller
space.  He did so in order to describe an extension of the metric to
the completion of Teichmüller space.  The completion of Teichmüller
space with respect to the Weil-Petersson metric also induces a
compactification of the moduli space of $M$, and this compactification
coincides with the Deligne-Mumford compactification, which arises in
the context of algebraic geometry
\cite{deligne69:_irred_of_space_of_curves}.  Thus, the Weil-Petersson
metric links the differential geometric and algebraic geometric
approaches to moduli space, which on the surface seem quite disparate.

Habermann and Jost \cite{hj-riemannian},
\cite{habermann98:_metric_rieman_surfac_and_geomet} later generalized
the Weil-Petersson metric in the following way.  The correspondence
between $\M / \pos$ and $\Mhyp$ is basically given by the fact that
$\Mhyp$ is a smooth global section of the principal $\pos$-bundle
\begin{equation*}
  \M \rightarrow \M / \pos.
\end{equation*}
The question that naturally arises is, what if one were to take a
different section of this bundle?  To retain the correspondence with
Teichmüller and moduli space, the section should be smooth and
invariant under the diffeomorphism group, but as long as these
requirements are satisfied, any section will give a metric on
Teichmüller space.  Note that though we take direct inspiration from
Habermann and Jost, the authors did not treat this exactly the same
way as we described in Section \ref{sec:summary-results}, but rather
retained some of the structures from the complex analytic definition
of Teichmüller space (described, e.g., in \cite{imayoshi-taniguchi}).
The construction we described in Section \ref{sec:summary-results} is,
in the spirit of Fischer and Tromba, a purely Riemannian one, and it
was chosen primarily because it allows us to directly apply our main
result.  The differences between our construction and that of
Habermann and Jost are described in Chapter
\ref{cha:appl-teichm-space}.

In \cite{hj-riemannian}, Habermann and Jost first considered the
section given by the so-called Bergman metric in each conformal class.
They gave a description of the completion of Teichmüller space with
respect to their generalization of the Weil-Petersson metric for this
special case.  In \cite{habermann98:_metric_rieman_surfac_and_geomet},
they considered all possible choices of sections, giving a sufficient
analytic criterion for incompleteness of their generalized
Weil-Petersson metric.

Thus, our application as described in Section
\ref{sec:summary-results} is in the same spirit as the papers of
Habermann and Jost, and we expect that our theorems will have other,
similar applications, especially to Teichmüller theory.  That we have
nevertheless chosen to prove the other main results listed in Section
\ref{sec:summary-results} for base manifolds of all dimensions and
topologies has various reasons.  First, the generalization to
arbitrary dimension was mostly straightforward.  Second, the manifold
of Riemannian metrics on an $n$-dimensional manifold arises in various
other contexts, and it is possible that our theorems might find
applications there.  The manifold of metrics has been considered in
general relativity by, e.g., DeWitt
\cite{dewitt67:_quant_theor_of_gravit}.  Furthermore, critical points
of functionals on the manifold of metrics have been used to determine
``best metrics'' on a given base manifold---for a nice survey of this
topic with compendious references, see
\cite[Ch.~11]{berger03:_panor_view_of_rieman_geomet}.  Finally, the
manifold of metrics is itself of great intrinsic interest, as it
exhibits interesting geometry.  We will review the work that has been
done on this last aspect in the next section.

\section{Overview of previous work}\label{sec:overv-prev-work}

Geometric structures on the manifold of metrics were perhaps first
considered by DeWitt \cite{dewitt67:_quant_theor_of_gravit}, who, as
mentioned above, was interested in applications to general relativity.
The metric on $\M$ considered by DeWitt is quantitatively similar to
the $L^2$ metric, but has different signature.

Ebin \cite{ebin70:_manif_of_rieman_metric} shortly thereafter used the
$L^2$ metric on $\M$ to obtain local slices for the action of the
diffeomorphism group on $\M$.  He used this, for one, to obtain results
about the topology of so-called \emph{superspace}, which is the
quotient $\M / \D$ of $\M$ by the group $\D$ of smooth,
orientation-preserving diffeomorphisms of $M$.  Superspace can be
viewed as the space of Riemannian geometries on $M$.  Ebin also used
his slice theorem to show that the set of metrics $\M'$ with trivial
isometry group is an open, dense subset of $\M$.

Later, Freed and Groisser \cite{freed89:_basic_geomet_of_manif_of}
studied the basic geometry of $\M$ with the $L^2$ metric.  They
computed the curvature and geodesics of $\M$ and two related
manifolds, the submanifold $\M_\mu \subset \M$ of metrics inducing a
fixed volume form $\mu$, and the manifold $\V$ of smooth volume forms
on $M$.  Additionally, Freed and Groisser used their results to study
the curvature and geodesics of the quotient manifold $\M' / \D$, where
$\M'$ is again the set of metrics with trivial isometry group.

Gil-Medrano and Michor
\cite{gil-medrano91:_rieman_manif_of_all_rieman_metric} generalized
the results of Freed and Groisser on $\M$ to the case of base
manifolds $M$ that are not necessarily compact.  Though the Ricci and
scalar curvature of $\M$ cannot be defined in the usual way, they
define and compute curvatures on $\M$ which they call ``Ricci-like''
and ``scalar-like'' curvature.  Moreover, Gil-Medrano and Michor give
a detailed analysis of the exponential mapping, which serves as an
excellent illustration of the problems that can arise when considering
weak instead of strong Riemannian manifolds.  They also prove the
existence and uniqueness of Jacobi fields on $\M$, and give an
explicit expression for these fields.

A generalized version of
\cite{gil-medrano91:_rieman_manif_of_all_rieman_metric} is the paper
\cite{gil-medrano92:_pseud_metric_spaces_of_bilin_struc} by
Gil-Medrano, Michor, and Neuwirther.  The results of this paper are
also given in \cite[\S 45]{kriegl97:_conven_settin_of_global_analy}.

We will rely heavily on the work of the above-mentioned authors and
are indebted to all of them for laying the foundations upon which this
thesis is built.

\section{Outline of the thesis}\label{sec:outline-thesis}

The thesis is arranged as follows.  In Chapter
\ref{cha:preliminaries}, we summarize the preliminary knowledge
necessary to carry out and understand the work that we will do in the
remainder of the thesis.  We begin with a discussion of the completion
of a metric space, which is meant to recall fundamental results on
this topic and collect all the facts we will need into a coherent
form.  Following that, we give the definition of Fréchet manifolds.
This is the category in which we will work, and we describe how spaces
of smooth mappings, like $\M$, can be viewed as Fréchet manifolds.  We
then go over a few somewhat nonstandard facts from Riemannian geometry
for which we could find no complete reference.

In Chapter \ref{cha:preliminaries}, we also discuss weak Riemannian
manifolds, a class of manifolds that includes $(\M, (\cdot , \cdot))$,
as we already mentioned.  In particular, we sketch an example by
Michor and Mumford
\cite{michor06:_rieman_geomet_spaces_of_plane_curves} of the
potentially pathological properties of such manifolds, as well as
giving our own proofs of some standard results from the theory of
Riemannian Hilbert manifolds that we have weakened so that they hold
for weak Riemannian manifolds as well.  With knowledge of these
structures at hand, we then go into details on the manifold of metrics
itself, more explicitly describing many of the previously known facts
mentioned in Section \ref{sec:overv-prev-work}.  Chapter
\ref{cha:preliminaries} closes with a list of conventions and notation
that we use throughout the thesis.

In Chapter \ref{cha:init-metr-prop}, we begin by proving the first of
the main results given in Section \ref{sec:summary-results}, namely
that $\M$ with its $L^2$ metric has the structure of a metric space.
One of the steps in this proof, which is also of use in later
chapters, is the fact that the function on $\M$ assigning to a metric
the square root of its total volume is Lipschitz.  The second half of
the chapter initiates the study of the completion of $(\M, (\cdot,
\cdot))$, where we first try to complete ``nice'' subspaces of $\M$.
We show that if we take a subset of metrics satisfying certain
uniformity conditions, then the completion of such a subset with
respect to $(\cdot, \cdot)$ coincides with its completion with respect
to the $L^2$ norm (not to be confused with the $L^2$ metric).  This
fact is used as a springboard for our further investigations of the
completion.

The completion of a metric space is a quotient space of the set of
Cauchy sequences in the space.  Since we wish to identify
$\overline{\M}$ with (a quotient of) the space $\M_f$ of measurable,
finite-volume, positive semidefinite sections of $S^2 T^* M$, we need
a rigorous notion for how a Cauchy sequence in $\M$ converges to an
element of $\M_f$.  This notion, which we call $\omega$-convergence,
is described in Chapter \ref{cha:almost-everywh-conv}, where we prove
that every Cauchy sequence in $\M$ subconverges to a unique element of
$\M_f / {\sim}$ (cf.~Section \ref{sec:summary-results}).  This allows
us to define the map $\Omega : \overline{\M} \rightarrow \Mfhat$
mentioned in Section \ref{sec:summary-results}, as well as to show
that $\Omega$ is an injection.  This chapter is the most technically
challenging of the thesis.

In the definition of $\omega$-convergence, we basically have pointwise
convergence of the metrics in a Cauchy sequence $\{g_k\}$ almost
everywhere, with the exception that on any set $E$ with $\Vol(E, g_k)
\rightarrow 0$, there is no convergence required and none can be asked
for.  The reason for this is the following proposition, which is in
our eyes one of the most striking and unexpected results of the
thesis:

\begin{prop*}
  Suppose that $g_0, g_1 \in \M$, and let $E := \carr (g_1 - g_0) = \{
  x \in M \mid g_0(x) \neq g_1(x) \}$.  Let $d$ be the Riemannian
  distance function of the $L^2$ metric $(\cdot, \cdot)$.  Then there
  exists a constant $C(n)$ depending only on $n := \dim M$ such that
  \begin{equation*}
    d(g_0, g_1) \leq C(n) \left( \sqrt{\Vol(E, g_0)} +
      \sqrt{\Vol(E,g_1)} \right).
  \end{equation*}

  In particular, we have
  \begin{equation*}
    \diam \left( \{ \tilde{g} \in \M \mid \Vol(M, \tilde{g}) \leq
      \delta \} \right) \leq 2 C(n) \sqrt{\delta}.
  \end{equation*}
\end{prop*}

\begin{figure}[t]
  \centering
  \includegraphics[width=11cm]{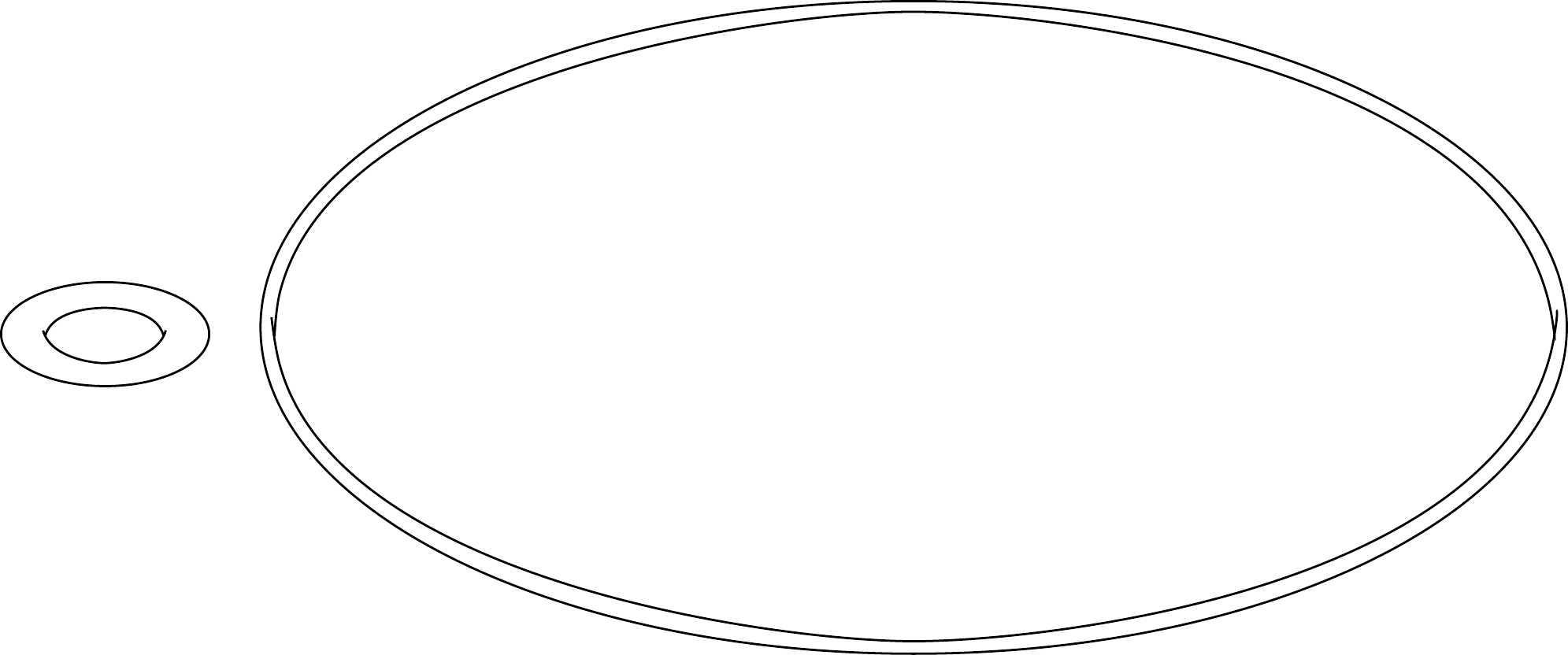}
  \caption{Two tori that are close together in $\M$ (with base
    manifold $M = T^2$) purely by virtue of having small volume.}
  \label{fig:tori}
\end{figure}

The surprising thing about this proposition is that it says that two
metrics can vary wildly, but as long as they do so on a set that has
small volume with respect to each, they are close together in the
$L^2$ metric.  For example, if $M = T^2$, the two-dimensional torus,
with its standard chart ($[0,1] \times [0,1]$ with edges identified),
we consider the metrics
\begin{equation*}
  g_0 =
  \begin{pmatrix}
    10 & 0 \\
    0 & 10^{-5}
  \end{pmatrix}
  \qquad
  g_1 =
  \begin{pmatrix}
    10^{10} & 0 \\
    0 & 10^{-14}
  \end{pmatrix}.
\end{equation*}
By the above proposition, these two very different metrics satisfy
$d(g_0, g_1) \leq C(n) / 100$, simply because they define tori with
small volume.  The difference between the geometries defined by $g_0$
and $g_1$ is depicted very qualitatively in Figure \ref{fig:tori}.

The above proposition is the reason why, in the second theorem of
Section \ref{sec:summary-results}, we identify $\overline{\M}$ with a
quotient space of the space $\M_f$ of semimetrics with measurable
coefficients and finite volume, instead of $\M_f$ itself.  The reasons
for this are discussed in more detail in Chapter
\ref{cha:almost-everywh-conv}.

In Chapter \ref{chap:sing-metrics}, we complete the proof of the
second main result of Section \ref{sec:summary-results} by showing
that the map $\Omega : \overline{\M} \rightarrow \Mfhat$ is a
surjection.  Combined with the already mentioned results of Chapter
\ref{cha:almost-everywh-conv}, we thus see that $\Omega$ is a
bijection, proving the main result on the completion of $\M$.

Finally, in Chapter \ref{cha:appl-teichm-space}, we give a more
detailed overview of the aspects of Teichmüller theory mentioned in
Section \ref{sec:motivation}.  One novelty of our presentation of this
well-tread area of mathematics is a compact and relatively elementary
proof of the existence of horizontal lifts for the principal bundle
$\Mhyp \rightarrow \Mhypd$.  We note, though, that this existence has
been long-known to experts in the field.

After presenting the known facts about Teichmüller theory and the
Weil-Petersson metric that we need, we give the generalizations of the
Weil-Petersson metric mentioned in Section \ref{sec:summary-results},
rigorously stating and proving the result on the completion of
Teichmüller space with respect to these metrics.

At this point, we would like to draw the reader's attention to two
reading aids that should help to avoid confusion.  First, on page
\pageref{cha:metr-conv-noti}, we lay out the relations between the
various Riemannian metrics, distance functions and convergence notions
used in the thesis.  The second aid is the list of symbols on page
\pageref{cha:list-frequently-used}, where we have attempted to include
all symbols used with any frequency throughout the text.  We hope that
that these two guides provides the reader with at least a trail of
bread crumbs to avoid getting lost while navigating the thesis.

\section{Outlook}\label{sec:outlook}

We have given just one application of the main result of our thesis,
the application to Teichmüller theory.  However, we envision more
applications to arise in the future, in particular applications to
determining the completion of superspace $\M / \D$, the space of
Riemannian geometries on $M$ mentioned in Section
\ref{sec:overv-prev-work}.

In particular, the $L^2$ metric is invariant under the pull-back
action of the group $\D$ of orientation-preserving diffeomorphisms of
$M$ (see Section \ref{sec:hyperb-metr-unif}), and so it induces a
well-defined distance function on the quotient.  Note that we do not
get a Riemannian metric, since the quotient is a singular space due to
non-freeness of the $\D$-action at any metric with nontrivial isometry
group.  Of course, we nevertheless hope that information about the
completion of $\M$ can lead us to information about the completion of
$\M / \D$.

These results are not immediate, however, for a number of reasons.
The simple fact that $\M$ is a metric space does not necessarily imply
that the orbit space $\M / \D$ carries a metric space structure as
well---the induced distance function may only be a pseudometric, as
\emph{a priori} two $\D$-orbits may be infinitesimally close to one
another.  The singular nature of $\M / \D$ makes it difficult to
relate distances on $\M / \D$ to those on $\M$.  Here, the existence
of Ebin's slice \cite{ebin70:_manif_of_rieman_metric} might be
helpful.  Alternatively, one could adopt the philosophy analogous to
using Teichmüller space for studying moduli space and first study an
intermediate, smooth space like Fischer's resolution of the
singularities of $\M / \D$
\cite{fischer86:_resol_singul_in_space_of_rieman_geomet}.

These considerations are, however, extremely preliminary, and are
merely given to illustrate one potential future direction this work
might lead us in.


\chapter{Preliminaries}\label{cha:preliminaries}

In this chapter, we define and explore the concepts necessary to carry
out the main body of the work.  The chapter is structured as follows:

We go over the most basic material in Section
\ref{sec:compl-metr-spac}, where we briefly recall the definitions and
fundamental facts regarding completions of metric spaces.  We also
give an alternative definition of the completion of a metric space
that is more suited to studying Riemannian manifolds.

We give a definition of Fréchet manifolds in Section
\ref{sec:frechet-manifolds}, since this will be the category in which
we work.  We go into depth on the class of Fréchet manifolds that
plays the greatest role in global analysis, that of manifolds of
mappings (actually, manifolds of sections of finite-dimensional fiber
bundles).

In Section \ref{sec:geom-prel}, we briefly review some of the geometry
that will be needed for the subsequent portions of the thesis.

We then generalize the notion of a Riemannian metric to Fréchet
manifolds in Section \ref{sec:weak-riem-manif}.  In particular, we are
interested in so-called \emph{weak Riemannian metrics} on Hilbert and
Fréchet manifolds, as the $L^2$ metric on the manifold of metrics is
such an object.  Weak Riemannian metrics are, as we will argue,
fundamental objects in global analysis, though the lack of a good
general theory for them makes their study more difficult than the
tamer strong Riemannian manifolds.  With some notable exceptions, the
research on weak Riemannian manifolds focuses on studying specific
cases, and the difficulties arising from the weak nature of the
metrics are often only implicit.  General results on weak Riemannian
manifolds are often given without proof, as the statements are
typically just a weakening of the corresponding statements for strong
Riemannian manifolds.  Nevertheless, we felt a precise treatment was
appropriate for this work.  Therefore, at the end of Section
\ref{sec:weak-riem-manif} we present some results which are relatively
straightforward generalizations of analogous results for strong
Riemannian manifolds---though necessarily weaker---and which the
author has not found explicitly proved anywhere else in the
literature.

The study of weak Riemannian metrics will allow us to define the
Riemannian manifold $\M$ of Riemannian metrics in Section
\ref{sec:manifold-metrics-m}, as well as to discuss what is already
known about this manifold, in particular what is already known about
its metric geometry.  For example, we will give a description of its
exponential mapping and discuss its curvature.  We will also discuss
two important classes of submanifolds, the orbits of the conformal
group (i.e., the group of positive functions) and the manifolds of
metrics that induce the same volume form.

Finally, we end the chapter with Section \ref{sec:conventions}, which
describes the nonstandard conventions that will be in place throughout
the text.

\section{Completions of metric spaces}\label{sec:compl-metr-spac}

In this short section, we look at completions of metric spaces.  We
will simply state the definition and explore a couple of consequences
of it, then give an alternative, equivalent viewpoint for path metric
spaces.

For the rest of the section, let $(X, \delta)$ be a metric space.

Recall that $X$ is called complete if every Cauchy sequence converges.
Even if $X$ is incomplete, there is a very natural way to construct a
complete space from $X$.  The basic idea is that if we want a space in
which every Cauchy sequence converges, then we should replace $X$ with
a space in which each point represents a Cauchy sequence in $X$.  Then
each Cauchy sequence in this new space ``converges to itself'' in a
certain sense.  This idea can be made more precise as follows.

The \emph{precompletion} \label{p:precompl} of $(X, \delta)$ is the
set $\overline{(X, \delta)}^{\textnormal{pre}}$, usually just denoted
by $\overline{X}^{\textnormal{pre}}$, consisting of all Cauchy
sequences of $X$, together with the distance function
\begin{equation*}
  \delta(\{x_k\}, \{y_k\}) := \lim_{k \rightarrow \infty}
  \delta(x_k, y_k).
\end{equation*}
(We denote the distance function of the precompletion of a space using
the same symbol as for the space itself; which distance function is
meant will always be clear from the context.)  We claim that $\delta$
is well-defined by the above definition, as $\delta(x_k, y_k)$ is a
Cauchy sequence in $\R$, so the limit exists.  To see this, choose $K$
large enough that $k, l \geq K$ implies $\delta(x_k, x_l) <
\epsilon/2$ and $\delta(y_k, y_l) < \epsilon/2$.  Then
\begin{equation*}
  \delta(x_l, y_l) \leq \delta(x_l, x_k) + \delta(x_k, y_k) +
  \delta(y_k, y_l) < \delta(x_k, y_k) + \epsilon,
\end{equation*}
and similarly with $k$ and $l$ swapped, showing $|\delta(x_k, y_k) -
\delta(x_l, y_l)| < \epsilon$.

It is immediate from the definition that $\delta$ defines a
pseudometric on $\overline{X}^{\textnormal{pre}}$ (i.e., $\delta$
satisfies all properties of a metric except that two distinct points
may have $\delta$-distance zero from one another).  Therefore, as with
any pseudometric space, we can define a metric space by declaring all
points with distance zero from one another to be equal.  In this case,
the resulting space $\overline{X}$ is called the
\emph{completion} \label{p:compl} of $X$.  In symbols, its definition
is
\begin{equation*}
  \overline{X} := \overline{X}^{\textnormal{pre}} / {\sim},
\end{equation*}
where $\sim$ is the equivalence relation defined by
\begin{equation}\label{eq:97}
  \{x_k\} \sim \{y_k\} \Longleftrightarrow \delta(\{x_k\}, \{y_k\}) =
  0.
\end{equation}

Of course, we wouldn't call it the completion if we didn't have good
reason to.  The next theorem proves this and shows two other
important properties of the completion $\overline{X}$ of $X$.

Before we state the theorem, we simply remark that if $\{x_k\}$ is a
Cauchy sequence in $X$ and $\{x_{k_l}\}$ is a subsequence, then
clearly $\{x_{k_l}\} \sim \{x_k\}$.  Thus, given an element of the
precompletion of $X$, we can always pass to a subsequence and still be
talking about the same element of the completion.

\begin{thm}\label{thm:29}
  The completion $\overline{X}$ of $X$ has the following properties:
  \begin{enumerate}
  \item $\overline{X}$ is a complete metric space.
  \item The canonical embedding of $X$ into $\overline{X}$ mapping a
    point $x$ to the constant sequence $\{x\}$ is an isometry, and the
    image of $X$ is a dense subspace of $\overline{X}$.
  \item \label{item:2} Any uniformly continuous function $f : X
    \rightarrow Y$, where $Y$ is a complete metric space, has a unique
    extension to a uniformly continuous function on $\overline{X}$.
  \end{enumerate}
\end{thm}
\begin{proof}
  To prove (1), let any Cauchy sequence $[\{x^l_k\}]$ in
  $\overline{X}$ be given.  The index $l$ is meant to be the index in
  $\overline{X}$, while $k$ is meant to be the index in $X$.  Thus,
  for each fixed $l$, $\{x^l_k\}$ is a Cauchy sequence in $X$ with
  index $k$.  The square brackets in the above represent that each
  element of $\overline{X}$ is an \emph{equivalence class} of Cauchy
  sequences.

  We claim that $[\{x^l_k\}]$ converges to the equivalence class of
  the diagonal sequence $\{x^k_k\}$; that is, for any $\epsilon > 0$,
  we can find representatives $\{x^l_k\} \in [\{x^l_k\}]$ and $M \in
  \N$ such that $l \geq M$ implies
  \begin{equation*}
    \delta(\{x^l_k\}, \{x^k_k\}) = \lim_{k \rightarrow \infty}
    \delta(x^l_k, x^k_k) < \epsilon.
  \end{equation*}
  To put it one last way, given $\epsilon > 0$, we must find an $M$
  such that for each $l \geq M$, there exists $N \in \N$ such that for
  $k \geq N$,
  \begin{equation}\label{eq:124}
    \delta(x^l_k, x^k_k) < \epsilon.
  \end{equation}
  
  Choose any representatives $\{x^l_k\} \in [\{x^l_k\}]$; by passing
  to subsequences if necessary, we can assume that for each $l \in
  \N$, $k,n \geq K$ implies that
  \begin{equation}\label{eq:99}
    \delta(x^l_k, x^l_n) < 2^{-K}
  \end{equation}
  Let's fix a particular $K$ that is large enough that $2^{-K} \leq
  \epsilon / 3$.

  Now, since $\{x^l_k\}$ is a Cauchy sequence, we can find $M \geq
  K$ such that if $l, m \geq M$, then
  \begin{equation}\label{eq:98}
    \delta(\{x^l_r\}, \{x^m_r\}) = \lim_{r \rightarrow \infty}
    \delta(x^l_r, x^m_r) < \epsilon / 6.
  \end{equation}
  Now simply set $N := M$, and let $k, l \geq M = N$ be given.  By
  \eqref{eq:98}, we can find $R \in \N$ such that $r \geq R$ implies
  \begin{equation}\label{eq:100}
    \delta(x^l_r, x^k_r) < \epsilon / 3.
  \end{equation}
  Thus, by the triangle inequality, if $k, l \geq M = N$ and $r \geq
  R$,
  \begin{equation*}
    \delta(x^l_k, x^k_k) \leq \delta(x^l_k, x^l_r) + \delta(x^l_r,
    x^k_r) + \delta(x^k_r, x^k_k) < \epsilon,
  \end{equation*}
  where we have used \eqref{eq:100} to estimate the middle term and
  \eqref{eq:99} to estimate the two other terms.  As this proves
  \eqref{eq:124}, statement (1) is shown.

  Statement (2) is not difficult, since
  \begin{equation*}
    \delta(\{x\}, \{y\}) = \lim_{k \rightarrow \infty} \delta(x, y) =
    \delta(x, y),
  \end{equation*}
  and to find a constant sequence arbitrarily close to any Cauchy
  sequence, we can simply take an appropriate element of said
  sequence.

  As for statement (3), this follows directly from (2) and the fact
  that a uniformly continuous function $f$ on a dense subset $A$ of a
  metric space $X$ always has a unique uniformly continuous extension
  to the entire space, provided the target space $Y$ is complete.
  This fact is readily verified by noting that a uniformly continuous
  function maps Cauchy sequences to Cauchy sequences.  The extension
  of the function to a point $x \in X \setminus A$ is defined as
  follows.  Take any sequence $x_k \rightarrow x$.  This is then a
  Cauchy sequence in $X$, so $\{f(x_k)\}$ is a Cauchy sequence in $Y$.
  But $Y$ is complete, so we can define $f(x) := \lim f(x_k)$.  It is
  straightforward to check that the extension thus defined is
  uniformly continuous.
\end{proof}

Recall that a \emph{path metric space} is a metric space for which the
distance between any two points coincides with the infimum of the
lengths of curves joining the two points.  Given this definition, we
expect that there be a description of the completion of a path metric
space that uses curves instead of Cauchy sequences, and indeed this is
so.  Before we give it, though, let's give the definition of a path
metric space in more detail.

Let $\alpha : [0,1] \rightarrow X$ be a continuous path, and let $0 =
t_1 < t_2 < \cdots < t_n = 1$ be any finite partition of the interval
$[0,1]$.  Then the \emph{length} of the polygonal path given by
$\{\alpha(t_1), \dots, \alpha(t_n)\}$ is defined to be
\begin{equation*}
  L_{t_1, \dots, t_n}(\alpha) := \sum_{k=1}^{n-1} \delta(\alpha(t_k), \alpha(t_{k+1})).
\end{equation*}
Finally, we define the \emph{length} of $\alpha$ to be
\begin{equation*}
  L(\alpha) := \sup \{ L_{t_1, \dots, t_n}(\alpha) \mid (t_1, \dots,
  t_n)\ \textnormal{is a partition of}\ [0,1] \}.
\end{equation*}
We take the supremum since as we add vertices to a polygonal path,
i.e., improve the approximation of $\alpha$, the triangle inequality
implies the lengths of the polygonal paths are nondecreasing.  Thus,
this definition will match up with, say the length of a differentiable
path in a Riemannian manifold.

We call a path $\alpha$ with $L(\alpha) < \infty$ \emph{rectifiable}
and say that $(X, \delta)$ is a path metric space if for any $x, y \in
X$,
\begin{equation*}
  \delta(x, y) = \inf \{ L(\alpha) \mid \alpha\ \textnormal{is a
    rectifiable curve joining}\ x\ \textnormal{and}\ y \}.
\end{equation*}
If the domain of $\alpha$ is an open interval, e.g., $(0,1)$, then we
define the length of $\alpha$ to be
\begin{equation*}
  L(\alpha) := \lim_{\epsilon \rightarrow 0}
  L(\alpha|_{[\epsilon, 1-\epsilon]}),
\end{equation*}
and similarly if the domain is a half-open interval.  We again call
such a curve rectifiable if its length is finite.

We will also call a rectifiable curve a \emph{finite-length path} or
simply a \emph{finite path}.

Given these definitions, we can formulate an alternate definition of
the completion of a path metric space.  Just as we can imagine a
Cauchy sequence to be ``open-ended'' but convergent in some larger
space containing $X$, we can imagine a path to be open on one end and
view the path as representing its endpoint, which may or may not exist
within $X$.

\begin{thm}\label{thm:30}
  Let $(X, \delta)$ be a path metric space.  Then the following
  description of the completion of $(X, \delta)$ is equivalent to the
  definition given above.

  Define the precompletion $\overline{X}^{\textnormal{pre}}$ of $X$ to
  be the set of rectifiable curves
  \begin{equation*}
    \alpha : (0,1] \rightarrow X.
  \end{equation*}
  It carries the pseudometric
  \begin{equation}\label{eq:101}
    \delta(\alpha_0, \alpha_1) := \lim_{t \to 0} \delta(\alpha_0(t), \alpha_1(t)).
  \end{equation}

  Then the completion of $(X, \delta)$ is the metric space associated
  to $\overline{X}^{\textnormal{pre}}$.  That is,
  \begin{equation*}
    \overline{X} := \overline{X}^{\textnormal{pre}} / {\sim},
  \end{equation*}
  where $\alpha_0 \sim \alpha_1 \Longleftrightarrow \delta(\alpha_0,
  \alpha_1) = 0$.
\end{thm}
\begin{proof}
  First, let's show that the limit in \eqref{eq:101} exists---this
  will follow if, for every sequence $t_k \rightarrow 0$,
  $\delta(\alpha_0(t_k), \alpha_1(t_k))$ is a Cauchy sequence.  But
  given $\epsilon > 0$, by rectifiability of the two curves, we can
  find $K \in \N$ such that $k \geq K$ implies that
  \begin{equation*}
    L(\alpha_i|_{(0,t_k]}) < \epsilon / 2
  \end{equation*}
  for $i = 0,1$.  Thus, if $k, l \geq K$, we have
  \begin{align*}
    \delta(\alpha_0(t_k), \alpha_1(t_k)) &\leq \delta(\alpha_0(t_k),
    \alpha_0(t_l)) + \delta(\alpha_0(t_l), \alpha_1(t_l)) +
    \delta(\alpha_1(t_l), \alpha_1(t_k)) \\
    &< \delta(\alpha_0(t_l), \alpha_1(t_l)) + \epsilon.
  \end{align*}
  Doing the same computation with $k$ and $l$ swapped proves that
  $\delta(\alpha_0(t_k), \alpha_1(t_k))$ is a Cauchy sequence.

  Now, to show equivalence of the two definitions, we demonstrate an
  isometry from the completion as defined using sequences to the
  completion as defined using paths.  For completeness (excusing the
  pun), we also write down the inverse mapping of this isometry.

  So let a Cauchy sequence $\{x_k\}$ be given.  Choose a subsequence
  $\{x_{k_l}\}$ (which, as previously noted, is equivalent to
  $\{x_k\}$) such that
  \begin{equation}\label{eq:102}
    \sum_{l=1}^\infty \delta(x_{k_l}, x_{k_{l+1}}) < \infty.
  \end{equation}
  Since $X$ is a path metric space, we can choose paths $\alpha_l$
  joining $x_{k_l}$ and $x_{k_{l+1}}$ such that $L(\alpha_l) \leq 2
  \delta(x_{k_l}, x_{k_{l+1}})$.  Then the concatenated path
  \begin{equation*}
    \alpha_{\{x_{k_l}\}} := \alpha_1 * \alpha_2 * \alpha_3 * \cdots
  \end{equation*}
  is rectifiable.

  To get a Cauchy sequence from a curve $\alpha : (0,1] \rightarrow
  X$, simply take any monotonically decreasing sequence $t_k \searrow
  0$ in $(0,1]$ and define
  \begin{equation*}
    x^\alpha_k := \alpha(t_k).
  \end{equation*}
  Then it is easy to see that finite length of $\alpha$ implies that
  $x^\alpha_k$ is a Cauchy sequence, for given $\epsilon > 0$, we can
  find $K \in \N$ such that $k \geq K$ implies $L(\alpha|_{(0, t_k]})
  < \epsilon$.  Thus $l \geq k \geq K$ implies
  \begin{equation*}
    \delta(x^\alpha_k, x^\alpha_l) \leq L(\alpha|_{(0,t_k]}) < \epsilon.
  \end{equation*}

  To see that these two mappings are well-defined on the completion,
  as defined via sequences on the one side and paths on the other, and
  to show that they are isometries, we need to show the following:
  \begin{enumerate}
  \item If $\{x_k\}$ and $\{y_k\}$ both satisfy \eqref{eq:102} (with
    $x_k$ and $y_k$, respectively, in place of $x_{k_l}$), then
    $\delta(\alpha_{\{x_k\}}, \alpha_{\{y_k\}}) = \delta(\{x_k\},
    \{y_k\})$.
  \item If $\alpha$ and $\beta$ are equivalent finite paths and $t_k
    \searrow 0$, then $\delta(\{x^\alpha_k\}, \{x^\beta_k\}) = 0$.
    Furthermore, different choices of sequences $t_k \searrow 0$ give
    rise to equivalent Cauchy sequences.
  \item If $\alpha$ is a finite path and $t_k \searrow 0$, then
    $\alpha_{\{x^\alpha_k\}} \sim \alpha$.
  \end{enumerate}
  From (1), we see that $\{x_k\} \mapsto \alpha_{\{x_k\}}$ is
  well-defined and an isometry from one completion to the other.  From
  (2), it follows that $\alpha \mapsto \{x^\alpha_k\}$ is well-defined
  on the completions, and (3) implies that these two mappings are
  inverses of one another.
  
  To prove (1), reparametrize $\alpha_{\{x_k\}}$ and $\alpha_{\{y_k\}}$
  so that
  \begin{equation}\label{eq:103}
    \alpha_{\{x_k\}} ( 1/k ) = x_k \quad
    \textnormal{and} \quad \alpha_{\{y_k\}} ( 1/k ) = y_k.
  \end{equation}
  Let $\epsilon > 0$ be given, and choose $L \in \N$ such that
  \begin{equation*}
    \sum_{l=L}^\infty \delta(x_{k_l}, x_{k_{l+1}}) < \epsilon / 4 \quad
    \textnormal{and} \quad \sum_{l=L}^\infty \delta(y_{k_l}, y_{k_{l+1}}) < \epsilon / 4,
  \end{equation*}
  so that by the construction of $\alpha_{\{x_k\}}$ and
  $\alpha_{\{y_k\}}$,
  \begin{equation}\label{eq:105}
    L(\alpha_{\{x_k\}}|_{(0,1/L]}) < \epsilon / 2 \quad
    \textnormal{and} \quad
    L(\alpha_{\{y_k\}}|_{(0,1/L]}) < \epsilon / 2.
  \end{equation}
  Then for $t \leq 1 / L$,
  \begin{align*}
    \delta(\alpha_{\{x_k\}}(t), \alpha_{\{y_k\}}(t)) &\leq
    \delta(\alpha_{\{x_k\}}(t), \alpha_{\{x_k\}}(1/L)) +
    \delta(\alpha_{\{x_k\}}(1/L), \alpha_{\{y_k\}}(1/L)) \\
    &\quad + \delta(\alpha_{\{y_k\}}(1/L), \alpha_{\{y_k\}}(t)) \\
    &\leq \delta(\{x_L\}, \{y_L\}) + L(\alpha_{\{x_k\}}|_{[t,1/L]}) +
    L(\alpha_{\{y_k\}}|_{[t,1/L]}) \\
    &\leq \delta(\{x_L\}, \{y_L\}) + \epsilon,
  \end{align*}
  where we have used \eqref{eq:103} in the second inequality and
  \eqref{eq:105} in the third.  Similarly, one can prove that for $t
  \leq 1 / L$,
  \begin{equation*}
    \delta(\{x_L\}, \{y_L\}) \leq \delta(\alpha_{\{x_k\}}(t),
    \alpha_{\{y_k\}}(t)) + \epsilon.
  \end{equation*}
  From the two above inequalities, it is easy to see that
  \begin{equation*}
    \lim_{t \to 0} \delta(\alpha_{\{x_k\}}(t), \alpha_{\{y_k\}}(t)) =
    \lim_{k \rightarrow \infty} \delta(x_k, y_k),
  \end{equation*}
  as was to be proved.

  The proofs of (2) and (3) are very similar, yet simpler, and so we
  omit them.  Besides, we have already proved the statement of the
  theorem, so these are just ``bonus'' statements about the inverse to
  the isometry $\{x_k\} \mapsto \alpha_{\{x_k\}}$.
\end{proof}

We are now equipped with all of the metric space tools we need to
study the completion of the manifold of metrics.

\section{Fréchet manifolds}\label{sec:frechet-manifolds}

The manifold of smooth metrics is itself a Fréchet manifold, and so
these will play an extremely important role in this work.  However, we
will not need to go into depth on Fréchet manifolds.  This is
because the manifold of metrics is an extremely simple type of
Fréchet manifold, namely an open subset of a Fréchet space.

An excellent source on Fréchet manifolds and the implicit function
theorem in the category of Fréchet spaces is
\cite{hamilton82:_inver_funct_theor_of_nash_and_moser}, and this is
our main reference for the first two subsections.  For more in-depth
and recent results on this and related categories, see
\cite{omori97:_infin_dimen_lie_group}, which focuses mainly on
Fréchet Lie groups.

After introducing Fréchet spaces and Fréchet manifolds, we will
discuss a particular class of Fréchet manifolds, namely manifolds of
smooth mappings.  The main result, which will allow us to define the
manifold of metrics, is that if $N$ is a finite-dimensional manifold
and $F$ is a finite-dimensional fiber bundle over $N$, then the set of
$C^\infty$ sections of $F$ carries the structure of a smooth Fréchet
manifold.  If $F$ is a vector bundle, then the set of $C^\infty$
sections has a linear structure, so it even forms a Fréchet space.

There is another category incorporating manifolds of smooth mappings,
the so-called convenient setting
\cite{kriegl97:_conven_settin_of_global_analy}.  This setting is
highly developed and allows one to deal with more general spaces than
Fréchet spaces.  We chose to use the Fréchet category because we
need only basic facts, and Fréchet manifolds are the most familiar
and simplest to introduce.

So, without further delay, we get into the definitions.

\subsection{Fréchet spaces}\label{sec:frechet-spaces}

\begin{dfn}\label{dfn:17}
  Let $E$ be a vector space over a field $\mathbb{K}$.  A
  \emph{seminorm} on $E$ is a function $\| \cdot \| : E \rightarrow
  \mathbb{K}$ with the following properties for all $v, w \in E$ and
  $\lambda \in \mathbb{K}$:
  \begin{enumerate}
  \item $\| v \| \geq 0$,
  \item $\| v + w \| \leq \| v \| + \| w \|$ and
  \item $\| \lambda v \| = |\lambda| \| v \|$.
  \end{enumerate}
\end{dfn}

Given a collection of seminorms $\{ \| \cdot \|_i \mid i \in I \}$ on
$E$, we can define a topology on $E$ by declaring that a sequence or
net $\{v_k\}$ converges to $v$ if and only if $\| v - v_k \|_i
\rightarrow 0$ for all $i \in I$.  A \emph{locally convex topological
  vector space} (or LCTVS) is a vector space together with a topology
defined in this way.  It happens that the topology of an LCTVS is
metrizable if and only if it is defined by a countable collection of
seminorms, and it is Hausdorff if and only if $v = 0$ whenever $\| v
\|_i = 0$ for all $i \in I$.  In a metrizable LCTVS, it suffices to
use sequences instead of nets when describing the topology via
convergence.

In a metrizable LCTVS, we call a sequence $\{v_k\}$ a \emph{Cauchy
  sequence} if given any $i \in \N$ and $\epsilon > 0$, we can find
$N(i, \epsilon) \in \N$ such that $\| v_k - v_l \|_i < \epsilon$ for
all $k, l \geq N(i, \epsilon)$.  We call the space \emph{complete} if
every Cauchy sequence converges.

With these preparations, we can make the following definition.

\begin{dfn}\label{dfn:18}
  A \emph{Fréchet space} is an LCTVS that is Hausdorff, metrizable
  and complete.
\end{dfn}

For example, every Banach or Hilbert space is a Fréchet space, with
topology given by a single norm.  For a more interesting example,
consider the interval $[0,1] \subset \R$ and the space $C^\infty[0,1]$
of smooth functions on this interval.  If we give this space the
topology defined by the $C^k$ norms,
\begin{equation*}
  \| f \|_k = \sum_{l=1}^k \sup_{x \in [0,1]}
  \left|
    \frac{d^l}{d x^l} f(x)
  \right|,
\end{equation*}
then $C^\infty[0,1]$ becomes a Fréchet space.  The Hausdorff
property and metrizability are clear, and completeness follows from
the fact that $C^k[0,1]$ is a Banach space with the $\| \cdot \|_k$
norm.  Therefore, if a sequence is Cauchy in each $\| \cdot \|_k$
norm, it converges to a function that is $C^k$ for each $k \in \N$,
i.e., a smooth function.

Note that we could have also used the $H^s$ norms to define
$C^\infty[0,1]$.  The proof that this defines a Fréchet space
topology is the same, but we have to make the extra step of using the
Sobolev embedding theorem to show that a Cauchy sequence converges to
a smooth limit function.  The advantage of using the $H^s$ norms is
that they come from scalar products, which yields some extra structure
to work with.  However, the topology on $C^\infty[0,1]$ is the same as
when we use the $C^k$ norms, which we can again see using the Sobolev
embedding theorem.

As suggested by the term locally convex topological vector space above,
a Fréchet space is a topological vector space, meaning that vector
addition and scalar multiplication are continuous maps.

Fréchet spaces have some fundamental differences from Banach spaces.
For example, the dual of a Fréchet space $E$ is not always a Fréchet
space.  In fact, the dual is a Fréchet space if and only $E$ is a
Banach space!  This implies that the space $L(E, F)$ of linear maps
between two Fréchet spaces $E$ and $F$ is a Fréchet space if and only
if $F$ is a Banach space.  Additionally, naive generalizations of the
Banach space implicit function theorem to Fréchet spaces
fail---instead, one must work in the category of so-called \emph{tame
  Fréchet spaces}, which require additional estimates on maps between
them that are not present in the Banach case, to get a satisfactory
implicit function theorem.  However, these matters are not important
to our concerns.

Despite the difficulties in working with Fréchet spaces, many
results from the theory of Banach spaces and Banach manifolds carry
over.  For example, the Hahn-Banach theorem holds, as does the open
mapping theorem.

Calculus in Fréchet spaces works in almost exactly the same manner as
it does in Banach spaces, if we define the derivative in the following
way.  Let $E$ and $F$ be Fréchet spaces, let $U \subseteq E$ be open,
and let $f : U \subseteq E \rightarrow F$ be a continuous map.  We
define the differential of $f$ at the point $x \in U$ in the direction
$v \in E$ to be
\begin{equation*}
  D f(x)v := \lim_{t \to 0}\frac{f(x + t v) - f(x)}{t}.
\end{equation*}
We define $f$ to be differentiable at $x$ in the direction $v$ if the
limit exists.  We define $f$ to be $C^1$ (or continuously
differentiable) if the limit exists for all $x \in U$ and $v \in E$,
and the map
\begin{equation*}
  D f : U \times E \rightarrow F
\end{equation*}
is continuous in both its arguments.  Note that $Df$ is a map from the
product $U \times E$ to $F$.  We \emph{do not} consider it as a map $U
\rightarrow L(E,F)$, because as we mentioned above, $L(E,F)$ is not
necessarily a Fréchet space---even though $D f(x)$ is indeed a linear
map from $E$ to $F$ for each $x \in U$.

To define the second derivative, we take the \emph{partial} derivative
of the map $D f$ in the first component, i.e., we take the derivative
as $D f$ varies only over $U$.  This is because $D f$ is linear in the
second component, and hence this partial derivative just gives $D f$
again.  It is also done to match up with the usual definition of the
derivative in Banach spaces.  Thus, the second derivative is a map
\begin{equation*}
  D^2 f : U \times E \times E \rightarrow F.
\end{equation*}

We can iterate the definitions above to define $C^k$ and $C^\infty$
maps between Fréchet spaces.  The chain rule holds for the
differential as thus defined.  We can also define integrals over
curves in the usual way, and if we do so then the fundamental theorem
of calculus holds.

With all of these results at hand, it is clear that calculus in
Fréchet spaces is formally very similar to that in Banach spaces.
Thus, we will not go into any more detail at this point---we again
refer the interested reader to
\cite{hamilton82:_inver_funct_theor_of_nash_and_moser}.  All others
may assume that the intuition and computation rules from calculus in
Banach spaces work fine here as well.

\subsection{Fréchet manifolds}\label{sec:frechet-manifolds-1}

Again, our reference for this subsection is
\cite{hamilton82:_inver_funct_theor_of_nash_and_moser}.

Just as the usual rules for calculus generalize to Fréchet spaces, so
does the definition of a manifold.  Thus, a \emph{Fréchet manifold}
modeled on a Fréchet space $E$ is a Hausdorff topological space $M$
with an atlas of coordinates $\{ (U_i, \phi_i) \mid i \in I \}$, where
each $U_i \subseteq M$ is open and each $\phi_i : U_i \subseteq M
\rightarrow E$ is a homeomorphism onto its image.  Furthermore, if
$U_i \cap U_j \neq 0$, we require that the transition map
\begin{equation*}
  \phi_j |_{U_i \cap U_j} \circ \phi_j^{-1} |_{\phi_j(U_i \cap U_j)} :
  \phi_j(U_i \cap U_j) \subseteq E \rightarrow E
\end{equation*}
is a smooth mapping of Fréchet spaces.

Tangent spaces/bundles, smooth/differentiable mappings, vector
bundles, fiber bundles, and so on are defined in the category of
Fréchet manifolds exactly analogously to the case of Banach
manifolds.

Fréchet Lie groups are Fréchet manifolds that are also groups and
on which the operations of multiplication and taking the inverse are
smooth.  One example of a Fréchet Lie group is the diffeomorphism
group of a compact manifold.  For more facts on these fascinating and
difficult objects, which are so important in global analysis, see
\cite{omori70:_group_of_diffeom_of_compac_manif} and
\cite{omori97:_infin_dimen_lie_group}.

\subsection{Manifolds of mappings}\label{sec:manifolds-mappings}

The fundamental object in the field of global analysis is the set of
sections of a smooth fiber bundle $F$ with $m$-dimensional fibers over
a smooth, $n$-dimensional manifold $M$.  Typically, one is interested
in restricting to sections with a certain regularity, say $C^k$ for $0
\leq k \leq \infty$ or $H^s$ for $s \geq 0$.  $C^k$ regularity is, of
course, well understood, and it is the goal of this section to outline
the notion of $H^s$ regularity.  Additionally, as analysts and
geometers, we prefer to work with smooth manifolds, and so we will
sketch the useful fact that the by restricting to certain types of
sections, we get a Hilbert/Banach/Fréchet manifold.

Let's get down to defining manifolds of mappings.  The facts presented
here are taken from the texts \cite{saunders89:_geomet_of_jet_bundl},
\cite{palais68:_found_of_global_non_linear_analy} and
\cite[Chap.~4]{palais65:_semin_atiyah_singer_index_theor}.  For a very
concise but readable outline, see \cite[\S
3]{ebin70:_manif_of_rieman_metric}.  It will simplify the presentation
somewhat, and is in fact sufficient for our purposes, to assume that
the base manifold $M$ is closed and oriented.

Manifolds of sections are constructed using the notion of a \emph{jet
  bundle}, which is essentially a bundle that contains information
about the Taylor expansions of sections of $F$.  The precise
definition is as follows.

Suppose we are given two local, $k$-times differentiable sections
$\varphi$ and $\psi$ of $F$.  Suppose that $\varphi$ and $\psi$ are
both defined on an open neighborhood of $p \in M$.  We say that
$\varphi$ and $\psi$ are \emph{$k$-equivalent at $p$} if $\varphi(p) =
\psi(p)$ and the following holds.  Let $(x^i, u^\alpha)$ be
coordinates on $F$ around $p$ such that $(x^i)$ are coordinates on the
base manifold $M$ and $(u^\alpha)$ are coordinates in the fiber
directions---i.e., $(x^i, u^\alpha)$ are the coordinates of a local
trivialization.  We require that for all multi-indices $I$, taking
values in $\{1, \dots, n\}$, with $1 \leq |I| \leq k$ and all $1 \leq
\alpha \leq m$ (recall $m$ is the dimension of the fibers):
\begin{equation}\label{eq:122}
  \left. \frac{\partial^{|I|} \varphi^\alpha}{\partial x^I} \right|_p
  = \left. \frac{\partial^{|I|} \psi^\alpha}{\partial x^I} \right|_p.
\end{equation}

Thus, two local sections are $k$-equivalent at $p$ if and only if
their values at $p$ are equal, as are their first $k$ derivatives at
$p$ in some local coordinate system around $p$.  Note that while the
value of the derivatives depends on the local coordinates, equality of
the derivatives as in \eqref{eq:122} does not (see \cite[Lemma
6.2.1]{saunders89:_geomet_of_jet_bundl}).

The equivalence class containing the local section $\varphi$ is
denoted by $j^k_p \varphi$ and is called the \emph{$k$-jet of
  $\varphi$ at $p$}.  The equivalence class of a local section
$\varphi$ thus consists of all local sections having Taylor expansion
up to order $k$---in local coordinates at $p$---equal to that of
$\varphi$.

The set of all $k$-jets of local sections of $F$, denoted
\begin{equation*}
  J^k F := \{ j^k_p \varphi \mid p \in M,\ \varphi\ \textnormal{is a
    local section of}\ F\ \textnormal{around}\ p \},
\end{equation*}
is called the \emph{$k$-th jet bundle}, as it has a natural structure
of a smooth, finite-dimensional fiber bundle over both $M$ and $F$.
To see this, we first write down the coordinate atlas that makes it
into a manifold.  As above, let $(x^i, u^\alpha)$ be coordinates on an
open set $U \subseteq F$, with $(x^i)$ coordinates on the base and
$(u^\alpha)$ coordinates on the fibers.  Let $U^k$ be the subset of
$J^k F$ given by
\begin{equation*}
  U^k := \{ j^k_p \varphi \mid \varphi(p) \in U \}.
\end{equation*}
Then we get coordinates $(x^i, u^\alpha, u^\alpha_I)$ on $U^k$, where
$I$ runs through all unordered multi-indices taking values in $\{1,
\dots, n\}$ with $1 \leq |I| \leq k$, and
\begin{equation}\label{eq:123}
  \begin{aligned}
    x^i(j^k_p \varphi) &= x^i(p), \\
    u^\alpha(j^k_p \varphi) &= u^\alpha(\varphi(p)), \\
    u^\alpha_I(j^k_p \varphi) &= \left. \frac{\partial^{|I|}
        \varphi^\alpha}{\partial x^I} \right|_p.
  \end{aligned}
\end{equation}
(The reason we require $I$ to be unordered is the symmetry of the
derivatives in local coordinates, i.e., because differentiations in
different coordinate directions commute with one another.)  We will
not show that this does indeed define a smooth atlas on $J^k F$, but
refer the interested reader to \cite{saunders89:_geomet_of_jet_bundl}.
We do note, however, that since there are only finitely many
multi-indices of order not greater than $k$ taking values in $\{1,
\dots, n\}$, there are only finitely many coordinates $u^\alpha_I$,
and hence $J^k E$ is finite-dimensional.

The bundle structures $J^k F \rightarrow M$ and $J^k F \rightarrow F$
are given by the so-called \emph{source} and \emph{target
  projections}:
\begin{equation*}
  \pi_k : J^k F \rightarrow M, \quad j^k_p \varphi \mapsto p,
\end{equation*}
and
\begin{equation*}
\pi_{k,0} : J^k F \rightarrow F, \quad j^k_p \varphi \mapsto \varphi(p),
\end{equation*}
respectively.  We can also view $J^l F$ as a bundle over $J^k F$ for
any $1 \leq k \leq l$; the bundle structure is given by the
\emph{$k$-jet projection}:
\begin{align*}
  \pi_{l,k} : J^l F &\rightarrow J^k F \\
  j^l_p \varphi &\mapsto j^k_p \varphi.
\end{align*}

There is a natural mapping, denoted $j^k$, sending local $C^l$
sections of $F \rightarrow M$ (for $l \geq k$) to local $C^{l-k}$
sections of $J^k F \rightarrow M$.  If $\varphi$ is a local section of
$F$, then this map is defined by
\begin{equation*}
  j^k \varphi(p) := j^k_p \varphi
\end{equation*}
The section $j^k \varphi$ is sometimes called the \emph{$k$-th
  prolongation of $\varphi$}, and $j^k$ is sometimes called the
\emph{$k$-jet extension map}.  If we only consider global sections of
$F$, then $j^k$ defines a map from $C^\infty(F)$ to $C^\infty(J^k F)$,
where for a fiber bundle $E \rightarrow M$, $C^\infty(E)$ denotes the
space of smooth sections of $E$.

\begin{rmk}\label{rmk:10}
  The notation $C^\infty(E)$ for the set of smooth sections of the
  \emph{fiber bundle} $E \rightarrow M$ should not be confused with
  the oft-used identical notation for the set of smooth functions on
  the \emph{manifold} $E$.  In this thesis, whenever we consider a
  bundle structure on a space $E$, by $C^k(E)$, $C^\infty(E)$,
  $H^s(E)$ (the last one we have yet to define), and so on, we will
  always mean the appropriate space of sections of the bundle.

  This point will hardly arise outside this chapter, though, so we
  hope this admittedly suboptimal notation will cause no large
  problems.
\end{rmk}

At this point, let us restrict to the case where $F$ is a
\emph{vector} bundle over $M$, as it will simplify the exposition
somewhat and will still be sufficient for our purposes.  With this
assumption, $J^k F$ has the structure of a vector bundle over $M$, not
just a fiber bundle.  This can be seen, heuristically, from the fact
that the values of any section $\varphi$ at a point $p$ belong to the
vector space $F_p$, and the $j$-th total differential (with respect to
$(x^i)$, as in the $u^\alpha_I$-coordinates of \eqref{eq:123}) of the
section $\varphi$ at $p$ can be seen in local coordinates as a
$j$-linear map from $\R^n$ to $F_p$ (recall $n = \dim M$).  This is an
extremely sketchy ``proof'' and not at all rigorous, so we refer the
reader to \cite[pp.~5--6]{palais68:_found_of_global_non_linear_analy} for
details.

Since $F$ and $J^k F$ are both vector bundles, $C^\infty(F)$ and
$C^\infty(J^k F)$ are both vector spaces.  It is then easy to see that
the $k$-jet extension map $j^k : C^\infty(F) \rightarrow C^\infty(J^k F)$ is a linear map.

A Riemannian metric $\gamma$ on $J^k F$ is given by a smooth choice of
positive-definite scalar product $\gamma(p)$ on $J^k_p F$, the fiber
of $\pi_k : J^k F \rightarrow M$ at $p$, for each $p \in M$.  Given a
Riemannian metric $\gamma$ on $J^k F$ and a smooth volume form $\mu$
on $M$, we get a scalar product $(\cdot, \cdot)_\gamma$ on
$C^\infty(J^k F)$ via
\begin{equation*}
  (\varphi, \psi)_\gamma = \integral{M}{}{\gamma(p)(\varphi(p),
    \psi(p))}{\mu(p)}.
\end{equation*}
We can pull this scalar product back along the $k$-jet extension map
$j^k$ to get a scalar product on $C^\infty(F)$.  We denote by $H^k(F)$
the completion of $C^\infty(F)$ with respect to this scalar product.
The space $H^k(F)$ is a Hilbert space over the reals, and its norm
depends on our choices of $\gamma$ and $\mu$.  However, the topology
of $H^k(F)$ does not, as \cite[\S
IX.2]{palais65:_semin_atiyah_singer_index_theor} shows.  Therefore, we
are justified in omitting $\gamma$ and $\mu$ from our notation and
calling $H^k(F)$ the \emph{space of $H^k$ sections of $F$}.

\begin{rmk}\label{rmk:18}
  The scalar product $(\cdot, \cdot)_\gamma$ is essentially an $L^2$
  scalar product on sections $\varphi$ of the $k$-th jet bundle.
  Since these sections contain the all derivatives of $\varphi$ of
  order $k$ and lower, it can be seen that the definitions above match
  up with the definitions of Sobolev spaces of functions on open sets
  of $\R^n$.  If we allow ourselves to speak imprecisely by mixing
  global and local notions, we can say that the completion of
  $C^\infty(F)$ with respect to the above-described scalar product
  contains all sections with $L^2$-integrable partial derivatives up
  to order $k$.
\end{rmk}

In a similar but simpler way, we can define a Banach space structure
on the space $C^k(F)$ of $C^k$ sections of $F$.  To do this, we again
choose a Riemannian structure $\gamma$ on $J^k F$, but this time
define a norm on $C^0(J^k F)$ by
\begin{equation*}
  \| \varphi \|_\gamma = \sup_{p \in M} \sqrt{\gamma(p)(\varphi(p), \varphi(p))}.
\end{equation*}
Since the $k$-jet extension map $j^k$ is a linear map defined on
$C^k(F)$, we pull the above norm back to $C^k(F)$ along $j^k$.  Then
$C^k(F)$ is a Banach space with respect to the pulled-back norm.

With these definitions, the Sobolev embedding theorem holds for spaces
of sections of vector bundles, just as it does for spaces of functions
over $\R^n$.  Thus if $s > n/2 + k$, there is a continuous linear
inclusion $H^s(F) \hookrightarrow C^k(F)$.  (See \cite[\S X.4,
Thm.~4]{palais65:_semin_atiyah_singer_index_theor}.)  A consequence is
the following statement.  Define $C^\infty(F)$ to be the Fréchet space
of smooth sections of $F$ with the topology given by the family of
$C^k(F)$-norms for $k \in \M$.  Then this topology on $C^\infty(F)$
coincides with the one given by the family of $H^s(F)$-norms for $s
\in \N$.  This latter view is the one we will take in this thesis,
since it allows us to work with the chain of \emph{Hilbert} manifolds
$H^0(F), H^1(F),\dots$, which have nicer properties than the
\emph{Banach} manifolds $C^0(F), C^1(F), \dots$.

To recap, for any vector bundle $F \rightarrow M$, we can consider the
set of all sections of $F$.  By taking sections with certain
properties, we can build Hilbert spaces
$H^s(F)$ of Sobolev sections of $F$, Banach spaces $C^k(F)$ of
$k$-times differentiable sections, and the Fréchet space
$C^\infty(F)$ of smooth sections.  The latter has the topology coming
either from the family of $C^k$ norms or the family of $H^s$ norms.

We have restricted the discussion to vector bundles for simplicity,
but we end this section by briefly remarking on the situation when $F$
is a fiber bundle.  In this case, we can build a Banach
\emph{manifold} (which will in general not be a linear space) $C^k(F)$
for $k=0,1,2,\dots$.  We can also define the sets $H^s(F)$ of $H^s$
sections of $F$ for $s=0,1,2,\dots$, but if we want $H^s(F)$ to be a
(Hilbert) manifold, then for technical reasons we have to restrict to
$s > n/2$, i.e., we have to require that $H^s(F) \subseteq
C^0(F)$. (See \cite[\S
11ff]{palais68:_found_of_global_non_linear_analy}.)  Using either the
$H^s$ or $C^k$ norms, we can give $C^\infty(F)$ a Fréchet manifold
structure.  The way that all of these results are proved is by locally
reducing the analysis of $H^s$ sections of a fiber bundle to the
analysis of $H^s$ sections of a related vector bundle.  We do not need
this directly, however, so instead of proving it we refer to \cite[\S
13]{palais68:_found_of_global_non_linear_analy} for the general case
and \cite[Ex.~4.1.2]{hamilton82:_inver_funct_theor_of_nash_and_moser}
for a nice, concise description of the $C^k$ and $C^\infty$ cases.

\begin{rmk}\label{rmk:20}
  It is also worth noting that if $N$ is another finite-dimensional
  manifold, then the set of mappings from $M$ to $N$ can be treated as
  in this subsection by viewing a map $M \rightarrow N$ as a section
  of the trivial bundle $M \times N$ over $M$.  If $N = M$, then we
  can construct the manifold $C^\infty(M,M)$ of smooth self-mappings
  of $M$.  It is not hard to see that the set $\D$ of smooth
  diffeomorphisms of $M$ is open in $C^\infty(M,M)$, and we therefore
  get a Fréchet manifold structure on $\D$.  As we mentioned above,
  $\D$ is even a Fréchet Lie group
  \cite{omori70:_group_of_diffeom_of_compac_manif},
  \cite{omori97:_infin_dimen_lie_group}.
\end{rmk}

\section{Geometric preliminaries}\label{sec:geom-prel}

At this point, we will go over some geometric notions and notation
that we will be using later in the thesis.  We'll first look at the
endomorphism bundle of a finite-dimensional manifold and the
eigenvalues of its sections.  Then we'll discuss a few concepts from
measure theory, and finish with the description of two special
manifolds of mappings that will play a role in what is to come.

\begin{cvt}\label{cvt:6}
  For the remainder of this thesis, we work over a fixed,
  finite-dimensional, oriented, closed base manifold $M$, and set $n
  := \dim M$.
\end{cvt}

\subsection{The endomorphism bundle of $M$}\label{sec:endom-bundle-m}

The endomorphism bundle $\textnormal{End}(M)$ is the bundle of
$(1,1)$-tensors on $M$. \label{p:end-m} A $(1,1)$-tensor at $p \in M$
is an element of $T_p M \otimes T^*_p M$, and so it can be identified
with an endomorphism of $T_p M$.  A smooth section of
$\textnormal{End}(M)$ is therefore a smooth vector bundle map of $T M$
into itself.  Furthermore, $(1,1)$-tensors and sections of
$\textnormal{End}(M)$ have a well-defined multiplication, which is
simply the multiplication of matrices (in local coordinates) or the
composition of linear transformations (invariantly described).  As a
$(1,1)$-tensor $H$ is a linear transformation, any property of
matrices that is invariant under a change of basis will be
well-defined (i.e., coordinate-independent) for an endomorphism of
$T_p M$.  Especially important for us is that this includes the
determinant, the trace, and the eigenvalues of $H$.

This also implies that if we are given a section $H$ of
$\textnormal{End}(M)$, then the determinant, trace, and eigenvalues of
$H$ are well-defined functions over $M$.  Furthermore, if $H$ is
measurable/continuous/smooth, then the determinant and trace will be
so as well, since they are smooth functions from the space of $n
\times n$ matrices into $\R$.

The regularity properties of the eigenvalues of a section of the
endomorphism bundle are not so immediate, but there are a couple of
things that we need to understand better.  To do this, we first prove
a statement about the eigenvalues of symmetric matrices, then
``globalize'' the statement.  We do this in two lemmas, after
reviewing a fact from linear algebra in the following proposition.

\begin{prop}[{\cite[Thm.~7.2.1]{horn90:_matrix_analy}}]\label{prop:7}

  A symmetric $n \times n$ matrix $T$ is positive definite
  (resp.~positive semidefinite) if and only if all eigenvalues of $T$
  are positive (resp.~nonnegative).

  In particular, if $T$ is positive definite (resp.~positive
  semidefinite), then $\det T > 0$ (resp.~$\det T \geq 0$).  If $T$ is
  positive semidefinite but not positive definite, then $\det T = 0$.
\end{prop}

\begin{lem}\label{lem:46}
  Let $\langle\!\langle \cdot , \cdot \rangle\!\rangle$ be any scalar
  product on $\R^n$, and let $\lambda^A_{\textnormal{min}}$ and
  $\lambda^A_{\textnormal{max}}$ denote the smallest and largest
  eigenvalues, respectively, of an $n \times n$ matrix $A$.  Then the
  map $A \mapsto \lambda^A_{\textnormal{min}}$ is a concave function
  from the space of self-adjoint $n \times n$ matrices to $\R$.  (Of
  course, we define ``self-adjoint'' with respect to $\langle\!\langle
  \cdot , \cdot \rangle\!\rangle$.)  Furthermore, $A \mapsto
  \lambda^A_{\textnormal{max}}$ is convex.

  In particular, each map is continuous.
\end{lem}
\begin{proof}
  Consider the following formula for the minimal eigenvalue of a
  self-adjoint matrix, which follows from the min-max theorem
  \cite[Thm.~XIII.1]{reed78:_method_of_moder_mathem_physic_iv}:
  \begin{equation}\label{eq:132}
    \lambda^A_{\textnormal{min}} = \min_{\substack{v \in \R^n \\
        \langle\!\langle v, v \rangle\!\rangle = 1}} \langle\!\langle v, A v \rangle\!\rangle.
  \end{equation}
  Therefore, if $A$ and $B$ are self-adjoint matrices, we have
  \begin{align*}
    \lambda^{(1-t) A + t B}_{\textnormal{min}} &= \min_{\substack{v \in \R^n \\
        \langle\!\langle v, v \rangle\!\rangle = 1}} \langle\!\langle
    v, ((1-t)
    A + t B) v \rangle\!\rangle \\
    &\geq \min_{\substack{v \in \R^n \\
        \langle\!\langle v, v \rangle\!\rangle = 1}} \langle\!\langle
    v,
    (1-t) A v \rangle\!\rangle + \min_{\substack{v \in \R^n \\
        \langle\!\langle v, v \rangle\!\rangle = 1}} \langle\!\langle
    v, t B v
    \rangle\!\rangle \\
    &= (1-t) \lambda^A_{\textnormal{min}} + t
    \lambda^B_{\textnormal{min}}.
  \end{align*}

  That the map sending a self-adjoint matrix to its maximal eigenvalue
  is convex follows in exactly the same way from the formula
  \begin{equation}\label{eq:133}
    \lambda^A_{\textnormal{max}} = \max_{\substack{v \in \R^n \\
        \langle\!\langle v, v \rangle\!\rangle = 1}} \langle\!\langle
    v, A v \rangle\!\rangle,
  \end{equation}
  which again follows from the min-max theorem.

  Continuity of the maps follows from the well-known result that a
  convex or concave function on a real, finite-dimensional vector
  space is continuous \cite[Thm.~10.1]{rockafellar70:_convex_analy}.
\end{proof}

\begin{lem}\label{lem:45}
  Let $h$ be any continuous, symmetric $(0,2)$-tensor field.  Suppose
  $g$ is a Riemannian metric on $M$, and let $H$ be the $(1,1)$-tensor
  field obtained from $h$ by raising an index using $g$.  (That is, locally
  $H^i_j = g^{ik} h_{kj}$.)  Then $H$ is a continuous section of the
  endomorphism bundle $\textnormal{End}(M)$.  Denote by
  $\lambda^H_{\textnormal{min}}(x)$ the smallest eigenvalue of $H(x)$.
  We have that
  \begin{enumerate}
  \item $\lambda^H_{\textnormal{min}}$ is a continuous function and
  \item if $h$ is positive definite, then $\min_{x \in M}
    \lambda^H_{\textnormal{min}}(x) > 0$.
  \end{enumerate}

  Furthermore, if $\lambda^H_{\textnormal{max}}(x)$ denotes the
  largest eigenvalue of $H(x)$, then $\lambda^H_{\textnormal{max}}$ is
  a continuous and hence bounded function.
 \end{lem}
\begin{proof}

  For any fixed $p \in M$, let a neighborhood $U$ of $p$ be given with
  the property that we can find a frame field for $T M|_U$, i.e.,
  there exist $n$ smooth vector fields $X_1, \dots, X_n$ over $U$ that
  together form a basis of $T_x M$ for each $x \in U$.  For every
  nonzero $n$-tuple $\alpha = (\alpha^1, \dots, \alpha^n) \in \R^n$,
  we define a vector field $X_\alpha$ over $U$ via
  \begin{equation*}
    X_\alpha := \alpha^1 X_1 + \cdots + \alpha^n X_n.
  \end{equation*}

  For each such $\alpha \in \R^n$, consider the function
  \begin{align*}
    Q^g_\alpha : U &\rightarrow \R \\
    x &\mapsto \frac{g(x)(X_\alpha(x), H(x)
      X_\alpha(x))}{g(x)(X_\alpha(x), X_\alpha(x))}.
  \end{align*}
  Thus, $Q^g_\alpha(x)$ is the Rayleigh quotient, with respect to
  $g(x)$, of $H(x)$ on the vector $X_\alpha(x)$.  It does not depend
  on $\alpha$, but only on the line on which $\alpha$ lies.  Thus, the
  family $\{ Q^g_\alpha \mid \alpha \in \R^n \}$ can be seen as a
  family of functions for $\alpha \in S^{n-1} \subset \R^n$.

  By the continuity of $g$, $H$ and $X_\alpha$, as well as the
  relative compactness of $U$ and the compactness of $S^1$, it is not
  hard to show that
  \begin{equation*}
    Q^g_{\textnormal{max}}(x) := \max_{\alpha \in S^{n-1}} Q^g_\alpha(x)
    \quad \textnormal{and} \quad Q^g_{\textnormal{min}}(x) :=
    \min_{\alpha \in S^{n-1}} Q^g_\alpha(x)
  \end{equation*}
  are continuous functions defined on $U$.  On the other hand, since
  $H(x)$ is self-adjoint with respect to $g(x)$, we can use the
  formulas in the proof of Lemma \ref{lem:46} to see that
  \begin{equation*}
    Q^g_{\textnormal{max}}(x) = \lambda^H_{\textnormal{max}}(x) \quad
    \textnormal{and} \quad Q^g_{\textnormal{min}}(x) =
    \lambda^H_{\textnormal{min}}(x).
  \end{equation*}
  From this, and since $p$ was chosen arbitrarily, the continuity of
  $\lambda^H_{\textnormal{max}}$ and $\lambda^H_{\textnormal{min}}$ is
  immediate.

  The upper bound on $\lambda^H_{\textnormal{max}}$ follows from its
  continuity.  That $\lambda^H_{\textnormal{min}}$ is bounded away
  from zero if $h$ is positive definite follows from the fact that if
  $p$ is the (arbitrary) point chosen above, then
  \begin{equation*}
    \lambda^H_{\textnormal{min}}(p) = \min_{\alpha \in S^{n-1}}
    g(p)(X_\alpha(p), H(p) X_\alpha(p)) = \min_{\alpha \in S^{n-1}} h(p)
    (X_\alpha(p), X_\alpha(p)) > 0,
  \end{equation*}
  so $\lambda^H_{\textnormal{min}}$ is a continuous positive function
  on $M$.
\end{proof}

\subsection{Lebesgue measure on
  manifolds}\label{sec:lebesg-meas-manif}

The concept of Lebesgue measurability carries over from $\R^n$ to
smooth (or even topological) manifolds very simply.  Let a maximal
atlas of coordinate charts $\{(U_\alpha, \phi_\alpha) \mid U_\alpha
\rightarrow V_\alpha \subseteq \R^n)\}$ \label{correction:charts} for
$M$ be given.  We say a subset $E \subset M$ is \emph{Lebesgue
  measurable} if we can find a covering of $E$ by charts $\{(U_\beta,
\phi_\beta)\}$ such that $\phi_\beta(E \cap U_\beta)$ is Lebesgue
measurable for each $\beta$.  This concept is independent of the
particular choice of covering: if $\{U_\gamma, \phi_\gamma\}$ is a
second covering of $E$ by charts, then each transition function
\begin{equation*}
  \phi_{\beta \gamma} = \phi_\gamma |_{U_\gamma
    \cap U_\beta}\circ \phi_\beta^{-1} |_{\phi_\beta(U_\beta \cap
    U_\gamma)}
\end{equation*}
is a smooth diffeomorphism, and hence it maps Lebesgue measurable sets
to Lebesgue measurable sets.  Of course, the transition function will
not necessarily preserve the quantitative measure of a
Lebesgue measurable set, but it \emph{will} map nullsets to nullsets.
Therefore, we can speak about nullsets on a smooth, finite-dimensional
manifold.

\begin{cvt}\label{cvt:7}
  Whenever we refer to a measure-theoretic concept on $M$, we
  implicitly mean that we work with Lebesgue measure or Lebesgue sets,
  unless we explicitly state otherwise.  (At some points Borel
  measures will also come up.)
\end{cvt}

With Lebesgue measurable sets well-defined, the concept of a
measurable function or a measurable map between manifolds is also
well-defined---these are simply those maps for which the preimage of
any measurable set is measurable.  If we have a local $(r,s)$-tensor
field $t$ on $M$, defined over a set $E \subseteq M$, we say that it
is \emph{measurable} or has \emph{measurable coefficients} if there is
a covering of $\{(U_\beta, \phi_\beta)\}$ of $E$ by coordinate charts
such that over each $U_\beta$, the coefficients of $t$ are measurable.
This is again independent of the covering chosen, since in a different
coordinate chart $(U_\gamma, \phi_\gamma)$, the coefficients of $t$
are determined from the coefficients in the original chart and the
transition function $\phi$ via
\begin{equation*}
  t^{i_1 \cdots i_r}_{j_1 \cdots j_s} = (D \phi^{-1})^{i_1}_{k_1} \cdots (D \phi^{-1})^{i_r}_{k_r} (t^{k_1 \cdots k^r}_{l_1 \cdots l_s}\circ
  \phi_{\beta \gamma})
  (D \phi)^{l_1}_{j_1} \cdots (D \phi)^{l_s}_{j_s}.
\end{equation*}
That is, the new coefficients are obtained from the old via
composition, addition, and multiplication with smooth functions, so
they are again measurable.

We can also speak about Lebesgue measures.  By the above definition,
it is immediate that any volume form $\mu$ on $M$ with measurable
coefficients induces a Lebesgue measure on $M$.  (By volume form, we
simply mean any $n$-form with positive coefficient.  Saying the
coefficient is positive is coordinate-independent because of the
orientation of $M$.)  We can also allow $\mu$ to be a nonnegative
$n$-form---i.e., one for which the coefficient is everywhere
nonnegative---and $\mu$ again induces a Lebesgue measure on $M$.
(Nonnegativity is again a coordinate-independent notion thanks to
orientability of $M$.)

We next mention the relation of the Lebesgue measurable sets
$\mathcal{L}$ to the Borel measurable sets $\mathcal{B}$.  It is not
hard to see that the same general relationship between these sets that
holds on $\R^n$ holds on $M$ as well.  Let's recall this
relationship---namely, that Lebesgue measure on $\R^n$ coincides with
the outer measure induced by Borel measure \cite[\S
1.7]{rana02:_introd_to_measur_and_integ}.  For the reader's
convenience, we briefly review this notion, as well as that of the
completion of a measure space.  All facts are taken from \cite[\S
1.5]{bogachev07:_measur_theor} unless otherwise mentioned.

Let $(X,
\Sigma, \mu)$ be a measure space.  The measure $\mu$ is called
\emph{complete} if for every $B \in \Sigma$ with $\mu(B) = 0$ and
every subset $A \subset B$, we have $A \in \Sigma$ (and therefore of
course $\mu(A) = 0$).  That is, $\mu$ is complete if every subset of a
nullset is $\mu$-measurable.

If $\mu$ is not complete, we can extend it to a complete measure as
follows.  We define the \emph{outer measure} $\mu^*$ induced by $\mu$
to be, for any $E \subseteq X$,
\begin{equation*}
 \mu^*(E) := \inf \left\{ \sum_{k=1}^\infty E_k \midmid E_k \in \Sigma,\ E
   \subseteq \bigcup_{k=1}^\infty E_k \right\}
\end{equation*}
for any set $E \subseteq X$.  We note that if $E \in \Sigma$, then
$\mu^*(E) = \mu(E)$.  Define a set $E \subseteq X$ to be
\emph{$\mu^*$-measurable} if for every $Y \subseteq X$,
\begin{equation*}
  \mu^*(Y) = \mu^*(Y \cap E) + \mu^*(Y \cap E^c).
\end{equation*}
We denote the class of $\mu^*$-measurable sets by $\Sigma^*$, and note
that this is a $\sigma$-algebra.  Furthermore, $\Sigma \subseteq
\Sigma^*$, and $(X, \Sigma^*, \mu^*)$ is complete.  Thus $\mu^*$ is an
extension of $\mu$ to a complete measure.

Now, let's see what this means in the special case of $M$ with a
measure $\mu$ on the Borel sets $\mathcal{B}$ of $M$.  Firstly, since
the above statement that Lebesgue measure coincides with the outer
measure induced by Borel measure can be localized, the outer measure
corresponding to $\mu$ is a measure $\mu^*$ on the Lebesgue sets
$\mathcal{L}$ of $M$.  We also see that a Borel measurable set is
Lebesgue measurable.  Furthermore, by \cite[\S
3.11]{rana02:_introd_to_measur_and_integ}, the following holds.

\begin{lem}\label{lem:53}
  For every $E \in \mathcal{L}$, there exist $F \in \mathcal{B}$ and
  $G \in \mathcal{L}$ such that
  \begin{enumerate}
  \item $E = F \cup G$,
  \item there exists a $\mu$-nullset $A \in \mathcal{B}$ such that $G
    \subset A$ and
  \item $\mu^*(E) = \mu(F)$.
  \end{enumerate}
\end{lem}

In other words, Lebesgue measurable sets can always be built from the
union of a Borel measurable set and a subset of a Borel nullset.

To close this subsection, for convenience we recall two standard
results from measure theory: the Lebesgue dominated convergence
theorem and Fatou's lemma.

\begin{thm}[{The Lebesgue dominated convergence theorem
    \cite[Thm.~5.4.9]{rana02:_introd_to_measur_and_integ}}]\label{thm:36}
  Let $(X, \Sigma, \nu)$ be a measure space, and let $\{f_k\}$ be a
  sequence of measurable functions on $X$ converging a.e.~to a
  function $f$.  Suppose further that there exists an $L^1$ function
  $g$ with $|f_k| \leq g$ a.e.  Then $f$ is also $L^1$ and
  \begin{equation*}
    \integral{X}{}{f}{d \mu} = \lim_{k \rightarrow \infty} \integral{X}{}{f_k}{d \mu}.
  \end{equation*}
  Furthermore,
  \begin{equation*}
    \lim_{k \rightarrow \infty} \integral{X}{}{|f - f_k|}{d \mu} = 0.
  \end{equation*}
\end{thm}

\begin{thm}[{Fatou's lemma
    \cite[Thm.~2.8.3]{bogachev07:_measur_theor}}]\label{thm:37}
  Let $(X, \Sigma, \nu)$ be a measure space, and let $\{f_k\}$ be a
  sequence of nonnegative measurable functions on $X$ converging
  a.e.~to a function $f$.  Suppose that there exists a constant $K <
  \infty$ such that
  \begin{equation*}
    \sup_k \integral{X}{}{f_k}{d \mu} \leq K.
  \end{equation*}
  Then the function $f$ is integrable and
  \begin{equation*}
    \integral{X}{}{f}{d \mu} \leq K.
  \end{equation*}
  In addition,
  \begin{equation*}
    \integral{X}{}{f}{d \mu} \leq \liminf_{k \rightarrow \infty}
    \integral{X}{}{f_k}{d \mu}.
  \end{equation*}
\end{thm}

\subsection{The manifolds of positive functions and volume
  forms}\label{sec:manif-posit-funct}

We will denote the set of positive $C^\infty$ functions on the base
manifold $M$ by $\pos$.  By the considerations of Subsection
\ref{sec:manifolds-mappings}, $\pos$ is a Fréchet manifold, since it
can be identified with the space of smooth sections of the trivial
fiber bundle $M \times (0, \infty)$.  (Alternatively, one can view it
as an open set in the Fréchet space of smooth sections of the vector
bundle $M \times \R$.)  It is not hard to see that $\pos$ is even a
Fréchet Lie group with respect to the group operation of pointwise
multiplication---that is, it is a Fréchet manifold such that the
multiplication of two elements is a smooth map, as is the map sending
an element to its multiplicative inverse.

Similarly, if we denote by $\V$ the set of smooth volume forms on $M$,
then this is a Fréchet manifold.  One can see this either by viewing
it as an open set of $\Omega^n(M)$, the Fréchet space of
highest-order differential forms on $M$ (it is the sections of the
line bundle $\Lambda^n T^* M$), or by viewing $\V$ as the smooth
sections of the fiber bundle of positive $n$-forms on $M$.

Given any volume form $\mu \in \V$ and any $n$-form $\alpha \in
\Omega^n(M)$, there exists a unique $C^\infty$ function, denoted by
$(\alpha / \mu)$, such that
\begin{equation}\label{eq:118}
  \alpha = \left( \frac{\alpha}{\mu} \right) \mu.
\end{equation}
This fact is easy to deduce from the coordinate representations of
$\alpha$ and $\mu$, along with the fact that the coefficient of $\mu$
is positive in any coordinate chart because $\mu$ is a volume form.

If $\alpha$ is also a smooth volume form, then $(\alpha / \mu)$ is
additionally a positive function.  If we consider the measures induced
by $\alpha$ and $\mu$, then $(\alpha / \mu)$ coincides with the
\emph{Radon-Nikodym derivative}
\cite[Dfn.~9.1.16]{rana02:_introd_to_measur_and_integ} of $\alpha$
with respect to $\mu$.  That is, for any measurable set $E \subseteq
M$, we have
\begin{equation*}
  \integral{E}{}{}{\alpha} = \integral{E}{}{\left(
      \frac{\alpha}{\mu} \right)}{\mu}.
\end{equation*}
We just note that the Radon-Nikodym derivative is defined in general
as follows.  Say we are given a space $X$ with a $\sigma$-algebra
$\Sigma$, as well as two $\sigma$-finite measures $\nu$ and $\nu_0$ on
$\Sigma$.  Furthermore, suppose that $\nu_0$ is \emph{absolutely
  continuous} with respect to $\nu$, that is, $\nu_0(E) = 0$ for all
$E \subseteq X$ with $\nu(E) = 0$.  Then there exists a nonnegative
measurable function $f$ on $X$, called the Radon-Nikodym derivative,
such that
\begin{equation*}
  \integral{E}{}{}{d \nu_0} = \integral{E}{}{f}{d \nu}.
\end{equation*}

The considerations above suggest a natural diffeomorphism between $\V$
and $\pos$.  Namely, if we choose any volume form $\mu \in \V$, then
we can define a map
\begin{equation}\label{eq:127}
  \nu \mapsto \left( \frac{\nu}{\mu} \right),
\end{equation}
which as we have seen maps $\V$ into $\pos$.  It is not hard to see
that this map is bijective.  To see that it is smooth, we simply note
that it is the restriction to $\V$ of the linear map $\alpha \mapsto
(\alpha / \mu)$, which maps $\Omega^n(M)$ into $C^\infty(M)$.

\begin{rmk}\label{rmk:13}
  Note that the function $(\alpha / \mu)$ can be more generally
  defined for \emph{any} $n$-form $\alpha$ and \emph{any} volume form
  $\mu$, including those that are not smooth, or continuous, or even
  measurable.  The function will be smooth/continuous/measurable if
  both $\alpha$ and $\mu$ are, as is easily seen in a coordinate
  chart.  It is easy to show (or one may consult
  \cite[Prop.~5.2.6]{rana02:_introd_to_measur_and_integ}) that if
  $\alpha$ is measurable and nonnegative, then it induces a measure on
  $M$, defined by fixing any volume form $\mu \in \V$ and setting
  \begin{equation*}
    \integral{E}{}{}{\alpha} := \integral{E}{}{\left(
        \frac{\alpha}{\mu} \right)}{\mu}
  \end{equation*}
  for any measurable $E \subseteq M$.  Furthermore, this measure is
  absolutely continuous with respect to $\mu$.
\end{rmk}

If $g$ is a Riemannian metric on $M$, then it induces a volume form on
$M$ given in local coordinates $x^1, \dots, x^n$ by
\begin{equation}\label{eq:119}
  \mu_g = \sqrt{\det g}\, dx^1 \cdots dx^n.
\end{equation}
If $g_0$ and $g_1$ are two Riemannian metrics on $M$, then locally,
the Radon-Nikodym derivative of $\mu_{g_1}$ with respect to
$\mu_{g_0}$ is given by
\begin{equation}\label{eq:8}
  \left( \frac{\mu_{g_1}}{\mu_{g_0}} \right) = \sqrt{\frac{\det
      g_1}{\det g_0}} = \sqrt{\det (g_0^{-1} g_1)}.
\end{equation}
Note that this is a well-defined function on $M$ by the discussion of
Subsection \ref{sec:endom-bundle-m}.

This completes our general geometric considerations.  We now move on
to the study of Riemannian metrics on Fréchet manifolds.

\section{Weak Riemannian manifolds}\label{sec:weak-riem-manif}

The manifold of metrics with its $L^2$ metric, the object of study of
this thesis, is an example of what is called a weak Riemannian
manifold.  In this section, we will describe and explore these objects
a little bit.

It is well known that on a finite-dimensional vector space, all
positive-definite scalar products are equivalent---i.e., every
positive-definite scalar product induces the same topology on the
space.  In infinite-dimensional vector spaces, this is no longer the
case---there are many inequivalent positive-definite scalar products,
with differing topologies---a simple example might be the $C^k$ and
$H^s$ topologies on the space of smooth functions $[0,1] \rightarrow
\R$.  As one might naturally expect, this linear phenomenon has an
analog in nonlinear spaces, i.e., manifolds.

In manifold theory, the nonlinear analog of a positive-definite scalar
product on a vector space is a Riemannian metric on a manifold.  Of
course, this is just a positive-definite scalar product on the
linearization of the manifold at each point (i.e., its tangent spaces)
that varies in a smooth way as we move from point to point.  (We'll
make a formal definition of infinite-dimensional Riemannian metrics
soon; for the moment, let's just take this as our heuristic definition
for purposes of the introductory discussion.)  There is also a
nonlinear analog of the difference between finite- and
infinite-dimensional spaces as described above.

On a finite-dimensional Riemannian manifold $(N, \gamma)$ modeled on
$\R^n$, the tangent space $T_x N$ is, via a choice of coordinates,
isomorphic to $\R^n$ for each $x \in N$.  The equivalence of all
scalar products on $\R^n$ implies that the scalar product induced by
the Riemannian metric, when viewed as a scalar product on $\R^n$, is
equivalent to the Euclidean scalar product.  In particular, it induces
the standard topology on $\R^n$.

In the case of an infinite-dimensional Riemannian manifold $(N,
\gamma)$ modeled on a Hilbert space $E$, we cannot necessarily say
that the scalar product induced by $\gamma$ on a tangent space $T_x N$
is equivalent to the Hilbert space scalar product of $E$.  Therefore,
the topology that $\gamma$ induces on $T_x N$ \emph{may differ from
  the topology of $E$}.  Thus, we can distinguish two types of
Riemannian metrics on a Fréchet (or, as a special case, Hilbert)
manifold.  We call $\gamma$ a \emph{strong Riemannian metric} if it
induces the model space topology on each tangent space, and a
\emph{weak Riemannian metric} if it induces a weaker topology.

This subtle but important distinction between the two types of metrics
leads to a vast gulf in the two theories one can develop around each
structure.  For a strong Riemannian metric, one can reproduce most of
the important results in finite-dimensional Riemannian geometry.  For
example, the Levi-Civita connection, geodesics, and the exponential
mapping exist.  A strong Riemannian metric induces a distance function
that gives a metric space structure on the manifold.  In addition, the
topology induced from this metric space structure agrees with the
manifold's intrinsic topology.

\emph{None of the above-mentioned results hold in general for weak
  Riemannian manifolds.}

In this section, we will go into detail on these and other differences
between weak and strong Riemannian manifolds, as well as explore what
statements one can make about weak Riemannian metrics in the cases
where the corresponding statements for strong metrics break down.

Before we continue with formal definitions and results, though, we
make a couple of philosophical remarks.  We have found relatively few
references that systematically cover what results of standard
Riemannian geometry do and do not hold in this context, as most
authors naturally treat only those aspects that arise in the examples
they are considering.  Furthermore, to the author's knowledge, all
standard textbooks about Riemannian geometry on Hilbert manifolds,
such as \cite{klingenberg95:_rieman_geomet} and
\cite{lang95:_differ_and_rieman_manif}, work only with strong
Riemannian metrics, without explicitly mentioning the distinction
between the two types of metrics.  Therefore, even the most basic
results on weak Riemannian metrics seem not to have been formally
written down.  Later in the section, we will prove a few general
results that will come in useful to us.

Despite there being, to our knowledge, no comprehensive formal
treatment of them, weak Riemannian metrics are fundamental objects in
global analysis, which deals primarily with manifolds of sections of
fiber bundles over a finite-dimensional manifold.  Of course, one is
typically most interested in $C^\infty$ sections, and the space of
$C^\infty$ sections of a vector bundle is a proper Fréchet
space---proper meaning that the topology does not come from a (single)
norm.  Such a space carries \emph{only} weak Riemannian metrics, since
the existence of a strong Riemannian metric would give, via a
coordinate chart, a norm inducing the topology of the model
space---but this is impossible.

If one considers $H^s$ sections of a vector/fiber bundle, then strong
Riemannian metrics can be found.  However, since the choice of $s \in
\N$ is essentially arbitrary, one would have to choose a different
Riemannian metric on the manifold of sections for each $s$.  This
somewhat unsatisfactory situation leads one to generally pick a single
metric (i.e., use the same formula for each $s$), often one inducing
the $L^2$ topology.  Thus, one is again led back to working with weak
Riemannian metrics.

Hopefully we have convinced the reader of the importance of weak
Riemannian metrics.  We now move on to defining them, exploring some
of their deficiencies as compared with strong Riemannian metrics, and
then elaborating what weaker results one can prove about them in
general.

\subsection{(Weak) Riemannian Fréchet
  manifolds}\label{sec:weak-riem-frech}

A Riemannian metric on a Fré\-chet manifold is defined exactly
analogously to one on a finite-dimensional manifold, modulo the
distinction between weak and strong metrics mentioned above.

Recall that on a Banach manifold $N$ modeled on a Banach space $E$,
each tangent space $T_x N$ is naturally isomorphic to the model space
$E$, the isomorphism being given by any choice of coordinates around
$x$ (this choice is, of course, usually very non-canonical).  The same
holds true for Fréchet manifolds, and we keep this in mind as we
make the following definition.

\begin{dfn}\label{dfn:19}
  Let $N$ be a Fréchet manifold modeled on a Fréchet space $E$.  A
  \emph{Riemannian metric} $\gamma$ on $N$ is a choice of scalar
  product $\gamma(x)$ on $T_x N$ for each $x \in N$, such that for
  each $x \in N$, the following holds:
  \begin{enumerate}
  \item $\gamma$ is smooth in the sense that if $U$ is any open
    neighborhood of $x$ and $V, W$ are vector fields defined on $U$,
    then $\gamma(\cdot)(V, W) : U \rightarrow \R$ is a smooth local
    function;
  \item $\gamma(x)$ is a continuous (i.e., bounded) bilinear mapping;
    and
  \item $\gamma(x)$ is positive definite on $T_x N$.
  \end{enumerate}

  Furthermore, $\gamma$ is called
  \begin{enumerate}
  \item \emph{strong} if the topology induced by $\gamma$ coincides
    with the topology of the model space $E$; and
  \item \emph{weak} otherwise, i.e., if the topology induced by
    $\gamma$ is weaker than the model space topology.
  \end{enumerate}

  The pair $(N, \gamma)$ is called a \emph{Riemannian Fréchet
    manifold}.
\end{dfn}

To put it another way, $(N, \gamma)$ is a strong Riemannian Fréchet
manifold if its tangent spaces are complete with respect to $\gamma$,
and it is weak if the tangent spaces are incomplete with respect to
$\gamma$.

\begin{rmk}\label{rmk:21}
  There is no such thing as a Riemannian metric inducing a topology on
  the tangent space that is stronger than the manifold topology.  This
  is because in that case some vectors would have infinite norm---just
  think of the $H^{s+1}$ norm on $H^s$ functions, for example.
\end{rmk}

The first definition of weak Riemannian Hilbert manifolds that we know
of (though our knowledge is surely incomplete) is in
\cite{ebin70:_manif_of_rieman_metric}, the paper that founded the
study of the geometry of the manifold of metrics.  The generalization
to weak Riemannian Fréchet manifolds is natural and has been used in
several works.  In no particular order, here is a list of papers that
consider weak Riemannian manifolds (specifically, those that
explicitly deal with the questions posed by ``weakness'' and are not
mentioned elsewhere in this thesis):
\cite{biliotti04:_expon_map_of_weak_rieman_hilber_manif},
\cite{constantin03:_geodes_flow_diffeom_group_of_circl},
\cite{ebin70:_group_of_diffeom_and_motion},
\cite{marsden72:_darboux_theor_fails_for_weak_sympl_forms},
\cite{michor:_overv_of_rieman_metric_spaces},
\cite{michor98:_geomet_of_viras_bott_group},
\cite{misiolek93:_stabil_of_flows_of_ideal},
\cite{misiolek96:_conjug_point_in_d_mu_t} and
\cite{misiolek99:_expon_maps_of_sobol_metric_loop_group}.  We have
made no attempt to make this list complete---it is simply a smattering
of examples.

Let $(N, \gamma)$ be a Riemannian Fréchet manifold.  Just as in the
case of finite-dimensional Riemannian manifolds, we can use $\gamma$
to define a distance between points of $N$ by taking the infima of
lengths of paths.

Let $a \leq b$ be real numbers, and let $\alpha : [a,b] \rightarrow N$
be a piecewise $C^1$ path.  Define
\begin{equation*}
  L(\alpha) :=
  \integral{a}{b}{\sqrt{\gamma(\alpha(t))(\dot{\alpha}(t),\dot{\alpha}(t))}}{d
    t}.
\end{equation*}
Then, for any $x, y \in N$, we define
\begin{equation*}
  d_\gamma (x, y) := \inf_\alpha L(\alpha),
\end{equation*}
where the infimum is taken over all piecewise $C^1$ paths that start
at $x$ and end at $y$.

It is easy to see that $d_\gamma$ is a \emph{pseudometric}.  That is,
it has all the properties of a metric (in the sense of metric spaces)
other than positive-definiteness.  That $d_\gamma(x,y) =
d_\gamma(y,x)$, $d_\gamma(x,y) \geq 0$ and $d_\gamma(x,x) = 0$ for all
$x, y \in N$ is clear.  The triangle inequality for $d_\gamma$ then
follows from the fact that if we have a path from $x$ to $y$ and a
path from $y$ to $z$, the concatenation of the two is a path from $x$
to $z$ with length the sum of the two original paths.

Positive definiteness of the distance function is a trickier issue,
and in fact it only holds in general for strong Riemannian metrics!
For weak metrics, it may fail.  In fact, the example of the next
subsection shows that it may fail in the most spectacular way
possible---for some weak Riemannian manifolds, $d_\gamma(x,y) = 0$ for
all points $x,y \in N$.

After we have described the example and seen how bad things can get,
we will see what parts of the theory break down and allow such things
to happen.  After that we will try to partially rebuild.

\subsection{Pathological behavior of a weak Riemannian metric on the
  manifold of embeddings of $S^1$ into $\R^2$}\label{sec:path-behav-weak}

The following example is from
\cite{michor06:_rieman_geomet_spaces_of_plane_curves}, to which we
refer for more details.  We will give only a very sketchy and
conceptual presentation of one of their results.  There is no harm in
skipping this subsection and continuing on to the discussion of the
Levi-Civita connection in the next subsection.  On the other hand, the
reader interested in a complete description of this example should
consult \cite{michor06:_rieman_geomet_spaces_of_plane_curves}.

Let $C^\infty(S^1, \R^2)$ denote the vector space of all smooth
mappings of $S^1$ into $\R^2$.  This is a Fréchet space, as we saw
in Subsection \ref{sec:manifolds-mappings}.  We consider the open set
$\mathcal{E} \subset C^\infty(S^1, \R^2)$ of smooth embeddings of
$S^1$ into $\R^2$---in other words, this is the space of smooth,
parametrized, closed curves in $\R^2$.  As an open set of a Fréchet
space, it is trivially a Fréchet manifold.

Let $\D$ denote the group of smooth diffeomorphisms of the circle.  It
is a Fréchet Lie group, and it acts on $C^\infty(S^1, \R^2)$ from
the right by composition, i.e., pull-back: for $\varphi \in \D$ and $f
\in C^\infty(S^1, \R^2)$, the action is $\varphi^* f = f \circ
\varphi$.  If we restrict this action to $\mathcal{E}$, then it is
free, and it turns out that the quotient $\mathcal{E} / \D$ is a
smooth Fréchet manifold.

There exists a natural $\D$-invariant Riemannian metric on
$\mathcal{E}$.  It is a weak metric, as it induces the $L^2$ topology
on the tangent spaces.  To define it, let $f \in \mathcal{E}$ be any
embedding.  Since $\mathcal{E}$ is an open set of $C^\infty(S^1,
\R^2)$, the tangent space $T_f \mathcal{E}$ is canonically isomorphic
to $C^\infty(S^1, \R^2)$ itself, and we can think of $T_f\mathcal{E}$
as the space of vector fields on $f(S^1) \subset \R^2$.  That is, if
$\pi : S^1 \times \R^2 \rightarrow S^1$ is the projection, then
$T_f\mathcal{E}$ consists of maps $h : S^1 \rightarrow S^1 \times
\R^2$ with $\pi \circ h = f$.  With this in mind, we define for any
$h, k \in C^\infty(S^1, \R^2) \cong T_f \mathcal{E}$:
\begin{equation*}
  (\!( h, k )\!)_f := \integral{S^1}{}{\langle h(\theta), k(\theta)
    \rangle |\partial_\theta f(\theta)|}{d \theta},
\end{equation*}
where $\langle \cdot, \cdot \rangle$ is the Euclidean scalar product
on $\R^2$.  Describing this metric in words, we integrate the scalar
product of $h$ and $k$ with respect to the Euclidean volume form
pulled back along $f$.

Since $(\!( \cdot , \cdot )\!)$ is $\D$-invariant (as is relatively
easily computed), it descends to a weak Riemannian metric on
$\mathcal{E} / \D$.  Though it is outside the scope of this thesis to
prove this here, the Riemannian metric thus obtained induces a
distance function as described above, but the distance between any two
points vanishes!  Thus the Riemannian metric $(\!( \cdot , \cdot )\!)$
is, in some sense, a very bad metric on $\mathcal{E} / \D$.

\begin{figure}[t]
  \centering
  \includegraphics[width=11cm]{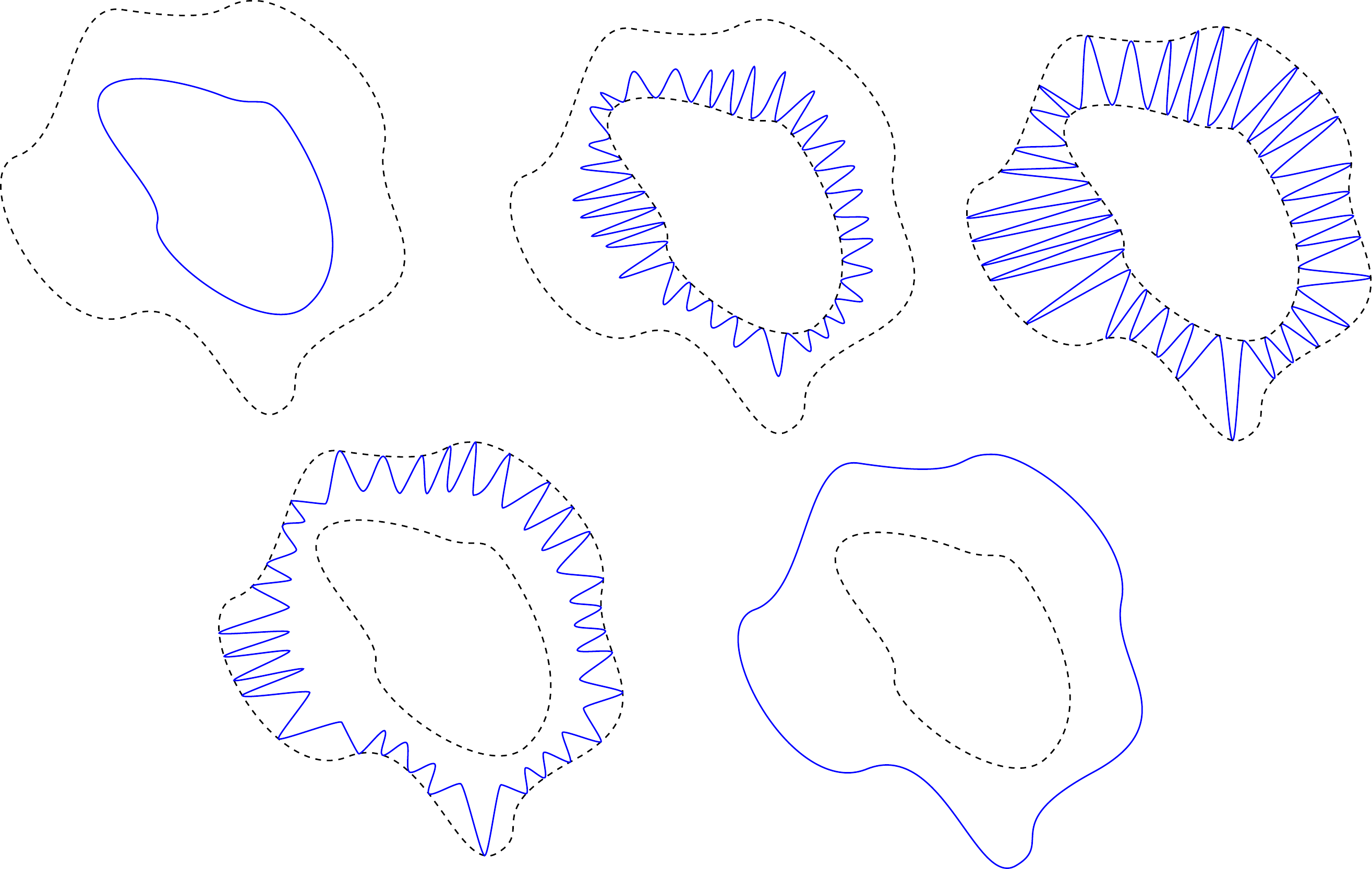}
  \caption{Wildly oscillating curves interpolating between $f_0$ (the
    inner curve) and $f_1$ (the outer curve).  First, the oscillating
    parts are extended until they reach $f_1$, then the rest
    follows.}\label{fig:circ-embed}
\end{figure}

Rather than prove this fact, we will simply give the idea of the
proof.  We can bound the distance between two points in $\mathcal{E} /
\D$ from above by the distance between any two points of their
preimages in $\mathcal{E}$. So take any path $f_t$ of curves
interpolating between $f_0$ and $f_1$---it happens that one can modify
this path to get a path whose image has arbitrarily small length when
projected to $\mathcal{E} / \D$, showing that the distance between the
endpoints in $\mathcal{E} / \D$ is zero.  To do this, we simply
construct a path $f_{n,t}$ from $f_t$ in which the curves oscillate
$n$ times as they interpolate between $f_0$ and $f_1$.  These
oscillating curves are illustrated in Figure \ref{fig:circ-embed}.  It
then happens that the length of $f_{n,t}$, when projected onto
$\mathcal{E} / \D$, goes to zero as $n \rightarrow \infty$.

So now we have an extremely pathological example of how bad the
distance function of a weak Riemannian metric can be.  Our next task
is to understand how such a phenomenon, which is impossible in the
finite-dimensional case, can occur.  To do so, we need to reexamine
some of the standard theorems of Riemannian geometry and see what can
be said about them in the infinite-dimensional case.

Before we conclude this subsection, let us just note that
\cite{michor05:_vanis_geodes_distan_spaces_of} generalizes the example
described here to embeddings of any given manifold into a given
Riemannian manifold.

\subsection{The Levi-Civita connection}\label{sec:levi-civita-conn}

On a finite-dimensional Riemannian manifold $(N, \gamma)$, and even on
a strong Riemannian Hilbert manifold, there is a unique connection
$\nabla$ that is both
\begin{enumerate}
\item \emph{metric}, i.e., $X\gamma(Y, Z) = \gamma(\nabla_X Y, Z) +
  \gamma(Y, \nabla_X Z)$ for all vector fields $X,$ $Y$ and $Z$; and
\item \emph{torsion-free}, i.e., $\nabla_X Y - \nabla_Y X = [X,Y]$ for
  all vector fields $X$ and $Y$.
\end{enumerate}
The existence and uniqueness of this connection relies on the Koszul
formula, which states that a connection is both metric and
torsion-free if and only if the following equation holds for all
vector fields $X$, $Y$, and $Z$:
\begin{equation}\label{eq:95}
  \begin{aligned}
    \gamma(\nabla_X Y, Z) &= X \gamma(Y, Z) + Y \gamma(X, Z) - Z
    \gamma(X, Y) \\
    &\quad - \gamma(X, [Y, Z]) - \gamma(Y, [X, Z]) + \gamma(Z, [X,
    Y]).
  \end{aligned}
\end{equation}
Existence and uniqueness of the element $\nabla_X Y$ at the point $x
\in N$ now follows from the Riesz representation theorem applied to
the Hilbert space $(T_x N, \gamma)$.

The Levi-Civita connection is then used to define geodesics as those
paths $\alpha$ for which $\nabla_{\dot{\alpha}}\dot{\alpha} = 0$.
Geodesics, in turn, are used to define the exponential mapping, as is
well known.

On a weak Riemannian manifold, this picture breaks down, as
\eqref{eq:95} fails to guarantee existence of the Levi-Civita
connection.  (If it exists, though, \eqref{eq:95} does guarantee its
uniqueness.)  Since the tangent spaces of $(N, \gamma)$ are incomplete
with respect to $\gamma$, \eqref{eq:95} only guarantees the existence
of $\nabla_X Y$ at $x \in N$ as an element of the completion of $T_x
N$ with respect to $\gamma$.  This is of course because the Riesz
representation theorem does not hold on incomplete spaces.

The result of this is: \emph{On a weak Riemannian manifold, the
  Levi-Civita connection does not exist in general.  As a consequence,
  geodesics and the exponential mapping do not exist in general,
  either.}

The usual strategy when dealing with weak Riemannian manifolds is the
following.  Without general theorems at one's disposal, various
properties that are automatic for strong Riemannian manifolds have to
be directly verified.  For example, in the next section, we will
sketch how, in \cite{ebin70:_manif_of_rieman_metric}, the existence of
the Levi-Civita connection for the manifold of metrics was shown.  In
essence, an explicit formula for $\nabla_X Y$ was computed using the
Koszul formula, and it was shown that the result is in fact a section
of the tangent bundle.

\subsection{The exponential mapping and distance function on a strong
  Riemannian manifold}\label{sec:expon-mapp-dist}

Subsection \ref{sec:path-behav-weak} gave an example of a weak
Riemannian manifold with an induced distance function that is not a
metric---i.e., that fails to be positive definite.  In contrast, for a
strong Riemannian manifold, the following theorem holds, as it does in
the finite-dimensional case:

\begin{thm}[{\cite[Thm.~1.9.5]{klingenberg95:_rieman_geomet}}]\label{thm:25}
  Let $(N, \gamma)$ be a strong Riemannian (Hilbert) manifold.  Then
  the induced distance function $d_\gamma$ is a metric on $N$, and the
  topology of $d_\gamma$ coincides with the topology of $N$.
\end{thm}

The natural question that arises is, what goes wrong in the case of a
weak Riemannian manifold?  To answer this, we recall the main steps in
the proof of Theorem \ref{thm:25}.  The first is:

\begin{thm}[{\cite[Thm.~1.8.15]{klingenberg95:_rieman_geomet}}]\label{thm:26}
  Let $(N, \gamma)$ be a strong Riemannian (Hilbert) manifold.  Then
  there exists an open neighborhood $U \subseteq TN$ of $N$ such that
  the exponential mapping is defined and differentiable on $U$.

  Furthermore, for every $x \in N$, there exist positive numbers
  $\epsilon = \epsilon(x)$ and $\eta = \eta(x)$, with $\epsilon <
  \eta$, and a neighborhood $V$ of $x$ such that the following holds:
  \begin{enumerate}
  \item The mapping
    \begin{equation*}
      \exp_x|_{B_\epsilon (0)} : B_\epsilon (0) \rightarrow V,
    \end{equation*}
    where $B_\epsilon (0)$ is the open ball of radius $\epsilon$
    (w.r.t.~$\gamma$) around $0 \in T_x N$, is a diffeomorphism.
  \item For any $y, z \in V$, there exists a unique geodesic from $y$
    to $z$ with length less than $\eta$.
  \item For each $y \in V$, $\exp_y|B_\eta (0)$ is a diffeomorphism
    onto an open neighborhood $V_y$ of $y$, with $V \subseteq V_y$.
  \end{enumerate}
\end{thm}

Using this theorem, we have some control over the domain of definition
and the range of the exponential mapping.  The next step is to control
the lengths of paths contained within the image of the exponential
mapping:

\begin{thm}[{\cite[Thm.~1.9.2]{klingenberg95:_rieman_geomet}}]\label{thm:27}
  Suppose $(N, \gamma)$ is a strong Riemannian manifold.  Let $x \in
  N$, and suppose that $\exp_x$ is defined on an open neighborhood
  $U_x$ of $0 \in T_x N$.  Let $v: [0,1] \rightarrow U_x$ be any path
  with $v(0) = 0$, and let $\tilde{v} : [0,1] \rightarrow U_x$ be the
  straight-line path in $T_x N$ between $0$ and $v(1)$.  Finally,
  define paths in $N$ by $\alpha(t) := \exp_x(v(t))$ and
  $\tilde{\alpha}(t) := \exp_x(\tilde{v}(t))$.

  Then $L(\tilde{\alpha}) \leq L(\alpha)$, and equality holds if $v(t)
  = \tilde{v}(t(s))$, where $t(s)$ is a reparametrization with $t'(s)
  \geq 0$.

  Conversely, if $L(\tilde{\alpha}) = L(\alpha)$ and $D_{s
    v(t)}\exp_x$ has maximal rank for all $0 \leq s, t \leq 1$, then
  $v(t) = \tilde{v}(t(s))$, where $t(s)$ is a reparametrization with
  $t'(s) \geq 0$.
\end{thm}

What this theorem essentially says is the following.  Let $\exp_x$ be
defined on $U_x \subseteq T_x N$, with range $V_x \subseteq N$, and
let $y \in V_x$.  Then among the class of paths in $V_x$ from $x$ to
$y$, the unique shortest path (up to reparametrization) is the radial
geodesic emanating from $x$ and ending at $y$.

What Theorem \ref{thm:27} does not tell us is that the radial geodesic
from $x$ to $y$ is the shortest path among the class of \emph{all
  paths} in $N$ from $x$ to $y$.  However, combining Theorems
\ref{thm:26} and \ref{thm:27} gives us what we want:

\begin{thm}[{\cite[Thm.~1.9.3]{klingenberg95:_rieman_geomet}}]\label{thm:28}
  Suppose $(N, \gamma)$ is a strong Riemannian manifold.  Let $x \in
  N$, and let $\epsilon$, $\eta$ and $V$ be as in Theorem
  \ref{thm:26}.  Suppose that $y \in V$.  Then any geodesic starting
  from $y$ of length less than $\eta$ is a path of minimal length
  between its endpoints.
\end{thm}

Thus, geodesics are, locally, length-minimizing paths.  Using Theorem
\ref{thm:28} and our previous statement that the distance function of
a Riemannian manifold is always a pseudometric, it is then trivial to
prove Theorem \ref{thm:25}.

\subsection{The exponential mapping and distance function on a weak
  Riemannian manifold}\label{sec:expon-mapp-dist-1}

We now return to weak Riemannian manifolds.  The question remains:
What goes wrong when we try to extend the results of Subsection
\ref{sec:expon-mapp-dist}?

The theorem that breaks down, it turns out, is Theorem \ref{thm:26}.
This is true even if we assume that the Levi-Civita connection exists.
It even breaks down if we assume that the exponential mapping exists
and is a diffeomorphism when restricted to some open neighborhood of
the zero section in $T M$---none of which are guaranteed on a weak
Riemannian manifold!

The problem is the following: on a strong Riemannian manifold $(N,
\gamma)$, a neighborhood of $0 \in T_x N$ contains an open
$\gamma$-ball of some sufficiently small radius.  However, if $(N,
\gamma)$ is a weak Riemannian manifold, since the topology induced by
$\gamma$ is weaker than the manifold topology of $T_x N$, an open
neighborhood of $0$ (in the manifold topology) \emph{need not
  necessarily} contain any open $\gamma$-balls.

This phenomenon does indeed occur---it is not too hard to see that it
occurs for the example of Subsection \ref{sec:path-behav-weak}, and we
will see below, in Section \ref{sec:manifold-metrics-m}, that the
manifold of metrics also exhibits this phenomenon.

In the case of the manifold of metrics, we will eventually be able to
show, in Section \ref{sec:basic-metr-geom}, that the $L^2$ metric does in fact induce a
metric space structure.  However, the metric space topology does not
agree with the manifold topology, and so strange phenomena that are
absent for strong Riemannian metrics occur.  For example, we will
later show in Lemma \ref{lem:9} that there is no metric ball of
\emph{any} positive radius around \emph{any} point of the manifold of
metrics!  This is, of course, tied very closely to the analogous fact
about the tangent space.

For now, though, we put aside the nastier behavior of weak Riemannian
manifolds and show what results actually \emph{do} hold for them in
general.  They will necessarily be weaker than the results for strong
Riemannian manifolds, but they will still come in handy later on and
are of interest in their own right.

Our goal is to prove statements analogous to, but weaker than, the
theorems of Subsection \ref{sec:expon-mapp-dist}.  We will follow a
very similar course, making only minor modifications to the statements
and proofs in \cite{klingenberg95:_rieman_geomet} as necessary.

Our first theorem is familiar from finite-dimensional Riemannian
geometry and is quite simple to prove.

\begin{prop}\label{prop:3}
  Let $(N,\gamma)$ be a weak Riemannian manifold on which the
  Levi-Civita connection exists.  Let $p \in N$ and $v \in T_p N$, and
  suppose that $v$ is in the domain of $\exp_p$.  Then the geodesic
  $\alpha(t) := \exp_p (t v)$, $t \in [0,1]$, has length $\| v
  \|_\gamma$.
\end{prop}
\begin{proof}
  The proof for Riemannian Hilbert manifolds is algebraic in nature
  and so carries over to weak Riemannian manifolds---here we just give
  a sketch.  Since the Levi-Civita connection is metric, its parallel
  transport along any curve is an isometry of the tangent spaces.
  That $\alpha$ is a geodesic implies that $\alpha'(t)$ is parallel
  along $\alpha$, and hence $\alpha'(t)$ has constant length.  Since
  $\alpha'(0) = v$, this length is $\| v \|_\gamma$.
\end{proof}

Unfortunately, we cannot prove much more that is useful about weak
Riemannian manifolds without first making a couple of assumptions on
the exponential mapping.  Basically, we want it to exist and to be a
diffeomorphism between some open sets---so we'll have to assume that
as well.  The next bit of terminology incorporates this, and also adds
one technical detail that we'll soon need.

\begin{dfn}\label{dfn:20}
  We call a weak Riemannian manifold $(N, \gamma)$ \emph{normalizable
    at $x$} if there are open neighborhoods $U_x \subseteq T_x N$ and
  $V_x \subseteq N$ containing $0$ and $x$, respectively, such that
  \begin{enumerate}
  \item the exponential mapping $\exp_x$ exists and is a
    $C^1$-diffeomorphism between $U_x$ and $V_x$; and
  \item the following function is continuous:
    \begin{equation*}
      \begin{aligned}
        R : T_x N &\rightarrow \R_+ \\
        v &\mapsto \sup \{ r \in \R_+ \mid r \cdot v \in U_x \}.
      \end{aligned}
    \end{equation*}
  \end{enumerate}
  Note that the neighborhoods $U_x$ and $V_x$ are required to be open
  in the manifold topology of $N$.  We do not require that $U_x$ be
  open in the topology induced by $\gamma$.

  We call $(N, \gamma)$ \emph{normalizable} if it is normalizable at
  each $x \in N$.
\end{dfn}

\begin{dfn}
  Let $(N,\gamma)$ be a weak Riemannian manifold and let $x \in N$.
  We denote by $S_x N \subset T_x N$ the unit sphere, i.e.,
  \begin{equation*}
    S_x N = \{ v \in T_x N \mid \| v \|_\gamma = 1 \}.
  \end{equation*}
\end{dfn}

For the rest of this section, let $(N, \gamma)$ be a weak Riemannian
manifold that is normalizable at a point $x \in N$, and retain the
notation of Definition \ref{dfn:20}.

The following lemma shows that the exponential mapping of a weak
Riemannian manifold that is normalizable at $x$ is defined on some
nonzero vector pointing in each direction in $T_x N$.

\begin{lem}\label{lem:16}
  For each $v \in T_x N$, $R(v) > 0$.
\end{lem}
\begin{proof}
  Let $v \in T_x N$ be given.  Since $T_x N$ with its manifold
  topology is a topological vector space and $U_x$ is a neighborhood
  of the origin, there is some $\epsilon > 0$ such that $\epsilon
  \cdot v \in U_x$.
\end{proof}

\begin{rmk}
  Lemma \ref{lem:16} does not imply that $R(v)$ is uniformly bounded
  away from zero, even if we restrict the domain of $R$ to $S_x N$ at
  each $x \in N$.
\end{rmk}

This next proposition is the analog of Theorem \ref{thm:27}, and is
proved similarly.

\begin{prop}\label{prop:4}
  Let $r(s) \cdot v(s) \in U_x$, $s \in [0,1]$, be a path in $U_x$
  such that $v(s) \in S_x N $, $r(s) \in \R_{\geq 0}$.  (That is, we
  express the path in polar coordinates.)  We define a path $\alpha$
  by $\alpha(s) := \exp_x (r(s) v(s))$, $s \in [0,1]$.  Then
  \begin{equation*}
    L(\alpha) \geq | r(1) - r(0) |,
  \end{equation*}
  with equality if and only if $v(s)$ is constant and $r'(s) \geq 0$.
\end{prop}
\begin{proof}
  By Definition \ref{dfn:20} and Lemma \ref{lem:16}, as well as the
  compactness of $[0,1]$, there exist $\epsilon, \delta > 0$ such that
  if
  \begin{equation*}
    (s,t) \in U_{\epsilon,\delta} := 
    \left\{
      (s,t) \in \R^2 \mid s \in [0,1],\ t \in [-\epsilon, r(s) + \delta]
    \right\},
  \end{equation*}
  then $t \cdot v(s) \in U_x$.

  We define a one-parameter family of paths in $N$ by
  \begin{equation*}
    c_s (t) := \exp(t \cdot v(s)), \quad (s,t) \in U_{\epsilon,\delta}
  \end{equation*}
  Note that for each fixed $s$, the path $t \mapsto c_s(t)$ is a
  geodesic with
  \begin{equation}\label{eq:23}
    \| \partial_t c_s(t) \|_\gamma \equiv \| \partial_t c_s(0)
    \|_\gamma = \| v(s) \|_\gamma = 1.
  \end{equation}
  Note also that the image of the family of paths $c_\cdot (\cdot)$ is
  a singular surface in $N$ parametrized by the coordinates $(s,t)$.

  Keeping this in mind, we compute
  \begin{equation} \label{eq:24}
    \begin{aligned}
      \partial_t \gamma \left(
        \partial_s c_s(t), \partial_t c_s(t) \right) &= \gamma \left(
        \frac{\nabla}{\partial t} \partial_s c_s(t), \partial_t c_s(t)
      \right) + \gamma \left(
        \partial_s c_s(t), \frac{\nabla}{\partial t} \partial_t c_s(t)
      \right) \\
      &= \gamma \left( \frac{\nabla}{\partial s} \partial_t
        c_s(t), \partial_t c_s(t)
      \right) \\
      &= \frac{1}{2} \partial_s \gamma \left(
        \partial_t c_s(t), \partial_t c_s(t)
      \right) \\
      &= 0.
    \end{aligned}
  \end{equation}
  Here, the second line holds because
  \begin{itemize}
  \item $s$ and $t$ are coordinate functions, and hence (covariant)
    derivatives in the two directions commute, and
  \item $t \mapsto c_s(t)$ is a geodesic, hence
    $\frac{\nabla}{\partial t} \partial_t c_s(t) =0$.
  \end{itemize}
  The last line follows directly from (\ref{eq:23}).

  From (\ref{eq:24}), we immediately see that
  \begin{equation*}
    \gamma \left(
      \partial_s c_s(t), \partial_t c_s(t) \right)
  \end{equation*}
  is independent of $t$.  However, we also have that $c_s(0) = x$ for
  all $s$, implying that $\partial_s c_s(0) = 0$, thus
  \begin{equation*}
    0 = \gamma(\partial_s c_s(0), \partial_t c_s(0)) = \gamma(\partial_s
    c_s(t), \partial_t c_s(t))
  \end{equation*}
  for all $t$.  That is, $\partial_s c_s(t)$ and $\partial_t c_s(t)$
  are orthogonal for all $s$ and $t$.

  We now estimate:
  \begin{align*}
    \| \alpha'(s) \|_\gamma^2 &= \left\| \frac{d}{ds} c_s(r(s)) \right\|_\gamma^2 \\
    &= \left\| \partial_s c_s(r(s)) + r'(s) \partial_r c_s(r(s))
    \right\|_\gamma^2 \\
    &= \left\|
      \partial_s c_s(r(s)) \right\|_\gamma^2 + |r'(s)|^2 \left\|
      \partial_r c_s(r(s))
    \right\|_\gamma^2 \\
    &\geq |r'(s)|^2.
  \end{align*}
  Here, in the third line, we have used orthogonality of $\partial_s
  c_s(t)$ and $\partial_t c_s(t)$.  In the last line, we have used
  (\ref{eq:23}).  Note that equality holds if and only if
  $\left\| \partial_s c_s(r(s)) \right\|_\gamma \equiv 0$.

  Finally, we see that
  \begin{equation*}
    L(\alpha) = \int_0^1 \| \alpha'(s) \|_\gamma \, ds \geq \int_0^1
    |r'(s)| \, ds
    \geq
    \left|
      \int_0^1 r'(s) \, ds
    \right| = |r(1) - r(0)|,
  \end{equation*}
  which proves the desired inequality.  We note that the first
  inequality is an equality if and only if $\left\| \partial_s
    c_s(r(s)) \right\|_\gamma \equiv 0$ (see the previous paragraph) and the
  second inequality is an equality if and only if $r'(s) \geq 0$ for
  all $s$.
\end{proof}

Finally, we get the analog of Theorem \ref{thm:28}.  The remark
afterwards points out in what way this is weaker than that theorem,
however.

\begin{prop}\label{prop:2}
  Suppose $y \in V_x$ with $\exp_x^{-1}(y) = v$.  Then the path
  \begin{equation*}
    \alpha : [0,1] \rightarrow V_x, \quad \alpha(t) = \exp_x (t \cdot v)
  \end{equation*}
  satisfies $L(\alpha) = \| v \|_\gamma$, and $\alpha$ is of minimal
  length among all paths in $V_x$ from $x$ to $y$.  Furthermore,
  $\alpha$ is the unique minimal path (up to reparametrization) in
  $V_x$ from $x$ to $y$.
\end{prop}

\begin{rmk}
  Note that we will only show that $\alpha$ is minimal only among
  paths (or geodesics) in $V_x$, not all paths (or geodesics) in $N$.
  In particular, we cannot conclude from Proposition \ref{prop:2} that
  $d_\gamma (x,y) = L(\alpha)$.
\end{rmk}

\begin{proof}
  The equality $L(\alpha) = \| v \|_\gamma$ holds by Proposition
  \ref{prop:3}.
  
  A path $\eta(s)$, $s \in [0,1]$, in $V_x$ from $x$ to $y$
  corresponds via $\exp_x^{-1}$ to a path $r(s) \cdot v(s)$ in $U_x$
  with $v(s) \in S_x N$, $r(0) = 0$ and $r(1) \cdot v(1) = v$,
  implying $|r(1)| = \| v \|_\gamma$. By Proposition \ref{prop:4}, we
  therefore have that
  \begin{equation}\label{eq:25}
    L(\eta) \geq \| v \|_\gamma = L(\alpha),
  \end{equation}
  immediately implying minimality of $\alpha$.

  Let equality hold in (\ref{eq:25}).  Again by Proposition
  \ref{prop:4}, this implies $v(s)$ is constant and $r'(s) \geq 0$ for
  all $s$.  However, this means that $\eta$ is just a
  reparametrization of $\alpha$, proving the second statement.
\end{proof}

As an obvious result of Proposition \ref{prop:2}, we get the following
criterion for a weak Riemannian manifold to be a metric space.  It
requires rather strong assumptions which could probably be weakened
significantly, but it will be sufficient for some purposes that we have
in mind---specifically, we will use it to show that certain
submanifolds of the manifold of metrics are metric spaces.

\begin{thm}\label{thm:4}
  Let $(N,\gamma)$ be a weak Riemannian manifold.  Suppose that for
  some $x \in N$, the exponential mapping $\exp_x$ is a diffeomorphism
  between an open (in the manifold topology) neighborhood $U_x$ of
  $0 \in T_x N$ and $N$.
  
  Then $(N,d_\gamma)$, where $d_\gamma$ is the Riemannian distance function of
  $\gamma$, is a metric space.
\end{thm}
\begin{proof}
  Let $y \in N$.  It remains to show that if $y \neq x$, then $d(x,y)
  > 0$.  But if $\exp_x^{-1}(y) = v$, then Proposition \ref{prop:2}
  shows that the shortest path from $x$ to $y$ in $N$ is $\exp_x(t
  \cdot v)$, which has length $\| v \|$.  Therefore $d(x,y) = \| v \|
  > 0$.
\end{proof}

Proposition \ref{prop:2} of course cannot tell us anything about
whether a general weak Riemannian manifold $(N, \gamma)$ is a metric
space, and given the example of Subsection \ref{sec:path-behav-weak},
neither can any other theorem, since at a point $x \in
N$, the exponential mapping $\exp_x$ need not be defined on any
$\gamma(x)$-open neighborhood of $0 \in T_x N$.  This means that we
cannot use the exponential mapping directly to control the lengths of
curves between two chosen points.

\begin{figure}[t]
  \centering
  \includegraphics{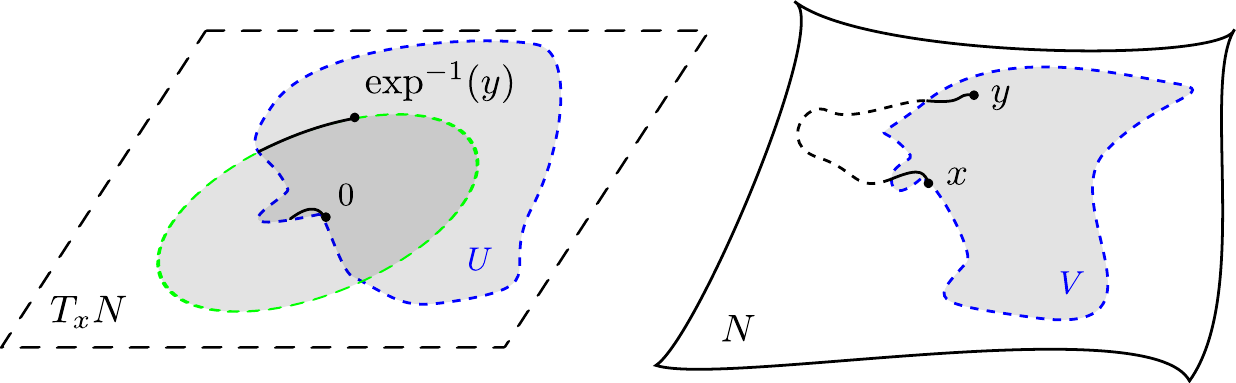}
  \caption{A path between two points in $V \subset N$ that travels
    out of and back into $V$ in possibly very short distance.  The
    dashed circle on the left represents the sphere of radius $\|
    \exp_x^{-1} (y) \|_{\gamma(x)}$ in $T_x N$.}
  \label{fig:1}
\end{figure}

Let's be more precise about this.  Assume that for some point $x \in
N$, $\exp_x$ is a diffeomorphism between open sets $U \in T_x N$ and
$V \in N$, but that $U$ contains no $\gamma(x)$-open ball.  Say we are
given a point $y \in N$, and let's even assume that $y \in V$ to
illustrate our point most dramatically.  We know that $U$ does not
contain any $\gamma(x)$-open ball around zero, and from Proposition
\ref{prop:2}, a radial path in $U$ is mapped by $\exp_x$ to a minimal
geodesic between its endpoints (minimal among the class of paths
remaining within $V$).  Thus we can imagine a radial path that starts
at $x$, leaves $V$ after an arbitrarily short distance, then reenters
$V$ such that its image under $\exp_x^{-1}$ lies on the sphere of
radius $\| \exp_x^{-1} (y) \|_{\gamma(x)}$.  This is illustrated in
Figure \ref{fig:1}.  In this case, the results stated so far do not
allow us to control the length of our path outside of $V$ or on the
second piece inside $V$, since Proposition \ref{prop:4} does not tell
us anything about paths with $r(t)$ constant (in the notation of that
proposition).  Our results therefore do not rule out paths of
arbitrarily small length.

\section{The manifold of metrics $\M$}\label{sec:manifold-metrics-m}

In this section, we define the manifold of smooth Riemannian metrics
$\M$ over a closed, finite-dimensional base manifold $M$.  We are
especially interested in the geometry of the so-called $L^2$ metric
$(\cdot, \cdot)$ on $\M$, which is a weak Riemannian metric on a
Fréchet manifold.  We will also discuss $\M^s$, the manifold of
Riemannian metrics with $H^s$ coefficients, which is a weak Riemannian
Hilbert manifold.  These objects will be defined in the first two
subsections.  In the third subsection, we will give a useful
decomposition of $\M$ into a product manifold, a decomposition that we
will refer back to later in the thesis.  Finally, we will mention some
facts about the geometry of $(\M, (\cdot, \cdot))$ that are already
known, such as formulas for its curvature and geodesics.

All of the facts in this section are culled from the three papers
\cite{ebin70:_manif_of_rieman_metric},
\cite{freed89:_basic_geomet_of_manif_of} and
\cite{gil-medrano91:_rieman_manif_of_all_rieman_metric}.  We refer the
reader to these for more details, and we will also reference specific
theorems at appropriate points.  We point out a few differences
between the papers.  The study of the geometry and topology of $(\M,
(\cdot, \cdot))$, as well as that of \emph{superspace} (the quotient
of $\M$ by the action of the diffeomorphism group) was initiated in
\cite{ebin70:_manif_of_rieman_metric} in the $H^s$ setting---i.e.,
this paper studied the manifold of metrics with $H^s$ coefficients
(see Subsection \ref{sec:defin-manif-metr}).  Much later,
\cite{freed89:_basic_geomet_of_manif_of} computed the curvature and
geodesics of $\M$ using some general theorems from the context of
strong Riemannian Hilbert manifolds.  Most of these general theorems
carry over to weak Riemannian manifolds, however, and the explicit
formulas of \cite{freed89:_basic_geomet_of_manif_of} all match up with
those of \cite{gil-medrano91:_rieman_manif_of_all_rieman_metric},
which computed the same things using tools strictly from the theory of
weak Riemannian manifolds.  Furthermore,
\cite{gil-medrano91:_rieman_manif_of_all_rieman_metric} computed the
analogs of Ricci curvature, scalar curvature and Jacobi fields in this
setting, and additionally did not require the base manifold $M$ to be
compact---simply without boundary.

\subsection{Definition of the manifold of metrics}\label{sec:defin-manif-metr}

Let $S^2 T^* M$ denote the second symmetric tensor power of the
cotangent bundle, and let $\s := \Gamma(S^2 T^* M)$ denote the vector
space of smooth, symmetric $(0,2)$-tensor fields on
$M$. \label{p:def-s} By the discussion in Subsection
\ref{sec:manifolds-mappings}, $\s$ is a Fréchet space with topology
coming from the $H^s$ norms induced by any smooth Riemannian metric
$g$ on $M$.  Furthermore, for $s \in \N \cup \{0\}$, we define $\s^s
:= H^s(S^2 T^* M)$, i.e., $\s^s$ is the vector space of $H^s$ sections
of $S^2 T^* M$. \label{p:def-ss} We equip $\s^s$ with the $H^s$ norm
induced by any smooth Riemannian metric $g$.

The first thing we note is that while the norm on $\s^s$ (and the
collection of norms on $\s$) depend on our choice of $g$, the
topologies of $\s$ and $\s^s$ do not.  This was pointed out in
Subsection \ref{sec:manifolds-mappings}.

Now, let $\M \subset \s$ and $\M^s \subset \s^s$ denote the subsets of
smooth Riemannian metrics and Riemannian metrics with $H^s$
coefficients, respectively. \label{p:manif-metr} That is, $\M$ and
$\M^s$ consist of those elements that induce positive definite scalar
products at each point.  We claim that for $s > n / 2$, $\M^s$ is an
open subset of $\s^s$, implying also that $\M$ is an open subset of
$\s$.  This follows easily from the Sobolev embedding theorem, for if
$s > n / 2$, then a bound on the $H^s$ norm of a tensor field implies
a bound on the $C^0$ norm.  Thus it is easy to see that if $g \in
\M^s$ is any $H^s$ (and hence continuous) metric and $h \in \s^s$ is
any tensor field with sufficiently small $H^s$ (and hence $C^0$) norm,
then $g + h$ will also be positive definite.  Note also that $\M$ and
$\M^s$ are positive cones, i.e., if $g_0$ and $g_1$ are metrics and
$\lambda, \mu > 0$, then $\lambda g_0 + \mu g_1$ is also a metric.

As open subsets of vector spaces, we trivially have that $\M$ is a
Fréchet manifold and $\M^s$ is a Hilbert manifold.  For the remainder
of the section, we will only discuss the manifold of smooth metrics
$\M$, as this is our main object of interest.  This is in the interest
of brevity and clarity of presentation only.  All results hold for
$\M^s$ as well if one uses $H^s$ objects instead of smooth objects and
puts a superscript ``$s$'' on all manifolds of mappings, i.e.,
considers spaces of $H^s$ instead of smooth mappings.  We will point
out a couple of examples along the way to show what we mean by this.

Since $\M$ is an open subset of $\s$, its tangent space at any point
$g \in \M$ is canonically identified with $\s$, i.e., $T_g \M \cong
\s$.  We will use this identification over and over throughout the
thesis.

\subsection{The $L^2$ metric}\label{sec:l2-metric}

Since $S^2 T^* M$ is a vector bundle associated to the tangent bundle,
a Riemannian metric $g \in \M$ induces a Riemannian metric on $S^2 T^*
M$.  Let's take a look at some fundamental linear algebra before we
write down this metric.

If $(V, \langle \cdot , \cdot \rangle_V)$ and $(W, \langle \cdot ,
\cdot \rangle_W)$ are vector spaces over the same ground field with
scalar products, we can form a scalar product on their tensor product
$V \otimes W$ in the following way.  For tensors of the form $v
\otimes w$, we define
\begin{equation}\label{eq:104}
  \langle v_1 \otimes w_1, v_2 \otimes w_2 \rangle_{V \otimes W} :=
  \langle v_1, v_2 \rangle_V \langle w_1, w_2 \rangle_W,
\end{equation}
and this definition is then extended via bilinearity to all of $V
\otimes W$.

The scalar product induced by the Riemannian metric $g$ on the
cotangent space $T^*_x M$ is given in local coordinates by
\begin{equation*}
  g(x)(\alpha, \beta) = g^{ij} \alpha_i \beta_j.
\end{equation*}
Hence, by \eqref{eq:104}, on the tensor product $T^*_x M \otimes T^*_x
M$, the scalar product on elements of the form $\alpha \otimes \beta$
is given by
\begin{equation*}
  g(x)(\alpha \otimes \beta, \gamma \otimes \delta) = g(x)(\alpha,
  \gamma) \cdot g(x)(\beta, \delta) = g^{ij} \alpha_i \gamma_j g^{kl}
  \beta_k \delta_l.
\end{equation*}
It is easy to see that the general formula, obtained by extending via
bilinearity, is the following.  For $h, k \in T^*_x M \otimes T^*_x
M$,
\begin{equation}\label{eq:106}
  g(x)(h, k) = g^{ij} h_{il} g^{lm} k_{jm}.
\end{equation}

\begin{rmk}\label{rmk:16}
  By the considerations above, a Riemannian metric $g \in \M$ gives
  rise to a Riemannian metric on any bundle associated to the tangent
  bundle (i.e., any bundle that can be built from $T M$ using tensor
  products, taking the dual, symmetrization, antisymmetrization etc.),
  since we know that $g^{-1}$ is a metric on $T^* M$ and
  \eqref{eq:104} shows us how to form a scalar product on tensor
  products of vector spaces.
\end{rmk}

Let's now restrict to symmetric tensors.  Denote by $\satx := S^2
T^*_x M$ \label{p:satx} the symmetrization of $T^*_x M \otimes T^*_x
M$.  Let $h, k \in \satx$, and let $H$ and $K$ be the tensors obtained
from $h$ and $k$, respectively, by raising an index with $g$.  Then
$H$ and $K$ are $(1,1)$-tensors, or in other words endomorphisms of
$T_x M$.  So in particular, we can multiply them.  Since they are
symmetric, we can use \eqref{eq:106} to get
\begin{equation}\label{eq:107}
  g(x)(h, k) = g^{ij} h_{il} g^{lm} k_{jm} = g^{ij} h_{il} g^{lm}
  k_{mj} = H^j_l K^l_j = \tr(H K) =: \tr_g(h k).
\end{equation}
The above expression is called the $g$-trace of $h k$.  It is
sometimes useful to write this in the notation of matrix
multiplication, so that
\begin{equation*}
  \tr_g(h k) = \tr(g^{-1} h g^{-1} k),
\end{equation*}
which is of course only valid in local coordinates.

If $h \in \satx$, we can similarly define its $g$-trace to be
\begin{equation*}
  \tr_g h = \tr H = \tr(g^{-1} h) = g^{ij} h_{ji}.
\end{equation*}

\begin{rmk}\label{rmk:8}
  Note that we could have defined a scalar product on $S^2 T^* M$ more
  generally by
  \begin{equation*}
    \tr_g(h k) + \alpha \tr_g (h) \tr_g(k)
  \end{equation*}
  for any $\alpha \in \R$.  These more general scalar products are
  studied in \cite{pekonen87:_dewit_metric}.  By setting $\alpha = 0$
  we get the $L^2$ metric back, and for $\alpha = -1$ we get the
  metric used by DeWitt \cite{dewitt67:_quant_theor_of_gravit}
  mentioned in Section \ref{sec:overv-prev-work}.  The scalar product
  is positive definite if $\alpha \geq -1 / n$; it is nondegenerate if
  $\alpha \neq - 1 / n$.

  There are two reasons we have chosen to study the $L^2$ metric in
  particular.  The first is that, as we have tried to show in this
  subsection, the $L^2$ metric arises canonically in the differential
  geometric context.  The second is the connection to Teichmüller
  theory that was mentioned in Section \ref{sec:motivation} and which
  will be elucidated in Chapter \ref{cha:appl-teichm-space}.
\end{rmk}

We now want to define the $g$-trace of a section or a product of two
sections of $S^2 T^* M$, which we can do by simply taking the
$g$-trace at each point.  We introduce the following notation for
this:

\begin{dfn}\label{dfn:21}
  Let $g$ be any Riemannian metric, and let $h$ and $k$ be elements of
  $\s$.  (We do not assume $g$, $h$ or $k$ to be smooth or even
  continuous.)  We denote the $g$-trace of $h k$ by
  \begin{equation*}
    \langle h, k \rangle_g := \tr_g(h k).
  \end{equation*}
  For each fixed choice of $g$, $h$, and $k$, it is a function mapping
  $M \rightarrow \R$.

  If it is necessary to explicitly denote at which point this
  expression is taken, we will write $\langle h(x) , k(x)
  \rangle_{g(x)}$.  Usually, though, the point $x$ will be clear from
  the context and omitted from the notation.
\end{dfn}

\begin{lem}\label{lem:48}
  For any fixed $g \in \M$ and $x \in M$, $\langle \cdot , \cdot
  \rangle_g$ is a positive definite scalar product on $\satx$.
  Furthermore, we can use it to define a smooth Riemannian metric
  $\langle \cdot , \cdot \rangle$ on the finite-dimensional manifold
  \begin{equation}\label{eq:120}
    \Matx := \{ g \in \satx \mid g > 0 \} = \{ g(x) \mid g \in \M \}
  \end{equation}
  by using the scalar product $\langle \cdot , \cdot \rangle_g$ on
  each tangent space $T_g \Matx = \satx$.  Of course, $g > 0$
  indicates that $g$ defines a positive-definite scalar product on
  $T_x M$.
\end{lem}
\begin{proof}
  We start with the proof that $\langle \cdot , \cdot \rangle_g$ is a
  positive-definite scalar product on $\satx$ for any fixed $g$.
  Bilinearity is clear, so we simply have to prove positive
  definiteness.  If $h \in \satx$, then
  \begin{equation*}
    \langle h , h \rangle_g = \tr(H^2)
  \end{equation*}
  by \eqref{eq:107}.  Let's fix any arbitrary coordinates around $x$
  and look at this expression locally.  From elementary linear
  algebra, we know that the trace of any matrix is equal to the sum of
  its eigenvalues.  Additionally, the eigenvalues of $H^2$ are the
  squares of the eigenvalues of $H$.  Therefore, if $\lambda_1^H,
  \dots, \lambda^H_n$ are the eigenvalues of $H$ and $h$ is nonzero,
  \begin{equation*}
    \tr(H^2) = (\lambda^H_1)^2 + \cdots + (\lambda^H_n)^2 > 0.
  \end{equation*}
  Of course, for this inequality to hold, we have to know that the
  eigenvalues of $H$ are real---but this was proved in Lemma
  \ref{lem:45}.  (Note that positive definiteness of $\langle \cdot ,
  \cdot \rangle_g$ actually also follows easily from that of $g$
  combined with \eqref{eq:104}.  Nevertheless, we will use the facts
  stated here later, so it is worthwhile to mention them.)

  As for the second statement, note first that $\Matx$ is indeed a
  finite-dimensional manifold, as it is an open set in the vector
  space $\satx$.  (This also gives us the identification of $T_g
  \Matx$ with $\satx$.)  Also, for any fixed $h, k \in \satx$, the
  function $g \mapsto \tr_g(h k) = \tr(g^{-1} h g^{-1} k)$ is clearly
  smooth over $\Matx$.  Combined with the positive definiteness of
  $\langle \cdot , \cdot \rangle_g$ for fixed $g$, this completes the
  proof that $\langle \cdot , \cdot \rangle$ is a Riemannian metric.
\end{proof}

Since each tangent space of $\M$ is identified with $\s$, a Riemannian
metric on $\M$ will give, for each Riemannian metric on $M$, a
positive scalar product on smooth sections of $S^2 T^* M$.  We have
just described a canonical positive definite scalar product on $S^2
T^*_x M$, and to pass to sections we do the obvious thing: we
integrate it.

\begin{dfn}\label{dfn:22}
  The \emph{$L^2$ metric} on $\M$ is defined to be
  \begin{equation*}
    (h, k)_g := \integral{M}{}{\tr_g(h k)}{\mu_g} \quad
    \textnormal{for all}\ h, k \in \s \cong T_g \M,
  \end{equation*}
  where $\mu_g$ is the volume form induced by $g$.

  For any given $g \in \M$, we denote by $\| \cdot \|_g$ the norm on
  $\s$ induced by $( \cdot , \cdot )_g$, that is,
  \begin{equation*}
    \| h \|_g := \sqrt{(h,h)_g} \quad \textnormal{for all}\ h \in \s.
  \end{equation*}

  Finally, we denote the distance function (a pseudometric) induced by
  $(\cdot, \cdot)$ simply by $d$.
\end{dfn}

The $L^2$ metric is indeed a smooth Riemannian metric---this is proved
in \cite[\S 4]{ebin70:_manif_of_rieman_metric}.  (In fact, $( \cdot ,
\cdot )$ is smooth in the $H^s$ topology on $\Ms$ for any $s > n/2$.)
We will not repeat the proof of smoothness here, but it is easy to see
bilinearity and positive definiteness---the latter follows simply from
positive definiteness of $\langle \cdot , \cdot \rangle_g$ at each
point of $M$.

The name of the $L^2$ metric is not there just for fun.  It is, in
fact, a weak Riemannian metric inducing the $L^2$ topology on each
tangent space, as the following theorem due to Palais \cite[\S
IX.2]{palais65:_semin_atiyah_singer_index_theor} shows.  (We have
already mentioned this theorem in Subsection
\ref{sec:manifolds-mappings}, but we restate it here in this context
and with an extra statement, the equivalence of the scalar products,
which is implied by the proofs in the above reference.)

\begin{thm}\label{thm:31}
  Let $g_0, g_1 \in \M$.  Then $( \cdot , \cdot )_{g_0}$ and $( \cdot
  , \cdot )_{g_1}$ are equivalent scalar products.  In particular,
  they both induce the same topology on $\s$, the $L^2$ topology.
\end{thm}

Ebin even pointed out in \cite[\S 4]{ebin70:_manif_of_rieman_metric}
that the theorem still holds if $g_0$ and $g_1$ are only assumed to be
continuous rather than smooth.

Let us make a brief technical note at this point.  When we use the
term ``$L^2$ topology'', what we really mean is that we give this name
to the topology induced from $( \cdot , \cdot )_g$ for some $g \in
\M$.  Of course, when we think of $L^2$ objects, we think of functions
that are square integrable, so we might ask whether a similar
interpretation holds for the completion of $\s$ with respect to $(
\cdot , \cdot )_g$.  In fact, looking at the coefficients of a tensor
field as local functions, defined over a coordinate chart, we claim that elements of the
completion of $\s$ with respect to $( \cdot , \cdot )_g$ are precisely
those tensor fields with coefficients that are $L^2$-integrable over
any chart.

The reason for this is that the proof of Theorem \ref{thm:31} is pointwise in
character---that is, not only are $( \cdot , \cdot )_{g_0}$ and $(
\cdot , \cdot )_{g_1}$ equivalent for any $g_0, g_1 \in \M$, but there
are constants $C,C' > 0$ such that for all $x \in M$ and $h,k \in \satx$,
\begin{equation*}
  \frac{1}{C} \tr_{g_1(x)}(h(x) k(x)) \leq \tr_{g_0(x)}(h(x) k(x)) \leq C \tr_{g_1(x)}(h(x) k(x))
\end{equation*}
and
\begin{equation*}
  \frac{1}{C'} \leq \left(
    \frac{\mu_{g_1}}{\mu_{g_0}} \right) (x) \leq C'.
\end{equation*}
(For the reader who desires more details, the proof of Lemma
\ref{lem:18} below will eventually make this clear.)  Thus, $( \cdot ,
\cdot )_{g_0}$ and $( \cdot , \cdot )_{g_1}$ are equivalent not just
on sections of $S^2 T^* M$ defined over all of $M$, but also
equivalent if we restrict them to sections defined only over some
subset of $M$.

Fix a coordinate chart $U$ and choose a metric $g^U$ with the property
that $g^U_{ij} \equiv \delta_{ij}$ on $U$.  Also fix an arbitrary
metric $g \in \M$.  On sections of $S^2 T^* M$ defined over $U$, $(
\cdot , \cdot )_g$ is equivalent to $( \cdot , \cdot )_{g^U}$ by the
arguments of the previous paragraph.  But the completion with respect
to $( \cdot , \cdot )_{g^U}$ of the space of sections of $S^2 T^* M$
over $U$ consists of exactly those sections with square integrable
coefficients, since locally
\begin{equation*}
  \tr_{g^U}(h k) = \delta^{ij} h_{il} \delta^{lm} k_{jm} = \sum_{i,j}
  h_{ij} k_{ij}.
\end{equation*}

This shows that the topology (and the completion) of $\s$ with respect
to $(\cdot, \cdot)_g$ is the same as the ``naive'' $L^2$ topology
coming from the (local) square integral of the coefficients of tensor
fields.  What this means is that we can use results on the $L^2$
topology for functions and apply them to $\s$ with the topology given
by $( \cdot , \cdot )_g$---always viewing the coefficients of tensor
fields as local functions.

\subsection{A product manifold structure for
  $\M$}\label{sec:prod-manif-struct}

Let's move on to studying the structure of $\M$ with respect to the
$L^2$ metric.  The goal of this subsection is to define a splitting of
$\M$ as the product of the set of metrics inducing the same volume
form and the set of volume forms on $M$.

Select any volume form $\mu \in \V$ and define
\begin{equation}\label{eq:115}
  \M_\mu := \{ g \in \M \mid \mu_g = \mu\ \},
\end{equation}
that is, $\M_\mu$ is the set of all metrics which induce the volume
form $\mu$.  Then $\M_\mu$ is a smooth submanifold of $\M$
(cf.~\cite[Lemma 8.8]{ebin70:_manif_of_rieman_metric}).

Consider the map $g \mapsto \mu_g$, mapping $\M$ to $\V$.  We
wish to compute the differential of this map, since this will help us
to figure out what the tangent space to $\M_\mu$ at a point is.  The
result is given in the following lemma.

\begin{lem}\label{lem:50}
  Let $g \in \M$ and $h \in \s$.  We have
  \begin{equation*}
    D \mu_g [h] = \frac{1}{2} \tr_g (h) \mu_g. 
  \end{equation*}
\end{lem}
\begin{proof}
  We wish to compute
  \begin{equation*}
    D \mu_g [h] = \left. \frac{d}{d t} \right|_{t=0} \mu_{g + t h}.
  \end{equation*}
  If we write this is local coordinates $x^1, \dots, x^n$ and let $I$
  denote the $n \times n$ identity matrix, we have
  \begin{equation}\label{eq:126}
    \begin{aligned}
      \left. \frac{d}{d t} \right|_{t=0} \mu_{g + t h} &=
      \left. \frac{d}{d
          t} \right|_{t=0} \sqrt{\det (g + t h)}\, dx^1 \cdots dx^n \\
      &= \left( \left. \frac{d}{d t} \right|_{t=0} \sqrt{\det (I + t
          g^{-1} h)} \right) \sqrt{\det g}\, dx^1 \cdots dx^n \\
      &= \left( \left. \frac{d}{d t} \right|_{t=0} \sqrt{\det (I + t
          g^{-1} h)} \right) \mu_g.
    \end{aligned}
  \end{equation}
  
  To compute the derivative term above, recall that for any square
  matrix $A$, $\exp \tr A = \det \exp A$, where $\exp$ is the matrix
  exponential.  Recall also that $\exp$ is a local diffeomorphism
  between a neighborhood $U$ of the zero matrix and a neighborhood $V$
  of the identity matrix in the space of $n \times n$ matrices.  So if
  $A_t$ is a one-parameter family of positive definite symmetric
  matrices with $A_0 = I$ and $A_t \in V$ then we can write $B_t =
  \log A_t$ uniquely.  This allows us to compute
  \begin{align*}
    \left. \frac{d}{d t} \right|_{t=0} \det A_t &= \left. \frac{d}{d
        t} \right|_{t=0} \det \exp B_t = \left. \frac{d}{d t}
    \right|_{t=0} \exp \tr B_t \\
    &= \left. (\tr B_t)' \exp \tr B_t \right|_{t=0} = \left. (\tr
      B_t)' \right|_{t=0},
  \end{align*}
  where the last equality follows from $B_0 = \log I = 0$.  Now, note
  that $A \mapsto \tr A$ is a linear map, so its differential is given
  by the map itself again.  Therefore
  \begin{equation}\label{eq:125}
    \left. \frac{d}{d t} \right|_{t=0} \det A_t = \left. (\tr B_t)' \right|_{t=0} = \tr B'_0.
  \end{equation}
  
  We now claim that $B'_0 = A'_0$.  This follows from the fact that
  for any matrix $X$,
  \begin{equation*}
    \left. \frac{d}{d t} \right|_{t=0} \exp(t X) = X,
  \end{equation*}
  If we define the notation $\Phi := \exp$ and $\Psi := \log$, we can
  write this another way:
  \begin{equation*}
    D \Phi (0) X = X \Longrightarrow D \Phi (0) = \id.
  \end{equation*}
  As $\Psi$ is the inverse function of $\Phi$ on the neighborhood $V$,
  $D \Psi(C) = D \Phi(\Psi(C))^{-1}$ for any $C \in V$.  Since
  $\Psi(I) = 0$, we have that $D \Psi(I) = D \Phi(0)^{-1} = \id$.
  This implies
  \begin{equation*}
    B'_0 = \left. \frac{d}{d t} \right|_{t=0} \log A_t = D \Psi(I)
    A'_0 = A'_0.
  \end{equation*}
  Substituting this into \eqref{eq:125}, we get
  \begin{equation*}
    \left. \frac{d}{d t} \right|_{t=0} \det A_t = \tr A'_0.
  \end{equation*}
  Using the above in \eqref{eq:126} and making a straightforward
  computation finally gives the result.
\end{proof}

Returning to $\M_\mu$, since the map $g \mapsto \mu_g$ is constant
over $\M_\mu$, we see from the previous lemma that
\begin{equation}\label{eq:112}
  T_g \M_\mu = \{ h \in \s \mid \tr_g h = 0 \}.
\end{equation}
That is, the tangent space to $\M_\mu$ at $g$ is given by the
$g$-traceless tensors.  Let us denote the set of $g$-traceless tensors
by $\st$.\label{p:def-stg}

Let $\mu \in \V$ be any smooth volume form on $M$.  Then, as pointed out in
Section \ref{sec:geom-prel}, for any $\nu \in \V$ there exists a
unique $C^\infty$ function, denoted $(\nu / \mu)$, such that
\begin{equation}\label{eq:110}
  \nu = \left( \frac{\nu}{\mu} \right) \mu.
\end{equation}
Furthermore, if $g \in \M$ and $f \in \pos$, i.e., $f$ is a smooth
positive function, then from the local expression for $\mu_g$
(cf.~\eqref{eq:119}) we see that
\begin{equation}\label{eq:109}
  \mu_{f g} = f^{n/2} \mu_g.
\end{equation}
From these facts, it is easy to see that if the metric $g$ induces the
volume form $\mu$ and $\nu \in \V$, then the unique metric conformal
to $g$ inducing the volume form $\nu$ is
\begin{equation*}
  \tilde{g} := \left( \frac{\nu}{\mu} \right)^{2/n} g.
\end{equation*}

This gives us the idea for a splitting of $\M$: by the considerations
of the last paragraph, there is a bijection between $\M$ and $\M_\mu
\times \pos$.  In Section \ref{sec:geom-prel}, we saw that $\pos$ is
diffeomorphic to $\V$, and so we can also say there is a bijection
between $\M$ and $\Mmu \times \V$.  This is more intuitive, as it
basically says that choosing a metric from $\M$ is the same as
choosing an element from $\Mmu$, which induces a fixed volume form,
and then picking a volume form.

In concrete terms, we define a map
\begin{equation}\label{eq:111}
  \begin{aligned}
    i_\mu : \M_\mu \times \V &\rightarrow \M \\
    (g, \nu) &\mapsto \left( \frac{\nu}{\mu} \right)^{2/n} g.
  \end{aligned}
\end{equation}
Thus, $i_\mu$ maps $(g, \nu)$ to the unique metric conformal to $g$
with volume form $\nu$.  It is straightforward to show that $i_\mu$ is
not only a bijection, but a diffeomorphism.

To compute the differential of $i_\mu$, recall from Section
\ref{sec:geom-prel} that $T_\nu \V = \Omega^n(M)$.  Also, since
$\alpha \mapsto (\alpha/\mu)$ is a linear map from $\Omega^n(M)$ to
$C^\infty(M)$, its differential is again given by the map itself,
i.e., $\beta \mapsto (\beta / \mu)$.  This gives us all we need to
compute
\begin{equation}\label{eq:117}
  D i_\mu(g,\nu) [h, \alpha] = \frac{2}{n} \left( \frac{\nu}{\mu}
  \right)^{\frac{2}{n} - 1} \left( \frac{\alpha}{\mu} \right) g + \left( \frac{\nu}{\mu} \right)^{2/n} h,
\end{equation}
where $\alpha \in \Omega^n(M) = T_\nu \V$ and $h \in \st = T_g
\M_\mu$.

As a submanifold of $\M$, $\M_\mu$ has a natural Riemannian metric
induced from the $L^2$ metric of $\M$.  We can use the map $i_\mu$ to
define a Riemannian metric on $\V$ as follows.  For every $g \in
\M_\mu$, we can embed $\V$ into $\M$ via
\begin{equation*}
  \V \overarrow{\cong} \{ g \} \times \V \overarrow{i_\mu} \M.
\end{equation*}
Let $(\!( \cdot , \cdot )\!)$ be the pullback of the $L^2$ metric $(
\cdot , \cdot )$ along this embedding.  To compute $(\!( \cdot , \cdot
)\!)$, note that $\mu_{i_\mu(g, \nu)} = \nu$, and for all $\alpha \in
\Omega^n(M)$ and $\nu, \lambda \in \V$, 
\begin{equation*}
  \left( \frac{\lambda}{\nu} \right)^{-1} = \bigg( \frac{\nu}{\lambda}
  \bigg) \quad \textnormal{and} \quad \bigg( \frac{\alpha}{\nu}
  \bigg) \bigg( \frac{\nu}{\lambda} \bigg) = \bigg(
    \frac{\alpha}{\lambda} \bigg).
\end{equation*}
Using this and \eqref{eq:117}, we can compute
\begin{align*}
  (\!( \alpha, \beta )\!)_\nu &= ( D i_\mu(g, \nu) [0, \alpha], D
  i_\mu(g,
  \nu) [0, \beta] )_{i_\mu(g, \nu)} \\
  &= \integral{M}{}{\tr_{(\nu/\mu)^{2/n} g}\left( \left( \frac{2}{n}
        \left( \frac{\nu}{\mu} \right)^{\frac{2}{n} - 1} \left(
          \frac{\alpha}{\mu} \right) g \right) \left(\frac{2}{n}
        \left( \frac{\nu}{\mu} \right)^{\frac{2}{n} - 1} \left(
          \frac{\beta}{\mu} \right) g \right) \right)}{\mu_{(\nu/\mu)^{2/n}
      g}} \\
  &= \integral{M}{}{\tr \left( \frac{4}{n^2} \left(
        \frac{\nu}{\mu} \right)^{-2} \left( \frac{\alpha}{\mu} \right)
      \left( \frac{\beta}{\mu} \right) I \right)}{\nu} \\
  &= \frac{4}{n} \integral{M}{}{\bigg( \frac{\alpha}{\nu} \bigg)
    \left( \frac{\beta}{\nu} \right)}{\nu}.
\end{align*}
Note that $(\!( \cdot , \cdot )\!)$ is actually independent of the
elements $g$ and $\mu$ we chose to define the embedding $\V
\hookrightarrow \M$, so it is a natural object.  In fact, $(\!( \cdot
, \cdot )\!)$ is just the constant factor $4/n$ times the most obvious
Riemannian metric on $\V$.

We know that $\V$ is diffeomorphic to $\pos$.  Furthermore, if $g \in
\M$ is any smooth metric, then the orbit of the conformal group $\pos$
through $g$, $\pos \cdot g$, is also diffeomorphic to $\pos$.  So
composing diffeomorphisms appropriately, we can also see that $\M
\cong \M_\mu \times \V \cong \M_\mu \times \pos \cdot g$.  Each
viewpoint may be useful, depending on the context.

The global splitting \eqref{eq:111} also, of course, gives a splitting
of the tangent space at each $g \in \M$.  Let's describe this briefly.

Let $(g, \nu) \in \M_\mu \times \V$.  From \eqref{eq:117}, it is easy
to see that
\begin{equation}\label{eq:113}
  D i_\mu(g, \nu) [0, T_\nu \V] = C^\infty(M) \cdot \left(
    \frac{\nu}{\mu} \right)^{2/n} g = C^\infty(M)
  \cdot g  
\end{equation}
In other words, the image of the tangent space of $\V$ under the
differential of $i_\mu$ is the set of \emph{pure trace tensors}.  Let
us denote the set of such tensors by $\sconf := C^\infty(M) \cdot
g$.\label{p:def-scg} (The superscript ``$c$'' stands for
``conformal.'')  Furthermore, we can also compute that if $f := (\nu /
\mu)^{n/2}$, then
\begin{equation}\label{eq:108}
  D i_\mu(g, \nu) [T_g \M_\mu, 0] = \s^T_{f g} = \s^T_{i_\mu(g,\nu)}.
\end{equation}
Note that this computation uses the fact that $\tr_{f g}(f h) = \tr_g
h$ for any $h \in \s$.

Let $g \in \M$ and $(\tilde{g}, \nu) := i_\mu^{-1}(g)$.  By
\eqref{eq:113} and \eqref{eq:108}, the splitting \eqref{eq:111} then
implies that
\begin{equation}\label{eq:114}
  T_g \M = \left( D i_\mu(\tilde{g}, \nu) [T_{\tilde{g}} \M_\mu, 0]
  \right) \oplus \left( D i_\mu(\tilde{g}, \nu) [0, T_\nu \V] \right) = \st \oplus \sconf.
\end{equation}
It is easy to see that this is, in fact, an \emph{orthogonal}
splitting of $T_g \M$ with respect to $(\cdot, \cdot)_g$.  For if $h
\in \st$ and $k = f g \in \sconf$, then we have
\begin{equation}\label{eq:131}
  \tr_g(h k) = \tr(g^{-1} h g^{-1} (f g)) = f \cdot \tr_g h = 0,
\end{equation}
since $\tr_g h = 0$ by assumption.

The splitting given in this subsection plays an important role in the
general theory of $\M$.  In particular, as we will see in the next
subsection, results on the curvature and geodesics of the $L^2$ metric
can be nicely stated and more easily visualized using this product
manifold structure.

\subsection{The curvature of $\M$}\label{sec:geod-curv-m}

The computation of the curvature (and, in the next section, of the
geodesics of $\M$) is greatly simplified by a heuristic consideration.
Namely, we can intuitively think of the $L^2$ metric on $\M$ as a
product metric with an infinite number of factors, one for each $x \in
M$.  ``Summing up'' the different terms in this product metric is then
done by integration.  Of course, this is only a formal construction,
but it is a useful practical aid.  More details about how this can be
made rigorous, in a much more general context, are given in
\cite[Appendix]{freed89:_basic_geomet_of_manif_of}.  In a case like
this, one often says that the computations are \emph{pointwise} in nature.

To illustrate what this means, we take the example of a geodesic in
$\M$.  A path in $\M$ is a one-parameter family $g_t$ of Riemannian
metrics on the base manifold $M$.  Thus, for every $x \in M$, $g_t(x)$
is a one-parameter family of positive definite elements of $\satx$
that glues together to a smooth metric over $M$ for each $t$.  The
geodesic $g_t$ is additionally completely determined by an initial
metric $g_0$ and a tangent vector $g'_0 \in T_{g_0} \M$.  When we say
that the geodesic equation is pointwise, what we mean is that we can
go one step further and say that the path the geodesic takes at a
point, $g_t(x)$, is determined completely by the values $g_0(x)$ and
$g'_0(x)$.

Now that we know what a pointwise computation is, we will keep these
considerations in mind as we continue.  However, before we can write
down formulas for the curvature of $\M$, we need to take care of an
issue that is technical in nature but central in its
implications---namely the existence of the Levi-Civita connection of
the $L^2$ metric.

Recall that in Subsection \ref{sec:levi-civita-conn}, we pointed out
that the Levi-Civita connection of a weak Riemannian manifold does not
necessarily exist.  If it does exist, however, it is unique.  The
problem was that the Koszul formula \eqref{eq:95} only guarantees the
existence of the Levi-Civita covariant derivative of a vector field at
a point as an element of the completion of the tangent space (with
respect to the Riemannian metric), not of the tangent space itself.

Thus, given two vector fields $h$ and $k$ on $\M$ ($h$ and $k$ are, at
each point of $\M$, smooth sections of $S^2 T^* M$), \eqref{eq:95} only
guarantees that the Levi-Civita covariant derivative $\nabla_h k|_g$
at a point $g \in \M$ is an element of $L^2(S^2 T^* M)$, since
$(\cdot, \cdot)$ induces the $L^2$ topology on each tangent space.

To show that the Levi-Civita connection does indeed exist, i.e., that
$\nabla_h k |_g$ is a smooth section of $S^2 T^* M$ for all vector
fields $h, k \in C^\infty (T \M)$ and all $g \in \M$, Ebin \cite[\S
4]{ebin70:_manif_of_rieman_metric} exhibited an explicit formula for
$\nabla_h k |_g$ on $\M^s$ and showed that $\nabla_h k|_g$ is $H^s$ if
$h$, $k$, and $g$ are.  Thus, it is also smooth if $h$, $k$, and $g$
are all smooth.  The precise formula is the following:
\begin{equation}\label{eq:116}
  \nabla_h k |_g = \left. \frac{d}{dt} \right|_{t=0} k(g + t h(g)) -
  \frac{1}{2} (h g^{-1} k + k g^{-1} h) +
  \frac{1}{4} ((\tr_g k) h + (\tr_g h) k - \tr_g(h k) g),
\end{equation}
where $h(g) \in T_g \M \cong \s$ is the value of the vector field $h$
at the basepoint $g$, and similarly $k(g + t h(g)) \in \s$ is the
value of $k$ at $g + t h(g)$ for small $t$.  (Bear in mind that a
smooth vector field on $\M$ is a smooth choice of an element of $\s$
for each $g \in \M$.)  It is easily seen from \eqref{eq:116} that
$\nabla_h k |_g$ is an $H^s$ section of $S^2 T^* M$ if $h$, $k$, and
$g$ are (see \cite[\S 4]{ebin70:_manif_of_rieman_metric} for an
explicit proof), and furthermore that this expression varies smoothly
with $g$.

Now that we know the Levi-Civita connection exists, we are assured
that the curvature and geodesics of $\M$ with its $L^2$ metric are
defined.  After some general discussion, the goal of this subsection
is to take a look at the curvature of $\M$ and, because it will play a
role later, also that of $\V$, $\pos$, and $\pos \cdot g$.

We now quote the theorem giving the curvature of $\M$:

\begin{thm}[{\cite[Prop.~2.6]{gil-medrano91:_rieman_manif_of_all_rieman_metric}}]\label{thm:32}
  Let $h,k,\ell \in T_g \M$, and let $H = g^{-1} h$, $K = g^{-1} k$
  and $L := g^{-1} \ell$.  We denote by $I$ the section of the
  endomorphism bundle $\textnormal{End}(M)$ that gives the identity
  map at each $x \in M$.

  The Riemannian curvature tensor of $\M$ with respect to the $L^2$
  metric is given by
  \begin{align*}
    g^{-1} R_g(h,k) \ell &= -\frac{1}{4} [[H,K], L] + \frac{n}{16} (\tr(H L) K - \tr
    (K L) H) \\
    &\quad + \frac{1}{16} (\tr(K) \tr(L) H - \tr(H) \tr(L) K) \\
    &\quad + \frac{1}{16} (\tr(H) \tr(K L) - \tr(K) \tr(H L)) I.
  \end{align*}
\end{thm}

\begin{rmk}\label{rmk:11}
  As is well known, in the literature on Riemannian geometry there are
  two conventions for defining the Riemannian curvature tensor.  The
  convention we use is the following.  If $N$ is a Riemannian manifold
  with Levi-Civita connection $\nabla$, we define
  \begin{equation*}
    R(X, Y) Z = \nabla_X \nabla_Y Z - \nabla_Y \nabla_X Z -
    \nabla_{[X,Y]} Z
  \end{equation*}
  for any vector fields $X$, $Y$ and $Z$ on $N$.

  The formula differs from the one in
  \cite{gil-medrano91:_rieman_manif_of_all_rieman_metric} by a
  negative sign, and is the same as the one used in
  \cite{freed89:_basic_geomet_of_manif_of}.
\end{rmk}

Here we make the observation that $R_g(h, k) \ell = 0$ if any one of
$h$, $k$ or $\ell$ is pure trace---i.e., of the form $f g$ for $f \in
C^\infty(M)$. This is readily checked using the above formula, but it
can also be seen via more geometric arguments (as is done in
\cite{freed89:_basic_geomet_of_manif_of}).  Using this observation, it
is possible to write the curvature in a more compact form, as well as
give a clean expression for the sectional curvature of $\M$.  For the
proof of the entire theorem we refer to the original source
\cite[Thm.~1.16 and Cor.~1.17]{freed89:_basic_geomet_of_manif_of}.

\begin{cor}\label{cor:6}
  Let notation be as in Theorem \ref{thm:32}.  If any of $h$, $k$, or
  $\ell$ is pure trace, then $R_g(h, k) \ell = 0$.  If $h, k, \ell \in
  \st$, then we have
  \begin{equation*}
    g^{-1} R_g (h, k) \ell = -\frac{1}{4} [[H, K], L] + \frac{n}{16}
    (\tr(H L) K - \tr(K L) H).
  \end{equation*}
  By the splitting \eqref{eq:114}, this determines the Riemannian
  curvature tensor completely.

  Furthermore, for $h, k \in \st$, the sectional curvature of $\M$ is
  given by
  \begin{equation*}
    K_g(h, k) = ( R_g(h, k) k, h)_g = \integral{M}{}{\left(
        \frac{1}{4} \tr([H, K]^2) + \frac{n}{16} (\tr(H K)^2 -
        \tr(H^2) \tr(K^2)) \right)}{\mu_g}.
  \end{equation*}
  If either of $h$ or $k$ is pure trace, then $K_g(h, k)$ vanishes.

  Finally, the above formula implies that $K_g(h, k) \leq 0$ for all
  $h, k \in T_g \M$ and all $g \in \M$.
\end{cor}

Using a result of Freed and Groisser \cite[Prop.~1.5]{freed89:_basic_geomet_of_manif_of}, we can also prove the following:

\begin{prop}\label{prop:22}
  Equip $\V$ with the weak Riemannian metric given by pullback along
  $i_\mu$ (see \eqref{eq:111}).  Equip $\pos$ with the metric given by
  pullback along the diffeomorphism \eqref{eq:127} with $\V$.
  Finally, for $g \in \M$, give the orbit $\pos \cdot g$ the metric it
  inherits as a submanifold of $\M$.

  Then $\V$, $\pos$, and $\pos \cdot g$ are all isometric.
  Furthermore, they are flat, i.e., their Riemannian curvature
  vanishes.
\end{prop}
\begin{proof}
  Freed and Groisser prove that $\V$ with the given metric is flat, so
  if we can show that $\V$, $\pos$, and $\pos \cdot g$ are all
  isometric, then the statement is immediate.

  By construction, it is clear that $\pos$ is isometric to $\V$ with
  the given metrics.
  
  As for $\pos \cdot g$, since $(\nu/\mu)$ can be any positive
  function given an appropriate choice of a volume form $\nu$, the
  image of $i_\mu$ is exactly $\pos \cdot g$.  Since $\V$ with the
  pullback metric is isometric to its image under $i_\mu$ as a
  submanifold of $\M$, this shows that $\pos \cdot g$ is isometric to
  $\V$.
\end{proof}

\subsection{Geodesics on $\M$}\label{sec:geodesics-m}

Now that we have given the curvature equation for $\M$, we'll take a
look at its geodesics.  It turns out that the geodesic equation can be
solved explicitly, and the result is the following (see
\cite[Thm.~2.3]{freed89:_basic_geomet_of_manif_of},
\cite[Thm.~3.2]{gil-medrano91:_rieman_manif_of_all_rieman_metric}):

\begin{thm}[The geodesic equation of $\M$]\label{thm:33}
  Let $g_0 \in \M$ and $h \in T_{g_0} \M = \s$.  Let $H := g_0^{-1} h$ and
  let $H^T$ be the traceless part of $H$.  Define two one-parameter
  families $q_t$ and $r_t$ of functions on $M$ as follows:
  \begin{equation*}
    q_t(x) := 1 + \frac{t}{4} \tr H, \quad r_t(x) :=
    \frac{t}{4} \sqrt{n \tr((H^T)^2)}.
  \end{equation*}
  Then the geodesic starting at $g_0$ with initial tangent $g'_0 = h$
  is given at each point $x \in M$ by
  \begin{equation*}
    g_t(x) =
    \begin{cases}
      \left( q_t^2(x) + r_t^2(x) \right)^{\frac{2}{n}} g_0(x) \exp \left(
        \frac{4}{\sqrt{n \tr(H^T(x)^2)}} \arctan
        \left( \frac{r_t(x)}{q_t(x)} \right) H^T(x) \right), & H^T(x)
      \neq 0, \\
      q_t(x)^{4/n} g_0(x), & H^T(x) = 0.
    \end{cases}
  \end{equation*}
  
  For precision, we specify the range of $\arctan$ in the above.  At a
  point where $\tr H \geq 0$, it assumes values in $(-\frac{\pi}{2},
  \frac{\pi}{2})$.  At a point where $\tr H < 0$, $\arctan(r_t / q_t)$
  assumes values as follows:
  \begin{enumerate}
  \item in $[0, \frac{\pi}{2})$ if $0 \leq t < - \frac{4}{\tr H}$,
  \item in $(\frac{\pi}{2}, \pi)$ if $- \frac{4}{\tr H} < t < \infty$,
  \end{enumerate}
  and we set $\arctan(r_t / q_t) = \frac{\pi}{2}$ if $t = -
  \frac{4}{\tr H}$.

  Finally, the geodesic is defined on the following domain.  If there
  are points where $H^T = 0$ and $\tr H < 0$, then let $t_0$ be the
  minimum of $\tr H$ over the set of such points.  In symbols,
  \begin{equation*}
    t_0 := \inf \{ \tr H(x) \mid H^T(x) = 0\ \textnormal{and}\ \tr
    H(x) < 0 \}.
  \end{equation*}
  Then the geodesic $g_t$ is defined for $t \in [0, -
  \frac{4}{t_0})$.

  If there are no points where both $H^T = 0$ and $\tr H < 0$, then
  $g_t$ is defined on $[0, \infty)$.
\end{thm}

\begin{rmk}\label{rmk:12}
  The geodesic given in Theorem \ref{thm:33} is parametrized
  proportionally to arc length.  That is, for each $\tau > 0$ such
  that $g_t$ is defined on $[0, \tau]$, we have
  \begin{equation*}
    L(g_t|_{[0,\tau]}) = \tau \| h \|_{g_0}.
  \end{equation*}
\end{rmk}

As for the distinguished submanifolds of $\M$ that we have studied,
their geodesics are given in the following two propositions.

\begin{prop}[{\cite[Prop.~2.1]{freed89:_basic_geomet_of_manif_of}}]\label{prop:23}
  If $g \in \M$, then $\pos \cdot g$ is a totally geodesic
  submanifold.  Therefore, the geodesic in $\pos \cdot g$ starting at
  $g_0$ with initial tangent $\alpha g_0$ is given by
  \begin{equation*}
    g_t = \left( 1 + n \frac{t}{4} \alpha \right)^{4/n} g_0.
  \end{equation*}
  As a result, the exponential mapping $\exp_{g_0}$ is a
  diffeomorphism from an open set $U \subset T_{g_0} (\pos \cdot
  g)$ onto $\pos \cdot g$.
\end{prop}

\begin{prop}[{\cite[Thm.~8.9]{ebin70:_manif_of_rieman_metric} and
    \cite[Prop.~1.27,
    Prop.~2.2]{freed89:_basic_geomet_of_manif_of}}]\label{prop:24}
  The submanifold $\M_\mu$ is not totally geodesic.  However, it is a
  globally symmetric space, and the geodesic starting at $g_0$ with
  initial tangent $g'_0 = h$ is given by
  \begin{equation*}
    g_t = g_0 \exp (t H),
  \end{equation*}
  where $H := g_0^{-1} h$.

  In particular, $\M_\mu$ is geodesically complete, and $\exp_g$ is a
  diffeomorphism from $T_g \M_\mu$ to $\M_\mu$ for any $g \in \M_\mu$.
\end{prop}

Note a consequence of Theorem \ref{thm:33} that is very important to
us.  Our goal being the description of the completion of $\M$, the
following corollary to Theorem \ref{thm:33} assures us that we
actually have something to study.

\begin{cor}\label{cor:18}
  The manifold of metrics is incomplete (geodesically and as a metric
  space) with respect to its $L^2$ metric.
\end{cor}
\begin{proof}
  Choose any $g_0 \in \M$, and choose any $h \in \s$ such that $\| h
  \|_{g_0} = 1$ and there is at least one point where $H^T = 0$ and
  $\tr H < 0$.  Then the geodesic $g_t$ starting at $g_0$ in the
  direction of $h$ has maximal domain of definition $[0, -
  \frac{4}{t_0})$, and its length over this domain is finite (in fact,
  equal to $- \frac{4}{t_0}$).
\end{proof}

We are interested in studying the completion of $\M$, and we have just
shown in Corollary \ref{cor:18} that it is incomplete.  Furthermore,
this is very simply expressed through the non-extensibility of a
geodesic, for which we have an explicit formula.  So it is worthwhile
to take a closer look at why geodesics can fail to be extensible, and
see what this does and does not tell us about the completion of $\M$.

First, by Theorem \ref{thm:33} a geodesic can fail to be forever
extensible only if it has a point where $h(x) = g'_0(x)$ is pure
trace, i.e., where $H^T(x) = 0$.  Why is this?  A quick look at the
geodesic equation provides the answer: if $H^T(x) \neq 0$ over all of
$M$, then $r_t(x) \neq 0$ as well.  This is because $\tr((H^T)^2) = 0$
if and only if $H^T = 0$, by Lemma \ref{lem:48}.  But then the scalar
coefficient in front, $(q_t^2 + r_t^2)^{2/n}$, is always positive.
Furthermore, since the matrix exponential maps symmetric matrices into
positive definite matrices and $H$ is $g_0$-symmetric (i.e., $g_0 H$
is symmetric), the exponential term does not destroy the
positive-definiteness of $g_0$.  Thus $g_t$ is positive-definite at
all points of $M$ for all $t \in [0,\infty)$, and hence is a metric
for all $t$.

Now, what can go wrong if $H^T(x) = 0$ for some $x \in M$?  In this
case, $r_t(x) = 0$ for all $t$, and the exponential term is absent in
the geodesic equation.  If we have $\tr H(x) \geq 0$, then $q_t(x) >
0$ for all $t$, and so again $g_t(x)$ is positive definite for all $t
\in [0, \infty)$.  But if $\tr H(x) < 0$, there is some $t'$ for which
$r_{t'}(x) = 0$.  Therefore, $g_{t'}(x) = 0$, and the geodesic has
left the manifold of metrics.  The geodesic can, however, be easily
identified with its limit point in the $C^\infty$ topology of $\s$,
which is a \emph{semimetric}, or a tensor field inducing a positive
semidefinite scalar product at each point.  However, only special
kinds of semimetrics can be realized as limit points of geodesics,
namely those that are either positive definite or zero at each point.
But it is also easy to convince oneself that all such semimetrics can
be realized as limit points of geodesics.  Thus we have arrived at our
first substantial piece of knowledge about the completion of $\M$.
Instead of writing it down as a proposition, we'll instead wait for
more general and rigorous statements to be made later.

A semimetric that is nonzero but not positive definite cannot be
realized in this way.  An extremely simple example is the semimetric
on the torus which is given in the standard chart by
\begin{equation*}
  g =
  \begin{pmatrix}
    1 & 0 \\
    0 & 0
  \end{pmatrix}.
\end{equation*}
Nevertheless, the question still remains as to whether such
semimetrics might also be representatives, in some sense that has to
be made precise, of points in the completion of $\M$.  Answering this
in the positive, in Section \ref{sec:meas-techn-result}, will be one
of our tasks in studying the completion and proving the main theorem.

Another question that presents itself at this point is whether a
finite path (or a Cauchy sequence, according to the correspondence
given in Section \ref{sec:compl-metr-spac}) in $\M$ can ``develop
infinities'', in the sense that one or more of its coefficients
becomes unbounded (in a fixed coordinate chart) as we run through its
domain of definition.  (See Definition \ref{dfn:25} below.)  Certainly
the coefficients of a geodesic always remain bounded on bounded
$t$-intervals.  Nevertheless, it will turn out that the metrics of a
finite-length path can develop infinities, but that the
infinities can also be neglected in a certain sense.  We will
explore this in Chapter \ref{cha:almost-everywh-conv}.

At this point, however, it will be profitable to give the exponential
mapping a somewhat closer inspection.

\begin{rmk}\label{rmk:15}
  We have not yet shown that $\M$ is a metric space and have already
  remarked that a weak Riemannian manifold does not always have a
  metric space structure, so in a sense Corollary \ref{cor:18} could
  be seen as a bit of a \emph{non sequitur}.  However, in Section
  \ref{sec:basic-metr-geom}, we
  will prove that $\M$ is a metric space with the distance function
  coming from $(\cdot, \cdot)$.  If we take this for granted, then the
  corollary makes sense.
\end{rmk}

\begin{rmk}\label{rmk:14}
  It should be noted that even if all geodesics on $\M$ could be
  extended indefinitely, it would not imply that $\M$ is complete.
  This is because the Hopf-Rinow theorem does not hold for all
  infinite-dimensional manifolds.  It does, in fact, hold for strong
  Riemannian Hilbert manifolds with nonpositive curvature (see
  \cite[\S 1.H]{mcalpin65:_infin_dimen_manif_and_morse_theor} and
  \cite[Cor.~IX.3.9]{lang95:_differ_and_rieman_manif}).  However, even
  though $\M$ has nonpositive curvature, it is a weak Riemannian
  manifold and so this theorem does not apply.
\end{rmk}

\subsection{A closer look at the exponential
  mapping}\label{sec:analys-expon-mapp}

In this subsection, we discuss the domain and range of the exponential
mapping.  The goal is to give an idea of why the exponential mapping
is an insufficient tool for studying the completion of $\M$.

First, though, let us use a childishly simple example to demonstrate
how the exponential mapping can sometimes be sufficient to describe
the completion of a Riemannian manifold.  Though this example is too
simple to be interesting, it illustrates an important philosophical
point and parallels the method we'll use in Section
\ref{sec:compl-orbit-space}.

If we take an open cylinder, say $N := S^1 \times (0,1)$, with its
standard flat metric $\gamma$, then the exponential mapping $\exp_x$
at any point $x \in N$ is an isometry from some open set $U \subseteq
T_x N$ onto $N$.  Therefore, we can identify the completion of $N$
with $\overline{U}$, the completion of $U$ with respect to
$\gamma(x)$.  Of course, this means that we can view the completion of
$N$ as equivalence classes of geodesics emanating from $x$.  If we
consider $N$ as being embedded in the closed cylinder $N' := S^1
\times [0,1]$, then two geodesics are equivalent if and only if they
have the same limit points as curves on $N'$.  This situation is
depicted in Figure \ref{fig:cyl}.

\begin{figure}[t]
  \centering
  \includegraphics[width=11cm]{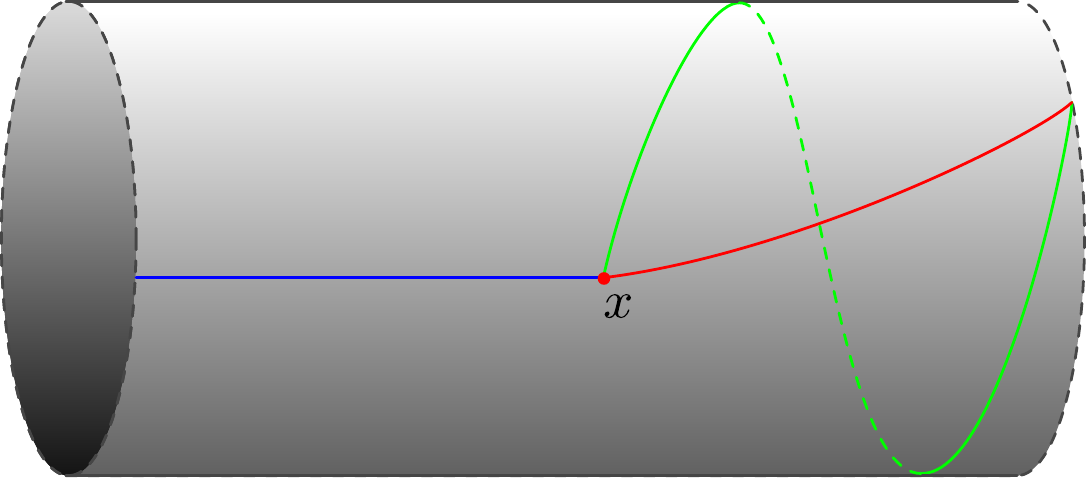}
  \caption{The open cylinder $S^1 \times (0,1)$.  The colored lines
    represent geodesics emanating from the point $x$, and via an
    appropriate geodesic we can reach any point on $S^1 \times
    \{0,1\}$.  The red and blue lines have different limit points and
    so are mutually inequivalent.  The green line, however, has the
    same limit point as the red one and so is equivalent to
    it.}\label{fig:cyl}
\end{figure}

There are two essential aspects of the above example that made it so simple to
treat.  First, $S^1 \times (0,1)$ is naturally embedded into a larger
space, $S^1 \times [0,1]$ (or even, if you like, $\R^3$) that contains
its completion.  Secondly, the exponential mapping is an isometry.

In our situation, studying the completion of $\M$, we luck out on the
first point, as $\M$ can be viewed as sitting inside the vector space
of all sections of $S^2 T^* M$---though of course we expect that this
space is much larger than necessary to accommodate the completion of
$\M$.  The second point certainly does not hold in our case---the
exponential mapping cannot be an isometry due to the fact that $\M$
has nonvanishing curvature.  But things are even worse, as it turns
out that the exponential mapping of $\M$ is \emph{highly
  nonsurjective}.  In a sense that we will see below, it is not even
locally surjective---so even the fact that the completion of a metric space
is very much a local concept does not help us here.

To see this, we can write down the domain and range of the exponential
mapping explicitly.  (This analysis is taken from \cite[\S 3.3 and
Thm.~3.4]{gil-medrano91:_rieman_manif_of_all_rieman_metric}.)  For $g
\in \M$, define the open set
\begin{equation*}
  U^g_x := \satx \setminus \left\{  \lambda g(x) \mid -\infty < \lambda \leq
  -\frac{4}{n} \right\} \subset T_{g(x)} \Matx.
\end{equation*}
Furthermore, we define an open subbundle $U^g \subset S^2 T^* M$ by
\begin{equation*}
  U^g = \bigcup_{x \in M} U^g_x.
\end{equation*}
(Note that $U^g$ is a \emph{fiber} bundle, not a vector bundle, so it
is a subbundle of $S^2 T^* M$ when viewed as a fiber bundle.)  Then it
is not hard to see from Theorem \ref{thm:33} that the maximal domain
of definition of $\exp_g$ consists of precisely the $C^\infty$
sections of $U^g$, i.e., those elements of $\s$ with image lying in
$U^g$.  Let us denote this by $\U^g := C^\infty(U^g)$.

Again from Theorem \ref{thm:33}, one can compute what the range of the
exponential mapping is, i.e., what $\exp_g(\U^g)$ is.  It turns out that
this is given by
\begin{equation*}
  \mathcal{A}^g := \left\{ g \exp H \mid H \in C^\infty(\textnormal{End}(M)),\
  \tr((H^T)^2) < \frac{1}{n} (4 \pi)^2 \right\} \subset \M,
\end{equation*}
where ``$\exp$'' in the above definition denotes the matrix
exponential and $H^T$ denotes the traceless part of $H$.

Of course, the set $\mathcal{A}^g$ omits many points of $\M$.  A
graphical illustration of this, which gives a very good impression of
just how remarkably nonsurjective $\exp_g$ is, can be found in
\cite[Fig.~1]{gil-medrano91:_rieman_manif_of_all_rieman_metric}.

One thing that goes right in this setting is the following theorem:

\begin{thm}\label{thm:34}
  For each $g \in \M$, $\exp_g$ is a real analytic diffeomorphism from
  $\U^g$ to $\mathcal{A}^g$.  Furthermore, if $\pi_\M : T \M
  \rightarrow \M$ is the projection from the tangent bundle of $\M$
  onto $\M$, then $(\pi_\M, \exp) : T \M \to \M \times \M$ is a real
  analytic diffeomorphism from an open neighborhood $\U$ of the zero
  section to an open neighborhood $\mathcal{A}$ of the diagonal.
  Explicitly,
  \begin{equation*}
    \U = \bigcup_{g \in \M} \U^g \quad \textnormal{and} \quad \mathcal{A} =
    \bigcup_{g \in \M} \mathcal{A}^g. 
  \end{equation*}

  These sets are maximal domains of definition for the exponential
  mapping and its inverse.
\end{thm}

This is a powerful theorem, and it certainly does not hold in general
for weak Riemannian manifolds (even if real analyticity is dropped).
However, its usefulness to us is limited.  The reason is that at any
point $g$, the neighborhood $\U^g$ does not contain any $L^2$-open
(i.e., $(\cdot, \cdot)_g$-open) set.  Therefore, we run into the
problem described at the end of Subsection
\ref{sec:expon-mapp-dist-1}---we get no information from the
exponential mapping about the distance between nearby points.
Therefore, we will have to revert to more direct methods of proof in
the coming chapters.

\section{Conventions}\label{sec:conventions}

Before we begin with the main body of the thesis, we will describe any
nonstandard conventions that will be used throughout the text.

The first thing we do is fix a reference metric, with respect to which
all standard concepts will be defined.

\begin{cvt}\label{cvt:3}
  For the remainder of the thesis, we fix an element $g \in \M$.
  Whenever we refer to the $L^p$ norm, $L^p$ topology, $L^p$
  convergence etc., we mean that induced by $g$ unless we explicitly
  state otherwise.  The designation nullset refers to
  Lebesgue measurable subsets of $M$ that have zero measure with
  respect to $\mu_g$.  If we say that something holds almost
  everywhere, we mean that it holds off of a $\mu_g$-nullset.

  If we have a tensor $h \in \s$, we denote by the capital letter $H$
  the tensor obtained by raising an index with $g$, i.e., locally
  $H^i_j := g^{ik} h_{kj}$.  Given a point $x \in M$ and an element $a
  \in \Matx$, the capital letter $A$ means the same---i.e., we assume
  some coordinates and write $A = g(x)^{-1} a$, though for readability
  we will generally omit $x$ from the notation.
\end{cvt}

Next, we'll fix an atlas of coordinates on $M$ that is convenient to
work with.

\begin{dfn}\label{dfn:1}
  We call a finite atlas of coordinates $\{(U_\alpha, \phi_\alpha)\}$
  for $M$ \emph{amenable} if for each $U_\alpha$, there exist a
  compact set $K_\alpha$ and a different coordinate chart $(V_\alpha,
  \psi_\alpha)$ (which does not necessarily belong to $\{ (U_\alpha,
  \phi_\alpha) \}$) such that
  \begin{equation*}
    U_\alpha \subset K_\alpha \subset V_\alpha \quad \textnormal{and}
    \quad \phi_\alpha = \psi_\alpha | U_\alpha.
  \end{equation*}\label{p:amenable-atlas}
\end{dfn}

\begin{cvt}\label{cvt:5}
  For the remainder of this thesis, we work over a fixed amenable
  coordinate atlas $\{(U_\alpha, \phi_\alpha)\}$ for all computations
  and concepts that require local coordinates.
\end{cvt}

The next lemma we'll prove shows one benefit of amenable coordinates:
smooth (or even continuous) metrics satisfy some kind of upper and
lower bounds in these coordinates.  Intuitively, the lemma says the
following: in amenable coordinates, the coordinate representations of
a smooth metric are somehow ``uniformly positive definite''.
Additionally, the coefficients satisfy a uniform upper bound.

\begin{lem}\label{lem:47}
  For any metric $\tilde{g} \in \M$, there exist constants
  $\delta(\tilde{g}) > 0$ and $C(\tilde{g}) < \infty$, depending only
  on $\tilde{g}$, with the property that for any $\alpha$, any $x \in
  U_\alpha$, and $1 \leq i,j \leq n$,
  \begin{equation}\label{eq:33}
    |\tilde{g}_{ij}(x)| \leq C(\tilde{g})\ \textnormal{and}\ \lambda^{\tilde{G}}_{\textnormal{min}}(x) \geq \delta(\tilde{g}),
  \end{equation}
  where we of course mean the value of $\tilde{g}_{ij}(x)$ in the
  chart $(U_\alpha, \phi_\alpha)$.
\end{lem}
\begin{proof}
  The lower bound on the minimal eigenvalue follows directly from
  Lemma \ref{lem:45}.

  The upper bound on the coefficients of $\tilde{g}$ follows from the
  fact that $| \tilde{g}_{ij} |$ is a continuous function in any given
  coordinate chart $(U_\alpha, \phi_\alpha)$, and we have assumed that
  $U_\alpha$ is contained in a compact set $K_\alpha$, which in turn
  is contained in another chart $(V_\alpha, \psi_\alpha)$ with
  $\phi_\alpha = \psi_\alpha | U_\alpha$.  Therefore, $|
  \tilde{g}_{ij} |$ is defined on $K_\alpha$ and assumes some maximum
  there---hence, it assumes some maximum $C_\alpha$ on $U_\alpha$.
  Since there are only finitely many charts $U_\alpha$, we can take
  $C(\tilde{g}) := \max_\alpha C_\alpha$.
\end{proof}

\begin{rmk}\label{rmk:2}
  The estimate $|\tilde{g}_{ij}(x)| \leq C(\tilde{g})$ also implies an
  upper bound in terms of $C(\tilde{g})$ on $\det \tilde{g}(x)$.  This
  is clear from the fact that the determinant is a homogeneous
  polynomial in $\tilde{g}_{ij}(x)$ with $n!$ terms and coefficients
  $\pm 1$.
\end{rmk}

The main point of using an amenable coordinate atlas is the following:
it gives us an easily understood and uniform---but nevertheless
coordinate-dependent---notion of how ``large'' or ``small'' a metric
is.  Namely, we look at how large the absolute values of its entries
are and how small its smallest eigenvalue is.  The dependence of this
notion on coordinates is perhaps somewhat dissatisfying at first
glance, but it should be seen as merely an aid in our quest to prove
statements that are, indeed, invariant in nature.

It is necessary to introduce somewhat more general objects than
Riemannian metrics in this thesis:

\begin{dfn}\label{dfn:24}
  Let $\tilde{g}$ be a section of $S^2 T^* M$.  Then $\tilde{g}$ is
  called a \emph{(Riemannian) semimetric} if it induces a positive
  semidefinite scalar product on $T_x M$ for each $x \in M$.
\end{dfn}

To make the above idea of uniformly largeness or positive definiteness
more precise for the case of a nonsmooth (semi)metric, we define two
notions.  The first is again some kind of ``uniform positive
definiteness'', and the second is a kind of uniform upper bound.

\begin{dfn}\label{dfn:23}
  Let $\tilde{g}$ be a semimetric on $M$ (which we do not assume to be
  even measurable).  Then $\tilde{g}$ is called \emph{inflated} if
  there exists a constant $\delta > 0$ such that
  \begin{equation*}
    \det \tilde{G}(x) \geq \delta
  \end{equation*}
  for a.e.~$x \in M$.  Otherwise $\tilde{g}$ is called
  \emph{deflated}.

  We define the set
  \begin{equation*}
    X_{\tilde{g}} := \{ x \in M \mid \tilde{g}(x)\ \textnormal{is not
      positive definite} \} \subset M,
  \end{equation*}
  which we call the \emph{deflated set} of $\tilde{g}$.

  We call $\tilde{g}$ \emph{bounded} if there exists a constant
  $C$ such that
  \begin{equation*}
    | \tilde{g}_{ij}(x) | \leq C
  \end{equation*}
  for a.e.~$x \in M$ and all $1 \leq i, j \leq n$.  Otherwise
  $\tilde{g}$ is called \emph{unbounded}.
\end{dfn}

Since the study of the completion of $\M$ boils down to the study of
Cauchy sequences in $\M$, it will turn out to be useful to define
notions related to the above for a sequence of elements of $\M$.

\begin{dfn}\label{dfn:25}
  Let $\{ g_k \} \subset \M$ be any sequence.  We define the sets
  \begin{align*}
    X_{\{ g_k \}} &:= \{ x \in M \mid \forall \delta > 0,\ \exists k
    \in \N\ \textnormal{s.t.}\ \det G_k(x) <
    \delta \}, \\
    S_{\{ g_k \}} &:= \{ x \in M \mid \forall C > 0,\ \exists k \in
    \N\ \textnormal{and}\ i,j \in \{1, \dots, n\}\ \textnormal{s.t.}\
    |(g_k)_{ij}(x)| > C \}.
  \end{align*}
  We call $X_{\{ g_k \}}$ the \emph{deflated set} and $S_{\{ g_k
    \}}$ the \emph{unbounded set} of $\{ g_k \}$.

  We say the sequence $\{g_k\}$ \emph{deflates at $x$} if $x \in
  X_{\{g_k\}}$.  We say it \emph{becomes unbounded at $x$} if $x
  \in S_{\{g_k\}}$.

  The sequence $\{ g_k \}$ is called \emph{inflated} if $X_{\{
    g_k \}}$ is a nullset, and \emph{deflated} otherwise.  It is
  called \emph{$C^0$-bounded} if $S_{\{ g_k \}}$ is a nullset, and
  \emph{$C^0$-unbounded} otherwise.
\end{dfn}

The last definition we need in this vein distinguishes elements of
smooth metrics from (possibly nonsmooth) semimetrics.

\begin{dfn}\label{dfn:8}
  A semimetric $\tilde{g}$ is called \emph{degenerate} if $\tilde{g}
  \not\in \M$, and \emph{nondegenerate} if $\tilde{g} \in \M$.
\end{dfn}

Note that by Remark \ref{rmk:13}, any measurable semimetric
$\tilde{g}$ on $M$ induces a nonnegative measure on $M$ that is
absolutely continuous with respect to the fixed volume form $\mu_g$.

A measurable Riemannian metric $\tilde{g}$ on $M$ gives rise to an
``$L^2$ scalar product'' on measurable functions in the following way.
For any two functions $\rho$ and $\sigma$ on $M$, we define
\begin{equation}\label{eq:121}
  (\rho, \sigma)_{\tilde{g}} = \integral{M}{}{\rho \sigma}{\mu_{\tilde{g}}}.
\end{equation}
(We denote this by the same symbol as the $L^2$ scalar product on
$\s$; which is meant will always be clear from the context.)  We put
``$L^2$ scalar product'' in quotation marks because unless we put
specific conditions on $\rho$, $\sigma$, and $\tilde{g}$,
\eqref{eq:121} is not guaranteed to be finite.  It suffices, for
example, to demand that $\rho$ and $\sigma$ are continuous and that
the total volume $\Vol(M, \tilde{g}) =
\integral{M}{}{}{\mu_{\tilde{g}}}$ of $\tilde{g}$ is finite.  As in
the case of the $L^2$ scalar product on $\s$, if $g_0$ and $g_1$ are
both continuous metrics, then $(\cdot, \cdot)_{g_0}$ and $(\cdot ,
\cdot)_{g_1}$ are equivalent scalar products on $C^\infty(M)$.
Therefore they induce the same topology, which we call the \emph{$L^2$
  topology}.

Now, let $\tilde{g}$ be a measurable semimetric---we want to introduce
a scalar product on functions induced from $\tilde{g}$ as well.  As a
semimetric, $\tilde{g}$ induces a nonnegative $n$-form in the same way
that a metric induces a volume form.  Locally, this is given by
\begin{equation*}
  \mu_{\tilde{g}} := \sqrt{\det \tilde{g}}\, dx^1 \cdots dx^n.
\end{equation*}
At points $x$ where $\tilde{g}(x)$ is not positive definite, we have
$\det \tilde{g}(x) = 0$ by Proposition \ref{prop:7}.  Therefore
$\mu_{\tilde{g}}(x) = 0$ as well, so $\mu_{\tilde{g}}$ is a volume
form if and only if $\tilde{g}$ is a metric.  Nevertheless, since it
is measurable and nonnegative, $\mu_{\tilde{g}}$ induces a Lebesgue
measure on $M$, and so we can define a positive semidefinite ``$L^2$
scalar product'' on functions via \eqref{eq:121}.  (It is only
positive semidefinite since if a function $\rho$ has the property that
$\supp \rho \cap M \setminus X_{\tilde{g}}$ has measure zero, then
$(\rho, \rho)_{\tilde{g}} = 0$.)  Again, if we want this to be a true
(finite) scalar product, we should, e.g., restrict to continuous
functions and finite-volume $\tilde{g}$ (those for which
$\integral{M}{}{}{\mu_{\tilde{g}}} < \infty$).

We define one more piece of notation before we close this section.

\begin{dfn}\label{dfn:27}
  By $\M_f$, we denote the space of measurable semimetrics on $M$ with
  finite volume, that is, semimetrics $\tilde{g}$ for which
  \begin{equation*}
    \integral{M}{}{}{\mu_{\tilde{g}}} < \infty.
  \end{equation*}
\end{dfn}


\chapter{First metric properties of $\M$}\label{cha:init-metr-prop}

In this chapter, we study the most easily accessible properties of
$\M$ as a metric space, which will form the basis for our continuing
investigations in later chapters.  Our first task, to be completed in
Section \ref{sec:basic-metr-geom}, is to show that $(\M, d)$ has the
structure of a metric space.  As we demonstrated in Subsection
\ref{sec:path-behav-weak}, this is not automatic for weak Riemannian
manifolds like $(\M, (\cdot, \cdot))$---the induced distance function
is only guaranteed to be a pseudometric.  Proving that $d$ is a metric
will be done by finding a manifestly positive-definite metric (in the
sense of metric spaces) on $\M$ that in some way bounds the
$d$-distance between two points from below, implying that it is
positive.

With this fact proved, we can move on to studying the completion of
$\M$, with the reassurance that the answer will be interesting.  (It's
of course of little interest to study the completion of a space in
which all points have zero distance from one another, as in Subsection
\ref{sec:path-behav-weak}.)  The strategy for obtaining the completion
will be to study the completions first of simple subspaces and then of
successively more complex subspaces of $\M$, until we have enough
information to describe the completion of the full space.

To begin this program, in Section \ref{sec:compl-an-amen}, we obtain
the completion of any so-called \emph{amenable subset}.  Recall that
in Definition \ref{dfn:23} we have defined two separate ``good''
properties of nonsmooth metrics.  The first is being bounded,
heuristically not becoming too large at any points.  The second is
inflation, heuristically not becoming too small.  Lemma \ref{lem:47}
shows that smooth metrics are both inflated and bounded, but the
constants of Lemma \ref{lem:47} depend on the metric in question.  An
amenable subset is one for which these constants can be chosen
\emph{uniformly} across the entire subset.  These subsets have the
nice property that the metric $d$ is equivalent to the metric induced
from the $L^2$ norm $\| \cdot \|_g$, in the sense that their Cauchy
sequences are the same.  This allows us to identify the completion of
an amenable subset with the $L^2$ completion of that subset.  This is
the first step in the strategy of bootstrapping our way to a
description of the completion.

\section{$\M$ is a metric space}\label{sec:basic-metr-geom}

As we have already remarked in Propositions \ref{prop:23} and
\ref{prop:24}, the exponential mappings of $\poss$ and $\Mmus$ are at
each point diffeomorphims between an open neighborhood in the tangent
space and the manifold itself.  Therefore, they both satisfy the
hypotheses of Theorem \ref{thm:4}, and we immediately get the
following two results.

\begin{thm}\label{thm:22}
  Let $\tilde{g} \in \M$.  Then $(\pos \cdot \tilde{g}, (\cdot,
  \cdot))$ is a metric space, where $(\cdot, \cdot)$ denotes the
  restriction of the $L^2$ metric on $\M$ to $\pos \cdot \tilde{g}$.
\end{thm}

\begin{thm}\label{thm:23}
  Let $\mu$ be any smooth volume form on $M$.  Then $(\Mmu, (\cdot,
  \cdot))$ is a metric space, where $(\cdot, \cdot)$ denotes the
  restriction of the $L^2$ metric on $\M$ to $\Mmu$.
\end{thm}

As we remarked at the end of Subsection \ref{sec:analys-expon-mapp},
we cannot infer any lower bounds on the distance between two points of
$\M$ from the exponential mapping, so we will have to directly find
these bounds.  To do this, we will first show Lipschitz continuity of
the function mapping a metric to the square root of its volume.  This
simple lemma will have far-reaching implications for our study.  The
first use of this lemma on the volume function is to aid us in
obtaining the lower bound on the $d$-distance between two points that
was described in the introduction.  This is, of course, after we
introduce an appropriate metric to bound $d$.

\subsection{Lipschitz continuity of the square root of the volume}\label{sec:lipsch-cont-square}

As just mentioned, we wish to show Lipschitz continuity of the square
root of the volume on $\M$.  In fact, the following lemma shows that
for any measurable $Y \subseteq M$, the function defined by
\begin{equation*}
  \tilde{g} \mapsto \sqrt{\Vol(Y, \tilde{g})}
\end{equation*}
is Lipschitz with respect to $d$.  Using this as a first step to
proving that $d$ is a metric takes its inspiration from \cite[\S
3.3]{michor05:_vanis_geodes_distan_spaces_of}.

\begin{lem}\label{lem:13}
  Let $g_0, g_1 \in \M$.  Then for any measurable subset $Y \subseteq
  M$,
  \begin{equation*}
    \left| \sqrt{\Vol(Y,g_1)} - \sqrt{\Vol(Y,g_0)} \right| \leq \frac{\sqrt{n}}{4} d(g_0,g_1).
  \end{equation*}
\end{lem}
\begin{proof}
  Let $g_t$, $t \in [0,1]$, be any path from $g_0$ to $g_1$, and
  define $h_t := g'_t$.  We compute
  \begin{equation}\label{eq:26}
    \begin{aligned}
      \partial_t \Vol(Y, g_t) &= \partial_t \int_Y \, \mu_{g_t} =
      \int_Y \partial_t \, \mu_{g_t} = \int_Y \frac{1}{2} \tr_{g_t}
      (h_t)
      \, \mu_{g_t} \\
      &\leq \left( \int_Y \, \mu_{g_t} \right)^{1/2} \left(
        \frac{1}{4}
        \int_Y \tr_{g_t}(h_t)^2 \, \mu_{g_t} \right)^{1/2} \\
      &\leq \frac{1}{2} \sqrt{\Vol(Y,g_t)} \left( \int_M
        \tr_{g_t}(h_t)^2 \, \mu_{g_t} \right)^{1/2},
    \end{aligned}
  \end{equation}
  where the first line follows from Lemma \ref{lem:50}, the second
  line follows from Hölder's inequality, and the last line from the
  nonnegativity of $\tr_{g_t}(h_t)^2$.  Now, let $A$ and $B$ be any $n
  \times n$ matrices, and denote their traceless parts by $A^T$ and
  $B^T$, respectively.  We then have the formula
  \begin{equation}\label{eq:64}
    \begin{aligned}
      \tr(AB) &= \tr\left(\left(A^T + \frac{1}{n} \tr(A)
          I\right)\left(B^T + \frac{1}{n} \tr(B) I\right)\right) \\
      &= \tr\left(A^T B^T\right) + \frac{1}{n} \tr(A) \tr(B).
    \end{aligned}
  \end{equation}
  The second line follows from the fact that traceless and pure trace
  matrices are orthogonal in the scalar product defined by $\tr(AB)$
  (cf.~\eqref{eq:131}---the computation is still valid if the matrices
  in question are not symmetric).  We have also used $\tr I = n$.

  Using \eqref{eq:64} with the $g_t$-trace and $A = B = h_t$, we see
  that
  \begin{equation*}
    \tr_{g_t}(h_t^2) = \tr_{g_t}\left((h^T_t)^2\right) + \frac{1}{n} \tr_{g_t}(h_t)^2,
  \end{equation*}
  implying
  \begin{equation*}
    \tr_{g_t}(h_t)^2 = n \left( \tr_{g_t}(h_t^2) -
      \tr_{g_t}\left((h^T_t)^2\right) \right) \leq n \tr_{g_t}(h_t^2),
  \end{equation*}
  since $\tr_{g_t}\left((h^T_t)^2\right) \geq 0$.  Applying this to
  (\ref{eq:26}) gives
  \begin{equation}\label{eq:27}
    \begin{aligned}
      \partial_t \Vol(Y, g_t) &\leq \frac{1}{2} \sqrt{\Vol(Y,g_t)}
      \left( n \int_M \tr_{g_t}(h_t^2) \, \mu_{g_t} \right)^{1/2} \\
      &\leq \frac{\sqrt{n}}{2} \sqrt{\Vol(Y,g_t)} \| h_t \|_{g_t}.
    \end{aligned}
  \end{equation}

  We next compute
  \begin{equation}\label{eq:49}
    \begin{aligned}
      \sqrt{\Vol(Y,g_1)} - \sqrt{\Vol(Y,g_0)} &= \int_0^1 \partial_t
      \sqrt{\Vol(Y,g_t)} \, dt = \int_0^1 \frac{1}{2} \frac{\partial_t
        \Vol(Y,g_t)}{\sqrt{\Vol(Y,g_t)}} \, dt \\
      &\leq \int_0^1 \frac{\sqrt{n}}{4} \| h_t \|_{g_t} \, dt 
      = \frac{\sqrt{n}}{4} L(g_t),
    \end{aligned}
  \end{equation}
  where the inequality follows from (\ref{eq:27}).  Since this holds
  for all paths from $g_0$ to $g_1$, and we can repeat the computation
  with $g_0$ and $g_1$ interchanged, it implies the result
  immediately.
\end{proof}

We note that Lemma \ref{lem:13} gives a positive lower bound on the
distance between two metrics in $\M$ that have different total
volumes---so we must now deal with the case where the two metrics have
the same total volume.

\subsection{A (positive definite) metric on
  $\M$}\label{sec:another-metric-m}

Our strategy for proving that $\M$ is a metric space is to find a
different metric (in the sense of metric spaces) on $\M$, the positive
definiteness of which is apparent and which bounds $d$ from below in
some way.  We do this in several steps.  The first is to define a
function on $\M \times \M$ and show that it is indeed a metric.

\begin{dfn}\label{dfn:15}
  Consider $\M_x = \{ \tilde{g} \in \satx \mid \tilde{g} > 0 \}$
  (cf.~\eqref{eq:120}).  Define a Riemannian metric $\langle \cdot ,
  \cdot \rangle^0$ on $\M_x$ given by
  \begin{equation*}
    \langle h , k \rangle^0_{\tilde{g}} = \tr_{\tilde{g}} (h k) \det
    g(x)^{-1} \tilde{g} \quad \forall h, k \in T_{\tilde{g}} \M_x \cong
    \satx.
  \end{equation*}
  (Recall that $g \in \M$ is our fixed reference element.)  We denote
  by $\theta^g_x$ the Riemannian distance function of $\langle \cdot ,
  \cdot \rangle^0$.
\end{dfn}

Note that $\theta^g_x$ is automatically positive definite, since it is
the distance function of a Riemannian metric on a finite-dimensional
manifold.  By integrating it in $x$, we can pass from a metric on
$\M_x$ to a function on $\M \times \M$ as follows:

\begin{dfn}\label{dfn:16}
  For any measurable $Y \subseteq M$, define a function $\Theta_Y : \M
  \times \M \rightarrow \R$ by
  \begin{equation*}
    \Theta_Y(g_0, g_1) = \integral{Y}{}{\theta^g_x(g_0(x), g_1(x))}{\mu_g(x)}.
  \end{equation*}
\end{dfn}

We have omitted the metric $g$ from the notation for $\Theta_Y$.  The
next lemma justifies this choice.

\begin{lem}\label{lem:28}
  $\Theta_Y$ does not depend on the choice of $g \in \M$ in the above
  definition.  That is, if we choose any other $\tilde{g} \in \M$ and
  define $\langle \cdot , \cdot \rangle^0$ and $\theta^{\tilde{g}}_x$
  with respect to this new reference metric, then
  \begin{equation*}
    \integral{Y}{}{\theta^g_x(g_0(x), g_1(x))}{\mu_g(x)} =
    \integral{Y}{}{\theta^{\tilde{g}}_x(g_0(x), g_1(x))}{\mu_{\tilde{g}}(x)}
  \end{equation*}
\end{lem}
\begin{proof}
  Let $\tilde{g} \in \M$ be any other metric.  Recall that
  $\theta^g_x$ was the distance function associated to the Riemannian
  metric $\langle \cdot, \cdot \rangle^0$ on $\Matx$, and the metric
  $g$ enters in the definition of this Riemannian metric.  Take a path
  $g_t(x)$ in $\Matx$.  For now, let's put $g$ and $\tilde{g}$ back in
  the notation, so that we can write formulas unambiguously.  For
  example, if we use $g$ to define $\langle \cdot, \cdot \rangle^0$,
  we write $L_g(g_t(x))$ for the length of $g_t(x)$ w.r.t.~$\langle
  \cdot , \cdot \rangle^0$; if we use $\tilde{g}$ in the definition,
  we write $L_{\tilde{g}}(g_t(x))$ for the length; and similarly for
  other notation.

  Using the definitions of $\Theta^g_Y$ and $\theta^g_x$, where infima
  are always taken over paths $g_t(x)$ from $g_0(x)$ to $g_1(x)$, and
  where $h_t(x) := g_t(x)'$, we
  can compute:
  \begin{align*}
    \Theta^g_Y(g_0,g_1) &= \int_Y \theta^g_x(g_0(x), g_1(x)) \, \mu_g(x) \\
    &= \int_Y \left( \inf L_g(g_t(x)) \right) \, \mu_g(x) \\
    &= \int_Y \left( \inf \int_0^1 \sqrt{\langle g_t(x)', g_t(x)'
        \rangle^0_{g_t(x)}} dt \right) \, \mu_g(x) \\
    &= \int_Y \left( \inf \int_0^1 \sqrt{\tr_{g_t(x)}(h_t(x)^2)
        \frac{\det g_t(x)}{\det g(x)}} \, dt \right) \sqrt{\det g(x)} \, dx^1 \cdots
    dx^n \\
    &= \int_Y \left( \inf \int_0^1 \sqrt{\tr_{g_t(x)}(h_t(x)^2)
        \frac{\det g_t(x)}{\det g(x)}} \sqrt{\det g(x)} \, dt \right) \, dx^1 \cdots
    dx^n \\
    &= \int_Y \left( \inf \int_0^1 \sqrt{\tr_{g_t(x)}(h_t(x)^2)
        \frac{\det g_t(x)}{\det \tilde{g}(x)}} \sqrt{\det \tilde{g}(x)} \, dt \right) \, dx^1 \cdots
    dx^n \\
    &= \Theta^{\tilde{g}}_Y(g_0,g_1),
  \end{align*}
  where the last line follows from running the first lines of the
  computation through in reverse.
\end{proof}

\begin{lem}\label{lem:44}
  Let any $Y \subseteq M$ be given.  Then $\Theta_Y$ is a pseudometric
  on $\M$, and $\Theta_M$ is a metric (in the sense of metric spaces).
  
  Furthermore, if $Y_1 \subset Y_2$, then $\Theta_{Y_1}(g_0, g_1)
  \leq \Theta_{Y_2}(g_0, g_1)$ for all $g_0, g_1 \in \M$.
\end{lem}
\begin{proof}
  Nonnegativity, vanishing distance for equal elements, symmetry and
  the triangle inequality are clear from the corresponding properties
  for $\theta^g_x$.

  That $\Theta_M$ is positive definite is also not hard to prove.
  Since $\theta^g_x$ is a metric on $\M_x$, $\theta^g_x(g_0(x),
  g_1(x)) > 0$ whenever $g_0(x) \neq g_1(x)$.  But since $g_0$ and
  $g_1$ are smooth metrics, if they differ at a point, they differ
  over an open neighborhood of that point.  Hence the integral of
  $\theta^g_x(g_0(x), g_1(x))$ must be positive.

  The second statement follows immediately from nonnegativity of
  $\theta^g_x$.
\end{proof}

\subsection{Proof of the main result}\label{sec:proof-main-result}

We have set up everything we need to prove the main result of this
section---that $d$ is a metric.  To do this, we use Lemma \ref{lem:13}
in order to control the volume of the metrics making up a path in
terms of the length of that path, combined with a Hölder's inequality
argument, and show that the pseudometrics $\Theta_Y$ provide a lower
bound for the distance between elements of $\M$ as measured by $d$.

\begin{prop}\label{prop:20}
  For any $Y \subseteq M$ and $g_0, g_1 \in \M$, we have the following
  inequality:
  \begin{equation*}
    \Theta_Y(g_0, g_1) \leq d(g_0, g_1) \left( \sqrt{n}\, d(g_0, g_1) +
      2 \sqrt{\Vol(M, g_0)} \right).
  \end{equation*}
  In particular, $\Theta_Y$ is a continuous pseudometric (w.r.t.~$d$).
\end{prop}
\begin{proof}
  By Lemma \ref{lem:44}, we need only prove the inequality for $Y =
  M$, and then it follows for any subset.

  We can clearly find a path $g_t$ from $g_0$ to $g_1$ with $L(g_t)
  \leq 2 d(g_0, g_1)$.  Then for any $\tau \in [0,1]$, we get
  \begin{equation*}
    2 d(g_0, g_1) \geq L(g_t) \geq L \left( g_t|_{[0,\tau]} \right) \geq d(g_0,
    g_\tau) \geq \frac{4}{\sqrt{n}} \left| \sqrt{\Vol(M, g_\tau)} -
      \sqrt{\Vol(M, g_0)} \right|,
  \end{equation*}
  where the last inequality is Lemma \ref{lem:13}.  In particular, we
  get
  \begin{equation}\label{eq:48}
    \sqrt{\Vol(M, g_\tau)} \leq \sqrt{\Vol(M, g_0)} +
    \frac{\sqrt{n}}{2} d(g_0, g_1) =: V
  \end{equation}
  for all $\tau \in [0,1]$.

  To find the length of $g_t$, we first integrate $\langle g'_t, g'_t
  \rangle$ over $x \in M$, then take the square root, and finally
  integrate over $t$.  Ideally, we would wish to change the order of
  integration, so that we first integrate over $t$, then over $x$.  We
  cannot do this exactly, but we can bound the computation of the
  length from below by an expression where we integrate in the
  opposite order, and this expression will involve $\theta^g_x$ and
  $\Theta_M$.  So let's see how this works.

  Let $h_t := g'_t$.  From Hölder's inequality,
  \begin{equation*}
    \int_M \sqrt{\tr_{g_t} (h_t^2)} \, d \mu_{g_t} \leq \left( \int_M
      \, d \mu_{g_t} \right)^{1/2} \left( \int_M \tr_{g_t} (h_t^2)
      \, d \mu_{g_t} \right)^{1/2},
  \end{equation*}
  which gives
  \begin{equation}\label{eq:30}
    \begin{aligned}
      \| h_t \|_{g_t} &= \left( \int_M \tr_{g_t} (h_t^2) \, d
        \mu_{g_t} \right)^{1/2} \geq \frac{1}{\sqrt{\Vol(M,g_t)}}
      \int_M
      \sqrt{\tr_{g_t} (h_t^2)} \, d \mu_{g_t} \\
      &\geq \frac{1}{V} \int_M \sqrt{\tr_{g_t} (h_t^2)} \, d
      \mu_{g_t},
    \end{aligned}
  \end{equation}
  where we have also used \eqref{eq:48}.  To remove the
  $t$-dependence from the volume element, we use
  \begin{equation*}
    \mu_{g_t} = \frac{\sqrt{\det g_t}}{\sqrt{\det g}} \mu_{g} =
    \sqrt{\det G_t} \mu_{g}.
  \end{equation*}
  We then rewrite (\ref{eq:30}) as
  \begin{equation}\label{eq:31}
    \| h_t \|_{g_t} \geq \frac{1}{V} \int_M
    \sqrt{\tr_{g_t}(h_t^2) \det G_t} \, \mu_{g} = \frac{1}{V}
    \integral{M}{}{\sqrt{\langle h_t(x) , h_t(x) \rangle^0_{g_t(x)}}}{\mu_g(x)},
  \end{equation}
  where we have used the Riemannian metric $\langle \cdot , \cdot
  \rangle^0$ on $\M_x$ (cf.~Definition \ref{dfn:15}).

  Since we have removed the $t$-dependence from the measure above, we
  can change the order of integration in the calculation of the length
  of $g_t$:
  \begin{equation}\label{eq:93}
    \begin{aligned}
      L(g_t) &= \integral{0}{1}{\| h_t \|_{g_t}}{d t} \geq \frac{1}{V}
      \integral{0}{1}{\integral{M}{}{\sqrt{\langle h_t(x),
            h_t(x) \rangle^0_{g_t(x)}}}{\mu_g(x)}}{d t} \\
      &= \frac{1}{V} \integral{M}{}{\integral{0}{1}{\sqrt{\langle
            h_t(x), h_t(x) \rangle^0_{g_t(x)}}}{d t}}{\mu_g(x)}.
    \end{aligned}
  \end{equation}
  Now we concentrate on the $t$-integral in the expression above.
  Since $g_t(x)$ is a path in $\M_x$ from $g_0(x)$ to $g_1(x)$ with
  tangents $h_t(x)$, the $t$-integral is actually the length of
  $g_t(x)$ with respect to $\langle \cdot , \cdot \rangle^0$.  But by
  definition, this length is bounded from below by $\theta^g_x(g_0(x),
  g_1(x))$.  Therefore, we can rewrite \eqref{eq:93} as
  \begin{equation*}
    L(g_t) \geq \frac{1}{V} \integral{M}{}{\theta^g_x(g_0(x),
      g_1(x))}{\mu_g(x)} = \frac{1}{V} \Theta_M(g_0, g_1).
  \end{equation*}
  But now the result is immediate given \eqref{eq:48} and the fact
  that we have assumed $L(g_t) \leq 2 d(g_0, g_1)$.
\end{proof}

The previous proposition allows us to achieve our goal for this
section.  Since $\Theta_M$ is a (positive-definite) metric by Lemma
\ref{lem:44}, $\Theta_M(g_0, g_1) > 0$ for any $g_0 \neq g_1$.  From
this, Proposition \ref{prop:20} immediately implies that $d(g_0, g_1)
> 0$ as well.  Since we have already mentioned that the distance
function induced by a weak Riemannian manifold is automatically a
pseudometric, we have proved:

\begin{thm}\label{thm:6}
  $(\M, d)$, where $d$ is the distance function induced from the $L^2$
  metric $(\cdot, \cdot)$, is a metric space.
\end{thm}

\section{The completion of an amenable subset}\label{sec:compl-an-amen}

Now that we know that $\M$ is a metric space, we begin the study of
its completion in this section.  According to the plan of attack laid
out at the beginning of the chapter, we will work on completing more
and more general subsets of $\M$.  This section is concerned with
so-called amenable subsets, defined below, consisting of metrics that
are somehow uniformly bounded and inflated.  The main result of the
section is that the completion of such a subset with respect to $d$
coincides with the completion with respect to the $L^2$ norm on $\s$,
the vector space in which $\M$ resides.

Note the difference to the case of a strong Riemannian manifold, where
Theorem \ref{thm:25} guarantees that the topology induced by the
Riemannian metric agrees with the manifold topology.  Here, the weaker
topology of the tangent spaces $T_{\tilde{g}} \M$ with the weak
Riemannian metric $(\cdot, \cdot)$ is reflected in the weaker topology
induced by the Riemannian distance function $d$ on an amenable subset.

\subsection{Amenable subsets and their
  properties}\label{sec:amen-subs-their}

Let's make the above-mentioned notion of being uniformly bounded and
inflated precise.  Recall that we work over an amenable atlas
(cf.~Definition \ref{dfn:1}).

\begin{dfn}\label{dfn:2}
  We call a subset $\U \subset \M$ \emph{amenable} if $\U$ is convex
  and we can find constants $C, \delta > 0$ such that for all
  $\tilde{g} \in \U$, $x \in M$ and $1 \leq i,j \leq n$,
  \begin{equation*}
    \lambda^{\tilde{G}}_{\textnormal{min}}(x) \geq \delta
  \end{equation*}
  (where we recall that $\tilde{G} = g^{-1} \tilde{g}$, with $g$ our
  fixed metric) and
  \begin{equation*}
    |\tilde{g}_{ij}(x)| \leq C.
  \end{equation*}
\end{dfn}

\begin{rmk}\label{rmk:1}
  We make a few remarks about the definition:
  \begin{enumerate}
  \item Recall from Definition \ref{dfn:23} that a semimetric
    $\tilde{g}$ is inflated if $\det \tilde{G}$ is bounded away from
    zero.  Above, we have instead used the condition
    $\lambda^{\tilde{G}}_{\textnormal{min}} \geq \delta$, but this
    does indeed imply that the metrics of an amenable subset are
    uniformly inflated.  This is because $\det \tilde{G} \geq \left(
      \lambda^{\tilde{G}}_{\textnormal{min}} \right)^n$, the
    determinant being the product of the eigenvalues.
  \item We could also have defined an amenable subset using the $C^0$
    topology on $\M \subset \s$.  Namely, let $\mathrm{cl}(\U) \subset
    \s$ be the closure of $\U$ in the $C^0$ topology of $\s$, and let
    $\partial \M$ be the boundary of $\M$ in this topology.
    ($\partial \M$ consists of semimetrics that fail to be positive
    definite and so have determinant $0$ at at least one point.)  Then
    $\U$ is amenable if and only if $\U$ is bounded in the $C^0$ norm
    on $\s$ and $\mathrm{cl}(\U) \cap \partial \M = \emptyset$.
  \item The requirement that $\U$ is convex is technical, and is there
    to insure that we can consider simple, straight-line paths between
    points of $\U$ to estimate the distance between them.
  \item \label{item:1} Recall that the function sending a matrix to
    its minimal eigenvalue is concave by Lemma \ref{lem:46}.  Also,
    the absolute value function on $\R$ is convex by the triangle
    inequality.  Therefore, the two bounds given in Definition
    \ref{dfn:2} are compatible with the requirement of convexity.
  \end{enumerate}
\end{rmk}

One useful property the metrics $\tilde{g}$ of an amenable subset have
is that the Radon-Nikodym derivatives $( \mu_{\tilde{g}} / \mu_g )$,
with respect to the reference volume form $\mu_g$, are bounded away
from zero and infinity independently of $\tilde{g}$.

\begin{lem}\label{lem:49}
  Let $\U$ be an amenable subset.  Then there exists a constant $K >
  0$ such that for all $\tilde{g} \in \U$,
  \begin{equation}\label{eq:130}
    \frac{1}{K} \leq \left( \frac{\mu_{\tilde{g}}}{\mu_g} \right) \leq K
  \end{equation}
\end{lem}
\begin{proof}
  First, we note that
  \begin{equation*}
    \left( \frac{\mu_{\tilde{g}}}{\mu_g} \right) = \det \tilde{G} \quad
    \textnormal{and} \quad \left( \frac{\mu_{\tilde{g}}}{\mu_g}
    \right)^{-1} = \left( \frac{\mu_g}{\mu_{\tilde{g}}}
    \right) = (\det \tilde{G})^{-1}.
  \end{equation*}
  So the bounds \eqref{eq:130} are equivalent to upper bounds on both
  $\det \tilde{G}$ and $(\det \tilde{G})^{-1}$.
  
  Now, if the eigenvalues of $\tilde{G}$ are $\lambda^{\tilde{G}}_1,
  \dots, \lambda^{\tilde{G}}_n$, then
  \begin{equation*}
    \det \tilde{G} = \lambda^{\tilde{G}}_1
    \cdots \lambda^{\tilde{G}}_n \geq \left(
      \lambda^{\tilde{G}}_{\textnormal{min}} \right)^n \geq \delta^n,
  \end{equation*}
  where $\delta$ is the constant guaranteed by the fact that
  $\tilde{g} \in \U$.  This allows us to bound $(\det \tilde{G})^{-1}$
  from above.

  To bound $\det \tilde{G}$ from above, it is sufficient to bound the
  absolute value of the coefficients of $\tilde{G} = g^{-1} \tilde{g}$
  from above.  But bounds on the coefficients of $\tilde{g}$ are
  already assured by the fact that $\tilde{g} \in \U$, and bounds on
  the coefficients of $g^{-1}$ are guaranteed by the fact that
  $g^{-1}$ is a fixed, smooth cometric on $M$.  So we are finished.
\end{proof}

Amenable subsets guarantee good behavior of the norms on $\s$ that are
defined by their members---namely, the norms are in some sense
``uniformly equivalent''.  More precisely, we have:

\begin{lem}\label{lem:18}
  Let $\U \subset \M$ be an amenable subset.  Then there exists a
  constant $K$ such that for all pairs $g_0, g_1 \in \U$ and all $h
  \in \s$,
  \begin{equation*}
    \frac{1}{K} \| h \|_{g_1} \leq \| h \|_{g_0} \leq
    K \| h \|_{g_1}.
  \end{equation*}
\end{lem}
\begin{proof}

  Instead of showing that the norms of any two metrics $g_0, g_1 \in
  \U$ are equivalent, we will show that the norm of any $\tilde{g} \in
  \U$ is equivalent to that of our reference metric $g$, i.e., there
  exists a constant $K$ independent of $\tilde{g}$ such that
  \begin{equation}\label{eq:128}
    \frac{1}{K} \| h \|_g \leq \| h \|_{\tilde{g}} \leq K \| h \|_g
  \end{equation}
  for all $h \in \s$.

  This is equivalent to the following statement.  Let
  \begin{equation*}
    T_{\tilde{g}} : (S^2 T^* M, \langle \cdot , \cdot
    \rangle_{\tilde{g}}) \rightarrow (S^2 T^* M, \langle \cdot, \cdot \rangle_g)
  \end{equation*}
  be the identity mapping on the level of sets, sending the bundle
  $S^2 T^* M$ with the Riemannian structure $\langle \cdot, \cdot
  \rangle_{\tilde{g}}$ to itself with the Riemannian structure
  $\langle \cdot , \cdot \rangle_g$.  Let $N(T_{\tilde{g}})(x)$ be the
  operator norm of $T_{\tilde{g}}(x) : \satx \rightarrow \satx$, and
  let $N(T_{\tilde{g}}^{-1})(x)$ be defined similarly.  Then
  \begin{align*}
    \| h \|_g^2 &= \integral{M}{}{\langle T_{\tilde{g}}(x) h(x),
      T_{\tilde{g}}(x) h(x) \rangle_{g(x)}}{\mu_g(x)} \\
    &\leq \integral{M}{}{(N(T_{\tilde{g}})(x))^2 \langle h(x), h(x)
      \rangle_{\tilde{g}(x)} \left( \frac{\mu_g}{\mu_{\tilde{g}}}
      \right)(x)}{\mu_{\tilde{g}}(x)}
  \end{align*}
  and similarly,
  \begin{equation*}
    \| h \|_{\tilde{g}}^2 \leq \integral{M}{}{(N(T_{\tilde{g}}^{-1})(x))^2
      \langle h(x), h(x) \rangle_{\tilde{g}(x)} \left(
        \frac{\mu_{\tilde{g}}}{\mu_g} \right)}{\mu_g(x)}.
  \end{equation*}
  So \eqref{eq:128} holds if and only if there are constants $K_0$ and
  $K_1$ such that
  \begin{equation*}
    N(T_{\tilde{g}})(x)^2, N(T_{\tilde{g}}^{-1})(x)^2
    \leq K_0 \quad \textnormal{and} \quad  (\mu_g / \mu_{\tilde{g}}), (\mu_{\tilde{g}} / \mu_g)
    \leq K_1.
  \end{equation*}

  This last statement is the one we'll prove.  The existence of the
  constant $K_1$ is guaranteed by Lemma \ref{lem:49}.  So we need to
  show the existence of the constant $K_0$.

  To do this, first note that $N(T_{\tilde{g}})$ and
  $N(T_{\tilde{g}}^{-1})$ are continuous functions on $M$ for fixed
  $\tilde{g}$.  This follows immediately from the fact that $g$ and
  $\tilde{g}$ are smooth.  (Of course, it would even suffice for them
  to be continuous.)  Secondly, we notice that $N(T_{\tilde{g}})(x)$
  and $N(T_{\tilde{g}}^{-1})(x)$ depend only on the coordinate
  representations of $\tilde{g}(x)$ and $g(x)$.
  
  Let $SP_n$ denote the set of all positive definite scalar products
  on $\R^n$, which we can identify with the set of all positive
  definite $n \times n$ symmetric matrices.  Let's define a function
    \begin{equation*}
    \tilde{N} : SP_n \times SP_n \rightarrow \R
  \end{equation*}
  by setting $\tilde{N}(a, b)$ to be equal to the operator norm of
  \begin{equation*}
    \id : (\R^n, a) \rightarrow (\R^n, b).
  \end{equation*}
  That is, $\tilde{N}(a,b)$ is the smallest number such that
  \begin{equation*}
    b(v,v) \leq \tilde{N}(a,b) \cdot a(v,v)
  \end{equation*}
  for all $v \in \R^n$.

  It is not hard to see that $\tilde{N}$ is continuous in both of its
  arguments, with the topology on $SP_n$ coming from its
  identification with the space of positive definite symmetric
  matrices.  Furthermore, by the arguments above, we have
  \begin{equation}\label{eq:129}
    N(T_{\tilde{g}})(x) = \tilde{N}(\tilde{g}(x), g(x)) \quad
    \textnormal{and} \quad N(T_{\tilde{g}}^{-1})(x) = \tilde{N}(g(x), \tilde{g}(x)),
  \end{equation}
  where we of course define $\tilde{N}$ in these cases using the
  coordinate representations of $g(x)$ and $\tilde{g}(x)$ in some
  chart around $x$.  (The value of $\tilde{N}$ won't depend on the
  chart.)  Furthermore, by the bounds satisfied by metrics in an
  amenable subset and the continuity of $g$, the set
  \begin{equation*}
    A := \{ g(x) \mid x \in M \} \cup \{ \tilde{g}(x) \mid x \in M,\ \tilde{g} \in \U \}
  \end{equation*}
  is relatively compact when viewed as a subset of the space of
  positive definite symmetric matrices.  Therefore $\tilde{N}|_{A
    \times A}$ is bounded.  But then \eqref{eq:129} immediately
  implies the existence of the constant $K_0$.
\end{proof}

Lemma \ref{lem:49} immediately implies that the function $\tilde{g}
\mapsto \Vol(M,\tilde{g})$ is bounded when restricted to any amenable
subset.  Recalling the form of the estimate in Proposition
\ref{prop:20} then shows the following lemma.

\begin{lem}\label{lem:26}
  Let $\U$ be an amenable subset and $g \in \M$.  Then there exists a
  constant $V$ such that for any $g_0, g_1 \in \U$ and $Y \subset M$,
  \begin{equation*}
    \Theta_Y(g_0, g_1) \leq 2 d(g_0, g_1) \left( \frac{2 \sqrt{n}}{4}d(g_0, g_1) +
      \sqrt{V} \right).
  \end{equation*}
  More precisely, $V = \sup_{\tilde{g} \in \U}\Vol(M,\tilde{g})$,
  which is finite by the discussion preceding the lemma.
\end{lem}

\subsection{The completion of $\U$ with respect to $d$ and $\| \cdot \|_g$}\label{sec:completion-u-with}

We are now ready to prove a result that, in particular, implies
equivalence of the topologies defined by $d$ and $\| \cdot \|_g$ on an
amenable subset $\U$.

\begin{thm}\label{thm:5}
  Consider the $L^2$ topology on $\M$ induced from the scalar product
  $(\cdot,\cdot)_g$ (where $g$ is fixed).  Let $\U \subset \M$ be any
  amenable subset.

  Then the $L^2$ topology on $\U$ coincides with the topology induced
  from the restriction of the Riemannian distance function $d$ of $\M$
  to $\U$.

  Additionally, the following holds:
  \begin{enumerate}
  \item There exists a constant $K$ such that
    \begin{equation*}
      d(g_0,g_1) \leq K \| g_1 - g_0 \|_g,
    \end{equation*}
    for all $g_0,g_1 \in \U$.
  \item \label{item:9} For any $\epsilon > 0$, there exists $\delta >
    0$ such that if $d(g_0,g_1) < \delta$, then $\| g_0 - g_1 \|_g <
    \epsilon$.
  \end{enumerate}
\end{thm}
\begin{proof}
  First, we show there is a constant $K$ such that
  \begin{equation*}
    d(g_0,g_1) \leq K \| g_1 - g_0 \|_g
  \end{equation*}
  for all $g_0, g_1 \in \U$.  Consider the path
  \begin{equation*}
    g_t := g_0 + t h, \quad h := g_1 - g_0,\ t \in [0,1],
  \end{equation*}
  which runs from $g_0$ to $g_1$.  Note that we can clearly find an
  amenable subset $\U'$ containing $\U$ and $g$.  We then have
  \begin{equation}\label{eq:38}
    L(g_t) = \int_0^1 \| (g_t)' \|_{g_t} \, dt = \int_0^1 \| h
    \|_{g_t} \, dt \leq \int_0^1 K \| h \|_g \, dt = K \| g_1 - g_0 \|_g
  \end{equation}
  where $K$ is the constant associated to $\U'$ guaranteed by Lemma
  \ref{lem:18}.  Since $d(g_0,g_1) \leq L(g_t)$ and the constant $K$
  depends only on the set $\U$, this inequality is shown.

  We now turn to proving statement (2).  Let $\epsilon > 0$ therefore
  be given.  Our plan is to use the Riemannian metric $\langle \cdot ,
  \cdot \rangle^0$ and its distance function $\theta^g_x$ to get
  pointwise bounds on $\| g_1 - g_0 \|_g$ based on $d(g_0,g_1)$.  We
  then use the bounds guaranteed by the fact that we work over an
  amenable subset in order to show that our pointwise estimates are
  uniform.  Finally, we use a variant of a thick-thin decomposition of
  $M$, where $\tr_g((g_1 - g_0)^2)$ is small on the ``thin'' part and
  the ``thick'' part has volume bounded in terms of $d(g_0,g_1)$.

  Since $\Matx$ is a finite-dimensional Riemannian manifold, the
  topology induced from $\theta^g_x$ is the same as the manifold
  topology, which in turn is given by any norm on $\satx$.  For
  instance this norm is given by the scalar product $\langle \cdot ,
  \cdot \rangle_{g(x)}$ on $\satx$, which we recall is given by
  \begin{equation}\label{eq:41}
    \langle h , k \rangle_{g(x)} = \tr_{g(x)}(h k)
  \end{equation}
  for $h, k \in \satx$.  That these two topologies are the same
  implies, in particular, that for all $\zeta > 0$ and $\tilde{g} \in
  \Matx$, we can find $\kappa > 0$ such that
  \begin{equation*}
    B^{\theta^g_x}_{\tilde{g}}(\zeta) \subset B^{\langle \cdot, \cdot \rangle_{g(x)}}_{\tilde{g}}(\kappa),
  \end{equation*}
  where
  \begin{align*}
    B^{\langle \cdot , \cdot \rangle_{g(x)}}_{\tilde{g}}(\kappa) &:= \left\{ \hat{g} \in \Matx \mid \sqrt{
      \langle \hat{g} - \tilde{g}, \hat{g} - \tilde{g} \rangle_{g(x)}} < \kappa
    \right\}, \\
    B^{\theta^g_x}_{\tilde{g}}(\zeta) &:= \left\{ \hat{g} \in \Matx
      \mid \theta^g_x(\hat{g}, \tilde{g}) < \zeta \right\}.
  \end{align*}

  Now, for $x \in M$ and $\tilde{g} \in \M$, we define a function
  $\eta_{x, \tilde{g}}(\zeta)$ by
  \begin{align*}
    \eta_{x,\tilde{g}}(\zeta) &:= \inf \left\{ \kappa \in \R \mid
      B^{\theta^g_x}_{\tilde{g}(x)}(\zeta) \subset B^{\langle \cdot , \cdot
        \rangle_{g(x)}}_{\tilde{g}(x)}(\kappa)
    \right\} \\
    &:= \inf \left\{ \kappa \in \R \mid \sqrt{ \langle \hat{g} -
        \tilde{g}(x), \hat{g} - \tilde{g}(x) \rangle_{g(x)}} < \kappa\
      \forall\ \hat{g}\ \mathrm{with}\ \theta^g_x(\hat{g},\tilde{g}(x)) <
      \zeta \right\}.
  \end{align*}
  Then, because of the smooth dependence of $\langle \cdot, \cdot
  \rangle^0$ and $\langle \cdot , \cdot \rangle_{g(x)}$ on $x$,
  $\eta_{x,\tilde{g}}(\zeta)$ is continuous separately in $x$ and
  $\tilde{g}$.  If we define
  \begin{equation*}
    \U_x :=
    \left\{
      \hat{g}(x) \mid \hat{g} \in \U
    \right\},
  \end{equation*}
  then $\U_x$ is a relatively compact subset of $\Matx$, since $\U$ is
  amenable.  Since $M$ is also compact, for any fixed $\zeta > 0$, we
  can define a function
  \begin{equation*}
    \eta(\zeta) := \sup_{\substack{x \in M \\ \tilde{g} \in \U}}
    \eta_{x,\tilde{g}}(\zeta)
    < \infty.
  \end{equation*}
  It follows from the definition that $\eta(\zeta) \rightarrow 0$
  for $\zeta \rightarrow 0$.

  Because of the relative compactness of $\U_x$ for each $x \in M$,
  together with compactness of $M$, there exists a constant $C_0$ such
  that $\theta^g_x(g_0(x), g_1(x)) \leq C_0$ for all $g_0,g_1 \in \U$
  and $x \in M$.  This implies immediately that
  \begin{equation*}
    \Theta_M(g_0, g_1) = \integral{M}{}{\theta^g_x (g_0(x), g_1(x))}{\mu_g(x)} \leq C_0 \Vol(M,g).
  \end{equation*}

  Now, choose $\zeta > 0$ small enough that
  \begin{equation*}
    \eta(\zeta) < \frac{\epsilon}{\sqrt{2 \Vol(M,g)}}.
  \end{equation*}

  By Lemma \ref{lem:26}, there exists a constant $V$ such that
  \begin{equation}\label{eq:46}
    \Theta_M(g_0, g_1) \leq 2 d(g_0, g_1) \left( \frac{2 \sqrt{n}}{4}d(g_0, g_1) +
      \sqrt{V} \right)
  \end{equation}
  for all $g_0, g_1 \in \U$.
  
  Choose $\delta$ small enough that
  \begin{equation*}
    2 \delta \left( \frac{2 \sqrt{n}}{4} \delta + \sqrt{V} \right) <
    \frac{\epsilon^2 \zeta}{2 \eta(C_0)^2}.
  \end{equation*}
  We claim that $d(g_0, g_1) < \delta$ implies that $\| g_1 - g_0 \|_g
  < \epsilon$.  Note that the choices of $\zeta$ and $C_0$ were
  made independently of $g_0$ and $g_1$, hence $\delta$ is independent
  of $g_0$ and $g_1$, as required.
  
  We define two closed subsets of $M$ by
  \begin{align*}
    M_+ &:= \left\{ x \in M \mid \theta^g_x (g_0(x), g_1(x)) \geq
      \zeta
    \right\}, \\
    M_- &:= \left\{ x \in M \mid \theta^g_x (g_0(x), g_1(x)) \leq
      \zeta \right\}.
  \end{align*}
  From \eqref{eq:46} and our choice of $\delta$, we have that
  \begin{equation}\label{eq:37}
    \int_M \theta^g_x (g_0(x), g_1(x)) \, \mu_g(x) = \Theta_M(g_0, g_1) <
    \frac{\epsilon^2 \zeta}{2 \eta(C_0)^2}.
  \end{equation}
  This inequality also holds if we integrate over $M_+$ instead of all
  of $M$, so
  \begin{equation*}
    \zeta \Vol(M_+, g) = \zeta \integral{M_+}{}{}{\mu_g} \leq
    \integral{M_+}{}{\theta^g_x (g_0(x), g_1(x))}{\mu_g(x)} < \frac{\epsilon^2
      \zeta}{2 \eta(C_0)^2},
  \end{equation*}
  implying
  \begin{equation*}
    \Vol(M_+, g_0) < \frac{\epsilon^2}{2 \eta(C_0)^2}.
  \end{equation*}

  From the definitions of $M_-$ and $\eta$, we have that
  \begin{equation*}
    \sqrt{\langle g_1(x) - g_0(x), g_1(x) - g_0(x) \rangle_{g(x)}} \leq
    \eta(\zeta)
  \end{equation*}
  on $M_-$.  From $\theta^g_x(g_0(x), g_1(x)) \leq C_0$,
  we have that
  \begin{equation*}
    \sqrt{\langle g_1(x) - g_0(x), g_1(x) - g_0(x) \rangle_{g(x)}}
    \leq \eta(C_0)
  \end{equation*}
  on all of $M$, and in particular on $M_+$.  Using this, we compute
  \begin{align*}
    \| g_1 - g_0 \|_g^2 &= \integral{M}{}{\langle g_1(x) - g_0(x),
      g_1(x) - g_0(x) \rangle_{g(x)}}{\mu_g(x)} \\
    &= \integral{M_-}{}{\langle g_1(x) - g_0(x), g_1(x) - g_0(x)
      \rangle_{g(x)}}{\mu_g(x)} \\
    &\quad + \integral{M_+}{}{\langle g_1(x) -
      g_0(x),
      g_1(x) - g_0(x) \rangle_{g(x)}}{\mu_g(x)} \\
    &\leq \eta(\zeta)^2 \integral{M_-}{}{}{\mu_g} + \eta(C_0)^2
    \integral{M_+}{}{}{\mu_g} \\
    &< \eta(\zeta)^2 \Vol(M,g) + \eta(C_0)^2 \frac{\epsilon^2}{2
      \eta(C_0)^2} \\
    &< \frac{\epsilon^2}{2} + \frac{\epsilon^2}{2} = \epsilon^2.
  \end{align*}
  This proves the second statement.
\end{proof}

Let's now equip $\M$ with the $H^s$ topology for some fixed $s > n/2$.
From Remark \ref{rmk:1}(2), we immediately get continuity (but not
Lipschitz continuity) of the Riemannian distance function on \emph{all
  of} $\M$, not just amenable subsets.

\begin{cor}\label{cor:1}
  The Riemannian distance function $d$ of $(\cdot,\cdot)$ is
  continuous in the $H^s$ topology on $\M$ for all fixed $s > n/2$.
\end{cor}
\begin{proof}
  Suppose we have $\tilde{g} \in \M$ and a sequence $g_n
  \rightarrow_{H^s} \tilde{g}$.  Since $s > n/2$, the Sobolev
  Embedding Theorem implies that $g_n \rightarrow_{C^0} \tilde{g}$.
  In particular, from Remark \ref{rmk:1}(2) we see that the set $\{g_n
  \mid n \in \N\} \cup \{ \tilde{g} \}$ is contained in some amenable
  subset $\U \in \M$.  Therefore, Theorem \ref{thm:5} gives
  $d(g_n,\tilde{g}) \rightarrow 0$, showing continuity.
\end{proof}

By Definition \ref{dfn:18} of a Fréchet space, Corollary \ref{cor:1}
then immediately implies:

\begin{cor}\label{cor:2}
  The Riemannian distance function $d$ of $(\cdot,\cdot)$ is
  continuous in the $C^\infty$ (manifold) topology on $\M$.
\end{cor}

Theorem \ref{thm:5} will give us our first result regarding the
completion of $\M$.  First, though, we need to make some definitions
and prove a statement about metric spaces.

\begin{dfn}\label{dfn:4}
  We define
  \begin{align*}
    \s^0 &:= H^0(S^2 T^* M) \\
    \M^0 &:=
    \left\{
      g^0 \in \s^0 \mid g^0(x) > 0\ \text{for almost all}\ x \in M
    \right\}.
  \end{align*}
  That is, $\s^0$ consists of all $H^0$ (i.e., $L^2$) symmetric
  $(0,2)$-tensor fields.  $\M^0$ consists of the elements of $\s^0$
  that induce a positive-definite scalar product on almost every
  tangent space of $M$.  Thus, $\s^0$ and $\M^0$ are the completions
  of $\s$ and $\M$, respectively, with respect to the \emph{fixed}
  norm $\| \cdot \|_g$.  (At the moment, this has nothing to do with
  the completion of $\M$ with respect to $d$.)

  If $\U \subset \M$ is any subset, we define
  \begin{equation*}
    \U^0 := \left\{
      g^0 \in \M^0 \midmid \exists g_n^0 \in \U,\ n \in \N : g^0_n
      \xrightarrow{L^2} g^0
    \right\},
  \end{equation*}
  that is, $\U^0$ is the $L^2$-completion of $\U$.
\end{dfn}

\begin{rmk}\label{rmk:6}
  A couple of remarks on the definition.
  \begin{enumerate}
  \item Note that $\M^0$ is \emph{not} open in the $L^2$ topology on
    $\s$.  In fact, even more is true: the interior of $\M^0$ is empty
    with respect to the $L^2$ topology.  (We will prove this
    explicitly in Lemma \ref{lem:9}.)  This fundamental point implies
    that we cannot place a manifold structure on $\M^0$ (or $\U^0$),
    at least not one with the natural model space $\s^0$.  Therefore,
    in light of Theorem \ref{thm:7} below, we will not be able to give
    a manifold structure to the completion of an amenable subset.
    That $\M^0$ is not open is also related to the fact that the
    exponential mapping of $\M$ is not defined on any $L^2$-open
    subset in any tangent space.
  \item Elements of $\U^0$ satisfy the same bounds as elements of $\U$
    at almost all $x \in M$.
  \end{enumerate}
\end{rmk}

Let's look back at Theorem \ref{thm:5} again.  The first statement
says that for any amenable subset $\U$ and any $g \in \M$, $d$ is
Lipschitz continuous with respect to $\| \cdot \|_g$ when viewed as a
function on $\U \times \U$.  The second statement says that $\| \cdot
\|_g$ is uniformly continuous on $\U \times \U$ with respect to $d$.
To put this knowledge to good use, we will need the following lemma:

\begin{lem}\label{lem:20}
  Let $X$ be a set, and let two metrics, $d_1$ and $d_2$, be defined
  on $X$.  Denote by $\phi : (X,d_1) \rightarrow (X, d_2)$ the map
  which is the identity on the level of sets, i.e., $\phi$ simply maps
  $x \mapsto x$.  Finally, denote by $\overline{X}^1$ and
  $\overline{X}^2$ the completions of $X$ with respect to $d_1$ and
  $d_2$, respectively.

  If both $\phi$ and $\phi^{-1}$ are uniformly continuous, then there
  is a natural homeomorphism between $\overline{X}^1$ and
  $\overline{X}^2$.
\end{lem}
\begin{proof}
  Recall the definition of the completion $\overline{X}^i$ of the
  metric space $(X,d_i)$, $i=1,2$, from Section
  \ref{sec:compl-metr-spac}.  It is formed of the equivalence classes
  of Cauchy sequences $\{x_n\}$, with metric (again denoted by $d_i$)
  given by
  \begin{equation*}
    d_i (\{x_n\}, \{y_n\}) = \lim d_i (x_n, y_n).
  \end{equation*}

  Since a uniformly continuous function maps Cauchy sequences to
  Cauchy sequences, our assumptions on $\phi$ and $\phi^{-1}$ imply
  that $d_1$ and $d_2$ have the same Cauchy sequences.  Thus, we only
  need to prove that the equivalence classes of these Cauchy sequences
  are the same in $\overline{X}^1$ and $\overline{X}^2$, that is,
  \begin{equation}\label{eq:137}
    \lim d_1 (\{x_n\}, \{y_n\}) = 0 \  \Longleftrightarrow \  \lim d_2 (\{x_n\}, \{y_n\}) = 0.
  \end{equation}
  But this is immediate from the uniform continuity of $\phi$ and
  $\phi^{-1}$.

  The natural homeomorphism is of course given by the unique uniformly
  continuous extension of $\phi$ to $\overline{X}^1$ (cf.~statement
  (\ref{item:2}) of Theorem \ref{thm:29}).
\end{proof}

We are now ready to state 

\begin{thm}\label{thm:7}
  Let $\U$ be an amenable subset.  Then we can identify
  $\overline{\U}$, the completion of $\U$ with respect to $d$, with
  $\U^0$, in the sense of Lemma \ref{lem:20}.  We can make the natural
  homeomorphism $\overline{\U} \rightarrow \U^0$ into an isometry by
  placing a metric on $\U^0$ defined by
  \begin{equation*}
    d(g_0, g_1) = \lim_{k \rightarrow \infty} d(g^0_k, g^1_k),
  \end{equation*}
  where $\{g^0_k\}$ and $\{g^1_k\}$ are any sequences in $\U$ that
  $L^2$-converge to $g_0$ and $g_1$, respectively.
\end{thm}
\begin{proof}
  Denote by $\hat{d}$ the metric induced from $\| \cdot \|_g$ on $\U$
  in the usual way for Hilbert spaces, i.e., $\hat{d}(g_1,g_2) = \|
  g_1 - g_2 \|_g$.  As in Lemma \ref{lem:20}, let $\phi : (\U, d)
  \rightarrow (\U, \hat{d})$ be the identity on the level of sets.
  Then Theorem \ref{thm:5} clearly implies that both $\phi$ and
  $\phi^{-1}$ are uniformly continuous ($\phi^{-1}$ is even a
  Lipschitz map).  Thus Lemma \ref{lem:20} gives the result.
\end{proof}

We have thus found a nice description of the completion of very
special subsets of $\M$.  As already discussed, our plan now is to
start removing the nice properties that allowed us to understand
amenable subsets so clearly, advancing through the completions of ever
larger and more generally defined subsets of $\M$.

To do this, however, we need to clear up our viewpoint and some
technicalities.  The issue is the following: it happens that one can
find examples of Cauchy sequences in $\M$ (i.e., points of the
precompletion) that do not $L^2$-converge to any point of $\s^0$.
(The skeptical reader can jump ahead to Section
\ref{sec:compl-orbit-space} for a proof of this fact, at least for the
case when $\dim M = 1$, $2$ or $3$.)  Nevertheless, we would like to
somehow be able to unambiguously identify points of $\overline{\M}$
with sections of $S^2 T^* M$.  Here, ``unambiguous'' means that each
Cauchy sequence is identified with a unique section, and all
equivalent Cauchy sequences are identified with the same section.  If
we could do this, we would have a bijection between $\overline{\M}$
and some subset of the sections of $S^2 T^* M$.  Without a uniform,
unambiguous notion of the ``limit point'' of a Cauchy sequence in
$\M$, such an identification is not well-defined.

Thus, we will delay further study of the completion of $\M$ and its
subsets until we see in exactly what way we can identify Cauchy
sequences with sections of $S^2 T^* M$.  The goal of the next chapter
is to resolve this with an appropriate convergence notion for
sequences in $\M$.  Then, in Chapter \ref{chap:sing-metrics}, we
determine precisely what sections of $S^2 T^* M$ actually do represent
Cauchy sequences in $\M$, thus describing the bijection mentioned
above.


\chapter{Cauchy sequences and $\omega$-convergence}\label{cha:almost-everywh-conv}

In this chapter, we introduce and study a fundamental notion of
convergence of our own invention for $d$-Cauchy sequences in $\M$.  We
call this \emph{$\omega$-convergence}, and its importance is made
clear through two theorems we will prove, an existence and a
uniqueness result.  The existence result, proved in Section
\ref{sec:existence-ad-limit}, says that every $d$-Cauchy sequence has
a subsequence that $\omega$-converges to a measurable semimetric,
which we will then show has finite total volume.  The uniqueness
result, proved in Section \ref{sec:uniqueness-ad-limit}, is that two
$\omega$-convergent Cauchy sequences in $\M$ are equivalent (in the
sense of \eqref{eq:97}) if and only if they have the same
$\omega$-limit.  These results allow us to identify an equivalence
class of $d$-Cauchy sequences with the unique $\omega$-limit that its
representatives subconverge to, and thus give a meaning to points of
$\overline{\M}$.

We might hope that our convergence notion for Cauchy sequences could
at least imply pointwise convergence of the metrics of the sequence to
some limit tensor field.  However, we will have to back off of this
hope somewhat, as it will turn out that one cannot demand that a
$d$-Cauchy sequence converge in any pointwise sense at points where
the metrics in the sequence deflate (cf.~Definition \ref{dfn:25}).
This is a consequence of the somewhat surprising result that one can
bound $d(g_0, g_1)$, for any $g_0, g_1 \in \M$, based only on the
``intrinsic volumes'' of the set on which $g_0$ and $g_1$ differ.
Intrinsic means here that this volume is measured with respect to
$g_0$ and $g_1$.  In particular, the bound \emph{does not depend} on
how much $g_0$ and $g_1$ differ as tensors, say in a fixed coordinate
system.  Hence two sequences of metrics can be $d$-close and yet have
very different pointwise limits (or no pointwise limits at all),
provided the only differ on small-volume subsets.  This will be made
more precise in Section \ref{sec:existence-ad-limit}, where we define
$\omega$-convergence.

We \emph{can} demand that $\omega$-convergence imply pointwise
convergence off of the deflated set.  We can then use this to show
that the volume forms $\mu_{g_k}$ of a Cauchy sequence $\{g_k\}$
converge pointwise almost everywhere.  This will allow us to prove, in
Section \ref{sec:ad-conv-conc}, that the volume of a subset of $M$ is
continuous with respect to the topology of $\omega$-convergence.

\section{Existence of the $\omega$-limit}\label{sec:existence-ad-limit}

We begin this section with an important estimate and some examples,
followed by the definition of $\omega$-convergence and some of its
basic properties.  After that, we start on the existence proof by
showing a pointwise version, i.e., an analogous result on $\Matx$.
Finally, we globalize this pointwise result to show the existence of
an $\omega$-convergent subsequence for any Cauchy sequence in $\M$.

\subsection{Volume-based estimates on $d$ and examples}\label{sec:volume-based-estim}

We have mentioned that $\omega$-convergence implies pointwise convergence
only off the deflated set of a sequence of metrics.  We also stated
that this is forced upon us by a bound on the distance between two
metrics that is based on the volume of the set on which they differ.
So before we give the definition of $\omega$-convergence, let's show this
result.  The proof is a bit technical, but the idea is very simple and
is described at the beginning of the proof.

\begin{prop}\label{prop:18}
  Suppose that $g_0, g_1 \in \M$, and let $E := \carr (g_1 - g_0) = \{
  x \in M \mid g_0(x) \neq g_1(x) \}$.  Then there exists a constant
  $C(n)$ depending only on $n = \dim M$ such that
  \begin{equation*}
    d(g_0, g_1) \leq C(n) \left( \sqrt{\Vol(E, g_0)} +
      \sqrt{\Vol(E,g_1)} \right).
  \end{equation*}

  In particular, we have
  \begin{equation*}
    \diam \left( \{ \tilde{g} \in \M \mid \Vol(M, \tilde{g}) \leq
      \delta \} \right) \leq 2 C(n) \sqrt{\delta}.
  \end{equation*}
\end{prop}
\begin{proof}
  The second statement follows immediately from the first, so we only
  prove the first.
  
  \begin{figure}[t]
    \centering
    \includegraphics[width=11cm]{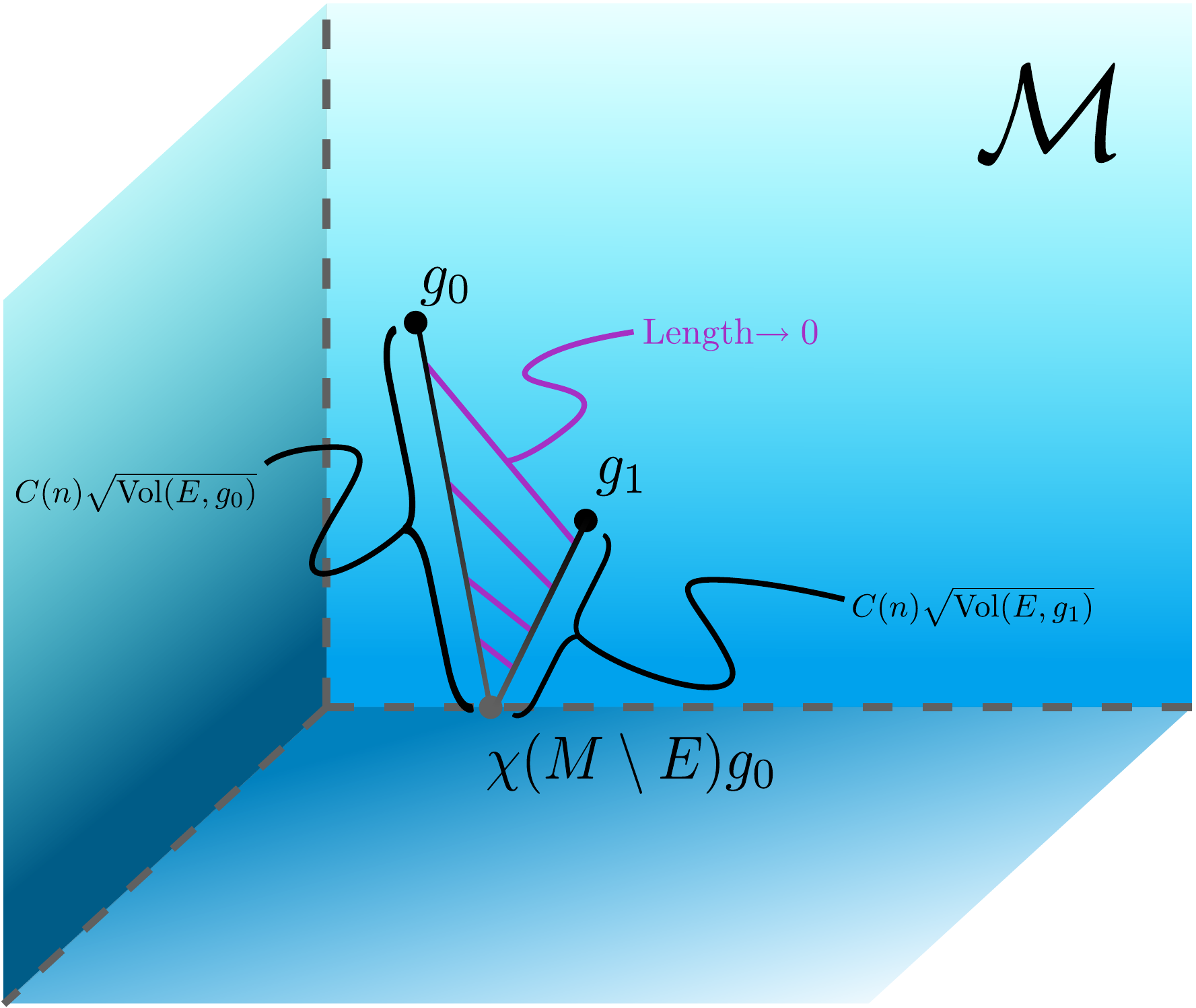}
    \caption{The heuristic idea for the proof of Proposition
      \ref{prop:18}.  A path in the family is given by traveling
      partway down the black line from $g_0$, then skipping over on
      one of the purple lines, and finally traveling up the black line
      to $g_1$.}\label{fig:svp}
  \end{figure}
  
  The heuristic idea is the following.  We want to construct a family
  of paths with three pieces, depending on a real parameter $s$, such
  that the metrics do not change on $M \setminus E$ as we travel along
  the paths.  Therefore, we pretend that we can restrict all
  calculations to $E$.  On $E$, the first piece of the path is the
  straight line from $g_0$ to $s g_0$ for some small positive number
  $s$.  It is easy to compute a bound for the length of this path
  based on $\Vol(E, g_0)$.  The second piece is the straight line from
  $s g_0$ to $s g_1$, which, as we will see, has length approaching
  zero for $s \rightarrow 0$.  The last piece is the straight line
  from $s g_1$ to $g_1$, which again has length bounded from above by
  an expression involving $\Vol(E, g_1)$.  This idea is illustrated in
  Figure \ref{fig:svp}.
  
  Our job is to now take this heuristic picture, which uses paths of
  $L^2$ metrics, and construct a family of paths of \emph{smooth}
  metrics that captures the essential properties.
  
  For each $k \in \N$ and $s \in (0,1]$, we define three families of
  metrics as follows.  Choose closed sets $F_k \subseteq E$ and open
  sets $U_k$ containing $E$ such that $\Vol(U_k, g_i) - \Vol(F_k, g_i)
  \leq 1/k$ for $i = 0,1$.  (This is possible because the Lebesgue
  measure is regular.)  Let $f_{k,s} \in C^\infty(M)$ be functions
  with the following properties:
  \begin{enumerate}
  \item $f_{k,s}(x) = s$ if $x \in F_k$,
  \item $f_{k,s}(x) = 1$ if $x \not\in U_k$ and
  \item $s \leq f_{k,s}(x) \leq 1$ for all $x \in M$.
  \end{enumerate}
  Now, for $t \in [0,1]$, define
  \begin{align*}
    \hat{g}^{k,s}_t &:= ((1-t) + t f_{k,s}) g_0 \\
    \bar{g}^{k,s}_t &:= f_{k,s} ((1-t) g_0 + t g_1) \\
    \tilde{g}^{k,s}_t &:= ((1-t) + t f_{k,s}) g_1.
  \end{align*}
  We view these as paths in $t$ depending on the family parameter
  $s$.  Furthermore, we define a concatenated path
  \begin{equation*}
    g^{k,s}_t := \hat{g}^{k,s}_t * \bar{g}^{k,s}_t * (\tilde{g}^{k,s}_t)^{-1},
  \end{equation*}
  where of course the inverse means we run through the path backwards.
  It is easy to see that $g^{k,s}_0 = g_0$ and $g^{k,s}_1 = g_1$ for
  all $s$.  Also note that each path making up $g^{k,s}_t$ is just a
  straight-line path.  The first is from $g_0$ to $f_{k,s} g_0$, the
  second is from $f_{k,s} g_0$ to $f_{k,s} g_1$, and the third is from
  $f_{k,s} g_1$ to $g_1$.

  We now investigate the lengths of each piece of $g^{k,s}_t$
  separately, starting with that of $\hat{g}^{k,s}_t$.  Recalling that
  by Convention \ref{cvt:3}, $G_0 = g^{-1} g_0$, we compute
  \begin{align*}
    L(\hat{g}^{k,s}_t) &= \integral{0}{1}{\| (\hat{g}^{k,s}_t)' \|_{\hat{g}^{k,s}_t}}{d t} \\
    &= \integral{0}{1}{\left( \integral{M}{}{\tr_{((1 - t) + t
            f_{k,s}) g_0} \left( ((f_{k,s} - 1)g_0)^2 \right)
          \sqrt{\det \left( ((1 - t) + t f_{k,s}) G_0 \right)}}{\mu_{g}} \right)^{1/2}}{d t} \\
    &= \integral{0}{1}{\left( \integral{U_k}{}{((1 - t) + t
          f_{k,s})^{\frac{n}{2} - 2} \tr_{g_0} \left( ((1 -
            f_{k,s})g_0)^2 \right) \sqrt{\det G_0}}{\mu_{g}} \right)^{1/2}}{d t}.
  \end{align*}
  since $\det (\lambda A) = \lambda^{n/2} \det A$ for any $n \times
  n$-matrix $A$ and $\lambda \in \R$.  Note that in the last line, we
  only integrate over $U_k$, which is justified by the fact that $1 -
  f_{k,s} = 0$ on $M \setminus U_k$.  Since $s > 0$, it is easy to see
  that
  \begin{equation*}
    (1 - f_{k,s})^2 \leq (1 - s)^2 < 1,
  \end{equation*}
  so that
  \begin{equation*}
    \tr_{g_0}\left( ((1 - f_{k,s})g_0)^2 \right) = n (1 - f_{k,s})^2 <
    n.
  \end{equation*}
  This gives us the estimate
  \begin{equation*}
    L(\hat{g}^{k,s}_t) < \integral{0}{1}{\left( n \integral{U_k}{}{((1
          - t) + t f_{k,s})^{\frac{n}{2} - 2}}{\mu_{g_0}}
      \right)^{1/2}}{d t}.
  \end{equation*}
  Now, to estimate this, we note that for $n \geq 4$, $\frac{n}{2} - 2
  \geq 0$ and therefore $f_{k,s} \leq 1$ implies that
  \begin{equation*}
    ((1 - t) + t f_{k,s})^{\frac{n}{2} - 2} \leq 1.
  \end{equation*}
  So in this case,
  \begin{equation}\label{eq:72}
    L(\hat{g}_t^{k,s}) < \sqrt{n \Vol(U_k, g_0)}.
  \end{equation}
  For $1 \leq n \leq 3$, $\frac{n}{2} - 2 < 0$ and therefore one can
  compute that $f_{k,s} \geq s > 0$ implies
  \begin{equation*}
    ((1 - t) + t f_{k,s})^{\frac{n}{2} - 2} \leq (1 - t)^{\frac{n}{2}
      - 2}.
  \end{equation*}
  In this case, then,
  \begin{equation}\label{eq:71}
    L(\hat{g}_t^{k,s}) < \sqrt{n \Vol(U_k, g_0)} \integral{0}{1}{(1 -
      t)^{\frac{n}{4} - 1}}{dt},
  \end{equation}
  and the integral term is finite since $\frac{n}{4} - 1 > -1$.
  Furthermore, the value of this integral depends only on $n$.
  Putting together \eqref{eq:72} and \eqref{eq:71} therefore gives
  \begin{equation}\label{eq:75}
    L(\hat{g}^{k,s}_t) \leq C(n) \sqrt{\Vol(U_k, g_0)},
  \end{equation}
  where $C(n)$ is a constant depending only on $n$.

  In exact analogy, we can show that
  \begin{equation}\label{eq:76}
    L(\tilde{g}^{k,s}_t) \leq C(n) \sqrt{\Vol(U_k, g_1)},
  \end{equation}
  where we can even use the same constant $C(n)$.

  Next, we look at the second piece of $g^{k,s}_t$.  Here we have,
  using that $g_1 - g_0 = 0$ on $M \setminus E$,
  \begin{align*}
    \| (\bar{g}^{k,s}_t)' \|_{\bar{g}^{k,s}_t}^2 &= \integral{M}{}{
      \tr_{f_{k,s} ((1 - t) g_0 + t g_1)} \left( (f_{k,s} (g_1 -
        g_0))^2\right) 
      \sqrt{\det \left( f_{k,s} ((1 - t) G_0 + t G_1)
        \right)}}{\mu_g} \\
    &= \integral{E}{}{f_{k,s}^{n/2} \tr_{(1 -
        t) g_0 + t g_1} \left( (g_1 - g_0)^2 \right)
      \sqrt{\det ((1 - t) G_0 + t G_1)}}{\mu_g}.
  \end{align*}
  Note that in the last line above, we are only integrating over $E$,
  and the factors of $f_{k,s}$ in the trace term have canceled each
  other out.  Also note that $f_{k,s}(x) = s$ if $x \in F_k$, and
  $f_{k,s}(x) \leq 1$ for all $x \in M$.  So it follows from the above
  that
  \begin{align*}
    \| (\bar{g}^{k,s}_t)' \|_{\bar{g}^{k,s}_t}^2 &\leq s^{n/2}
    \integral{F_k}{}{\tr_{(1 - t) g_0 + t g_1} \left( (g_1 - g_0)^2
      \right) \sqrt{\det ((1 - t) G_0 + t G_1)}}{\mu_g} \\
    &\quad + \integral{E \setminus F_k}{}{\tr_{(1 - t) g_0 + t g_1}
      \left( (g_1 - g_0)^2 \right) \sqrt{\det ((1 - t) G_0 + t
        G_1)}}{\mu_g}
  \end{align*}
  For each fixed $t$ and $k$, the first term in the above clearly goes
  to zero as $s \rightarrow 0$.  By our assumption on the sets $F_k$,
  the second term goes to zero as $k \rightarrow \infty$ for each
  fixed $t$ (it does not depend on $s$ at all).  But since $t$ only
  ranges over the compact interval $[0,1]$ and all terms in the
  integrals depend smoothly on $t$, both of these convergences are
  uniform in $t$.  From this, it is easy to see that
  \begin{equation}\label{eq:135}
    \lim_{k \rightarrow \infty} \lim_{s \rightarrow 0}
    L(\bar{g}^{k,s}_t) = 0.
  \end{equation}

  With all of this preparation, we can finally use \eqref{eq:75},
  \eqref{eq:76} and \eqref{eq:135} to estimate
  \begin{equation*}
    d(g_0, g_1) \leq \inf_{k,s} L(g^{k,s}_t) \leq \lim_{k \to \infty}
    \lim_{s \to 0} L(g^{k,s}_t) \leq C(n) \left( \sqrt{\Vol(E, g_0)} +
      \sqrt{\Vol(E, g_1)} \right)
  \end{equation*}
  by our assumptions on the sets $U_k$.
\end{proof}

Before we move on with general considerations, we give two simple
examples that illustrate some important principles here.  The first
principle is, as we mentioned, that metrics that differ on a
small-volume subset of $M$ are close together, no matter how their
coefficients differ individually.  Thus, as the first example shows, a
Cauchy sequence need not converge on a set with volume zero in the
limit.  The second example demonstrates that very different paths or
sequences can be equivalent in $\overline{\M}^{\textnormal{pre}}$,
even if they become unbounded.  It also hints at a principle that
we'll elaborate on in Subsection \ref{sec:limit-volume-sing}, namely
that we can essentially \emph{ignore} that a Cauchy sequence in $\M$
becomes unbounded, taking a sequence or path of metrics that become
unbounded at some points and replacing it with a sequence or path that
remains bounded.

\begin{eg}[A $d$-Cauchy sequence that does not converge pointwise]\label{eg:3}
  Let our base manifold $M$ now be the torus $T^2$.  In the standard
  chart on the torus ($[0,1] \times [0,1]$ with edges identified), we
  define a sequence of metrics by
  \begin{equation*}
    g_k :=
    \begin{pmatrix}
      \abs{\cos k} & 0 \\
      0 & k^{-1}
    \end{pmatrix}.
  \end{equation*}
  (These are, indeed, positive definite matrices, since $\cos k \neq
  0$ for all $k \in \N$.)  On the one hand, this sequence does not
  converge pointwise, thanks to the oscillating $\abs{\cos k}$
  coefficient.  On the other hand, since clearly
  \begin{equation*}
    \lim_{k \rightarrow \infty} \Vol(T^2, g_k) = \lim_{k \rightarrow
      \infty} \sqrt{\frac{|cos k|}{k}}= 0,
  \end{equation*}
  Proposition \ref{prop:18} allows us to see that $\{g_k\}$ is indeed
  a Cauchy sequence.

  Note that since $\abs{\cos k}$ is bounded, we can select a subsequence
  $\{g_{k_m}\}$, equivalent to the original sequence, that does
  converge.  This works for the example here, but as the next example
  shows, there are Cauchy sequences and finite paths with no
  convergent subsequence.
\end{eg}

\begin{eg}[Very different, but equivalent, finite paths, and an
  example of unboundedness]\label{eg:2}
  We again let $M = T^2$, and we define a family of metrics by
  \begin{equation*}
    g^{r,s}_t :=
    \begin{pmatrix}
      e^{r t} & 0 \\
      0 & e^{-s t}
    \end{pmatrix}
  \end{equation*}
  for $t \in [1,\infty)$ and $r, s > 0$.  We consider this to be a
  path depending on $t$ for each fixed choice of $r$ and $s$.  Each
  $g^{r,s}_t$ has pointwise limit, as $t \rightarrow \infty$, the
  ``tensor''
  \begin{equation*}
    g_\infty =
    \begin{pmatrix}
      \infty & 0 \\
      0 & 0
    \end{pmatrix}.
  \end{equation*}
  Thus $g^{r,s}_t$ becomes unbounded over the entire base
  manifold.

  If we let $h^{r,s}_t = (g^{r,s}_t)'$, then it is not hard to
  directly compute that
  \begin{equation*}
    \| h^{r,s}_t \|_{g^{r,s}_t} = \sqrt{r^2 + s^2} \, e^{(s-r) t/4},
  \end{equation*}
  which is integrable on $[1, \infty)$ if and only if $s > r$, and
  therefore $g^{r,s}_t$ is finite if $s > r$.  On the other hand, we
  also have that
  \begin{equation*}
    \lim_{t \rightarrow \infty}\Vol(T^2, g^{r,s}_t) = 0
  \end{equation*}
  if and only if $s > r$.  (If $s = r$, the volume is constant, and if
  $s < r$, the volume diverges.)   Thus, by Proposition \ref{prop:18}
  (or a direct computation, if one is so inclined), we have
  \begin{equation*}
    \lim_{t \rightarrow \infty} d(g^{r,s}_t, g^{a,b}_t) = 0
  \end{equation*}
  whenever $r > s > 0$ and $a > b > 0$.  In other words, though the
  coefficients of $g^{r,s}_t$ and $g^{a,b}_t$ can differ greatly,
  these finite paths are equivalent because the volume of the set on
  which they differ vanishes in the limit.  In fact, any two paths
  $g^1_t$ and $g^2_t$ with
  \begin{equation*}
    \Vol(M, g^1_t) \rightarrow 0 \quad \textnormal{and} \quad \Vol(M,
    g^2_t) \rightarrow 0
  \end{equation*}
  are equivalent.  Therefore, we can pick a representative from the
  equivalence class $[g^{r,s}_t] \in \overline{\M}$ that does not
  become unbounded, but rather converges to a true tensor (with
  coefficients assuming values in $\R$).  A canonical choice might be
  a finite path with pointwise limit the zero section of $S^2 T^* M$.
\end{eg}

\subsection{$\omega$-convergence and its basic properties}\label{sec:ad-convergence-its}

So we now clearly see that we have to back off from the demand that
Cauchy sequences converge pointwise on their deflated sets.
Nevertheless, we can expect other nice behavior of Cauchy sequences,
and what we do expect is given in the next definition.  The definition
itself looks a bit technical, but is actually rather simple.
Therefore, after stating it in full, we will explain each of its parts
in more detail.

First, though, recall that we define general measure-theoretic notions
(e.g., the notion of something holding almost everywhere, or a.e.)
using the fixed reference metric $g$ (cf.~Convention \ref{cvt:3}).
Furthermore, we need one definition before that of
$\omega$-convergence.

\begin{dfn}\label{dfn:7}
  We denote by $\Mm$ the set of all measurable semimetrics on $M$.
  That is, $\Mm$ is the set of all sections of $S^2 T^* M$ that have
  measurable coefficients and that induce a positive semidefinite
  scalar product on $T_x M$ for each $x \in M$.
  
  Define an equivalence relation ``$\sim$'' on $\Mm$ by $g_0 \sim
  g_1$ if and only if
  \begin{enumerate}
  \item their deflated sets $X_{g_0}$ and $X_{g_1}$ differ at most
    by a nullset, and
  \item $g_0(x) = g_1(x)$ for a.e.~$x \in M \setminus (X_{g_0} \cup X_{g_1})$.
  \end{enumerate}
  We denote the quotient space of $\Mm$ by
  \begin{equation*}
    \Mmhat := \Mm / {\sim}.
  \end{equation*}
\end{dfn}

\begin{dfn}\label{dfn:13}
  Let $\{g_k\}$ be a sequence in $\M$, and let $[g_\infty] \in
  \Mmhat$.  Recall that we denote the deflated set of the sequence
  $\{g_k\}$ by $X_{\{g_k\}}$ and the deflated set of an individual
  semimetric $\tilde{g}$ by $X_{\tilde{g}}$ (cf.~Definitions
  \ref{dfn:23} and \ref{dfn:25}).  We say that $g_k$
  \emph{$\omega$-converges} to $[g_\infty]$ if for every
  representative $g_\infty \in [g_\infty]$, the following holds:
  \begin{enumerate}
  \item \label{item:4} $\{g_k\}$ is $d$-Cauchy,
  \item \label{item:5} $X_{g_\infty}$ and $X_{\{g_k\}}$ differ at most
    by a nullset,
  \item \label{item:6} $g_k(x) \rightarrow g_\infty(x)$ for a.e.~$x
    \in M \setminus X_{\{g_k\}}$, and
  \item \label{item:7} $\sum_{k=1}^\infty d(g_k, g_{k+1}) < \infty$.
  \end{enumerate}
  We call $[g_\infty]$ the \emph{$\omega$-limit} of the sequence
  $\{g_k\}$ and write $g_k \overarrow{\omega} [g_\infty]$.

  More generally, if $\{g_k\}$ is a $d$-Cauchy sequence containing a
  subsequence that $\omega$-converges to $[g_\infty]$, then we say
  that $\{g_k\}$ \emph{$\omega$-subconverges} to $[g_\infty]$.
\end{dfn}

So let's go through the definition one part at a time.

Condition (\ref{item:4}) is simply there for convenience, so we don't have to
repeatedly assume that a sequence is $\omega$-convergent \emph{and} Cauchy.

Condition (\ref{item:5}) says that the limit metric is deflated at a
point $x \in M$ if and only if $x$ is a point where $\{g_k\}$
deflates (up to a nullset where this fails to hold).

Condition (\ref{item:6}) says that $\{g_k\}$ has a pointwise limit at
almost every point off the deflated set.  Note that this limit will
necessarily be positive definite, since if $x \in M \setminus
X_{\{g_k\}}$, then there exists some $\delta(x) > 0$ such that
\begin{equation*}
  \det g_k(x) \geq \delta(x)
\end{equation*}
for all $k \in \N$ and in every chart from the amenable atlas that
contains $x$.

Finally, condition (\ref{item:7}) is technical and will aid us in
proofs.  Conceptually, it means that we can find paths $\alpha_k$
connecting $g_k$ to $g_{k+1}$ such that the concatenated path
$\alpha_1 * \alpha_2 * \cdots$ has finite length (cf.~the proof of
Theorem \ref{thm:30}).  \label{correction:conv-conditions} Given
condition (\ref{item:4}), we can always achieve this by passing to a
subsequence.  (We remark here, however, that these two conditions are
not independent.  In fact, (\ref{item:7}) implies (\ref{item:4}).)

Now that we have this definition out of the way, let's move on to
proving some properties of it.  We first state an entirely trivial
consequence of Definitions \ref{dfn:7} and \ref{dfn:13}.

\begin{lem}\label{lem:29}
  Let $[g_\infty] \in \M$, and let $\{g_k\}$ be a sequence in $\M$.
  Suppose that for one given representative $g_\infty \in [g_\infty]$,
  $\{g_k\}$ together with $g_\infty$ satisfies conditions
  \eqref{item:4}--\eqref{item:7} of Definition \ref{dfn:13}.  Then
  these conditions are also satisfied for $\{g_k\}$ together with
  every other representative of $[g_\infty]$.

  Therefore, if can we verify these conditions for one representative
  of an equivalence class, this already implies $\{g_k\}
  \overarrow{\omega} [g_\infty]$.
\end{lem}

We can thus consistently say that $\{g_k\}$ $\omega$-converges to an
individual semimetric $g_\infty \in \Mm$ if the two together satisfy
conditions \eqref{item:4}--\eqref{item:7} of Definition \ref{dfn:13}.
By the lemma, this is completely synonymous with saying that $\{g_k\}$
$\omega$-converges to the equivalence class $[g_\infty] \in \Mmhat$.
It is of course easier to show that $\{g_k\}$ $\omega$-converges to
one semimetric, rather than a whole equivalence class.  In the
following we will generally simply prove $\omega$-convergence to the
canonical choice of representative, the semimetric $g_\infty \in
[g_\infty]$ for which $g_\infty(x) = 0$ for all $x \in X_{g_\infty}$.

The next property of $\omega$-convergence is also obvious.

\begin{lem}\label{lem:41}
  If $\{g^0_k\}$ and $\{g^1_k\}$ both $\omega$-converge to the same
  element $[g_\infty] \in \Mmhat$, then $\{g^0_k\}$ and $\{g^1_k\}$
  have the same deflated set, up to a nullset.
\end{lem}
\begin{proof}
  This is immediate from property (\ref{item:5}) of Definition
  \ref{dfn:13}, as $g_\infty(x) = 0$ if and only if $x$ is in the
  deflated set of both $\{g^0_k\}$ and $\{g^1_k\}$.
\end{proof}

Recall that the main goal of this section is to show that each Cauchy
sequence in $\M$ has an $\omega$-convergent subsequence.  To do this,
we will first prove a pointwise result in the following subsection.

\subsection{(Riemannian) metrics on $\Matx$ revisited}\label{sec:metr-matx-revis}

In this subsection, we take a closer look at the Riemannian metrics
$\langle \cdot , \cdot \rangle$ and $\langle \cdot , \cdot \rangle^0$
(see Lemma \ref{lem:48} and Definition \ref{dfn:15}, respectively)
that we have defined on the finite-dimensional manifold $\Matx$.
Since the distance function $\theta^g_x$ is induced from $\langle
\cdot , \cdot \rangle^0$, and the metric $\Theta_M$ on $\M$ is defined
via $\theta^g_x$, we can get information on $\Theta_M$ by studying
$\langle \cdot , \cdot \rangle^0$.  And information on $\Theta_M$
yields information on $d$ via the estimate of Lemma \ref{prop:20}.
Furthermore, by recalling the definitions of the two Riemannian
metrics, we can see that $\langle \cdot , \cdot \rangle^0$ is
intimately related to $\langle \cdot , \cdot \rangle$:
\begin{equation}\label{eq:50}
  \langle h, k \rangle_{\tilde{g}}^0 = \langle h, k
  \rangle_{\tilde{g}} \det \tilde{G} \quad \textnormal{for all}\
  \tilde{g} \in \Matx\ \textnormal{and}\ h,k \in T_{\tilde{g}} \Matx \cong \satx.
\end{equation}
Thus, we will first study the simpler Riemannian metric $\langle \cdot
, \cdot \rangle$ and find out what properties of $\langle \cdot ,
\cdot \rangle^0$ we can deduce in this way.

We first make the observation that $\langle \cdot , \cdot \rangle$
differs very much in character from its integrated version $( \cdot ,
\cdot )$.  In particular, $\Matx$ is complete with respect to $\langle
\cdot , \cdot \rangle$!  This is is not hard to see, as we can solve
the geodesic equation of $\langle \cdot , \cdot \rangle$ directly.
Following the analogous computation for $( \cdot , \cdot )$ on $\M$
carried out in \cite[Thm.~2.3]{freed89:_basic_geomet_of_manif_of}, we
first calculate the Christoffel symbols of $\langle \cdot , \cdot
\rangle$ and then use them to solve the geodesic equation.  We note
that our computation is basically just a simplified version of that in
\cite{freed89:_basic_geomet_of_manif_of}.

Before we start, let's clear up some notation.

\begin{dfn}\label{dfn:11}
  By $d_x$, we denote the distance function induced on $\Matx$ by
  $\langle \cdot , \cdot \rangle$.  We denote
  the $\langle \cdot , \cdot \rangle$-length of a path $a_t$ in
  $\Matx$ by $L^{\langle \cdot , \cdot \rangle}(a_t)$ and the $\langle
  \cdot , \cdot \rangle^0$-length by $L^{\langle \cdot, \cdot
    \rangle^0}(a_t)$.
\end{dfn}

Now we compute the Christoffel symbols.

\begin{prop}\label{prop:13}
  Let $h$ and $k$ be constant vector fields on $\Matx$, and denote the
  Levi-Civita connection of $\langle \cdot , \cdot \rangle$ by
  $\nabla$.  Then the Christoffel symbols of $\langle \cdot , \cdot
  \rangle$ are given by
  \begin{equation*}
    \Gamma(h,k) = \nabla_h k |_{\tilde{g}} = - \frac{1}{2} \left( h \tilde{g}^{-1} k +
      k \tilde{g}^{-1} h \right).
  \end{equation*}
\end{prop}
\begin{proof}
  All computations are done at the base point $\tilde{g}$, which we
  will omit from the notation for convenience.  Let $\ell$ be any
  other constant vector field on $\Matx$.  By the Koszul formula,
  \begin{equation*}
    2 \langle \nabla_h k, \ell \rangle = h \langle k , \ell \rangle +
    k \langle \ell , h \rangle - \ell \langle h , k \rangle - \langle
    h , [k, \ell] \rangle - \langle
    k , [\ell, h] \rangle + \langle
    \ell , [h, k] \rangle.
  \end{equation*}
  Notice, however, that the last three terms drop out, since $h$, $k$,
  and $\ell$ are all constant, so their brackets with each other are
  zero.  Therefore we have
  \begin{equation}\label{eq:66}
    2 \langle \nabla_h k, \ell \rangle = h \langle k , \ell \rangle +
    k \langle \ell , h \rangle - \ell \langle h , k \rangle.
  \end{equation}

  Now, it is well-known (and easy to verify by differentiating
  $\hat{g} \hat{g}^{-1} = I$) that the derivative of the map
  $\hat{g} \mapsto \hat{g}^{-1}$ at the point $\tilde{g}$ is given by
  $a \mapsto -\tilde{g}^{-1} a \tilde{g}^{-1}$.  Using this, along
  with the definition of $\langle \cdot , \cdot \rangle$ and the fact
  that $a,b \mapsto \tr(ab)$ is bilinear, we get (denoting the
  derivative of a function $f$ in the direction $h$ by $h[f]$)
  \begin{align*}
    h \langle k , \ell \rangle &= h[\tr_{\tilde{g}}(k \ell)] =
    h \left[ \tr\left( (\tilde{g}^{-1} k) (\tilde{g}^{-1} \ell) \right) \right] \\
    &= -\tr\left( (\tilde{g}^{-1} h \tilde{g}^{-1} k) (\tilde{g}^{-1}
      \ell) \right) - \tr\left( (\tilde{g}^{-1} k) (\tilde{g}^{-1} h
      \tilde{g}^{-1} \ell) \right).
  \end{align*}
  Repeating the same computation for the other permutations and
  substituting the results into \eqref{eq:66} yields
  \begin{align*}
    2 \langle \nabla_h k, \ell \rangle &= -\tr\left( (\tilde{g}^{-1} h
      \tilde{g}^{-1} k) (\tilde{g}^{-1} \ell) \right) - \tr\left(
      (\tilde{g}^{-1} k) (\tilde{g}^{-1} h
      \tilde{g}^{-1} \ell) \right) \\
    &\quad\, - \tr\left( (\tilde{g}^{-1} k \tilde{g}^{-1} \ell)
      (\tilde{g}^{-1} h) \right) - \tr\left( (\tilde{g}^{-1} \ell)
      (\tilde{g}^{-1} k
      \tilde{g}^{-1} h) \right) \\
    &\quad\, + \tr\left( (\tilde{g}^{-1} \ell \tilde{g}^{-1} h)
      (\tilde{g}^{-1} k) \right) + \tr\left( (\tilde{g}^{-1} h)
      (\tilde{g}^{-1} \ell \tilde{g}^{-1} k) \right) \\
    &= -\tr \left( \tilde{g}^{-1} h \tilde{g}^{-1} k \tilde{g}^{-1}
      \ell \right) - \tr \left( \tilde{g}^{-1} k \tilde{g}^{-1} h
      \tilde{g}^{-1} \ell \right) \\
    &= -\langle h \tilde{g}^{-1} k, \ell \rangle - \langle k
    \tilde{g}^{-1} h, \ell \rangle,
  \end{align*}
  where in the second-to-last line we have used the invariance of the
  trace under cyclic permutations.  The result now follows directly.
\end{proof}

Using this, it is a relatively simple matter to solve the geodesic
equation of $\langle \cdot , \cdot \rangle$.

\begin{prop}\label{prop:12}
  The geodesic $g_t$ in $(\Matx, \langle \cdot , \cdot \rangle)$ with
  initial data $g_0$, $g'_0$ is given by
  \begin{equation*}
    g_t = g_0 e^{t g_0^{-1} g'_0}.
  \end{equation*}
  In particular, $(\Matx, d_x)$ is a complete metric space.
\end{prop}

\begin{rmk}\label{rmk:7}
  Note that this formula is exactly the same as the geodesic equation
  for $(\M_\mu, ( \cdot , \cdot ))$ (cf.~Proposition \ref{prop:24}).
  This is no accident, and occurs because the volume form is constant
  over $\M_\mu$.  Therefore there is no contribution to the
  Christoffel symbols coming from the volume form.  As a result, all
  of the dynamics come from the integrand of $( \cdot , \cdot )$, and
  the integrand is exactly $\langle \cdot , \cdot \rangle$.

  Nevertheless, since we haven't derived the geodesic equation for
  $\M_\mu$, we prefer to prove Proposition \ref{prop:12} directly,
  without resorting to indirect arguments.  The proof is not hard, anyway.
\end{rmk}

\begin{proof}[Proof of the proposition]
  Let $a_t := g'_t$.  Since $g_t$ is a geodesic, we have $\nabla_{a_t}
  a_t = 0$.  Therefore
  \begin{equation}\label{eq:67}
    0 = a'_t + \Gamma(a_t, a_t) = a'_t - a_t g_t^{-1} a_t
  \end{equation}
  by Proposition \ref{prop:13}.  Now, since $g'_t = a_t$, the
  $t$-derivative of $t \mapsto g_t^{-1}$ is the same as the derivative
  of $g_t \mapsto g_t^{-1}$ in the direction of $a_t$.  Hence,
  $(g_t^{-1} a_t)' = g_t^{-1} a'_t - g_t^{-1} a_t g_t^{-1} a_t$.  It
  is then easy to see that multiplying \eqref{eq:67} on the left by
  $g_t^{-1}$ gives
  \begin{equation*}
    (g_t^{-1} a_t)' = 0.
  \end{equation*}
  Thus $g_t^{-1} g'_t$ is constant, or $\log(g_t)' = g_t^{-1} g'_t
  \equiv g_0^{-1} g'_0$.  The geodesic equation now follows, and it
  remains to show that $(\M_x, d_x)$ is complete.

  Since $A \mapsto e^A$ maps symmetric matrices into positive definite
  matrices and $t g_0^{-1} g'_0$ is $g_0$-symmetric, $g_0 e^{t
    g_0^{-1} g'_0}$ is a positive definite matrix for all $t \in
  (-\infty, \infty)$.  Thus $g_t$ is positive definite for all $t$,
  and so $(\Matx, \langle \cdot , \cdot \rangle)$ is geodesically
  complete.  Since $\Matx$ is finite-dimensional, the Hopf-Rinow
  theorem applies to show that $(\Matx, d_x)$ is complete.
\end{proof}

From Proposition \ref{prop:12} and \eqref{eq:50}, we would suspect
that a finite-length path $g_t$ in $(\Matx, \langle \cdot , \cdot
\rangle^0)$ can deflate or become unbounded only if $\det g_t$
converges to zero---otherwise, the $\langle \cdot , \cdot
\rangle^0$-length of $g_t$ is related to the $\langle \cdot , \cdot
\rangle$-length by a constant, and a path with finite $\langle \cdot ,
\cdot \rangle$-length lies completely within $\Matx$.  The following
lemma and proposition confirm this suspicion.  The lemma is the
pointwise version of Lemma \ref{lem:13}, and the proposition is our
first concrete step towards proving existence of the
$\omega$-limit---it is the necessary pointwise result.

\begin{lem}\label{lem:34}
  Let $a_0, a_1 \in \Matx$.  Then
  \begin{equation*}
    \left|
      \sqrt{\det A_1} - \sqrt{\det A_0}
    \right| \leq \frac{\sqrt{n}}{2} \theta^g_x (a_0, a_1).
  \end{equation*}
  (Recall Convention \ref{cvt:3} for the definitions of $A_i$.)
\end{lem}
\begin{proof}
  Let $a_t$, $t \in [0,1]$, be any path from $a_0$ to $a_1$, and
  recall that $A_t = g^{-1} a_t$ (cf.~Convention \ref{cvt:3}).
  Following the proof of Lemma \ref{lem:13}, we have
  \begin{align*}
    \partial_t \sqrt{\det A_t} &= \frac{1}{2} \left( \tr_{a_t} a'_t
    \right) \sqrt{\det A_t} = \frac{1}{2} \left( \left( \tr_{a_t} a'_t
      \right)^2 \det A_t \right)^{1/2} \\
    &\leq \frac{1}{2} \left( n \tr_{a_t} ((a'_t)^2)  \det
      A_t \right)^{1/2} = \frac{\sqrt{n}}{2} \sqrt{\langle a'_t , a'_t
      \rangle^0_{a_t}},
  \end{align*}
  where the inequality follows, as in the proof of Lemma \ref{lem:13},
  from \eqref{eq:64}.  This now implies that
  \begin{equation*}
    \sqrt{\det A_1} - \sqrt{\det A_0} = \integral{0}{1}{\partial_t
      \sqrt{\det A_t}}{d t} \leq \frac{\sqrt{n}}{2}
    \integral{0}{1}{\sqrt{\langle a'_t , a'_t \rangle^0_{a_t}}}{d t} =
    \frac{\sqrt{n}}{2} L(a_t).
  \end{equation*}
  Since this holds for all paths, we can replace the far right-hand
  side with $\frac{\sqrt{n}}{2} \theta^g_x (a_0, a_1)$.  Now repeating
  the computation with $a_0$ and $a_1$ swapped completes the proof.
\end{proof}

\begin{prop}\label{prop:14}
  Let $a_k$ be a $\theta^g_x$-Cauchy sequence.  Then either
  \begin{enumerate}
  \item $\det A_k \rightarrow 0$ for $k \rightarrow \infty$, or
  \item there exist constants $C,\eta > 0$ such that $|(a_k)_{ij}|
    \leq C$ and $\det A_k \geq \eta$ for all $1 \leq i,j \leq n$ and
    $k \in \N$.
  \end{enumerate}
\end{prop}
\begin{proof}
  Keeping Lemma \ref{lem:34} in mind, it is more convenient to work
  with the square root of the determinant.  This is, of course,
  completely equivalent for our purposes.

  Now, by Lemma \ref{lem:34}, the map $a \mapsto \sqrt{\det A}$ is
  $\theta^g_x$-Lipschitz.  Since $a_k$ is $\theta^g_x$-Cauchy, it is
  easy to see that $\lim_{k \to \infty} \sqrt{\det A_k}$ exists, so
  let's call this limit $L$.

  If for every $\eta > 0$, there exists $k$ such that $\sqrt{\det A_k}
  \leq \eta$, then clearly $L = 0$.

  It remains to show that if there exist $i$ and $j$ such that for all
  $C > 0$, there is a $k$ such that $|(a_k)_{ij}| > C$, then
  $\sqrt{\det A_k} \rightarrow 0$.  We will assume that $L > 0$ and
  show a contradiction.

  Let's say that we are given $b_0, b_1 \in \Matx$ with $\det B_0,
  \det B_1 \geq \delta$.  Let
  \begin{align*}
    L_{-\delta} &:=
    \inf \left\{
      L^{\langle \cdot, \cdot \rangle^0}(b_t) \mid b_t\ \textnormal{is a path from $b_0$ to $b_1$ with}\ \det
      B_t \leq \delta/2\ \textnormal{for some}\ t \in (0,1)
    \right\}, \\
    L_{+\delta} &:=
    \inf \left\{
      L^{\langle \cdot, \cdot \rangle^0}(b_t) \mid b_t\ \textnormal{is a path from $b_0$ to $b_1$ with}\ \det
      B_t \geq \delta/2\ \textnormal{for all}\ t \in [0,1]
    \right\}.
  \end{align*}
  It is easy to see that $\theta^g_x(b_0, b_1) = \min (L_{-\delta},
  L_{+\delta})$.  Now let $b_t$ be a path as in the definition of
  $L_{-\delta}$, and assume $\tau \in (0,1)$ is such that $\det B_\tau
  \leq \delta/2$.  Then using Lemma \ref{lem:34}, we have
  \begin{align*}
    L^{\langle \cdot, \cdot \rangle^0}(b_t) &= L^{\langle \cdot, \cdot \rangle^0}(b_t|_{[0,\tau]}) + L^{\langle \cdot, \cdot \rangle^0}(b_t|_{[\tau,1]}) \\
    &\geq \frac{\sqrt{n}}{2} \left| \sqrt{\det B_0} - \sqrt{\det
        B_\tau} \right| + \frac{\sqrt{n}}{2} \left| \sqrt{\det B_1} -
      \sqrt{\det B_\tau}
    \right| \\
    &\geq \sqrt{n} \left( \sqrt{\delta} - \sqrt{\frac{\delta}{2}}
    \right) = \sqrt{n} \left( 1 - \frac{1}{\sqrt{2}} \right)
    \sqrt{\delta}.
  \end{align*}
  Therefore $L_{-\delta} \geq \sqrt{n} (1 - 1/\sqrt{2})
  \sqrt{\delta}$.  Then, if $b_t$ is a path as in the definition of
  $L_{+\delta}$, we have
  \begin{equation*}
    L(b_t) = \integral{0}{1}{\sqrt{\langle b'_t , b'_t \rangle^0}}{d
      t} = \integral{0}{1}{\sqrt{\langle b'_t , b'_t \rangle} \det B_t}{d t}
    \geq \sqrt{\frac{\delta}{2}} \integral{0}{1}{\sqrt{\langle b'_t ,
        b'_t \rangle}}{d t} \geq \sqrt{\frac{\delta}{2}} \, d_x(b_0, b_1).
  \end{equation*}
  This gives $L_{+\delta} \geq \sqrt{\delta / 2} \, d_x(b_0, b_1)$.
  Putting all of this together, we get that
  \begin{equation}\label{eq:65}
    \theta^g_x(b_0, b_1) \geq \min \{ \sqrt{n} (1 - 1/\sqrt{2}) \sqrt{\delta} ,
    \sqrt{\delta / 2} \, d_x(b_0, b_1) \}
  \end{equation}
  whenever $\det B_0, \det B_1 \geq \delta$.

  Now, let's apply the considerations of the last paragraph to the
  problem at hand.  Let $i$ and $j$ be, as above, the indices for which
  $|(a_k)_{ij}|$ is unbounded, and choose a subsequence, which we
  again denote by $a_k$, such that $|(a_k)_{ij}| \geq k$ for all $k
  \in \N$.  Passing to this subsequence does not change the limit
  $\lim_{k \to \infty} \sqrt{\det A_k}$.

  Next, choose $K \in \N$ such that $k \geq K$ implies $\sqrt{\det
    A_k} \geq L/2$ and $k,l \geq K$ implies $\theta^g_x(a_k, a_l) \leq
  \frac{1}{2} \sqrt{n} (1 - 1/\sqrt{2}) \sqrt{L/2}$.  The latter
  assumption is possible since $a_k$ is Cauchy.  By \eqref{eq:65}, if
  $k \geq K$, we also have
  \begin{equation*}
    \theta^g_x(a_K, a_k) \geq \min \{ \sqrt{n} (1 - 1 / \sqrt{2}) \sqrt{L/2} ,
    \sqrt{L / 4} \, d_x(a_K, a_k) \}.
  \end{equation*}
  But $\theta^g_x(a_K, a_k) \geq \sqrt{n} (1 - 1 / \sqrt{2})
  \sqrt{L/2}$ violates our assumptions on $K$. Furthermore, $d_x(a_K,
  a_k) \rightarrow \infty$ since $|(a_k)_{ij}| \rightarrow \infty$ and
  $(\Matx, \langle \cdot , \cdot \rangle)$ is complete.  Therefore, if
  $\theta^g_x(a_K, a_k) \geq \sqrt{L / 4} \, d_x(a_K, a_k)$ for all
  $k$, then our assumptions on $K$ are violated as well.  Thus we have
  achieved the desired contradiction.
\end{proof}

Since for every pair of constants $C, \eta > 0$, the set of elements
$\tilde{g}$ of $\Matx$ with $|\tilde{g}_{ij}| \leq C$ and $\det
\tilde{G} \geq \eta$ for all $1 \leq i,j \leq n$ is compact, we
immediately get the following corollary of Proposition \ref{prop:14}:

\begin{cor}\label{cor:10}
  Let $\{g_k\}$ be a $\theta^g_x$-Cauchy sequence.  Then either
  \begin{enumerate}
  \item $\det G_k \rightarrow 0$ for $k \rightarrow \infty$, or
  \item there exists an element $g_\infty \in \Matx$ such that $g_k
    \rightarrow g_\infty$, with convergence in the manifold topology
    of $\M_x$.
  \end{enumerate}
\end{cor}
\begin{proof}
  By the discussion preceding the corollary, if $\det G_k$ is bounded
  away from zero, then $\{g_k\}$ is contained within a compact subset
  of $\Matx$.  Since a compact subset of a metric space is complete
  and $\{g_k\}$ is Cauchy, it $\theta^g_x$-converges to some limit
  $g_\infty$.  Finally, since $\Matx$ is finite dimensional, the
  topology induced by $\theta^g_x$ coincides with the topology of
  $\Matx$ as a manifold (or open subset of $\satx$), so in fact $g_k
  \rightarrow g_\infty$.
\end{proof}

This is essentially the pointwise equivalent of $\omega$-convergence.  In
the next subsection, we will globalize this result.  Before we do
that, though, we use this opportune moment to prove two last pointwise
results, which will be useful in Section
\ref{sec:uniqueness-ad-limit}.  The first is the pointwise analog of
Proposition \ref{prop:18}.

\begin{prop}\label{prop:25}
  Let $\tilde{g}, \hat{g} \in \M_x$.  Then there exists a constant
  $C'(n)$, depending only on $n$, such that
  \begin{equation*}
    \theta^g_x(\tilde{g}, \hat{g}) \leq C'(n) \left( \sqrt{\det \tilde{G}} +
      \sqrt{\det \hat{G}} \right).
  \end{equation*}
\end{prop}
\begin{proof}
  For this proof, we will denote the $\langle \cdot , \cdot
  \rangle^0$-length of a path simply by $L$.  The metric $\langle
  \cdot , \cdot \rangle$ does not play a role here, and so there is no
  need to distinguish between the two lengths.
  
  The proof goes very similarly to Proposition \ref{prop:18}, but is
  simpler because we do not need to do any argument approximating
  paths of $L^2$ metrics by $C^\infty$ metrics.  Since the ideas are
  the same, the proof can be safely skipped, but we include it here
  for completeness.

  First, define paths $\tilde{g}^s_t$ and $\hat{g}^s_t$, for $0 < s
  \leq 1$ and $t \in [s,1]$, by
  \begin{equation*}
    \tilde{g}^s_t := t \tilde{g} \quad \textnormal{and} \quad
    \hat{g}^s_t := t \hat{g}.
  \end{equation*}
  We consider these as a family of paths in the time variable $t$ with
  domain depending on the family parameter $s$.

  Second, define a family $\bar{g}^s_t$ of paths in $t$ depending on
  the family parameter $s$ by
  \begin{equation*}
    \bar{g}^s_t := s \left( (1 - t) \tilde{g} + t \hat{g} \right),
  \end{equation*}
  where again $0 < s \leq 1$ but this time $t \in [0,1]$.

  Then the concatenation $g^s_t := (\tilde{g}^s_t)^{-1} * \bar{g}^s_t
  * \hat{g}^s_t$ (here, $(\tilde{g}^s_t)^{-1}$ means we run through
  that path backwards) is, for each $s$, a path from $\tilde{g}$ to
  $\hat{g}$.  We will prove that
  \begin{equation*}
    \lim_{s \to 0} L(g^s_t) \leq
    C'(n) \left( \sqrt{\det \tilde{G}} + \sqrt{\det \hat{G}} \right),
  \end{equation*}
  which will imply the result immediately.

  First, note that $L(\tilde{g}^s_t) \leq \lim_{s \to 0}
  L(\tilde{g}^s_t)$ for all $s$.  To compute the right-hand side, note
  that
  \begin{equation*}
    \langle (\tilde{g}^s_t)', (\tilde{g}^s_t)'
    \rangle^0_{\tilde{g}^s_t} = \tr_{t \tilde{g}}(\tilde{g}^2) \det (t
    \tilde{G}) = (n \det \tilde{G}) t^{\frac{n}{2} - 2}.
  \end{equation*}
  Therefore,
  \begin{equation*}
    L(\tilde{g}^s_t) \leq \lim_{s \to 0} L(\tilde{g}^s_t) = \sqrt{n \det \tilde{G}}
    \integral{0}{1}{t^{\frac{n}{4} - 1}}{d t}.
  \end{equation*}
  Since $\frac{n}{4} - 1 > -1$, the above integral is finite, with a
  value depending only on $n$.  Hence we have
  \begin{equation*}
    L(\tilde{g}^s_t) \leq C'(n) \sqrt{\det \tilde{G}}.
  \end{equation*}

  In exactly the same way, we can show
  \begin{equation*}
    L(\hat{g}^s_t) \leq C'(n) \sqrt{\det \hat{G}},
  \end{equation*}
  even using the same constant.

  Now, if we can show that $\lim_{s \to 0} L(\bar{g}^s_t) = 0$, we will be
  finished.  So we compute
  \begin{align*}
    \langle (\bar{g}^s_t)', (\bar{g}^s_t)' \rangle^0_{\bar{g}^s_t} &= \tr_{s((1 - t)
      \hat{g} + t \tilde{g})} \left( s^2 (\tilde{g} - \hat{g})^2
    \right)
    \det(s((1 - t) \tilde{G} + t \hat{G})) \\
    &= s^{n/2} \tr_{((1 - t) \tilde{g} + t \hat{g}} \left( (\hat{g} -
      \tilde{g})^2 \right) \det ((1-t) \tilde{G} + t \hat{G}) \\
    &= s^{n/2} \langle (g^1_t)', (g^1_t)' \rangle^0_{g^1_t}.
  \end{align*}
  This implies that
  \begin{equation*}
    L(\bar{g}^s_t) = s^{n/2} L(\bar{g}^1_t),
  \end{equation*}
  from which $\lim_{s \to 0} L(\bar{g}^s_t) = 0$ is immediate.  This
  completes the proof.
\end{proof}

The last pointwise result we need combines Corollary \ref{cor:10} and
Proposition \ref{prop:25} to give a description of the completion of
the metric space $(\Matx, \theta^g_x)$.

\begin{thm}\label{thm:35}
  For any given $x \in M$, let $\textnormal{cl}(\Matx)$ denote the
  closure of $\Matx \subset \satx$ with regard to the natural
  topology.  Then $\cl(\Matx)$ consists of all positive semidefinite
  $(0,2)$-tensors at $x$.  Let us denote the boundary of $\Matx$, as a
  subspace of $\satx$, by $\partial \Matx$.

  Define an equivalence relation on $\textnormal{cl}(\Matx)$ by $g_0
  \sim g_1$ if and only if $g_0, g_1 \in \partial \Matx$.  Thus, we
  simply identify the boundary of $\Matx$ together to a point.
  
  Then the completion of $(\Matx, \theta^g_x)$ can be identified with
  the space $\textnormal{cl}(\Matx) / {\sim}$.  The distance
  function is given by
  \begin{equation*}
    \theta^g_x(g_0, g_1) = \lim_{k \rightarrow
      \infty}\theta^g_x(g^0_k, g^1_k),
  \end{equation*}
  where $\{g^0_k\}$ and $\{g^1_k\}$ are any sequences in $\Matx$
  converging (in the topology of $\satx$) to $g_0$ and $g_1$,
  respectively.
\end{thm}
\begin{proof}
  Note that $\tilde{g} \mapsto \det \tilde{G}$ is a continuous map
  from $\satx$ to the reals, that the map is positive when restricted
  to $\Matx$, and that it is constantly zero when restricted to
  $\partial \Matx$.  The latter facts are implied by Proposition
  \ref{prop:7}.

  Let $\{g_k\}$ be any sequence in $\Matx$.  By Corollary
  \ref{cor:10}, if $\{g_k\}$ is Cauchy then either $g_k \rightarrow
  g_\infty \in \Matx$ (with convergence in the topology of $\satx$),
  or $\det G_k \rightarrow 0$.  By Proposition \ref{prop:25}, all
  sequences with $\det G_k \rightarrow 0$ are equivalent Cauchy
  sequences, and so they are identified in $\overline{(\Matx,
    \theta^g_x)}$.  Since the determinant is a continuous map, as
  noted above, we can thus identify such sequences with any given
  sequence converging to $\partial \Matx$.

  Finally, if $g_k \rightarrow g_\infty \in \Matx$ in the topology of
  $\satx$, then we also have that $\theta^g_x(g_k, g_\infty)
  \rightarrow 0$, because $\Matx$ is finite dimensional and so the
  topology of $\theta^g_x$ coincides with the manifold topology.  This
  implies that $\{g_k\}$ is Cauchy.  By the same reasoning, we can
  show that if $\{\tilde{g}_k\}$ is a second sequence converging to
  $g_\infty$ in the topology of $\satx$, then $\{g_k\}$ and
  $\{\tilde{g}_k\}$ are equivalent.

  We have thus shown that $\{g_k\}$ is Cauchy if and only if either
  $g_k \rightarrow g_\infty \in \Matx$ or $\det G_k \rightarrow 0$
  holds.  We have also shown that all sequences with $\det G_k
  \rightarrow 0$ are equivalent, and that all sequences converging to
  the same element of $\Matx$ are equivalent.  The statement of the
  theorem now follows.
\end{proof}

\subsection{The existence proof}\label{sec:limit-volume-sing}

Corollary \ref{cor:10} gives us strong hints as to what to expect from
Cauchy sequences in $\M$.  Thinking heuristically, a $d$-Cauchy
sequence $\{g_k\}$ in $\M$ should be a Cauchy sequence in $\theta^g_x$
for ``most'' points $x$ by the estimate of Proposition \ref{prop:20}.
Then we know that at ``most'' points $x$, either $\{g_k(x)\}$
converges or $\det G_k(x) \rightarrow 0$.  That is, $\{g_k\}$
converges at ``most'' points outside the deflated set.  The goal of
this subsection is to make this heuristic idea precise and use it to
prove existence of the $\omega$-limit.

\begin{lem}\label{lem:35}
  Let $\{g_k\}$ be a Cauchy sequence in $\M$.  By passing to a
  subsequence if necessary, we can assume that
  \begin{equation*}
    \sum_{k=1}^\infty d(g_k, g_{k+1}) < \infty.
  \end{equation*}

  Then the following holds:
  \begin{equation*}
    \sum_{k=1}^\infty \Theta_M (g_k, g_{k+1}) < \infty.
  \end{equation*}
  Furthermore, define functions $\Omega$ and $\Omega_N$ for each $N
  \in \N$ by
  \begin{equation*}
    \Omega_N := \sum_{k=1}^N \theta^g_x(g_k(x), g_{k+1}(x)),
    \quad \Omega := \sum_{k=1}^\infty \theta^g_x(g_k(x), g_{k+1}(x)).
  \end{equation*}
  Then $\Omega$ is a.e.~finite, $\Omega \in L^1(M, g)$ and $\Omega_N
  \overarrow{L^1} \Omega$.  Furthermore, by definition, $\Omega_N$
  converges to $\Omega$ pointwise.
\end{lem}
\begin{proof}
  The first statement is clear, as is the statement that $\Omega_N
  \rightarrow \Omega$ pointwise.  So we move on to the other
  statements.

  Lemma \ref{lem:13} implies that $\sqrt{\Vol(M, g_k)}$ is a Cauchy
  sequence in $\R$.  Therefore it is bounded, and we can find a
  constant $V$ such that $\sqrt{\Vol(M, g_k)} \leq V$ for all $k$.
  Thus, by Proposition \ref{prop:20},
  \begin{equation*}
    \Theta_M(g_k, g_{k+1}) \leq 2 d(g_k, g_{k+1})
    \left( \frac{\sqrt{n}}{2} d(g_k, g_{k+1}) + V \right).
  \end{equation*}
  But for large $k$, since $\{g_k\}$ is Cauchy, we must have $d(g_k,
  g_{k+1}) \leq 1$, so
  \begin{equation*}
    \Theta_M(g_k, g_{k+1}) \leq \sqrt{n} \, d(g_k,
    g_{k+1})^2 + V d(g_k, g_{k+1}) \leq (\sqrt{n} + V) d(g_k, g_{k+1}).
  \end{equation*}
  The first statement is now immediate.

  To prove the second statement, we recall the monotone convergence
  theorem of Lebesgue and Levi
  \cite[Thm.~2.8.2]{bogachev07:_measur_theor}.  Let $(X, \Sigma, \nu)$
  be a measure space, and let $h_i$, for $i \in \N$, be measurable
  functions $X \rightarrow [0,+\infty]$.  Suppose that $h_i \leq
  h_{i+1}$ for all $i \in \N$ and a.e.~$x \in X$, and furthermore that
  $\sup_i \integral{}{}{h_i}{d \nu} < \infty$.  Then the function
  defined by $h(x) := \lim h_i(x)$ is a.e.~finite, and
  \begin{equation*}
    \lim_{i \to \infty} \integral{X}{}{h_i}{d\nu} = \integral{X}{}{h}{d\nu}.
  \end{equation*}
  This theorem implies a criterion for exchanging infinite sums and
  integrals.  In particular, let $f_i$ be a sequence of nonnegative
  measurable functions on $X$.  Let $F_N$ be the partial sum of the
  first $N$ elements and define $F := \lim_{N \to \infty} F_N$.
  Suppose that $\sup_N \integral{}{}{F_N}{d \nu} < \infty$.  Then $F$
  and $F_N$ clearly satisfy the requirements of the monotone
  convergence theorem, so we have
  \begin{equation*}
    \sum_{i=1}^\infty \integral{X}{}{f_i}{d \nu} = \lim_{N \to \infty}
    \sum_{i=1}^N \integral{X}{}{f_i}{d \nu} = \lim_{N \to \infty}
    \integral{X}{}{F_N}{d \nu} = \integral{X}{}{F}{d \nu} =
    \integral{X}{}{\left( \sum_{i=1}^\infty f_i \right)}{d \nu}.
  \end{equation*}

  We can apply this to $\Omega$ and $\Omega_N$ to obtain
  \begin{equation*}
    \integral{M}{}{\Omega}{\mu_g} = \lim_{N \to \infty}
    \integral{M}{}{\Omega_N}{\mu_g} = \lim_{N \to \infty}
    \sum_{k=1}^N \integral{M}{}{\theta^g_x(g_k, g_{k+1})}{\mu_g} = \sum_{k=1}^\infty \Theta_M(g_k, g_{k+1}) < \infty,
  \end{equation*}
  where finiteness follows from the first part of the lemma.  This
  proves that $\Omega$ is a.e.~finite and $\Omega \in L^1(M,g)$.  It
  remains to show that $\Omega_N \overarrow{L^1} \Omega$.  But this is
  now immediate from
  \cite[Thm.~8.5.1]{rana02:_introd_to_measur_and_integ}, which states
  that if $1 \leq p < \infty$, $f_i \rightarrow f$ a.e.~and $\| f_i
  \|_p \rightarrow \| f \|_p$, then $f_i \overarrow{L^p} f$.
\end{proof}

Using this lemma, we can prove what we heuristically described
before---that given a Cauchy sequence $\{g_k\}$, we can find a
subsequence $\{g_{k_m}\}$ such that $\{g_{k_m}(x)\}$ is
$\theta^g_x$-Cauchy for ``most'' $x$, allowing us to apply Proposition
\ref{prop:14} at these points.

\begin{prop}\label{prop:15}
  Let $\{g_k\}$ be a Cauchy sequence in $\M$ such that
  \begin{equation*}
    \sum_{k=1}^\infty d(g_k, g_{k+1}) < \infty.
  \end{equation*}
  Then $\{g_k(x)\}$ is a $\theta^g_x$-Cauchy sequence for a.e.~$x \in M$.
\end{prop}
\begin{proof}
  By our assumption, all the conclusions of Lemma \ref{lem:35} hold.
  In particular, $\Omega_N \rightarrow \Omega$ pointwise and $\Omega$
  is a.e.~finite.  Therefore, for a.e.~$x \in M$,
  \begin{equation}\label{eq:69}
    \sum_{k=1}^\infty \theta^g_x(g_k, g_{k+1}) = \Omega(x) < \infty.
  \end{equation}
  
  It is then simple to show that $\{g_k(x)\}$ is $\theta^g_x$-Cauchy at a
  point where \eqref{eq:69} holds, for if $l \leq m$,
  \begin{equation*}
    \theta^g_x(g_l, g_m) \leq \sum_{k=l}^m \theta^g_x(g_k, g_{k+1})
  \end{equation*}
  by the triangle inequality.  But \eqref{eq:69} shows that the
  right-hand side of the above is small for $l$ and $m$ large, proving
  that $\{g_k(x)\}$ is $\theta^g_x$-Cauchy.
\end{proof}

The previous proposition allows us to globalize Corollary
\ref{cor:10}.  The precise statement is the following:

\begin{cor}\label{cor:13}
  Let $\{g_k\}$ be a Cauchy sequence in $\M$ such that
  \begin{equation*}
    \sum_{k=1}^\infty d(g_k, g_{k+1}) < \infty.
  \end{equation*}
  Then for a.e.~$x \in M$, $\{g_k(x)\}$ is $\theta^g_x$-Cauchy and
  either:
  \begin{enumerate}
  \item \label{item:12} $\det G_{t_k}(x) \rightarrow 0$ for $k \rightarrow \infty$, or
  \item \label{item:13} $g_k(x)$ is a convergent sequence in $\Matx$.
  \end{enumerate}
  Furthermore, \eqref{item:12} holds for a.e.~$x \in X_{\{g_k\}}$, and
  \eqref{item:13} holds for a.e.~$x \in M \setminus X_{\{g_k\}}$.
\end{cor}
\begin{proof}
  By Proposition \ref{prop:15}, $\{g_k(x)\}$ is $\theta^g_x$-Cauchy
  for a.e.~$x$.  Then Corollary \ref{cor:10} implies the result
  immediately.
\end{proof}

This corollary essentially delivers us the proof of the existence
result.

\begin{thm}\label{thm:39}
  For every Cauchy sequence $\{g_k\}$, there exists an element
  $[g_\infty] \in \Mmhat$ and a subsequence $\{g_{k_l}\}$ such that
  $\{g_{k_l}\}$ $\omega$-converges to $[g_\infty]$.

  Explicitly, $[g_\infty]$ is the unique equivalence class containing
  the element $g_\infty \in \Mm$ defined as follows.  At points $x \in
  M$ where $\{g_{k_l}(x)\}$ is $\theta^g_x$-Cauchy,
  \begin{enumerate}
  \item $g_\infty(x) := 0$ for $x \in X_{\{g_{k_l}\}}$ and
  \item $g_\infty(x) := \lim g_{k_l}(x)$ for $x \in M \setminus
    X_{\{g_{k_l}\}}$.
  \end{enumerate}
  At points $x \in M$ where $\{g_{k_l}(x)\}$ is not
  $\theta^g_x$-Cauchy, we set $g_\infty(x) := 0$.
\end{thm}
\begin{proof}
  Let $\{g_{k_l}\}$ be a subsequence of $\{g_k\}$ such that
  \begin{equation*}
    \sum_{l=1}^\infty d(g_{k_l}, g_{k_l+1}) < \infty.
  \end{equation*}
  Then $\{g_{k_l}\}$ satisfies properties (\ref{item:4}) and
  (\ref{item:7}) of Definition \ref{dfn:13}, as well as the hypotheses
  of Corollary \ref{cor:13}.  Thus $\{g_{k_l}\}$ is
  a.e.~$\theta^g_x$-Cauchy, and so $g_\infty$ is defined a.e.~by the
  two conditions given above.  From this, it is immediate that
  $\{g_{k_l}\}$ together with $g_\infty$ also satisfies properties
  (\ref{item:5}) and (\ref{item:6}) of Definition \ref{dfn:13}.  Thus,
  $\{g_{k_l}\}$ $\omega$-converges to $g_\infty$, and by Lemma
  \ref{lem:29} it therefore converges to $[g_\infty]$---provided we
  can show that $g_\infty \in \Mm$.
  
  Let's prove this last fact.  Clearly $g_\infty$ is a semimetric, so
  we must show that $g_\infty$ is measurable.  Now, on $M \setminus
  X_{\{g_{k_l}\}}$, $g_\infty$ is the a.e.-limit of measurable metrics, so
  it is measurable restricted to this set.  Furthermore, $g_\infty(x)
  = 0$ for every $x \in X_{\{g_{k_l}\}}$, so if we can show that
  $X_{\{g_{k_l}\}}$ is measurable, then we are done.  But the following
  formula shows that $X_{\{g_{k_l}\}}$ can be built from countable unions
  and intersections of open sets:\label{p:deflated-sets}
  \begin{align*}
    X_{\{g_{k_l}\}} &= \left\{ x \in M \mid \forall \delta > 0,\ \exists
      l\ \mathrm{s.t.}\ \det g_{k_l}(x) < \delta
    \right\} \\
    &= \bigcap_{N \in \N} \bigcup_{l \in \N} \left\{ x \in M \mid \det
      g_{k_l}(x) < \frac{1}{N} \right\}.
  \end{align*}
\end{proof}

Knowing now that the $\omega$-limit of a Cauchy sequence of $\M$
exists (after passing to a subsequence), we go further into the
properties of $\omega$-convergence.

\section{$\omega$-convergence and the concept of
  volume}\label{sec:ad-conv-conc}

In this brief section, we wish to prove that the volumes of measurable
subsets behave well under $\omega$-convergence.  Specifically, we want
to show that if $\{g_k\}$ $\omega$-converges to $[g_\infty]$ and $Y
\subseteq M$ is measurable, then for any representative $g_\infty \in
[g_\infty]$,
\begin{equation}\label{eq:139}
  \Vol(Y, g_k) \rightarrow \Vol(Y, g_\infty).
\end{equation}

To see that the above expression is well-defined, recall that a
measurable semimetric $\tilde{g}$ on $M$ induces a nonnegative volume
form and measure $\mu_{\tilde{g}}$ on $M$ (cf.~Subsection
\ref{sec:conventions}) that is absolutely continuous with respect to
the fixed volume form $\mu_g$.  Furthermore, given any two
representatives $g^0_\infty, g^1_\infty \in [g_\infty]$, we have that
$\mu_{g^0_\infty} = \mu_{g^1_\infty}$ as measures---it is clear from
Definition \ref{dfn:7} that $\mu_{g^0_\infty}$ and $\mu_{g^1_\infty}$
can differ at most on a nullset.  Thus $\Vol(Y, g^0_\infty) = \Vol(Y,
g^1_\infty)$.

The proof of \eqref{eq:139} is achieved via the Lebesgue dominated
convergence theorem (cf.~Theorem \ref{thm:36}).  So let $\{g_k\}$
$\omega$-converge to $g_\infty$, and let's see what we need to do to
apply this theorem.
First, we need to show that $(\mu_{g_k} / \mu_g)
\overarrow{\textnormal{a.e.}}  (\mu_{g_\infty} / \mu_g)$.  If we can
also find a function $f \in L^1(M, g)$ such that $( \mu_{g_k} / \mu_g
) \leq f$ a.e., then the Lebesgue dominated convergence theorem would
imply that
\begin{equation*}
  \Vol(Y, g_\infty) = \integral{Y}{}{\left(
      \frac{\mu_{g_\infty}}{\mu_g} \right)}{\mu_g} = \lim_{k
    \rightarrow \infty} \integral{Y}{}{\left( \frac{\mu_{g_k}}{\mu_g}
    \right)}{\mu_g} = \lim_{k \rightarrow \infty} \Vol(Y, g_k).
\end{equation*}
We begin by showing a.e.-convergence.

\begin{lem}\label{lem:52}
  Let $\{g_k\}$ $\omega$-converge to $g_\infty \in \Mm$.  Then
  \begin{equation*}
    \left( \frac{\mu_{g_k}}{\mu_g} \right) \overarrow{\textnormal{a.e.}} \left(
      \frac{\mu_{g_\infty}}{\mu_g} \right).
  \end{equation*}
\end{lem}
\begin{proof}
  Recall that
  \begin{equation*}
    \left( \frac{\mu_{g_k}}{\mu_g} \right) = \sqrt{\det G_k} \quad \textnormal{and}
    \quad \left(
      \frac{\mu_{g_\infty}}{\mu_g} \right) = \sqrt{\det G_\infty}.
  \end{equation*}
  So we can prove the statement by working with the determinants above
  instead of the Radon-Nikodym derivatives.

  We first prove that for a.e.~$x \in X_{\{g_k\}}$, $\det G_k(x)
  \rightarrow 0 = \det G_\infty$ as $k \rightarrow \infty$.  By the
  definition of the deflated set, for every $x \in X_{\{g_k\}}$ and
  $\epsilon > 0$, there exists $k \in \N$ such that
  \begin{equation}\label{eq:88}
    \det G_k(x) < \epsilon.
  \end{equation}
  But we also know from Proposition \ref{prop:15} and property
  (\ref{item:7}) of Definition \ref{dfn:13} that $\{g_k(x)\}$ is
  $\theta^g_x$-Cauchy for a.e.~$x \in M$.  Hence, by Lemma
  \ref{lem:34}, $\left\{\sqrt{\det G_k(x)}\right\}$ is a Cauchy
  sequence in $\R$ at such points.  Therefore it has a limit, and by
  \eqref{eq:88} we know that this limit must be $0$.

  Now, for a.e.~$x \in M \setminus X_{\{g_k\}}$, $g_k(x) \rightarrow
  g_\infty(x)$.  Since the determinant is a continuous map from the
  space of $n \times n$ matrices into $\R$, this immediately implies
  that $\det G_k(x) \rightarrow \det G_\infty(x)$ for a.e.~$x \in M
  \setminus X_{\{g_k\}}$.  Combined with the last paragraph, this
  proves the desired result.
\end{proof}

Our next task is to find an $L^1$ function that dominates
$(\mu_{g_k} / \mu_g)$.

\begin{lem}\label{lem:36}
  Let $\{g_k\}$ be a Cauchy sequence such that
  \begin{equation*}
    \sum_{k=1}^\infty d(g_k, g_{k+1}) < \infty,
  \end{equation*}
  and let $\Omega$ be the function of Lemma \ref{lem:35}.  Then
  \begin{equation*}
    \left( \frac{\mu_{g_k}}{\mu_g} \right)(x) \leq \frac{\sqrt{n}}{2}
    \Omega(x) + \left( \frac{\mu_{g_1}}{\mu_g} \right)(x)
  \end{equation*}
  for a.e.~$x \in M$ and all $k \in \N$.
\end{lem}
\begin{proof}
  Fix some $k$ for the moment.  By Proposition \ref{prop:15},
  $\{g_k(x)\}$ is $\theta_x^g$-Cauchy for a.e.~$x \in M$.  Let $x \in
  M$ be a point where this holds.  Then by Lemma \ref{lem:34}, the
  triangle inequality, and the definitions of $\Omega_N$ and $\Omega$,
  we have
  \begin{align*}
    \left| \sqrt{\det G_k} - \sqrt{\det G_1} \right| &\leq
    \frac{\sqrt{n}}{2} \theta^g_x(g_k, g_1) \leq
    \frac{\sqrt{n}}{2} \sum_{m=1}^{k-1} \theta^g_x(g_m,
    g_{m+1}) \\
    &= \frac{\sqrt{n}}{2} \Omega_{k-1}(x) \leq \frac{\sqrt{n}}{2}
    \Omega(x).
  \end{align*}
  In particular,
  \begin{equation*}
    \sqrt{\det G_k(x)} \leq \frac{\sqrt{n}}{2} \Omega(x) + \sqrt{\det G_1(x)}.
  \end{equation*}
  The result is now immediate.
\end{proof}

Now, since $\mu_{g_1}$ is smooth, it has finite volume, implying that
$(\mu_{g_1} / \mu_g) \in L^1(M, g)$.  We have already seen in Lemma
\ref{lem:35} that $\Omega \in L^1(M, g)$.  Therefore Lemma
\ref{lem:36} gives the necessary function dominating $(\mu_{g_k} /
\mu_g)$, and we can apply the Lebesgue dominated convergence theorem
as discussed before the lemmas to obtain:

\begin{thm}\label{thm:19}
  Let $\{g_k\}$ $\omega$-converge to $g_\infty \in \Mm$, and let $Y
  \subseteq M$ be any measurable subset.  Then $\Vol(Y, g_k)
  \rightarrow \Vol(Y, g_\infty)$.
\end{thm}

An immediate corollary of this theorem is that the total volume of the
$\omega$-limit is finite:

\begin{cor}\label{cor:15}
  If $g_\infty$ is the $\omega$-limit of a sequence $\{g_k\}$ in $\M$, then
  $\Vol(M, g_\infty) < \infty$.  That is, $g_\infty \in \M_f$.
\end{cor}
\begin{proof}
  By Lemma \ref{lem:13}, if $\{g_k\}$ is a $d$-Cauchy sequence, then
  $\{\Vol(M, g_k)\}$ is a Cauchy sequence of positive real numbers.
  Therefore it converges to some finite nonnegative real number, and
  by Theorem \ref{thm:19} this number must be $\Vol(M, g_\infty)$.
\end{proof}

Furthermore, as we might have suspected from the beginning, the volume
of the deflated set of an $\omega$-convergent sequence vanishes in the
limit.

\begin{cor}\label{cor:17}
  Let $\{g_k\}$ $\omega$-converge to $g_\infty \in \M_f$.  Then the
  deflated set $X_{\{g_k\}}$ satisfies $\Vol(X_{\{g_k\}}, g_k)
  \rightarrow 0$.
\end{cor}
\begin{proof}
  As noted in the proof of Theorem \ref{thm:39}, $X_{\{g_k\}}$ is
  measurable.  Now, the definition of $\omega$-convergence implies
  that $\Vol(X_{\{g_k\}}, g_\infty) = 0$, since $g_\infty(x) = 0$ for
  all $x \in X_{\{g_k\}}$.  So Theorem \ref{thm:19} gives the result.
\end{proof}

Given Corollary \ref{cor:15}, it behooves us to make the following
definition, following which we refine the result of Theorem
\ref{thm:39} using Corollary \ref{cor:15}.

\begin{dfn}\label{dfn:9}
  Let $\Mfhat \subset \Mmhat$ denote the subset of those equivalence
  classes of semimetrics whose representatives are all elements of
  $\Mf$, i.e., finite-volume measurable semimetrics.
\end{dfn}

By the discussion at the beginning of the section, any two
representatives of an equivalence class in $\Mmhat$ have the same
total volume.  Therefore, if one representative of an equivalence
class has finite volume, then all do.  Moreover, for every $\tilde{g}
\in \Mf$, $[\tilde{g}] \in \Mfhat$.

The refinement of Theorem \ref{thm:39} is:
  
\begin{thm}\label{thm:40}
  For every Cauchy sequence $\{g_k\}$, there exists an element
  $[g_\infty] \in \Mfhat$ such that $\{g_k\}$ $\omega$-subconverges to
  $[g_\infty]$.
\end{thm}

\section{Uniqueness of the $\omega$-limit}\label{sec:uniqueness-ad-limit}

The goal of this section is to prove the uniqueness of the $\omega$-limit in
the sense mentioned in the introduction to the chapter: we will show
that two $\omega$-convergent Cauchy sequences in $\M$ are equivalent if and
only if they have the same $\omega$-limit.

We prove each direction in a separate subsection.  After proving this
uniqueness result, combining it with the existence result and the
properties of $\omega$-convergence given above will show that for
every equivalence class of Cauchy sequences in $\M$, there is a unique
equivalence class of finite-volume, measurable semimetrics that each
of its representatives subconverges to.  Thus, $\omega$-convergence is
a suitable convergence notion for choosing a limit point for Cauchy
sequences in $\M$.

\subsection{First uniqueness result}\label{sec:first-uniq-result}

We first prove the statement that if two $\omega$-con\-ver\-gent
Cauchy sequences are equivalent, then their $\omega$-limits agree.  To
do so, we will extend the pseudometric $\Theta_Y$ (cf.~Definition
\ref{dfn:16}) to the precompletion of $\M$.  For this, we need an easy
lemma.

\begin{lem}\label{lem:10}
  Let $Y \subseteq M$ be measurable.  If $\{g_k\}$ is a $d$-Cauchy
  sequence, then it is also $\Theta_Y$-Cauchy.
\end{lem}
\begin{proof}
  As noted in the proof of Lemma \ref{lem:35}, since $\{g_k\}$ is
  $d$-Cauchy, the sequence $\sqrt{\Vol(M, g_k)}$ in $\R$ is bounded,
  say by a constant $V$.  Thus, by the estimate of Proposition
  \ref{prop:20}, for any $k, l \in \N$, we have
  \begin{equation*}
    \Theta_Y(g_k, g_l) \leq d(g_k, g_l) \left( \sqrt{n}\, d(g_k, g_l) +
      2 V \right).
  \end{equation*}
  From this the statement of the lemma is clear.
\end{proof}

Now we give the extension of $\Theta_Y$ mentioned above.

\begin{prop}\label{prop:21}
  Let $Y \subseteq M$ be measurable.  Then the pseudometric $\Theta_Y$
  on $\M$ can be extended to a pseudometric on
  $\overline{\M}^{\textnormal{pre}}$, the precompletion of $\M$, via
  \begin{equation}\label{eq:14}
    \Theta_Y(\{g^0_k\}, \{g^1_k\}) := \lim_{k \rightarrow \infty}
    \Theta_Y(g^0_k, g^1_k)
  \end{equation}
  This pseudometric is weaker than $d$ in the sense that $d(\{g^0_k\},
  \{g^1_k\}) = 0$ implies that $\Theta_Y(\{g^0_k\}, \{g^1_k\}) = 0$ for any
  Cauchy sequences $\{g^0_k\}$ and $\{g^1_k\}$.  More precisely, we
  have
  \begin{equation}\label{eq:91}
    \Theta_Y (\{g^0_k\}, \{g^1_k\}) \leq d(\{g^0_k\}, \{g^1_k\}) \left( \sqrt{n}\, d(\{g^0_k\},
      \{g^1_k\}) + 2 \sqrt{\Vol(M, g_0)} \right),
  \end{equation}
  where $g_0$ is any element of $\Mf$ with $g^0_k \overarrow{\omega} [g_0]$.

  Furthermore, if $\{g^0_k\}$ and $\{g^1_k\}$ are sequences in $\MV$
  that $\omega$-converge to $g_0$ and $g_1$, respectively, then the formula
  \begin{equation}\label{eq:92}
    \Theta_Y(\{g^0_k\}, \{g^1_k\}) = \integral{Y}{}{\theta_x^g(g_0(x), g_1(x))}{\mu_g(x)}
  \end{equation}
  holds for all $g_0, g_1 \in \MV$.
\end{prop}

\begin{rmk}\label{rmk:26}
  In \eqref{eq:91}, we choose any $\omega$-limit of $\{g^0_k\}$.  The
  existence of such a limit has already been proved, but not its
  uniqueness.  On the other hand, if $\tilde{g}_0$ is a different
  $\omega$-limit of $\{g^0_k\}$, Theorem \ref{thm:19} guarantees that
  $\Vol(M, \tilde{g}_0) = \Vol(M, g_0)$.  Therefore, the estimate
  \eqref{eq:91} is independent of the choice of $\omega$-limit.
\end{rmk}

\begin{proof}[Proof of Proposition \ref{prop:21}]
  The construction of a pseudometric on the precompletion of a metric
  space can be carried over to the case where we begin with a
  pseudometric space.  Therefore, the limit in \eqref{eq:14} is
  well-defined due to the fact that $\{g^0_k\}$ and $\{g^1_k\}$ are
  Cauchy sequences with respect to $\Theta_Y$, and \eqref{eq:14}
  indeed defines a pseudometric.
  
  The inequality \eqref{eq:91} is proved via the following simple
  computation, which uses \eqref{eq:14}, Proposition \ref{prop:20},
  and Theorem \ref{thm:19}:
  \begin{align*}
    \Theta_Y(\{g^0_k\}, \{g^1_k\}) &= \lim_{k \rightarrow \infty}
    \Theta_Y(g^0_k, g^1_k) \\
    &\leq \lim_{k \rightarrow \infty} d(g^0_k, g^1_k) \left(
      \sqrt{n}\, d(g^0_k,
      g^1_k) + 2 \sqrt{\Vol(M, g^0_k)} \right) \\
    &= d(\{g^0_k\}, \{g^1_k\}) \left( \sqrt{n}\, d(\{g^0_k\},
      \{g^1_k\}) + 2 \sqrt{\Vol(M, g_0)} \right).
  \end{align*}

  As for the last statement, note first that $\theta^g_x(g_0(x),
  g_1(x))$ is well-defined by Theorem \ref{thm:35}, since $g_0$ and
  $g_1$ are positive semidefinite tensors at each point $x \in M$.  To
  prove \eqref{eq:92}, we will first use Fatou's Lemma (cf.~Theorem
  \ref{thm:37}) to show that $\theta^g_x(g_0(x), g_1(x))$ is
  integrable.  We will then use this to apply the Lebesgue dominated
  convergence theorem.

  So we start by letting $\{g^0_k\}$ and $\{g^1_k\}$ be sequences in
  $\M$ $\omega$-converging to $g_0$ and $g_1$, respectively.
  
  By Proposition \ref{prop:15}, for a.e.~$x \in M$, $\{g^0_k(x)\}$ and
  $\{g^1_k(x)\}$ are $\theta^g_x$-Cauchy.  At such points, by
  definition,
  \begin{equation}\label{eq:89}
    \theta^g_x(g_0(x), g_1(x)) = \lim_{k \rightarrow \infty}
    \theta^g_x(g^0_k(x), g^1_k(x)).
  \end{equation}
  So defining 
  \begin{equation*}
    f_k(x) := \theta^g_x(g^0_k(x), g^1_k(x)), \quad f(x) := \theta^g_x(g_0(x),
    g_1(x)),
  \end{equation*}
  we have $f_k \rightarrow f$ a.e.

  Now, note that
  \begin{equation*}
    \Theta_Y(g^0_k, g^1_k) = \integral{Y}{}{f_k(x)}{\mu_g(x)}.
  \end{equation*}
  We have already seen that $\lim_{k \rightarrow \infty}
  \Theta_Y(g^0_k, g^1_k)$ exists, so $\{\Theta_Y(g^0_k, g^1_k)\}$ is
  in particular a bounded sequence of real numbers.  Thus
  \begin{equation*}
    \sup_k \integral{Y}{}{f_k(x)}{\mu_g(x)} = \sup_k \Theta_Y(g^0_k,
    g^1_k) < \infty,
  \end{equation*}
  where we have used Fatou's lemma.

  Now we wish to verify the assumptions of the Lebesgue dominated
  convergence theorem for $f_k$ and $f$.  We note that for each $l >
  k$, the triangle inequality gives
  \begin{align*}
    f_k(x) &= \theta^g_x(g^0_k(x), g^1_k(x)) \\
    &\leq \sum_{m=k}^{l-1}
    \theta^g_x(g^0_m(x), g^0_{m+1}(x)) + \theta^g_x(g^0_l(x),
    g^1_l(x)) + \sum_{m=k}^{l-1} \theta^g_x(g^1_m(x), g^1_{m+1}(x)) \\
    &\leq \sum_{m=1}^{l-1}
    \theta^g_x(g^0_m(x), g^0_{m+1}(x)) + \theta^g_x(g^0_l(x),
    g^1_l(x)) + \sum_{m=1}^{l-1} \theta^g_x(g^1_m(x), g^1_{m+1}(x)).
  \end{align*}
  Note that the only difference between the second and last lines is
  that the sums start at $m = 1$ instead of $m = k$.  Taking the limit
  $l \rightarrow \infty$ of the above gives, for a.e.~$x \in M$,
  \begin{equation*}
    f_k(x) \leq \sum_{m=1}^\infty
    \theta^g_x(g^0_m(x), g^0_{m+1}(x)) + f(x) + \sum_{m=1}^\infty \theta^g_x(g^1_m(x), g^1_{m+1}(x)),
  \end{equation*}
  where we have used \eqref{eq:89}.  Now we claim that the right-hand
  side of the above inequality is $L^1$-integrable.  We already showed
  $f$ is integrable using Fatou's Lemma.  As for the two infinite
  sums, they are each also integrable by Lemma \ref{lem:35} and
  $\omega$-convergence of $g^i_k$,\label{p:gik} $i=0,1$ (specifically,
  property (\ref{item:7}) of Definition \ref{dfn:13} and Lemma
  \ref{lem:35}).  Thus each $f_k$ is bounded a.e.~by an $L^1$ function
  not depending on $k$.
  
  Knowing all of this, we can apply the Lebesgue dominated convergence
  theorem to show
  \begin{equation*}
    \Theta_Y(\{g^0_k\}, \{g^1_k\}) = \lim_{k \rightarrow \infty} \Theta_Y(g^0_k, g^1_k) = \lim_{k
      \rightarrow \infty} \integral{Y}{}{f_k}{\mu_g} =
    \integral{Y}{}{f}{\mu_g} = \integral{Y}{}{\theta^g_x(g_0(x),
      g_1(x))}{\mu_g(x)},
  \end{equation*}
  which completes the proof.
\end{proof}

With this proposition, proving the first uniqueness result becomes a
relatively simple matter.

\begin{thm}\label{thm:20}
  Let two $\omega$-convergent sequences $\{g^0_k\}$ and $\{g^1_k\}$,
  with $\omega$-limits $[g_0]$ and $[g_1]$, respectively, be given.
  If $g^0_k$ and $g^1_k$ are equivalent, i.e., if
  \begin{equation*}
    \lim_{k \rightarrow \infty} d(g^0_k, g^1_k) = 0,
  \end{equation*}
  then $[g_0] = [g_1]$.
\end{thm}
\begin{proof}
  Suppose the contrary; then for any representatives $g_0 \in [g_0]$
  and $g_1 \in [g_1]$, one of two possibilities holds:
  \begin{enumerate}
  \item \label{item:14} $X_{g_0}$ and $X_{g_1}$ differ by a set of positive measure,
    or
  \item \label{item:15} $X_{g_0} = X_{g_1}$, up to a nullset, but $g_0$ and $g_1$
    differ on a set $E$ with $E \cap (X_{g_0} \cup X_{g_1}) =
    \emptyset$ and $\Vol(E, g) > 0$, where $g$ is our fixed metric.
  \end{enumerate}
  We will show that neither of these possibilities can actually occur.

  To rule out \eqref{item:14}, let $X_i := X_{\{g^i_k\}}$ denote the
  deflated set of the sequence $\{g_k^i\}$ for $i=0,1$.  Then we claim
  $X_0 = X_1$, up to a nullset.  If this is not true, then by swapping
  the two sequences if necessary, we see that $Y := (X_0 \setminus
  X_1)$ has positive volume with respect to $g_1$ and zero volume with
  respect to $g_0$.  ($Y$ is simply the set on which $\{g^0_k\}$
  deflates and $\{g^1_k\}$ doesn't.)  But then by Lemma \ref{lem:13},
  \begin{equation*}
    \lim_{k \rightarrow \infty} d(g^0_k, g^1_k) \geq \lim_{k
      \rightarrow \infty} \sqrt{\Vol(Y, g^1_k)} = \sqrt{\Vol(Y, g_1)}
    > 0,
  \end{equation*}
  where we have used Theorem \ref{thm:19}.  This contradicts the
  assumptions of the theorem, so in fact $X_0 = X_1$ up to a nullset.
  Since by property (\ref{item:5}) of Definition \ref{dfn:13} $X_{g_i}
  = X_{\{g^i_k\}}$ up to a nullset as well, \eqref{item:14} cannot
  hold.

  So suppose that \eqref{item:15} holds.  Note that on $E$, $g_0$ and
  $g_1$ are both positive definite.  Since $E$ has positive
  $g$-volume, we can conclude from Proposition \ref{prop:21}
  (specifically \eqref{eq:92}) that $\Theta_E(\{g^0_k\}, \{g^1_k\}) >
  0$.  But then this and \eqref{eq:91} also imply that
  \begin{equation*}
    \lim_{k \rightarrow \infty} d(g^0_k, g^1_k) = d(\{g^0_k\},
    \{g^1_k\}) > 0.
  \end{equation*}
  This contradicts the assumptions of the theorem, and so
  \eqref{item:15} cannot hold either.
\end{proof}

\subsection{Second uniqueness result}\label{sec:second-uniq-result}

Our goal in this subsection is to prove the following statement: up to
equivalence, there is only one $d$-Cauchy sequence $\omega$-converging
to a given element of $\Mfhat$.  That is, if we have two sequences
$\{g^0_k\}, \{g^1_k\}$ that both $\omega$-converge to the same
$[g_\infty] \in \Mfhat$, then
\begin{equation*}
  d(\{g^0_k\}, \{g^1_k\}) = \lim_{k \to \infty} d(g^0_k, g^1_k) = 0.
\end{equation*}

After we've proved this statement, we combine it with the existence
result from Section \ref{sec:existence-ad-limit} and the results on
volumes from Section \ref{sec:ad-conv-conc}, as mentioned in the
introduction to this section.

We will first prove the above statement for sequences that remain
within a given amenable subset $\U$, and will then use this to extend
the proof to arbitrary sequences.  Before any of this, though, we
state a definition and a result from measure theory that we'll need.

\begin{dfn}[{\cite[Dfn.~8.5.2]{rana02:_introd_to_measur_and_integ}}]\label{dfn:28}
  Let $(X, \Sigma, \nu)$ be a measure space, and let $\mathcal{F}$ be
  a collection of measurable functions.  We say that $\mathcal{F}$ is
  \emph{equicontinuous at $\emptyset$} if for any $\epsilon > 0$ and
  any sequence $\{E_k\}$ of measurable sets with
  \begin{equation*}
    \bigcap_{k=1}^\infty E_k = \emptyset,
  \end{equation*}
  there exists $K_0$ such that
  \begin{equation*}
    \integral{E_k}{}{|f|}{d \nu} < \epsilon
  \end{equation*}
  for all $f \in \mathcal{F}$ and $k > K_0$.
\end{dfn}

We note that in particular, if $\nu(X) < \infty$ and we are given a
collection of functions $\mathcal{F}$ for which we can find some
constant $C$ with
\begin{equation*}
  | f(x) | \leq C \quad \textnormal{for a.e.}\ x \in X\
  \textnormal{and every}\ f \in \mathcal{F},
\end{equation*}
then $\mathcal{F}$ is equicontinuous at $\emptyset$.

\begin{thm}[{\cite[Thm.~8.5.14]{rana02:_introd_to_measur_and_integ}}]\label{thm:1}
  Let $(X, \Sigma, \nu)$ be a measure space with $\nu(X) < \infty$,
  and let $f$ be a measurable function on $X$.  Furthermore, let $f_k$
  be a sequence of functions in $L^p(X, \nu)$.  Then the following
  statements are equivalent.
  \begin{enumerate}
  \item $f_k \rightarrow f$ in $L^p(X, \nu)$.
  \item $\{ |f_k|^p \mid k \in \N \}$ is equicontinuous at $\emptyset$
    and $f_k \rightarrow f$ in measure.
  \end{enumerate}
\end{thm}

\begin{rmk}\label{rmk:9}
  We make a couple of remarks about this theorem that we will need later:
  \begin{enumerate}
  \item By \cite[Thm.~8.3.3]{rana02:_introd_to_measur_and_integ},
    a.e.~convergence implies convergence in measure.  Therefore,
    Theorem \ref{thm:1} implies that if $\{ |f_k|^p \mid k \in \N \}$ is
    equicontinuous at $\emptyset$ and $f_k \rightarrow f$ a.e., then
    $f_k \rightarrow f$ in $L^p(X, \nu)$.
  \item By \cite[Thm.~8.3.6]{rana02:_introd_to_measur_and_integ}, if
    $f_k \rightarrow f$ in measure, then there exists a subsequence
    $\{ f_{k_l} \}$ such that $f_{k_l} \rightarrow f$ a.e.  Combining
    this with Theorem \ref{thm:1} implies that if $f_k \rightarrow f$
    in $L^p$, then there exists a subsequence $f_{k_l}$ such that
    $f_{k_l} \rightarrow f$ a.e.
  \end{enumerate}
\end{rmk}

We now state the second uniqueness result as confined to the context
of amenable subsets.

\begin{prop}\label{prop:1}
  Let $\U$ be an amenable subset, and let $\U^0$ be the
  $L^2$-completion of $\U$.  If two sequences $\{g^0_k\}$ and
  $\{g^1_k\}$ in $\U$ both $\omega$-converge to $[g_\infty] \in \Mfhat$,
  then $\{g^0_k\}$ and $\{g^1_k\}$ are equivalent.  That is,
  \begin{equation*}
    \lim_{k \to \infty} d(g^0_k, g^1_k) = 0.
  \end{equation*}
  Furthermore, up to differences on a nullset, $[g_\infty]$ only
  contains one representative, $g_\infty$, and $\{g^0_k\}$ and
  $\{g^1_k\}$ both $L^2$-converge to $g_\infty$.  In particular,
  $g_\infty \in \U^0$.
\end{prop}
\begin{proof}
  Note that Definition \ref{dfn:2} of an amenable subset implies that
  the deflated sets of $\{g^0_k\}$ and $\{g^1_k\}$ are empty.
  Therefore, all representatives of $[g_\infty]$ differ at most by a
  nullset, and property (\ref{item:6}) of Definition \ref{dfn:13}
  implies that $g^0_k, g^1_k \overarrow{\textnormal{a.e.}} g_\infty$.
  
  Since all $g^0_k$ and $g^1_k$ satisfy the same bounds a.e.~in each
  coordinate chart, it is easy to see that the set
  \begin{equation*}
    \{ | (g_l^k)_{ij} |^2 \mid 1 \leq i,j \leq n,\ k \in \N \}
  \end{equation*}
  is equicontinuous at $\emptyset$ in each coordinate chart for both
  $l = 0$ and $l = 1$.  Therefore, Remark \ref{rmk:9} gives that
  $\{g^0_k\}$ and $\{g^1_k\}$ converge in $L^2$ to $g_\infty$, proving
  the second statement.  This also implies that
  \begin{equation*}
    \lim_{k \to \infty} \| g^1_k - g^0_k \|_{g} = 0.
  \end{equation*}
  But now, invoking Theorem \ref{thm:5} gives
  \begin{equation*}
    \lim_{k \to \infty} d(g^0_k, g^1_k) = 0.
  \end{equation*}
\end{proof}

The next lemma establishes the strong correspondence between $L^2$-
and $\omega$-convergence within amenable subsets.

\begin{lem}\label{lem:43}
  Let $\U \subset \M$ be amenable, and let $\tilde{g} \in \U^0$.  Then
  for any sequence $\{ g_k \}$ in $\U$ that $L^2$-converges to
  $\tilde{g}$, there exists a subsequence $\{ g_{k_l} \}$ that
  $\omega$-converges to $\tilde{g}$.

  In particular, for any element $\tilde{g} \in \U^0$, we can always
  find a sequence in $\U$ that both $L^2$- and $\omega$-converges to
  $\tilde{g}$.
\end{lem}
\begin{proof}
  Let $\{ g_k \}$ be any sequence $L^2$-converging to $\tilde{g} \in
  \U^0$.  Then $\tilde{g}$ together with any subsequence of $\{ g_k
  \}$ already satisfies properties (\ref{item:4}) and (\ref{item:5})
  of Definition \ref{dfn:13}.  This is clear from Theorem \ref{thm:5}
  and Definition \ref{dfn:2} of an amenable subset.  (Property
  (\ref{item:5}) is empty here, as $\{g_k\}$ has empty deflated set
  by the definition of an amenable subset.)  Since $\{ g_k \}$ is
  $d$-Cauchy by Theorem \ref{thm:5}, it is also easy to see that there
  is a subsequence $\{ g_{k_m} \}$ of $\{ g_k \}$ satisfying property
  (\ref{item:7}) of $\omega$-convergence.

  To verify property (\ref{item:6}), note that $L^2$-convergence of
  $\{ g_{k_m} \}$ implies that there exists a subsequence
  $\{g_{k_l}\}$ of $\{g_{k_m}\}$ that converges to $\tilde{g}$ a.e.
  (Cf.~Remark \ref{rmk:9}.)
\end{proof}

Given the results that we have so far, we can give an alternative
description of the completion of an amenable set using $\omega$-convergence
instead of $L^2$-convergence.

\begin{prop}\label{prop:26}
  Let $\U \subset \M$ be an amenable subset.  Then the completion
  $\overline{\U}$ of $\U$ as a metric subspace of $\M$ can be
  identified with $\U^0$, the $L^2$ completion of $\U$, using
  $\omega$-convergence.  That is, there is a natural bijection between
  $\overline{\U}$ and $\U^0$ given by identifying each equivalence
  class of Cauchy sequences $[\{g_k\}]$ with the unique element of
  $\U^0$ that they $\omega$-subconverge to.
\end{prop}
\begin{proof}
  The existence result---Theorem \ref{thm:39}---the first uniqueness
  result---Theorem \ref{thm:20}---and Proposition \ref{prop:1}
  together imply that for every equivalence class $[\{g_k\}]$ of
  $d$-Cauchy sequences in $\U$, there is a unique $L^2$ metric
  $g_\infty \in \U^0$ such that every representative of $[\{g_k\}]$
  $\omega$-subconverges to $g_\infty$, and that the representatives of a
  different equivalence class cannot also $\omega$-subconverge to
  $g_\infty$.  This gives us the map from $\overline{\U}$ to $\U^0$
  and shows that it is injective.  Furthermore, by Lemma \ref{lem:43},
  there is a sequence in $\U$ $\omega$-subconverging to every element of
  $\U^0$.  Thus, this map is also surjective.
\end{proof}

With this identification, we can define a metric on $\U^0$ by
declaring the bijection of the previous proposition to be an
isometry.  The result is the following:

\begin{dfn}\label{dfn:14}
  Let $\U$ be an amenable subset.  By $d_\U$, we denote the metric on
  the completion of $\U$, which we identify with the $L^2$-completion
  $\U^0$ via Proposition \ref{prop:26}.  Thus, for $g_0, g_1 \in \U^0$
  and any sequences $g^0_k \overarrow{\omega} g_0$, $g^1_k \overarrow{\omega}
  g_1$, we have
  \begin{equation*}
    d_\U(g_0, g_1) = \lim_{k \rightarrow \infty} d(g^0_k, g^1_k).
  \end{equation*}
  Note that by the preceding results, we can equivalently define
  $d_\U$ by assuming that $\{g^0_k\}$ and $\{g^1_k\}$ $L^2$-converge
  to $g_0$ and $g_1$, respectively.
\end{dfn}

The next lemma shows that the metric $d_\U$ is nicely compatible with
the metric $d$.

\begin{lem}\label{lem:42}
  Let $\U \subset \M$ be amenable, and suppose $g_0, g_1 \in \U$ and
  $g_2 \in \U^0$.  Then
  \begin{enumerate}
  \item \label{item:16} $d(g_0, g_1) = d_\U(g_0, g_1)$, and 
  \item \label{item:3} $d(g_0, g_1) \leq d_\U(g_0, g_2) + d_\U(g_2, g_1)$.
  \end{enumerate}
\end{lem}
\begin{proof}
  
  Statement (1) is true simply by the definition of $d_\U$.  Statement
  (2) is proved by applying statement (1) and the triangle inequality
  for $d_\U$.
\end{proof}

With a little bit of effort, we can use previous results to extend
Proposition \ref{prop:18}, a statement about $\M$, to the completion
of an amenable subset.  We first prove a very special case in a lemma,
followed by the full result.

\begin{lem}\label{lem:51}
  Let $\U$ be any amenable subset and $g^0, g^1 \in \U$.
  Let $C(n)$ be the constant of Proposition \ref{prop:18}, and let $E
  \subseteq M$ be measurable.  Then
  \begin{equation*}
    d_\U(g^0, \chi(M \setminus E) g^0 + \chi(E) g^1)
    \leq C(n) \left( \sqrt{\Vol(E, g^0)} + \sqrt{\Vol(E,
        g^1)} \right)
  \end{equation*}
\end{lem}
\begin{proof}
  For each $k \in \N$, choose closed subsets $F_k$ and open subsets
  $U_k$ such that $F_k \subseteq E \subseteq U_k$ and $\Vol(U_k, g) -
  \Vol(F_k, g) \leq 1/k$.  Furthermore, choose functions $f_k \in
  C^\infty(M)$ satisfying
  \begin{enumerate}
  \item $0 \leq f_k(x) \leq 1$ for all $x \in M$,
  \item $f_k(x) = 1$ for $x \in F_k$ and
  \item $f_k(x) = 0$ for $x \not\in U_k$.
  \end{enumerate}
  Then it is not hard to see that the sequence defined by
  \begin{equation*}
    g_k := (1 - f_k) g^0 + f_k g^1
  \end{equation*}
  $L^2$-converges to $\chi(M \setminus E) g^0 + \chi(E)
  g^1$, so in particular
  \begin{equation}\label{eq:87}
    d_\U(g^0, \chi(M \setminus E) g^0 + \chi(E) g^1) =
    \lim_{k \rightarrow \infty} d(g^0, g_k).
  \end{equation}
  Furthermore, since $g^0$ and all $g_k$ are smooth, Proposition
  \ref{prop:18} gives
  \begin{equation}\label{eq:74}
    d(g^0, g_k) \leq C(n) \left( \sqrt{\Vol(U_k, g^0)} +
      \sqrt{\Vol(U_k, g_k)} \right).
  \end{equation}
  By our assumptions on the sets $U_k$, it is clear that $\Vol(U_k,
  g^0) \rightarrow \Vol(E, g^0)$.  So if we can show that
  $\Vol(U_k, g_k) \rightarrow \Vol(E, g^1)$, then \eqref{eq:87}
  and \eqref{eq:74} combine to give the desired result.

  Now, because $g_k = g^1$ on $F_k$, we have
  \begin{equation*}
    \Vol(U_k, g_k) = \integral{F_k}{}{}{\mu_{g^1}} +
    \integral{U_k \setminus F_k}{}{}{\mu_{g_k}}.
  \end{equation*}
  The first term converges to $\Vol(E, g^1)$ for $k \rightarrow
  \infty$ by the definition of $F_k$.  We claim that the second term
  converges to zero.  Note that since the bounds of Definition
  \ref{dfn:2} are pointwise convex, we can enlarge $\U$ to an amenable
  subset containing $g_k$ for each $k \in \N$.  (By the definition,
  each $g_k$ is, at each point $x \in M$, a sum $(1-s) g^0(x) + s
  g^1(x)$ with $0 \leq s \leq 1$.)  Therefore, by Lemma \ref{lem:49},
  there exists a constant $K$ such that
  \begin{equation*}
    \left( \frac{\mu_{g_k}}{\mu_g} \right) \leq K.
  \end{equation*}
  But using this, our claim is clear from the assumptions on $U_k$ and
  $F_k$.
\end{proof}

\begin{thm}\label{thm:15}
  Let $\U$ be any amenable subset with $L^2$-completion $\U^0$.
  Suppose that $g_0, g_1 \in \U^0$, and let $E := \carr (g_1 - g_0) =
  \{ x \in M \mid g_0(x) \neq g_1(x) \}$.  Then there exists a
  constant $C(n)$ depending only on $n = \dim M$ such that
  \begin{equation*}
    d_\U (g_0, g_1) \leq C(n) \left( \sqrt{\Vol(E, g_0)} +
      \sqrt{\Vol(E,g_1)} \right).
  \end{equation*}
  In particular, we have
  \begin{equation*}
    \diam_\U \left( \{ \tilde{g} \in \U^0 \mid \Vol(M, \tilde{g}) \leq
      \delta \} \right) \leq 2 C(n) \sqrt{\delta}.
  \end{equation*}
\end{thm}
\begin{proof}
  Using Lemma \ref{lem:43}, choose any two sequences $\{g^0_k\}$ and
  $\{g^1_k\}$ in $\U$ that both $L^2$- and $\omega$-converge to $g_0$
  and $g_1$, respectively.  Then by the triangle inequality and Lemma
  \ref{lem:42}\eqref{item:16}, for each $k \in \N$,
  \begin{equation}\label{eq:68}
    d_\U(g_0, g_1) \leq d_\U(g_0, g^0_k) + d(g^0_k, g^1_k) + d_\U(g^1_k, g_1).
  \end{equation}
  By Theorem \ref{thm:5}, the first and last terms above approach zero
  as $k \rightarrow \infty$.  Furthermore, we claim that the middle
  term satisfies
  \begin{equation*}
    \lim_{k \rightarrow \infty} d(g^0_k, g^1_k) \leq C(n) \left(
      \sqrt{\Vol(E, g_0)} + \sqrt{\Vol(E, g_1)} \right),
  \end{equation*}
  which would complete the proof.

  By the triangle inequality (\ref{item:3}) of Lemma \ref{lem:42}, we
  have
  \begin{equation}\label{eq:136}
    d(g^0_k, g^1_k) \leq d_\U(g^0_k, \chi(M \setminus E) g^0_k +
    \chi(E) g^1_k) + d_\U(\chi(M \setminus E) g^0_k +
    \chi(E) g^1_k, g^1_k).
  \end{equation}
  By Lemma \ref{lem:51}, the first term of the above satisfies
  \begin{equation*}
    d_\U(g^0_k, \chi(M \setminus E) g^0_k + \chi(E) g^1_k) \leq C(n)
    \left( \sqrt{\Vol(E, g^0_k)} + \sqrt{\Vol(E, g^1_k)} \right).
  \end{equation*}
  Applying Theorem \ref{thm:19} allows us to conclude
  \begin{equation*}
    \lim_{k \rightarrow \infty} d_\U(g^0_k, \chi(M \setminus E) g^0_k
    + \chi(E) g^1_k) \leq C(n) \left( \sqrt{\Vol(E, g_0)} +
      \sqrt{\Vol(E, g_1)} \right).
  \end{equation*}
  Therefore, if we can show that the second term of \eqref{eq:136}
  converges to zero as $k \rightarrow \infty$, then we will have the
  desired result.  But $\{g^0_k\}$ $L^2$-converges to $g_0$ and
  $\{g^1_k\}$ $L^2$-converges to $g_1$.  Additionally, $\chi(M
  \setminus E) g_0 = \chi(M \setminus E) g_1$.  Therefore,
  \begin{equation*}
    \lim_{k \rightarrow \infty} \chi(M \setminus E) g^0_k = \chi(M
    \setminus E) g_0 = \chi(M \setminus E) g_1 = \lim_{k
      \rightarrow \infty} \chi(M \setminus E) g^1_k,
  \end{equation*}
  where the limits are taken in the $L^2$ topology.  This implies
  that, again in the $L^2$ topology,
  \begin{equation*}
    \lim_{k \rightarrow \infty} \left( \chi(M \setminus E) g^0_k +
      \chi(E) g^1_k \right) = \lim_{k \rightarrow \infty} g^1_k.
  \end{equation*}
  By Definition \ref{dfn:14}, then,
  \begin{equation*}
    \lim_{k\ \rightarrow \infty} d_\U(\chi(M \setminus E) g^0_k +
    \chi(E) g^1_k, g^1_k) = 0, 
  \end{equation*}
  which is what was to be shown.
\end{proof}

Next, we need another technical result that will help us in extending
the second uniqueness result from amenable subsets to all of $\M$.

\begin{prop}\label{prop:19}
  Say $g_0 \in \M$ and $h \in \s$, and let $E \subseteq M$ be any open
  set.  Define an $L^2$ tensor $g_1 \in \s^0$ by $g_1 := g_0 + h^0$,
  where $h^0 := \chi(E) h$.  Assume that we can find an amenable
  subset $\U$ such that $g_1 \in \U^0$.  Finally, define a path $g_t$
  of $L^2$ metrics by $g_t := g_0 + t h^0$, $t \in [0,1]$.

  Then without loss of generality (by enlarging $\U$ if necessary),
  $g_t \in \U^0$ for all $t$, so in particular $d_\U(g_0, g_1)$ is
  well-defined.  Furthermore,
  \begin{equation}\label{eq:47}
    d_\U(g_0, g_1) \leq L(g_t) := \integral{0}{1}{\| h^0 \|_{g_t}}{d t},
  \end{equation}
  i.e., the length of $g_t$, when measured in the naive way, bounds
  $d_\U(g_0, g_1)$ from above.

  Lastly, suppose that on $E$, the metrics $g_t$, $t \in [0, 1]$, all
  satisfy the bounds
  \begin{equation*}
    | (g_t)_{ij}(x) | \leq C \quad \textnormal{and} \quad
    \lambda^{G_t}_{\textnormal{min}}(x) \geq \delta
  \end{equation*}
  for some $C, \delta > 0$, all $1 \leq i,j \leq n$ and a.e.~$x \in
  E$.  (That this is satisfied for some $C$ and $\delta$ is guaranteed
  by $g_t \in \U^0$.)  Then there is a constant $K = K(C, \delta)$
  such that
  \begin{equation*}
    d_\U(g_0, g_1) \leq K \| h^0 \|_g.
  \end{equation*}
\end{prop}
\begin{proof}
  The existence of the enlarged amenable subset $\U$ is clear from the
  construction of $g_t$.  So we turn to the proof of \eqref{eq:47}.
  
  Let any $\epsilon > 0$ be given.  By Theorem \ref{thm:5}, we can
  choose $\delta > 0$ such that for any $\tilde{g}_0, \tilde{g}_1 \in
  \U$, $\| \tilde{g}_1 - \tilde{g}_0 \|_g < \delta$ implies
  $d(\tilde{g}_0, \tilde{g}_1) < \epsilon$.

  Next, for each $k \in \N$, we choose closed sets $F_k \subseteq E$
  and open sets $U_k \supseteq E$ with the property that $\Vol(U_k, g)
  - \Vol(F_k, g) < 1/k$.  Given this, let's even restrict ourselves to
  $k$ large enough that
  \begin{equation}\label{eq:138}
    \| \chi(U_k \setminus F_k) h \|_g < \min \{\delta, \epsilon\}.
  \end{equation}

  We then choose $f_k \in C^\infty(M)$ satisfying
  \begin{enumerate}
  \item $f_k(x) = 1$ if $x \in F_k$,
  \item $f_k(x) = 0$ if $x \not\in U_k$ and
  \item $0 \leq f_k(x) \leq 1$ for all $x \in M$,
  \end{enumerate}

  The first consequence of our assumptions above is
  \begin{equation}\label{eq:77}
    \| g_1 - (g_0 + f_k h) \|_g \leq \| \chi(U_k \setminus F_k) h) \|_g < \delta.
  \end{equation}
  The second inequality is \eqref{eq:138}, and the first inequality
  holds for two reasons.  First, on both $F_k$ and $M \setminus U_k$,
  $g_0 + f_k h = g_0 + \chi(F_k) h = g_1$.  Second, on $U_k \setminus
  F_k$, $g_1 - (g_0 + f_k h) = (1 - f_k) h$, and by our third
  assumption on $f_k$, $0 \leq 1 - f_k \leq 1$.  Now, inequality
  \eqref{eq:77} allows us to conclude, by our assumption on $\delta$,
  that
  \begin{equation}\label{eq:78}
    d_\U(g_0 + f_k h, g_1) < \epsilon.
  \end{equation}
  Since by the triangle inequality
  \begin{equation*}
    d_\U(g_0, g_1) \leq d_\U(g_0, g_0 + f_k h) + d_\U(g_0 + f_k h, g_1)
    < d_\U(g_0, g_0 + f_k h) + \epsilon,
  \end{equation*}
  we must now get some estimates on $d_\U(g_0, g_0 + f_k h)$ to prove
  \eqref{eq:47}.
  
  To do this, define a path $g_t^k$ in $\M$, for $t \in [0,1]$, by
  $g_t^k := g_0 + t f_k h$.  Then we have, as is easy to see,
  \begin{equation}\label{eq:79}
    d(g_0, g_0 + f_k h) \leq L(g_t^k) = \integral{0}{1}{\|
      f_k h \|_{g_t^k}}{d t}
  \end{equation}
  This is almost what we want, but we first have to replace $f_k h$
  with $h^0 = \chi(E) h$.  Also note that the $L^2$ norm in
  \eqref{eq:79} is that of $g_t^k$.  To put this in a form useful for
  proving \eqref{eq:47}, we therefore also have to to replace $g^k_t$
  with $g_t$.

  Using the facts that on $F_k$, $f_k h = h^0$ and
  $g^k_t = g_t$, as well as that $f_k = 0$ on $M \setminus
  U_k$, we can write
  \begin{equation}\label{eq:80}
    \begin{aligned}
      \| f_k h \|_{g_t^k}^2 &= \integral{M}{}{\tr_{g_t^k} \left( (f_k
          h)^2 \right)}{\mu_{g_t^k}} \\
      &= \integral{F_k}{}{\tr_{g_t}
        \left( ( h^0 )^2 \right)}{\mu_{g_t}} + \integral{U_k
        \setminus F_k}{}{\tr_{g_t^k} \left( (f_k h)^2 \right)}{\mu_{g_t^k}}.
    \end{aligned}
  \end{equation}

  For the first term above, we clearly have
  \begin{equation}\label{eq:81}
    \integral{F_k}{}{\tr_{g_t} \left( ( h^0 )^2 \right)}{\mu_{g_t}} \leq \| h^0 \|_{g_t}^2.
  \end{equation}

  As for the second term, it can be rewritten and estimated by
  \begin{equation*}
    \integral{U_k \setminus F_k}{}{\tr_{g_t^k} \left(
        (f_k h)^2 \right)}{\mu_{g_t^k}} = \| \chi(U_k
    \setminus F_k) f_k h \|_{g_t^k} \leq \| \chi(U_k
    \setminus F_k) h \|_{g_t^k},
  \end{equation*}
  where the inequality follows from our third assumption on $f_k$
  above.  Now, recall that $g_t$ is contained within an amenable
  subset $\U$.  It is possible to enlarge $\U$, without changing the
  property of being amenable, so that $\U$ contains $g_t^k$ for all $t
  \in [0,1]$ and all $k \in \N$.  (That the enlarged subset satisfies
  bounds as in Definition \ref{dfn:2} is clear from the corresponding
  bounds on $g_0$ and $g_t$, and the fact that they are convex,
  pointwise conditions---cf.~part (\ref{item:1}) of Remark
  \ref{rmk:1}.)  Therefore, by Lemma \ref{lem:18}, there exists a
  constant $K' = K'(g_0, g_1)$---i.e., $K'$ does not depend on
  $k$---such that
  \begin{equation*}
    \| \chi(U_k \setminus F_k) h \|_{g_t^k} \leq K' \| \chi(U_k
    \setminus F_k) h \|_g.
  \end{equation*}
  But by \eqref{eq:138}, we have that $\| \chi(U_k \setminus F_k) h
  \|_g < \epsilon$.  Combining this with \eqref{eq:80} and
  \eqref{eq:81}, we therefore get
  \begin{equation*}
    \| f_k h \|_{g_t^k}^2 \leq \| h^0 \|_{g_t} + K' \epsilon.
  \end{equation*}

  The above inequality, substituted into \eqref{eq:79}, gives
  \begin{equation*}
    d(g_0, g_1^k) \leq \integral{0}{1}{\left( \| h^0 \|_{g_t} +
        K' \epsilon \right)}{d t} =
    L(g_t) + K' \epsilon.
  \end{equation*}
  The final step in the proof is then to estimate, using the above
  inequality and \eqref{eq:78}, that
  \begin{equation*}
    d_\U(g_0, g_1) \leq d(g_0, g_1^k) + d_\U(g_1^k, g_1) < L(g_t)
    + (1 + K') \epsilon.
  \end{equation*}
  Since $\epsilon$ was arbitrary and $K'$ is independent of $k$, we
  are finished with the proof of \eqref{eq:47}.

  Finally, the third statement follows from the following estimate,
  which is proved in exactly the same way as Lemma \ref{lem:18}:
  \begin{equation*}
    \| h^0 \|_{g_t} = \left( \integral{E}{}{\tr_{g_t}\left( h^2
        \right)}{\mu_{g_t}} \right)^{1/2} \leq K(C, \delta) \| h^0 \|_g.
  \end{equation*}
\end{proof}

With Theorem \ref{thm:15} and Proposition \ref{prop:19} as part of our
toolbox, we are now ready to take on the proof of the second
uniqueness result in its full generality.

So let two $d$-Cauchy sequences $\{g^0_k\}$ and $\{g^1_k\}$, as well
as some $g_\infty \in \M_f$, be given.  Suppose further that
$\{g^0_k\}$ and $\{g^1_k\}$ both $\omega$-converge to $g_\infty$ for $k
\rightarrow \infty$.  We will prove that
\begin{equation}\label{eq:70}
  \lim_{k \rightarrow \infty} d(g^0_k, g^1_k) = 0.
\end{equation}
The heuristic idea of our proof is very simple, which is belied by the
rather technical nature of the rigorous proof.  The point, though, is
essentially that for all $l \in \N$, we break $M$ up into two sets,
$E_l$ and $M \setminus E_l$.  The set $E_l$ has positive volume with
respect to $g_\infty$, but $\{g^0_k\}$ and $\{g^1_k\}$ $L^2$-converge
to $g_\infty$ on $E_l$, so the contribution of $E_l$ to $d(g^0_k,
g^1_k)$ vanishes in the limit $k \rightarrow \infty$.  The set $M
\setminus E_l$ contains the deflated sets of $\{g^0_k\}$ and
$\{g^1_k\}$, so the sequences need not converge on $M \setminus E_l$.
However, we choose things such that $\Vol(M \setminus E_l, g_\infty)$
vanishes in the limit $l \rightarrow \infty$, so that Proposition
\ref{prop:18} implies that the contribution of $M \setminus E_l$ to
$d(g^0_k, g^1_k)$ vanishes after taking the limits $k \rightarrow
\infty$ and $l \rightarrow \infty$ in succession.

The rigorous proof is achieved in three basic steps, which we will
describe after some brief preparation.

For each $l \in \N$, let
\begin{equation}\label{eq:85}
  E_l :=
  \left\{
    x \in M \midmid \det g^i_k(x) > \frac{1}{l},\ | (g^i_k)_{rs}(x) |
    < l\ \forall i=0,1;\ k \in \N;\ 1 \leq r,s \leq n
  \right\},
\end{equation}
where these local notions are of course defined with respect to our
fixed amenable atlas (cf.~Convention \ref{cvt:5}), and the
inequalities in the definition should hold in each chart containing
the point $x$ in question.  Thus, $E_l$ is a set over which the
sequences $g_i^k$ neither deflate nor become unbounded.  We first note
that for each $k \in \N$, there exists an amenable subset $\U_k$ such
that the metrics
\begin{equation*}
  g^0_k,\ g^1_k\ \textnormal{and}\  g^0_k + \chi(E_l)(g^1_k - g^0_k)
\end{equation*}
are contained in $\U_k^0$.  This is possible due to smoothness of
$g^0_k$ and $g^1_k$, as well as pointwise convexity of the bounds of
Definition \ref{dfn:2}.

The steps in our proof are
the following.  We will show first that
\begin{equation}\label{eq:82}
 \lim_{k \rightarrow \infty} d_{\U_k}(g^0_k, g^0_k + \chi(E_l) (g^1_k -
  g^0_k)) = 0
\end{equation}
for all fixed $l \in \N$.  Second,
\begin{equation}\label{eq:83}
\lim_{k \rightarrow \infty} d_{\U_k}(g^0_k + \chi(E_l) (g^1_k - g^0_k),
g^1_k) \leq 2 C(n)
  \sqrt{\Vol(M \setminus E_l, g_\infty)}  
\end{equation}
for all fixed $k \in \N$ (where $C(n)$ is the constant from Theorem
\ref{thm:15}).  And third,
\begin{equation}\label{eq:84}
\lim_{l \rightarrow \infty} \Vol(E_l, g_\infty) = \Vol(M, g_\infty).
\end{equation}
Since the triangle inequality of Lemma \ref{lem:42}(2) implies that
\begin{equation*}
  d(g^0_k, g^1_k) \leq d_{\U_k}(g^0_k, g^0_k + \chi(E_l) (g^1_k -
  g^0_k)) + d_{\U_k}(g^0_k + \chi(E_l) (g^1_k -
  g^0_k), g^1_k)
\end{equation*}
for all $l \in \N$, taking the limits $k \rightarrow \infty$ followed
by $l \rightarrow \infty$ of both sides then gives \eqref{eq:70}.

We now prove each of \eqref{eq:82}, \eqref{eq:83} and \eqref{eq:84} in
its own lemma.

\begin{lem}\label{lem:38}
  \begin{equation*}
 \lim_{k \rightarrow \infty} d_{\U_k}(g^0_k, g^0_k + \chi(E_l) (g^1_k -
  g^0_k)) = 0
\end{equation*}
\end{lem}
\begin{proof}
  We know that
  \begin{equation*}
    g^0_k, g^0_k + \chi(E_l) (g^1_k - g^0_k) \in \U_k^0,
  \end{equation*}
  where $\U_k$ is an amenable subset.  Therefore, for each fixed $k
  \in \N$, Proposition \ref{prop:19} applies to give
  \begin{equation}\label{eq:86}
    d_{\U_k}(g^0_k, g^0_k + \chi(E_l) (g^1_k - g^0_k)) \leq K_l \| \chi(E_l)
    (g^1_k - g^0_k) \|_g,
  \end{equation}
  where $K_l$ is some constant depending only on $l$.  (That the
  constant only depends on $l$ is the result of the fact that $g^0_k$
  and $g^1_k$ satisfy the bounds given in \eqref{eq:85} on $E_l$,
  which only depend on $l$.)  

  Now, recalling the definition \eqref{eq:85} of $E_l$, we note that
  for all $1 \leq i,j \leq n$ and all $k \in \N$, we have $|
  (g^1_k)_{ij}(x) - (g^0_k)_{ij}(x) |^2 \leq 4 l^2$ for $x \in E_l$,
  and hence the family of (local) functions
  \begin{equation*}
    \{ \chi(E_l) ((g^1_k)_{ij} - (g^0_k)_{ij}) \mid 1 \leq i,j
    \leq n,\ k \in \N \}
  \end{equation*}
  is equicontinuous at $\emptyset$.  Furthermore, since property
  (\ref{item:6}) of Definition \ref{dfn:13} implies that $\chi(E_l)
  g_a^k \rightarrow \chi(E_l) g_\infty$ a.e.~for $a=0,1$, we have that
  $\chi(E_l) (g^1_k - g^0_k) \rightarrow 0$ a.e.  Therefore, Remark
  \ref{rmk:9} implies that
  \begin{equation*}
    \| \chi(E_l) ( g^1_k - g^0_k ) \|_g
    \rightarrow 0
  \end{equation*}
  for $k \rightarrow \infty$.  Together with \eqref{eq:86}, this
  implies the result immediately.
\end{proof}

\begin{lem}\label{lem:40}
  \begin{equation*}
    \lim_{k \rightarrow \infty} d_{\U_k}(g^0_k + \chi(E_l) (g^1_k - g^0_k),
    g^1_k) \leq 2 C(n) \sqrt{\Vol(M \setminus E_l, g_\infty)}
  \end{equation*}
\end{lem}
\begin{proof}
  First note that $g^1_k = g^0_k + \chi(E_l) (g^1_k - g^0_k)$ on
  $E_l$.  Therefore, by Theorem \ref{thm:15},
  \begin{equation*}
    d_{\U_k}(g^0_k + \chi(E_l) (g^1_k - g^0_k),
    g^1_k) \leq C(n) \left( \sqrt{\Vol(M \setminus E_l, g^0_k)} +
      \sqrt{\Vol(M \setminus E_l, g^1_k)} \right).
  \end{equation*}
  But now the result follows immediately from Theorem \ref{thm:19},
  since $\Vol(M \setminus E_l, g_i^k) \rightarrow \Vol(M \setminus
  E_l, g_\infty)$ for $i = 0,1$.
\end{proof}

\begin{lem}\label{lem:39}
  \begin{equation*}
    \lim_{l \rightarrow \infty} \Vol(E_l, g_\infty) = \Vol(M, g_\infty).
  \end{equation*}
\end{lem}
\begin{proof}
  Recall that $X_{g_\infty} \subseteq M$ denotes the deflated set of
  $g_\infty$, i.e., the set where $g_\infty$ is not positive definite.
  This set has volume zero w.r.t.~$g_\infty$, since $\mu_{g_\infty} =
  0$ a.e.~on $X_{g_\infty}$.  Therefore $\Vol(M, g_\infty) = \Vol(M
  \setminus X_{g_\infty}, g_\infty)$.

  We note that $\chi(E_l)$ converges a.e.~to $\chi(M \setminus
  X_{g_\infty})$ and that $\chi(E_l)(x) \leq 1$ for all $x \in M$.
  Since $g_\infty$ has finite volume, the constant function 1 is
  integrable w.r.t.~$\mu_{g_\infty}$, and therefore the Lebesgue
  dominated convergence theorem (Theorem \ref{thm:36}) implies that
  \begin{equation*}
    \lim_{l \rightarrow \infty} \Vol(E_l, g_\infty) = \lim_{l
      \rightarrow \infty} \integral{M}{}{\chi(E_l)}{\mu_{g_\infty}}
    = \integral{M}{}{\chi(M \setminus X_{g_\infty})}{\mu_{g_\infty}}
    = \Vol(M \setminus X_{g_\infty}, g_\infty).
  \end{equation*}
\end{proof}

As already noted, Lemmas \ref{lem:38}, \ref{lem:40} and \ref{lem:39}
combine to give the desired result.  We summarize what we have just
proved in a theorem.

\begin{thm}\label{thm:17}
  Let $[g_\infty] \in \Mfhat$.  Suppose we have two sequences
  $\{g^0_k\}$ and $\{g^1_k\}$ with $g^0_k, g^1_k \overarrow{\omega}
  [g_\infty]$ for $k \rightarrow \infty$.  Then
  \begin{equation*}
    \lim_{k \rightarrow \infty} d(g^0_k, g^1_k) = 0,
  \end{equation*}
  that is, $\{g^0_k\}$ and $\{g^1_k\}$ are equivalent in the precompletion
  $\overline{\M}^{\mathrm{pre}}$ of $\M$.
\end{thm}

As we have already discussed, combining this theorem with the
existence result (Theorem \ref{thm:40}) and the first uniqueness
result (Theorem \ref{thm:20}) gives us an identification of
$\overline{\M}$ with a subset of $\Mfhat$.  We summarize this in a theorem:

\begin{thm}\label{thm:38}
  There is a natural identification of $\overline{\M}$, the completion
  of $\M$, with a subset of $\Mfhat$, the measurable semimetrics with
  finite volume on $M$ modulo the equivalence given in Definition
  \ref{dfn:7}.

  This identification is given by an injection $\Omega : \overline{\M}
  \hookrightarrow \Mfhat$, where we map an equivalence class
  $[\{g_k\}]$ of $d$-Cauchy sequences to the unique element of
  $\Mfhat$ that all of its members $\omega$-subconverge to.  This map
  is an isometry onto its image if we give $\Omega(\overline{\M})$ the
  metric $\bar{d}$ defined by
  \begin{equation*}
    \bar{d}([g_0], [g_1]) := \lim_{k \rightarrow \infty} d(g^0_k, g^1_k)
  \end{equation*}
  where $\{g^0_k\}$ and $\{g^1_k\}$ are any sequences in $\M$
  $\omega$-converging to $[g_0]$ and $[g_1]$, respectively.
\end{thm}

This is an extremely useful theorem, as it allows us to drop the
distinction between an $\omega$-convergent sequence and the element of
$\Mfhat$ that it converges to.  By Lemma \ref{lem:29}, we can even
identify an $\omega$-convergent sequence with any representative of
the equivalence class in $\Mfhat$ that it converges to.  From now on,
we will employ this trick to simplify formulas and proofs.

Our job in the next chapter will be to show that the identification
described in Theorem \ref{thm:38} is actually a surjection.  This will
allow us to identify $\overline{\M}$ with the space $\Mfhat$ itself,
instead of just a subset thereof.  In doing so, we will prove the main
result of this thesis.


\chapter{The completion of $\M$}
\label{chap:sing-metrics}

In this chapter, our previous efforts come to fruition and we are able
to complete our description of $\overline{\M}$ by proving, in Section
\ref{sec:finite-volume-metr}, that the map $\Omega : \overline{\M}
\rightarrow \Mfhat$ defined in the previous chapter is a bijection.

To prepare ourselves for this proof, Section
\ref{sec:compl-orbit-space} first looks at a simpler example of a
completion, namely that of the orbit space of the conformal group---a
submanifold of $\M$ that we first encountered in Section
\ref{sec:manifold-metrics-m}.  This example is not just illustrative
of our situation---formally it is extremely similarly to our proof of
the surjectivity of $\Omega$, though the latter is, of course,
significantly more challenging technically.  Nevertheless, the
computations of this example will be directly employed in the
surjectivity proof.

Section \ref{sec:meas-induc-degen} provides some necessary preparation
for the surjectivity proof by going into more depth on the behavior of
volume forms under $\omega$-convergence.  After this, Section
\ref{sec:meas-techn-result} presents a partial result on the image of
$\Omega$.  Namely, we show that all equivalence classes of measurable,
bounded semimetrics (cf.~Definition \ref{dfn:23}) are contained in
$\Omega(\overline{\M})$.  This marks the final preparation we need to
prove the main result.

\section{Completion of the orbit space of $\pos$}\label{sec:compl-orbit-space}

For our fixed but arbitrary metric $g \in \M$, consider the orbit
space $\pos \cdot g$.  (Later we will consider this space for other
metrics $\tilde{g} \in \M$ rather than just our fixed $g$.  But since
$g$ was chosen arbitrarily, anything we prove about $\pos \cdot g$
will hold for $\pos \cdot \tilde{g}$ as well.)  Recalling that $\pos$
is the Fréchet Lie group of smooth, positive functions on $M$, we see
that the orbit consists of metrics of the form $\rho g$, where $\rho$
is a positive $C^\infty$ function.  As we have already seen in
Subsection \ref{sec:manif-posit-funct}, since $\pos$ is an open subset
of $C^\infty(M)$, each tangent space to $\pos \cdot g$ is canonically
identified with $C^\infty(M) \cdot g$, the set of what we called pure
trace tensors.

By Proposition \ref{prop:23}, there is an open set $U \subset C^\infty(M)$ with the
property that the exponential mapping $\exp_g$ is a diffeomorphism
between the set $U \cdot g$ and $\pos \cdot g$.  For convenience, we
define a mapping
\begin{equation}\label{eq:53}
  \begin{aligned}
    \psi : U &\overarrow{\cong} \pos \cdot g \\
    \lambda &\mapsto \exp_g(\lambda g) = \left( 1 + \frac{n}{4}
      \lambda \right)^{\frac{4}{n}} g.
  \end{aligned}
\end{equation}
Note that $\psi$ is not an isometry on radial geodesics, since the
$L^2$ norm induced by $g$ on functions is a non-unit scalar multiple
of the $L^2$ norm induced by $g$ on pure trace tensors:
\begin{equation*}
  \begin{aligned}
    (\kappa g, \lambda g)_g = \integral{M}{}{\tr_g((\kappa g)(\lambda
      g))}{\mu_g} = \integral{M}{}{\kappa \lambda \tr (I)}{\mu_g} = n
    \integral{M}{}{\kappa \lambda}{\mu_g} = n (\kappa, \lambda)_g,
  \end{aligned}
\end{equation*}
where we have denoted the $n \times n$ identity matrix by $I$.  By the
above, if we define a radial geodesic by $g_t = \psi(t \lambda)$ for
$t \in [0,1]$, then we get $L(g_t) = \| \lambda g \|_g = \sqrt{n} \|
\lambda \|_g$.

We can even determine the set $U$ explicitly.  Let $\lambda \in
C^\infty(M)$.  Algebraically, we could define $\psi$ for any such
$\lambda$, but if we want $\psi(\lambda)$ to be a metric, we must at
least require that $\lambda(x) \neq -\frac{4}{n}$ for all $x \in M$.
Furthermore, since $\psi$ is defined using $\exp_g$, we should have
that if $\psi(\lambda)$ is defined, then $\psi(t \lambda)$ is defined
(and is a metric) for $t \in [0,1]$.  This rules out the possibility
that $\lambda(x) < -\frac{4}{n}$ at some point $x \in M$, so we see
that
\begin{equation}\label{eq:56}
  U =
  \left\{
    \lambda \in C^\infty(M) \midmid \lambda(x) > -\frac{4}{n}\ \text{for
      all}\ x \in M
  \right\}.
\end{equation}

By Proposition \ref{prop:22}, $\pos \cdot g$ is flat.  In the case of
a strong Riemannian Hilbert manifold, as in the case of a
finite-dimensional manifold, this would imply that $\exp_g$ is an
isometry, and hence that $\psi$ is an isometry up to a scalar factor.
But since $\pos \cdot g$ is a weak Riemannian manifold, to make this
conclusion we would first have to prove such a general result.  This
is not necessary, however, as we can show directly that the desired
conclusion holds in our case.

\begin{prop}\label{prop:9}
  Up to a scalar factor of $\sqrt{n}$, $\psi$ is an isometry.  More
  precisely, if $d_{\pos \cdot g}$ is the distance function induced on
  $\pos \cdot g$ as a submanifold of $(\M, (\cdot, \cdot))$, we have
  \begin{equation*}
    d_{\pos \cdot g}(\psi(\kappa), \psi(\lambda)) = \sqrt{n} \| \lambda - \kappa \|_g
  \end{equation*}
  for all $\kappa, \lambda \in U$.
\end{prop}
\begin{proof}
  We first note that $\pos \cdot g = \pos \cdot \psi(\kappa)$.  As
  above, we can find a neighborhood $V \in C^\infty(M)$ such that the
  map
  \begin{equation}\label{eq:52}
    \begin{aligned}
      \varphi : V &\rightarrow \pos \cdot g \\
      \sigma &\mapsto \exp_{\psi(\kappa)} (\sigma \psi(\kappa)) =
      \left( 1 + \frac{n}{4} \sigma \right)^{\frac{4}{n}} \psi(\kappa)
    \end{aligned}
  \end{equation}
  is a diffeomorphism.

  Now Proposition \ref{prop:2} implies that $d_{\pos \cdot
    g}(\psi(\kappa), \psi(\lambda)) = \sqrt{n} \| \varphi^{-1}
  \psi(\lambda) \|_g$, since the shortest path between $\psi(\kappa)$
  and $\psi(\lambda)$ is the unique radial geodesic emanating from
  $\psi(\kappa)$ and ending at $\psi(\lambda)$.  Therefore, we must
  prove that $\| \varphi^{-1} \psi(\lambda) \|_g = \| \lambda - \kappa
  \|_g$.

  We define $\sigma := \varphi^{-1} \psi(\lambda)$.  Then
  $\varphi(\sigma) = \psi(\lambda)$ implies, by (\ref{eq:53}) and
  (\ref{eq:52}), that
  \begin{equation*}
    \left( 1 + \frac{n}{4} \sigma \right)^{\frac{4}{n}} \left( 1 +
      \frac{n}{4} \kappa \right)^{\frac{4}{n}} g = \left( 1 +
      \frac{n}{4} \lambda \right)^{\frac{4}{n}} g.
  \end{equation*}
  Solving for $\sigma$ gives
  \begin{equation*}
    \varphi^{-1} \psi(\lambda) = \sigma = \frac{4}{n} \left( \left( 1
        + \frac{n}{4} \lambda \right) \left(1 + \frac{n}{4} \kappa
      \right)^{-1} - 1 \right).
  \end{equation*}
  Now, since $\mu_{\rho g} = \rho^{n/2} \mu_g$ for any $\rho \in
  \pos$, we have
  \begin{equation*}
    \mu_{\psi(\kappa)} = \left( 1 + \frac{n}{4} \kappa \right)^2 \mu_g.
  \end{equation*}
  Using these two equations, we finish the proof with a computation:
  \begin{equation}\label{eq:54}
    \begin{aligned}
      \| \varphi^{-1} \psi(\lambda) \|_{\psi(\kappa)}^2 &= \frac{16}{n^2}
      \integral{M}{}{\Big( \left( 1 + \frac{n}{4} \lambda \right)
          \left(1 + \frac{n}{4} \kappa
          \right)^{-1} - 1 \Big)^2}{\mu_{\psi(\kappa)}} \\
      &= \frac{16}{n^2} \integral{M}{}{\Big( \left( 1 + \frac{n}{4}
            \lambda \right) \left(1 + \frac{n}{4} \kappa \right)^{-1}
          -
          1 \Big)^2 \left(1 + \frac{n}{4} \kappa \right)^2}{\mu_g} \\
      &= \frac{16}{n^2} \integral{M}{}{\left( \left( 1 + \frac{n}{4}
            \lambda \right) - \left(1 + \frac{n}{4} \kappa \right)
        \right)^2}{\mu_g} \\
      &= \integral{M}{}{(\lambda - \kappa)^2}{\mu_g} = \| \lambda -
      \kappa \|_g^2.
    \end{aligned}
  \end{equation}
\end{proof}

Using this proposition, we can immediately determine $\overline{\pos
  \cdot g}$, the completion of an orbit of the conformal group.

\begin{thm}\label{thm:11}
  $\overline{\pos \cdot g}$ is isometric to the set of tensors of the
  form $\rho g$ with $\rho$ measurable, $\rho(x) \geq 0$ a.e., and
  $\Vol(M, \rho g) < \infty$.  Equivalently, this set is those metrics
  $\rho g$ where $\rho \in L^{n/2}(M)$, i.e.,
  $\integral{}{}{\rho^{n/2}}{\mu_g} < \infty$, and $\rho(x) \geq 0$
  a.e.

  The distance function on $\overline{\pos \cdot g}$ is given by
  $d_{\overline{\pos \cdot g}}(\rho_1 g, \rho_2 g) = \sqrt{n} \| \psi^{-1}(\rho_2 g) -
  \psi^{-1}(\rho_2 g) \|_g$.
\end{thm}

\begin{rmk}\label{rmk:5}
  Although $L^{n/2}(M)$ is not a normed space for $n = 1$, we simply
  define it as the set of measurable functions with integrable square
  root.
\end{rmk}

\begin{proof}[Proof of Theorem \ref{thm:11}]
  Let's look at the first statement.  The equivalence of the two
  formulations in the theorem is clear from the fact that $\mu_{\rho
    g} = \rho^{n/2} \mu_g$, so $\Vol(M, \rho g) =
  \integral{}{}{}{\mu_{\rho g}} = \integral{}{}{\rho^{n/2}}{\mu_g}$.
  We will therefore show only the second statement.
  
  Since $\psi$ is an isometry, up to a scalar factor, it is clear that
  $\overline{\pos \cdot g} \cong \psi(\overline{U})$, where $U$ is the
  open neighborhood on which $\psi$ is a diffeomorphism.  But from
  (\ref{eq:56}) and the fact that $\| \cdot \|_g$ is the $L^2$ norm on
  functions, we immediately see that
  \begin{equation*}
    \overline{U} =
    \left\{
      \lambda \in L^2(M) \midmid \lambda(x) \geq -\frac{4}{n}\ \mathrm{a.e.}
    \right\},
  \end{equation*}
  and therefore
  \begin{equation*}
    \overline{\pos \cdot g} = \psi(\overline{U}) =
    \left\{
      \left( 1 + \frac{n}{4} \lambda \right)^{\frac{4}{n}} g \midmid
      \lambda \in L^2(M),\ \lambda(x) \geq -\frac{4}{n}\ \mathrm{a.e.}
    \right\}.
  \end{equation*}
  If we define $\rho = \rho(\lambda) := \left( 1 + \frac{n}{4} \lambda
  \right)^{4/n}$, then it remains to prove that $\rho \in L^{n/2}(M)$
  for any $\lambda \in L^2(M)$.  But
  \begin{align*}
    \integral{M}{}{\rho^{n/2}}{\mu_g} = \integral{M}{}{\left( 1 +
        \frac{n}{4} \lambda \right)^2}{\mu_g} =
    \integral{M}{}{}{\mu_g} + \frac{n}{2}
    \integral{M}{}{\lambda}{\mu_g} + \frac{n^2}{16} \integral{M}{}{\lambda^2}{\mu_g}.
  \end{align*}
  The first term in the above expression is finite by compactness of
  $M$, and the third is finite since $\lambda \in L^2(M)$.  Using
  this, one can then see the second term is finite by Hölder's
  inequality.

  As for the statement about the distance function, this follows from
  the fact that $\psi$ extends uniquely to an isometry (up to the
  scalar factor $\sqrt{n}$) from $\overline{U}$ to $\overline{\pos
    \cdot g}$.  This is thanks to statement (\ref{item:2}) of Theorem
  \ref{thm:29}.
\end{proof}

This theorem immediately tells us what the completion of $\M$ is when
$M$ is one-dimensional---of course, there is only one diffeomorphism
class of compact one-dimensional manifolds, so in this case $M = S^1$.
The theorem gives us complete information here because any smooth
metric on $S^1$ can be obtained from the standard metric $g$ by
multiplication with a smooth function.  Therefore $\M = \pos \cdot g$,
and Theorem \ref{thm:11} immediately implies:

\begin{cor}
  We work over a one-dimensional base manifold $M$, so that $M =
  S^1$.  Let $g$ be the standard metric on $S^1$.  Then
  \begin{equation*}
    \overline{\M} \cong
    \left\{
    \rho g  \mid \sqrt{\rho} \in L^1(M, g),\ \rho(x) \geq 0\
    \mathrm{a.e.}
    \right\}.
  \end{equation*}
\end{cor}

Of course, we still have an infinite number of cases left to deal with
if we want to find the completion of $\M$ in arbitrary dimension.  We
need a few preliminary results in order to proceed.

\section{Measures induced by measurable
  semimetrics}\label{sec:meas-induc-degen}

For use in Section \ref{sec:finite-volume-metr}, we need to record a
couple of properties of the measure $\mu_{\tilde{g}}$ induced by an
element $\tilde{g} \in \M_f$.

\subsection{Weak convergence of measures}\label{sec:weak-conv-meas}

The first property we wish to prove is the following. Suppose
$\tilde{g} \in \M_f$ and $\{ g_k \}$ is a sequence $\omega$-converging
to $\tilde{g}$.  Furthermore, let $\rho \in C^0(M)$ be any continuous
function.  Then we claim that
\begin{equation*}
  \lim_{k \rightarrow \infty} \| \rho \|_{g_k} = \| \rho \|_{\tilde{g}},
\end{equation*}
where we recall that for any measurable semimetric $\bar{g}$ and any
function $\sigma$ on $M$,
\begin{equation*}
  \| \sigma \|_{\bar{g}} = \left(
    \integral{M}{}{\sigma^2}{\mu_{\bar{g}}} \right)^{1/2}.
\end{equation*}

To prove this, we need to introduce the notion of \emph{weak
  convergence} (sometimes also called \emph{weak-* convergence}) of
Borel measures.  We do this in the general setting before we apply it
to our situation.  So let $X$ be a topological space, and denote by
$\mathfrak{M}(X)$ the set of nonnegative, totally finite measures on
the Borel algebra of $X$.  (Recall that a totally finite measure is
one for which every measurable set has finite measure.)  Suppose that
our space $X$ is completely regular.  By this we mean that points and
closed sets are separated by continuous functions, i.e., given any
closed set $F \subset X$ and any point $x \not\in F$, there exists a
continuous function $f: X \rightarrow \R$ with $f(x) = 0$ and $f(y) =
1$ for all $y \in F$.  Most common spaces satisfy this condition; in
particular, every topological manifold (and hence our base manifold
$M$) is completely regular.  In this setting, we can make the
following definition.

\begin{dfn}\label{dfn:6}
  The sequence $\{\nu_k\} \subset \mathfrak{M}(X)$ is said to
  \emph{converge weakly} to $\nu \in \mathfrak{M}(X)$ if for every
  bounded continuous function $f$ on $X$,
  \begin{equation*}
    \integral{X}{}{f}{d \nu_k} \rightarrow \integral{X}{}{f}{d \nu}.
  \end{equation*}
\end{dfn}

To prove that $\omega$-convergence of metrics implies weak convergence
of the induced measures, we need the Portmanteau theorem
\cite[Thm.~8.1]{e70:_topol_and_measur}, a portion of which we quote
here:

\begin{thm}[Portmanteau theorem]\label{thm:10}
  Let $\nu$ be a measure in $\mathfrak{M}(X)$, and let $\{\nu_k\}$ be
  a sequence in $\mathfrak{M}(X)$.  Then the following conditions are
  equivalent:
  \begin{enumerate}
  \item \label{item:17} $\nu_k$ converges weakly to $\nu$,
  \item \label{item:10} $\limsup \nu_k (F) = \nu (F)$ for all closed sets $F
    \subset X$,
  \end{enumerate}
\end{thm}

With this theorem at hand, it is a simple matter to prove the claim
from above.

\begin{lem}\label{lem:54}
  Let $\tilde{g} \in \M_f$, and let $\rho \in C^0(M)$ be any
  continuous function.  If the sequence $\{ g_k \}$ $\omega$-converges
  to $\tilde{g}$, then $\mu_{g_k}$ converges weakly to
  $\mu_{\tilde{g}}$, so in particular
  \begin{equation*}
    \lim_{k \rightarrow \infty} \| \rho \|_{g_k} = \| \rho \|_{\tilde{g}}.
  \end{equation*}
\end{lem}
\begin{proof}
  We wish to apply Theorem \ref{thm:10}, which refers to Borel
  measures.  According to our conventions, the measures $\mu_{g_k}$
  and $\mu_{\tilde{g}}$ are considered as measures on the Lebesgue
  algebra of $M$, but since the Borel algebra is a subalgebra of the
  Lebesgue algebra, we can use Theorem \ref{thm:10} by simply
  restricting these measures to the Borel algebra.
  
  By Theorem \ref{thm:19}, condition (\ref{item:10}) of Theorem
  \ref{thm:10} holds.  Therefore, $\mu_{g_k}$ converges weakly to
  $\mu_{\tilde{g}}$, implying the lemma immediately.
\end{proof}

\subsection{$L^p$ spaces}\label{sec:lp-spaces}

We now move on to the next fact we need.  In this subsection, we prove
that if $\tilde{g} \in \M_f$, i.e., $\tilde{g}$ is a measurable,
finite-volume semimetric, then the set of $C^\infty$ functions is
dense in $L^p(M, \tilde{g})$ for $1 \leq p < \infty$, just as in the
case of a smooth volume form.  (Of course, by $L^p(M, \tilde{g})$ we
mean those functions on $M$ whose absolute value to the $p$-th power
is integrable with respect to $\mu_{\tilde{g}}$.)

To prove this claim, we first prove a statement about measures on
$\R^n$ that is proved almost identically to
\cite[Cor.~4.2.2]{bogachev07:_measur_theor}, where the statement is
made for Borel measures.  To prove it for Lebesgue measures, only one
tiny modification is necessary.

\begin{thm}\label{thm:13}
  Let a nonnegative measure $\nu$ on the algebra of Lebesgue sets in
  $\R^n$ be bounded on bounded sets. Then the class $C_0^\infty(\R^n)$
  of smooth functions with bounded support is dense in $L^p(\R^n,
  \nu)$, $1 \leq p < \infty$.
\end{thm}
\begin{proof}
  By the proof of \cite[Cor.~4.2.2]{bogachev07:_measur_theor}, if $F$
  is any Borel measurable set with $\nu(F) < \infty$, then $F$ can be
  approximated to arbitrary accuracy by sets from the algebra
  $\mathcal{C}$ generated by cubes with edges parallel to the
  coordinate axes.  (By this we mean that for any given $\epsilon >
  0$, we can find a set $A \in \mathcal{C}$ such that $\nu(F \setminus
  A) + \nu(A \setminus F) < \epsilon$.)

  Now, say that $E$ is a Lebesgue measurable set with $\nu(E) <
  \infty$.  Then by Lemma \ref{lem:53}, $E = F \cup G$, where $F$ is
  Borel measurable and $\nu(G) = 0$.  By approximating $F$ with sets
  from $\mathcal{C}$, we can therefore approximate $E$ with sets from
  $\mathcal{C}$.

  This means that linear combinations of the characteristic functions
  of sets in $\mathcal{C}$ are dense in $L^p(\R^n, \nu)$.  But we can
  easily approximate such functions by smooth functions with
  compact support---it suffices to be able to approximate any open
  cube, which is easily done.
\end{proof}

Now, since any $\tilde{g} \in \M_f$ has finite volume, its induced
measure $\mu_{\tilde{g}}$ clearly satisfies the hypotheses of the
theorem in any coordinate chart.  Therefore, we have:

\begin{cor}\label{cor:12}
  If $\tilde{g} \in \M_f$, then $C^\infty(M)$ is dense in
  $L^p(M,\tilde{g})$.
\end{cor}

\section{Bounded semimetrics}\label{sec:meas-techn-result}

In this section, we go one step further in our understanding of the
injection $\Omega : \overline{\M} \rightarrow \Mfhat$ that was
introduced in Theorem \ref{thm:38}.  Specifically, we want to see that
the image $\Omega(\overline{\M})$ contains all equivalence classes of
bounded, measurable semimetrics (cf.~Definition \ref{dfn:23}).

Our strategy for proving this is to first prove the fact for smooth
semimetrics by showing that for any smooth semimetric $g_0$, there is
a finite path $g_t$, $t \in (0,1]$, in $\M$ with $\lim_{t \to 0} g_t =
g_0$ (where we take the limit in the $C^\infty$ topology of $\s$).
Similarly to the constructions in Section \ref{sec:compl-metr-spac},
it is then simple to construct a sequence $\{g_{t_k}\}$ from $g_t$
such that $g_{t_k} \overarrow{\omega} g_0$ for $k \rightarrow \infty$.
If we simply let $t_k$ be any monotonically decreasing sequence
converging to zero, then it is trivial to show $\omega$-convergence of
this sequence.

\subsection{Paths to the boundary}\label{sec:paths-boundary}

Before we get into the proofs, we put ourselves in the proper setting,
for which we first need to introduce the notion of a quasi-amenable subset.
These are defined by weakening the requirements for an amenable subset
(cf.~Definition \ref{dfn:2}), giving up the condition of being
``uniformly inflated'':

\begin{dfn}\label{dfn:26}
  We call a subset $\U \subset \M$ \emph{quasi-amenable} if $\U$ is
  convex and we can find a constant $C$ such that for all $\tilde{g}
  \in \U$, $x \in M$ and $1 \leq i,j \leq n$,
  \begin{equation}\label{eq:134}
    |\tilde{g}_{ij}(x)| \leq C.
  \end{equation}
\end{dfn}

Quasi-amenable subsets are bounded subsets of $\s$, but they can run
right up to the boundary of $\M$ as a topological subset of $\s$.  We
denote this boundary by $\partial \M$.\label{p:btop}  Since $\M$ consists of all
smooth $(0,2)$-tensor fields on $M$ which induce positive definite
scalar products on $T_x M$ at all $x \in M$, we have that each tensor
field in $\partial \M$ induces a smooth, positive \emph{semi}definite
scalar product at each point of $M$.  That is,
\begin{equation*}
  \partial \M = \{ h \in \mathcal{S} \mid h \not\in \M\
  \textnormal{and}\ h(x)(X,X) \geq 0 \ \textnormal{for all}\ 
  X\in T_x M \}.
\end{equation*}
So $\partial \M$ consists of all smooth semimetrics that somewhere
fail to be positive definite.

Let $\U$ be any quasi-amenable subset, and denote by
$\cl(\U)$ \label{p:cl-U} the closure of $\U$ in the $C^\infty$
topology of $\s$.  Thus, $\cl(\U)$ may contain some smooth
semimetrics.

Now, suppose some $g_0 \in \cl(\U) \cap \btop$ is given, and let $g_1
\in \U$ have the property that $h := g_1 - g_0 \in \M$, i.e., that $h$
is positive definite.  Exploiting the linear structure of $\M$, we
define the simplest path imaginable from $g_0$ to $g_1$:
\begin{equation}\label{eq:17}
  g_t := g_0 + t h.
\end{equation}
Then by the convexity of $\U$, $g_t$ is a path $(0,1] \rightarrow \U$
with limit (in the topology of $\s$) as $t \rightarrow 0$ equal to
$g_0$.

\begin{rmk}\label{rmk:19}
  We make two remarks about this setup:
  \begin{enumerate}
  \item Requiring that $h > 0$ is a technical assumption that we will
    use later; we do not believe it to be essential to the end result.
  \item It is not hard to see that any $g_0 \in \partial \M$ is
    contained in $\cl(\U)$ for an appropriate quasi-amenable subset
    $\U$.
  \end{enumerate}
\end{rmk}

Recall that the length of $g_t$ is given by
\begin{equation}\label{eq:29}
  \begin{aligned}
    L(g_t) &= \integral{0}{1}{\| g'_t \|_{g_t}}{d t} =
    \integral{0}{1}{\left(
        \integral{M}{}{\tr_{g_t}((g'_t)^2)}{\mu_{g_t}}
      \right)^{1/2}}{d t} \\
    &= \integral{0}{1}{\left( \integral{M}{}{\tr_{g_t}(h^2) \sqrt{\det
            (g^{-1} g_t)}}{\mu_g} \right)^{1/2}}{d t}
  \end{aligned}
\end{equation}
To prove that $g_t$ is a finite path, we must therefore estimate the
integrand,
\begin{equation*}
  \tr_{g_t} (h^2) \sqrt{\det (g^{-1} g_t)}.  
\end{equation*}
This will follow from pointwise estimates combined with a
compactness/continuity argument.

\subsection{Pointwise estimates}\label{sec:pointwise-estimates}

Let $A = (a_{ij})$ and $B = (b_{ij})$ be real, symmetric $n \times n$
matrices, with $A_t := A+tB$ for $t \in (0,1]$.  We will assume that
$B > 0$ and that $A \geq 0$.  (In this scheme, $A$ and $B$ play the
role of $g_0(x)$ and $h(x)$, respectively, at some point $x \in M$.)
Furthermore, we fix an arbitrary matrix $C$ that is invertible and
symmetric (this plays the role of $g(x)$).

Therefore, to get a pointwise estimate on $\tr_{g_t} (h^2) \sqrt{\det
  g^{-1} g_t}$, we need to estimate $\tr_{A_t}(B^2) \sqrt{\det (C^{-1}
  A_t)}$.  We prove the desired estimate in three lemmas.

For any symmetric matrix $D$, let $\lambda^D_{\mathrm{min}} =
\lambda^D_1 \leq \cdots \leq \lambda^D_n = \lambda^D_{\mathrm{max}}$
be its eigenvalues numbered in increasing order.

\begin{lem}\label{lem:1}
  \begin{align*}
    \lambda^{A_t}_{\mathrm{min}} &\geq \lambda^A_{\mathrm{min}} + t
    \lambda^B_{\mathrm{min}} \\
    \lambda^{A_t}_{\mathrm{max}} &\leq \lambda^A_{\mathrm{max}} + t
    \lambda^B_{\mathrm{max}} \leq \lambda^A_{\mathrm{max}} +
    \lambda^B_{\mathrm{max}}
  \end{align*}
\end{lem}
\begin{proof}
  By Lemma \ref{lem:46}, the function mapping a self-adjoint matrix to
  its minimal (resp.~maximal) eigenvalue is concave (resp.~convex).
  This, combined with the facts that $\lambda^B_{\textnormal{max}} >
  0$ (since $B > 0$) and $t \leq 1$, gives the result immediately.
\end{proof}

\begin{lem}\label{lem:3}
  \begin{equation*}
    \tr_{A_t} \left( B^2 \right) \sqrt{\det C^{-1} A_t} \leq
    \frac{n \left(
        \lambda^B_{\mathrm{max}} \right)^2 \left( \lambda^{A_t}_{\mathrm{max}}
      \right)^{\frac{n-1}{2}}}{\sqrt{\det C} \left(
        \lambda^A_{\mathrm{min}} + t \lambda^B_{\mathrm{min}}
      \right)^{3/2}}
  \end{equation*}
\end{lem}
\begin{proof}
  We focus on the trace term first.  Note
  \begin{equation*}
    \tr_{A_t}( B^2 ) = \tr \left( \left( A_t^{-1} B \right)^2 \right).
  \end{equation*}
  Since $B$ is a symmetric matrix, there exists a basis for which $B$
  is diagonal, so that $B = \diag(\lambda_1^B,\ldots,\lambda_n^B)$.
  In this basis, if we denote $A_t^{-1} = (a_t^{ij})$, then we have
  \begin{equation}\label{eq:9}
    \begin{aligned}
      \tr \left( \left( A_t^{-1} B \right)^2 \right) &= \sum_{ij} a_t^{ij} \lambda_j^B
      a_t^{ji} \lambda_i^B \\
      &= \sum_{ij} \left(a^{ij}_t\right)^2 \lambda_i^B \lambda_j^B \\
      &\leq \left( \lambda^B_{\mathrm{max}} \right)^2 \sum_{ij}
      \left(a^{ij}_t\right)^2 \\
      &= \left( \lambda^B_{\mathrm{max}} \right)^2 \tr \left(
        A_t^{-2}\right),
    \end{aligned}
  \end{equation}
  where the second line follows from symmetry of $A_t^{-1}$ and the
  last line follows from
  \begin{equation*}
    \tr \left( A_t^{-2} \right) = \sum_{ij} a_t^{ij} a_t ^{ji} =
    \sum_{ij} \left( a_t^{ij} \right)^2.
  \end{equation*}

  Now, recall from the discussion in the proof of Lemma \ref{lem:48}
  that the trace of the square of a matrix is given by the sum of the
  squares of its eigenvalues.  Therefore, 
  \begin{equation}\label{eq:10}
    \tr \left( A_t^{-2}\right) = \sum_i \left( \lambda^{A_t}_i \right)^{-2} \leq
    n \left( \lambda^{A_t}_{\mathrm{min}} \right)^{-2}.
  \end{equation}
  This takes care of the trace term.

  For the determinant term, we clearly have
  \begin{equation}\label{eq:11}
    \det A_t = \lambda^{A_t}_1 \cdots \lambda^{A_t}_n \leq
    \lambda^{A_t}_{\mathrm{min}} \left( \lambda^{A_t}_{\mathrm{max}} \right)^{n-1}.
  \end{equation}
  Combining equations \eqref{eq:9}, \eqref{eq:10} and \eqref{eq:11}
  with the estimate of Lemma \ref{lem:1} now immediately yields the
  result.
\end{proof}

Since $A \geq 0$, we know that $\lambda^A_{\mathrm{min}} \geq 0$.
Therefore we can also immediately write the estimate of Lemma
\ref{lem:3} in a weaker, ``worst-case'' form:

\begin{lem}\label{lem:4}
  \begin{equation*}
    \tr_{A_t} \left( B^2 \right) \sqrt{\det C^{-1} A_t} \leq
    \frac{n \left(
        \lambda^B_{\mathrm{max}} \right)^2 \left( \lambda^{A_t}_{\mathrm{max}}
      \right)^{\frac{n-1}{2}}}{\sqrt{\det C} \left(
        \lambda^B_{\mathrm{min}}
      \right)^{3/2}}  \frac{1}{t^{3/2}}
  \end{equation*}
\end{lem}

\subsection{Finiteness of $L(g_t)$}\label{sec:finiteness-lg_t}

We want to use the pointwise estimate of Lemma \ref{lem:4} to prove
the main result of the section.

It is clear that to pass from the pointwise result of Lemma
\ref{lem:4} to a global result, we will have to estimate the maximum
and minimum eigenvalues of $h$, as well as the maximum eigenvalue of
$g_t$.  We begin by noting that since we work over an amenable
coordinate atlas (cf.~Definition \ref{dfn:1}), all coefficients of
$h$, $g$ and $g_0$ are bounded in absolute value.  Therefore, so are
their determinants.  In particular, since $g > 0$ and $h > 0$, we can
assume that $\det g \geq C_0$ and $C_1 \geq \det h \geq C_2$ over each
chart of the amenable atlas for some constants $C_0,C_1,C_2 > 0$.

\begin{lem}\label{lem:21}
  The quantities $\lambda^h_{\mathrm{max}}$ and
  $\lambda^{g_t}_{\mathrm{max}}$, as local functions on each
  coordinate chart, are uniformly bounded, say
  $\lambda^h_{\mathrm{max}}(x) \leq C_3$ and $\lambda^{g_t}_{
    \mathrm{max}}(x) \leq C_4$ for all $x$ and $t$.
\end{lem}
\begin{proof}
  Recall the formula \eqref{eq:133} for the maximal eigenvalue of a
  symmetric matrix.  If $\langle\!\langle \cdot , \cdot
  \rangle\!\rangle$ is the Euclidean scalar product in a chart around
  the point $x \in M$, then
  \begin{equation*}
    \lambda^h_{\textnormal{max}}(x) = \max_{\substack{v \in T_x M \\
        \langle\!\langle v, v \rangle\!\rangle = 1}}
    \langle\!\langle v, h(x) v
    \rangle\!\rangle \quad
    \textnormal{and} \quad
    \lambda^{g_t}_{\textnormal{max}}(x) = \max_{\substack{v \in T_x M \\
        \langle\!\langle v, v \rangle\!\rangle = 1}}
    \langle\!\langle v, g_t(x) v
    \rangle\!\rangle
  \end{equation*}

  Keep in mind that we work over an amenable atlas and that the unit
  sphere in each $T_x M$ (with respect to the Euclidean scalar product
  $\langle\!\langle \cdot , \cdot \rangle\!\rangle$) is compact.
  Since $h$ is continuous and $g_t \in \U$ for all $t \in (0,1]$ we
  can find some constant that bounds $|h_{ij}(x)|$ and
  $|(g_t)_{ij}(x)|$ uniformly for all $x \in M$, all $1 \leq i, j \leq
  n$, and all $t \in (0,1]$.

  From this uniform bound, it is easy to see that there are constants
  $C_3$ and $C_4$ such that
  \begin{equation*}
    \langle\!\langle v, h(x) v \rangle\!\rangle \leq C_3, \quad
    \langle\!\langle v, g_t(x) v \rangle\!\rangle \leq C_4
  \end{equation*}
  for all $x \in M$, $t \in (0,1]$ and $v \in T_x M$ with
  $\langle\!\langle v, v \rangle\!\rangle = 1$.  Since passing to the
  maximum preserves these inequalities, we get the desired bounds on
  the eigenvalues.
\end{proof}

\begin{lem}\label{lem:7}
  The quantity $\lambda^h_{\min}$, as a function over each coordinate
  chart, is uniformly bounded away from 0, say $\lambda^h_{\min} \geq
  C_5 > 0$.
\end{lem}
\begin{proof}
  Letting as usual $\lambda^h_1(x) \leq \cdots \leq \lambda^h_n(x)$ be
  the eigenvalues of $h(x)$ listed in increasing order, we have
  \begin{equation*}
    \det h(x) = \lambda^h_1(x) \cdots \lambda^h_n(x) \leq
    \lambda^h_{\mathrm{min}}(x) \lambda^h_{\mathrm{max}}(x)^{n-1}.
  \end{equation*}
  Therefore, by Lemma \ref{lem:21} and the discussion before it,
  \begin{equation*}
    \lambda^h_{\mathrm{min}}(x) \geq \lambda^h_{\mathrm{max}}(x)^{1-n}
    \det h(x) \geq C_3^{1-n} C_2 =: C_5.
  \end{equation*}
\end{proof}

\begin{thm}\label{thm:2}
  Define a path $g_t$ as in \eqref{eq:17}.  Then
  \begin{equation*}
    L(g_t) < \infty.
  \end{equation*}
\end{thm}
\begin{proof}
  At each point $x \in M$ we have
  \begin{equation}\label{eq:15}
    \begin{aligned}
      \tr_{g_t(x)}(h(x)^2) \sqrt{\det(g(x)^{-1} g_t(x))} &\leq \frac{n
        \left( \lambda^{h}_{\mathrm{max}}(x) \right)^2 \left(
          \lambda^{g_t}_{\mathrm{max}}(x)
        \right)^{\frac{n-1}{2}}}{\sqrt{\det g(x)} \left(
          \lambda^{h}_{\mathrm{min}}(x) \right)^{3/2}}
      \frac{1}{t^{3/2}} \\
      &\leq \frac{1}{\sqrt{C_0}} \frac{C_3^2}{C_5^{3/2}}
      C_4^{\frac{n-1}{2}} \frac{1}{t^{3/2}} =: \frac{C_6}{t^{3/2}},
    \end{aligned}
  \end{equation}
  where the first inequality follows from Lemma \ref{lem:4}, and the
  second line follows from the discussion before Lemma \ref{lem:21},
  as well as Lemma \ref{lem:21} itself and Lemma \ref{lem:7}.
  
  By the integrability of $t^{-3/4}$, then,
  \begin{align*}
    L(g_t) &= \int_0^1 \left( \int_M \tr_{g_t(x)}(h(x)^2)
      \sqrt{\det(g(x)^{-1} g_t(x))} \, \mu_g \right)^{1/2} \, dt \\
    &\leq \int_0^1 \left( \int_M \frac{C_6}{t^{3/2}} \, \mu_g
    \right)^{1/2} \, dt = \sqrt{C_6 \Vol(M,g)} \int_0^1 \frac{1}{t^{3/4}}
    \, dt < \infty.
  \end{align*}
\end{proof}

\subsection{Bounded, nonsmooth
  semimetrics}\label{sec:nonsm-linfty-semim}

We now move on to showing that the equivalence class of any
bounded semimetric, not just smooth ones, is contained in
$\Omega(\overline{\M})$.  The results we've just proved will come in
handy.

Let's review what we already know about the image of $\Omega$.  From
Proposition \ref{prop:26}, we know that the completion of an amenable
subset $\U$ can be identified with its $L^2$-completion $\U^0$.  So
the equivalence class of any measurable metric that can be obtained as
the $L^2$ limit of a sequence of metrics from an amenable subset
belongs to $\Omega(\overline{\M})$.  Furthermore, as we noted in the
introduction to this section, it is easy to see that Theorem
\ref{thm:2} implies that for any smooth semimetric $\tilde{g}$, there
exists a sequence in $\M$ that $\omega$-converges to $\tilde{g}$.
Thus $[\tilde{g}]$ also belongs to $\Omega(\overline{\M})$.

Recall that by the discussion following Theorem \ref{thm:38}, it is
not necessary to distinguish between equivalence classes in
$\Omega(\overline{\M})$ (or individual semimetrics that represent
them) and sequences in $\M$ that $\omega$-converge to them---i.e.,
points of $\overline{\M}$.  Thus, for the types of (semi)metrics
listed in the last paragraph, we will continue to drop this
distinction in what follows---expressions like $d(g_0, g_1)$ are
well-defined even when $g_0$ and $g_1$ are not smooth metrics, as long
as we have $[g_0], [g_1] \in \Omega(\overline{\M})$.

We will achieve our goal in this section essentially through studying
the completion of a quasi-amenable subset (cf.~Definition
\ref{dfn:26}) analogously to the methods we used for amenable subsets
in Section \ref{sec:compl-an-amen}.

To begin with, we want to prove a result about quasi-amenable subsets
that is a generalization of Theorem \ref{thm:5}.  That result was for
amenable subsets, and so we expect the result for quasi-amenable
subsets to be weaker.  This is indeed the case, but before we can
prove the larger result, we first need to prove a couple of lemmas.

\begin{lem}\label{lem:8}
  Let $\U \subset \M$ be quasi-amenable.  Recall that we denote the
  closure of $\U$ in the $C^\infty$ topology of $\s$ by $\cl(\U)$, and
  we denote the boundary of $\M$ in the $C^\infty$ topology of $\s$ by
  $\btop$. Then for each $\epsilon > 0$, there exists $\delta > 0$
  such that $d(g_0, g_0 + \delta g) < \epsilon$ for all $g_0 \in
  \cl(\U) \cap \btop$.
\end{lem}
\begin{proof}
  For any $g_0 \in \cl(\U) \cap \btop$, we consider the path $g_t :=
  g_0 + t h$, where $h := \delta g$ and $t \in (0,1]$.  The proof
  consists of reexamining the estimates of Theorem \ref{thm:2} and
  showing that they only depend on upper bounds on the entries of
  $g_0$ (and $g$, but we get these automatically when we work over an
  amenable atlas), and that the bound on the length of $g_t$ goes to
  zero as $\delta \rightarrow 0$.

  So, recall the main estimate \eqref{eq:15} of Theorem \ref{thm:2}:
  \begin{equation*}
    \tr_{g_t(x)}(h(x)^2) \sqrt{\det(g(x)^{-1} g_t(x))} \leq
    \frac{n \left( \lambda^{h}_{\mathrm{max}}(x)
      \right)^2  \left( \lambda^{g_t}_{\mathrm{max}}(x)
      \right)^{\frac{n-1}{2}}}{\sqrt{\det g(x)} \left(
        \lambda^{h}_{\mathrm{min}}(x) \right)^{3/2}}
    \frac{1}{t^{3/2}}.
  \end{equation*}
  
  Since $\det g(x)$ is constant w.r.t.~$\delta$, we ignore this term.
  By Lemma \ref{lem:1},
  \begin{equation*}
    \lambda^{g_t}_{\mathrm{max}}(x) \leq
    \lambda^{g_0}_{\mathrm{max}}(x) + \lambda^{h}_{\mathrm{max}}(x) =
    \lambda^{g_0}_{\mathrm{max}}(x) + 
    \delta \lambda^{g}_{\mathrm{max}}(x),
  \end{equation*}
  where the final inequality follows since the eigenvalues of $\delta
  g(x)$ are clearly just $\delta$ times the eigenvalues of $g(x)$.
  Therefore, using the same arguments as in Lemma \ref{lem:21},
  $\lambda^{g_t}_{\mathrm{max}}(x)$ is bounded from above, uniformly
  in $x$ and $t$, by a constant that decreases as $\delta$ decreases.
  Furthermore, this constant does not depend on our choice of $g_0 \in
  \cl(\U) \cap \btop$, since the proof of Lemma \ref{lem:21} depended
  only on uniform upper bounds on the entries of $g_0$, and we are
  guaranteed the same upper bounds on all elements of $\cl(\U) \cap
  \btop$ since $\U$ is quasi-amenable.

  We now focus our attention on the term
  \begin{equation*}
    \frac{\left( \lambda^{h}_{\mathrm{max}}(x)
      \right)^2}{\left( \lambda^{h}_{\mathrm{min}}(x) \right)^{3/2}} =
    \frac{\left( \delta \lambda^{g}_{\mathrm{max}}(x)
      \right)^2}{\left( \delta \lambda^{g}_{\mathrm{min}}(x)
      \right)^{3/2}} = \frac{\left( \lambda^{g}_{\mathrm{max}}(x)
      \right)^2}{\left( \lambda^{g}_{\mathrm{min}}(x)
      \right)^{3/2}} \sqrt{\delta}.
  \end{equation*}
  This expression clearly goes to zero as $\delta \rightarrow 0$.
  Therefore, we have shown that the constant $C_6$ in the estimate
  \eqref{eq:15} depends only on the choice of $\U$ and $\delta$, and
  that $C_6 \rightarrow 0$ as $\delta \rightarrow 0$.  The result now
  follows.
\end{proof}

The next lemma implies, in particular, that $\btop$ is \emph{not
  closed} in the $L^2$ topology of $\s$, nor is it in the topology of
$d$ on $\Omega(\overline{\M})$.  It also implies that around any point
in $\M$, there exists no $L^2$- or $d$-open neighborhood.

\begin{lem}\label{lem:9}
  Let $\U \in \M$ be any quasi-amenable subset.  Then for all
  $\epsilon > 0$, there exists a function $\rho_\epsilon \in
  C^\infty(M)$ with the properties that for all $g_1 \in \U$,
  \begin{enumerate}
  \item $\rho_\epsilon g_1 \in \btop$,
  \item $\rho_\epsilon(x) \leq 1$ for all $x \in M$,
  \item $\| g_1 - \rho_\epsilon g_1 \|_g < \delta$ and
  \item $d(g_1, \rho_\epsilon g_1) < \epsilon$.
  \end{enumerate}
\end{lem}
\begin{proof}
  Let $x_0 \in M$ be any point, and for each $n \in \N$, choose a
  function $\rho_n \in C^\infty(M)$ satisfying
  \begin{enumerate}
  \item $\rho_n(x_0) = 0$,
  \item $0 \leq \rho_n(x) \leq 1$ for all $x \in M$ and
  \item $\rho_n \equiv 1$ outside an open set $Z_n$ with $\Vol(Z_n, g)
    \leq 1/n$.
  \end{enumerate}
  Then clearly $\| g_1 - \rho_n g_1 \|_g \rightarrow 0$ as $n
  \rightarrow \infty$, and this convergence is uniform in $g_1$
  because of the upper bounds guaranteed by the fact that $g_1 \in \U$.

  Furthermore, if we estimate the length of the path $g_t^n := \rho_n
  g_1 + t (g_1 - \rho_n g_1)$, there will be no contribution to the
  integral from points of $M\setminus Z_n$, and on $Z_n$, we can find
  a constant $C_6$ as in \eqref{eq:15} that \emph{does not depend on
    $n$}, simply by assuming the worst case that $\rho_n \equiv 0$ on
  $Z_n$ for all $n$.  Furthermore, $C_6$ does not depend on $g_1$,
  just on the choice of $\U$, by the same arguments as in the proof of
  Lemma \ref{lem:8}.

  Therefore we get that
  \begin{equation*}
    L(g_t^n) \leq \sqrt{C_6 \Vol(Z_n, g)} \int_0^1 \frac{1}{t^{3/4}} \, dt \leq
    \sqrt{\frac{C_6}{n}} \int_0^1 \frac{1}{t^{3/4}} \, dt,
  \end{equation*}
  which converges to zero as $n \rightarrow \infty$.  Choosing $n$
  large enough completes the proof.
\end{proof}

The next theorem is the desired analog of Theorem \ref{thm:5}.  Note
that only one half of Theorem \ref{thm:5} holds in this case, and even
this is proved only in a weaker form.

\begin{thm}\label{thm:8}
  Let $\U \subset \M$ be quasi-amenable.  Then for all $\epsilon > 0$,
  there exists $\delta > 0$ such that if $g_0, g_1 \in \cl(\U)$ with
  $\| g_0 - g_1 \|_g < \delta$, then $d(g_0, g_1) < \epsilon$.

  In particular, $d$ is uniformly continuous in the $L^2$ topology of
  $\M$ when restricted to $\cl(\U)$, and if $\phi : (\cl(\U), \| \cdot
  \|_g) \rightarrow (\cl(\U), d)$ is the identity mapping on the level
  of sets (i.e., $\phi(g) = g$), then $\phi$ is uniformly continuous.
\end{thm}
\begin{proof}
  First, we enlarge $\U$ if necessary to include \emph{all} metrics
  satisfying the bound given in Definition \ref{dfn:26}.  This
  enlarged $\U$ is then clearly convex by the triangle inequality for
  the absolute value, and hence it is still a quasi-amenable subset.

  Now, let $\epsilon > 0$ be given.  We prove the statement first for
  $g_0, g_1 \in \cl(\U) \cap \btop$, then use this to
  prove the general case.

  By Lemma \ref{lem:8}, we can choose $\delta_1 > 0$ such that
  $d(\hat{g}, \hat{g} + \delta_1 g) < \epsilon/3$ for all $\hat{g} \in
  \cl(\U) \cap \btop$.  We define an amenable subset of $\M$ by
  \begin{equation*}
    \U' :=
    \left\{
      \hat{g} + \delta_1 g \mid \hat{g} \in \U
    \right\}.
  \end{equation*}
  This set is, indeed, amenable, since for each $x \in M$, Lemma
  \ref{lem:46} implies that
  \begin{equation*}
    \lambda^{\hat{g} + \delta_1 g}_{\textnormal{min}}(x) \geq
    \lambda^{\hat{g}}_{\textnormal{min}}(x) + \lambda^{\delta_1
      g}_{\textnormal{min}}(x) \geq \delta_1 \lambda^g_{\textnormal{min}}(x).
  \end{equation*}
  Now, by Theorem \ref{thm:5}, there exists $\delta > 0$ such that if
  $\tilde{g}_0, \tilde{g}_1 \in \U'$ with $\| \tilde{g}_0 -
  \tilde{g}_1 \|_g < \delta$, then $d(\tilde{g}_0, \tilde{g}_1) <
  \epsilon/3$.  Let $g_0, g_1 \in \cl(\U) \cap \btop$ be such that $\|
  g_0 - g_1 \|_g < \delta$.  If we define $\tilde{g}_i := g_i +
  \delta_1 g$ for $i = 1,2$, then it is clear that $\| \tilde{g}_0 -
  \tilde{g}_1 \|_g = \| g_0 - g_1 \|_g < \delta$.  Given this and the
  definition of $\delta_1$, we have
  \begin{equation*}
    d(g_0, g_1) \leq d(g_0, \tilde{g}_0) +
    d(\tilde{g}_0, \tilde{g}_1) + d(\tilde{g}_1, g_1) < \epsilon.
  \end{equation*}

  Now we prove the general case.  Let $\epsilon > 0$ be given.  By the
  special case we just proved, we can choose $\delta > 0$ such that if
  $\tilde{g}_0, \tilde{g}_1 \in \cl(\U) \cap \btop$ with $\|
  \tilde{g}_0 - \tilde{g}_1 \|_g < \delta$, then $d(\tilde{g}_0,
  \tilde{g}_1) < \epsilon/3$.  Let $g_0, g_1 \in \U$ be any elements
  with $\| g_0 - g_1 \|_g < \delta$.  By Lemma \ref{lem:9} and our
  enlargement of $\U$, we can choose a function $\rho \in C^\infty(M)$
  such that for $i = 0,1$,
  \begin{enumerate}
  \item $\rho g_i \in \cl(\U) \cap \btop$,
  \item $\rho(x) \leq 1$ for all $x \in M$, and
  \item $d(g_i, \rho g_i) < \epsilon/3$.
  \end{enumerate}
  (If $g_i \in \cl(\U) \cap \btop$ for both $i=1$ and $2$, we might as
  well just choose $\rho \equiv 1$.)  In particular, the second
  property of $\rho$ implies that
  \begin{equation*}
    \| \rho g_1 - \rho g_0 \|_g \leq \| g_1 - g_0 \|_g < \delta.
  \end{equation*}
  
  Then we immediately get
  \begin{equation*}
    d(g_0, g_1) \leq d(g_0, \rho g_0) + d(\rho g_0, \rho g_1) +
    d(\rho g_1, g_1) < \epsilon.
  \end{equation*}
  This proves the general case and thus the theorem.
\end{proof}

\begin{rmk}\label{rmk:22}
  Let's take a brief moment to discuss why only one half of Theorem
  \ref{thm:5} holds for quasi-amenable subsets.  The problem is that
  the determinants of elements of a quasi-amenable subset need not
  satisfy any uniform lower bounds.  Hence two metrics $g_0$ and $g_1$
  from a quasi-amenable subset can differ greatly with respect to $\|
  \cdot \|_g$, yet do so only on a subset of $M$ that has small volume
  with respect to $g_0$ and $g_1$ themselves.  In this situation,
  Proposition \ref{prop:18} implies that $d(g_0, g_1)$ will also be
  small.  So we cannot say that $\| g_1 - g_0 \|_g$ is small whenever
  $d(g_0, g_1)$ is, and something like statement (\ref{item:9}) of
  Theorem \ref{thm:5} cannot hold for quasi-amenable subsets.
\end{rmk}

With the above theorem at hand, it is now possible to obtain the
information on $\Omega(\overline{\M})$ that we desired.  First notice
that a bounded semimetric is precisely a semimetric that can be
obtained as the $L^2$ limit of a sequence of metrics contained within
some quasi-amenable subset.

Using the relationship between $d$ and $\| \cdot \|_g$ determined in
Theorem \ref{thm:8}, we can prove the following.

\begin{prop}\label{prop:27}
  Let $[\tilde{g}] \in \Mfhat$ be an equivalence class containing at
  least one bounded, measurable semimetric.  Then for any bounded
  representative $\tilde{g} \in [\tilde{g}]$, there exists a sequence
  $\{ g_k \}$ in $\M$ that both $L^2$- and $\omega$-converges to
  $\tilde{g}$.  Thus $[\tilde{g}] \in \Omega(\overline{\M})$.

  Moreover, suppose $\tilde{g} \in \U^0$ for some quasi-amenable
  subset $\U \subset \M$.  Then for any sequence $\{g_l\}$ in $\U$
  that $L^2$-converges to $\tilde{g}$, $\{g_l\}$ is $d$-Cauchy and
  there exists a subsequence $\{g_k\}$ that also $\omega$-converges to
  $\tilde{g}$.
\end{prop}
\begin{proof}
  By the discussion preceding the proposition, for every
  representative $\tilde{g} \in [\tilde{g}]$, we can find a
  quasi-amenable subset $\U \subset \M$ such that $\tilde{g} \in
  \U^0$.  (Recall that $\U^0$ denotes the completion of $\U$
  w.r.t.~$\| \cdot \|_g$, cf.~Definition \ref{dfn:4}.)  Thus, there
  exists a sequence $\{ g_l \}$ that $L^2$-converges to $\tilde{g}$.
  It is $d$-Cauchy by Theorem \ref{thm:8}.  We wish to show that it
  contains a subsequence $\{ g_k \}$ that also $\omega$-converges to
  $\tilde{g}$, so we still need to verify properties
  (\ref{item:5})--(\ref{item:7}) of Definition \ref{dfn:13} (we just
  noted that $\{g_k\}$ is $d$-Cauchy, so property (\ref{item:4})
  holds).

  By passing to a subsequence, we can assume that property
  (\ref{item:7}) is satisfied for $\{g_l\}$.  We verify property
  (\ref{item:6}) in the same way as in the proof of Lemma
  \ref{lem:43}.  That is, $L^2$-convergence of $\{ g_l \}$ implies by
  Remark \ref{rmk:9} that there exists a subsequence $\{ g_k \}$ of
  $\{ g_l \}$ that converges to $\tilde{g}$ a.e.  Finally,
  a.e.-convergence of $\{g_k\}$ to $\tilde{g}$ and continuity of the
  determinant function imply that property \eqref{item:5} holds.
\end{proof}

Thus, like we did for more restricted types of metrics before, this
proposition allows us to cease to distinguish between bounded
semimetrics and sequences $\omega$-converging to them.

\section{Unbounded metrics and the proof of the main
  result}\label{sec:finite-volume-metr}

Up to this point, we have an injection $\Omega : \overline{\M}
\rightarrow \Mfhat$, and we have determined that the image
$\Omega(\overline{\M})$ contains all equivalence classes containing
bounded semimetrics.  In this section, we prove that $\Omega$ is
surjective.  We will make good use of what we already know about
$\Omega(\overline{\M})$ in order to do so.

The following theorem is the surjectivity statement.  It is proved
using the same philosophy as in the construction of the completion of
$\pos \cdot g$ that was given in Section \ref{sec:compl-orbit-space}.
We simply need to adapt the arguments given there to our situation.

\begin{thm}\label{thm:12}
  Let any $[\tilde{g}] \in \Mfhat$ be given.  Then there exists a
  sequence $\{ g_k \}$ in $\M$ such that
  \begin{equation*}
    g_k \overarrow{\omega} [\tilde{g}].
  \end{equation*}
  Thus, $\Omega : \overline{\M} \rightarrow \Mfhat$ is surjective.
\end{thm}
\begin{proof}
  In view of Proposition \ref{prop:27}, it remains only to prove this
  for the equivalence class of a measurable, unbounded semimetric
  $\tilde{g} \in \Mf$.

  Given any element $\hat{g} \in \M_f$, we can define $\exp_{\hat{g}}$
  on tensors of the form $\sigma \hat{g}$, where $\sigma$ is any
  function, purely algebraically.  We simply set
  \begin{equation}\label{eq:60}
    \exp_{\hat{g}}(\sigma \hat{g}) := \left( 1 + \frac{n}{4} \sigma
    \right)^{4/n} \hat{g},
  \end{equation}
  so that the expression coincides with the usual one if $\hat{g}
  \in \M$ and $\sigma \in C^\infty(M)$ with $\sigma > -\frac{4}{n}$
  (cf.~\eqref{eq:56}).  If $\sigma$ is additionally measurable, then
  $\exp_{\hat{g}}(\sigma \hat{g})$ will also be measurable.

  Now, let $\tilde{g} \in \Mf$.  Then we can find a measurable,
  positive function $\xi$ on $M$ such that $g_0 := \xi \tilde{g}$ is a
  bounded semimetric.  The same calculation as in the proof of Theorem
  \ref{thm:11} shows that finite volume of $\tilde{g}$ implies $\rho
  := \xi^{-1} \in L^{n/2}(M, g_0)$.

  Define the map $\psi$ by $\psi(\sigma) := \exp_{g_0}(\sigma
  g_0)$, and let
  \begin{equation}\label{eq:140}
    \lambda := \frac{4}{n} \left( \rho^{n/4} - 1 \right).
  \end{equation}
  Then clearly $\psi(\lambda) = \rho g_0 = \tilde{g}$.  Moreover, we
  claim that $\lambda \in L^2(M, g_0)$ and hence, by Corollary
  \ref{cor:12}, we can find a sequence $\{\lambda_k\}$ of smooth
  functions that converge in $L^2(M, g_0)$ to $\lambda$ .  That
  $\lambda \in L^2(M, g_0)$ follows from two facts.  First, $\rho \in
  L^{n/2}(M, g_0)$, implying that $\rho^{n/4} \in L^2(M, g_0)$.
  Second, finite volume of $g_0$ implies that the constant function $1
  \in L^2(M, g_0)$ as well.

  Since $\lambda_k \rightarrow \lambda$ in $L^2(M, g_0)$, Remark
  \ref{rmk:9} implies that by passing to a subsequence, we can also
  assume that $\lambda_k \rightarrow \lambda$ pointwise a.e., where we
  note that here, ``almost everywhere'' means with respect to
  $\mu_{g_0}$.  With respect to the fixed, smooth, strictly positive
  volume form $\mu_g$, this actually means that $\lambda_k(x)
  \rightarrow \lambda(x)$ for a.e.~$x \in M \setminus X_{g_0}$, since
  $X_{g_0}$ is a nullset with respect to $\mu_{g_0}$.  Note also that
  $X_{g_0} = X_{\tilde{g}}$, since we assumed that the function $\xi$
  is positive.  Therefore $\lambda_k(x) \rightarrow \lambda(x)$ for
  a.e.~$x \in M \setminus X_{\tilde{g}}$.

  Furthermore, since from \eqref{eq:140} and positivity of $\xi$ it is
  clear that $\lambda > - \frac{4}{n}$, we can choose the sequence $\{
  \lambda_k \}$ such that $\lambda_k > - \frac{4}{n}$ for all $k \in
  \N$.  This implies, in particular, that $X_{\psi(\lambda_k)} =
  X_{g_0} = X_{\tilde{g}}$, which is easily seen from \eqref{eq:60}.
  
  We make one last assumption on the sequence $\{ \lambda_k \}$.
  Namely, by passing to a subsequence, we can assume that
  \begin{equation}\label{eq:141}
    \sum_{k = 1}^\infty \norm{\lambda_{k+1} - \lambda_k}_{g_0} < \infty.
  \end{equation}
  
  Using a limiting argument, we can show a statement analogous to, but
  weaker than, Proposition \ref{prop:9}.  Namely, if $d$ is the metric
  on $\Omega(\overline{\M})$ defined in Theorem \ref{thm:38}, then
  \begin{equation}\label{eq:51}
    d( \psi(\sigma), \psi(\tau) ) \leq \sqrt{n} \| \tau - \sigma \|_{g_0}
  \end{equation}
  for all $\sigma, \tau \in C^\infty (M)$
  with $\sigma, \tau > -\frac{4}{n}$.  We delay the proof of this
  statement to Lemma \ref{lem:33} below, though, and first finish the
  proof of the theorem.
  
  We wish to construct a sequence that $\omega$-converges to
  $\tilde{g}$ using the sequence $\{ \psi(\lambda_k) \}$.  We can't
  use $\{ \psi(\lambda_k) \}$ directly, since it is a sequence in
  $\Omega(\overline{\M})$, not $\M$ itself.  So we first verify the
  properties of $\omega$-convergence for $\{ \psi(\lambda_k) \}$ and
  then construct a sequence in $\M$ that approximates $\{
  \psi(\lambda_k) \}$ well enough that it still satisfies all the
  conditions for $\omega$-convergence.
  
  Since the sequence $\{\lambda_k\}$ is convergent in $L^2(M, g_0)$,
  it is also Cauchy in $L^2(M, g_0)$.  Using the inequality
  \eqref{eq:51}, it is then immediate that $\{\psi(\lambda_k)\}$ is a
  Cauchy sequence in $(\Omega(\overline{\M}), d)$.  This verifies
  property (\ref{item:4}) of $\omega$-convergence (cf.~Definition
  \ref{dfn:13}).
  
  We next verify property (\ref{item:6}).  Note that $X_{\tilde{g}}
  \subseteq X_{\{ \psi(\lambda_k) \}}$, since we have already shown
  that $X_{\psi(\lambda_k)} = X_{\tilde{g}}$.  (Keep in mind here the
  subtle point that $X_{\psi(\lambda_k)}$ is the deflated set of the
  \emph{individual} semimetric $\psi(\lambda_k)$, while $X_{\{
    \psi(\lambda_k) \}}$ is the deflated set of the \emph{sequence}
  $\{\psi(\lambda_k)\}$.  Refer to Definitions \ref{dfn:23} and
  \ref{dfn:25} for details.)  The inclusion implies that
  \begin{equation*}
    M \setminus X_{\{ \psi(\lambda_k) \}} \subseteq M \setminus X_{\tilde{g}},
  \end{equation*}
  so it suffices to show that $\psi(\lambda_k)(x) \rightarrow
  \tilde{g}(x)$ for a.e.~$x \in M \setminus X_{\tilde{g}}$.  But this
  is clear from the definition of $\psi$ and the fact, proved above,
  that $\lambda_k(x) \rightarrow \lambda(x)$ for a.e.~$x \in M
  \setminus X_{\tilde{g}}$.
  
  To verify property (\ref{item:5}), we claim that $X_{\{
    \psi(\lambda_k) \}} = X_{\tilde{g}}$, up to a nullset.  In the
  previous paragraph, we already showed that $X_{\tilde{g}} \subseteq
  X_{\{\psi(\lambda_k)\}}$. Furthermore, for a.e.~$x \in M \setminus
  X_{\tilde{g}}$, $\{ \psi(\lambda_k)(x) \}$ converges to
  $\tilde{g}(x)$, which is positive definite, so for a.e.~$x \in M
  \setminus X_{\tilde{g}}$, $\lim \det \psi(\lambda_k) > 0$.  This
  immediately implies that $X_{\{\psi(\lambda_k)\}} \subseteq
  X_{\tilde{g}} $, up to a nullset.
  
  The last property to verify is (\ref{item:7}).  But this is
  immediate from \eqref{eq:141} and \eqref{eq:51}.

  So we have shown that $\{\psi(\lambda_k)\}$ satisfies the properties
  of $\omega$-convergence, save that it is a sequence of measurable
  semimetrics, rather than a sequence of smooth metrics as required.
  To get a sequence in $\M$ that $\omega$-converges to $\tilde{g}$,
  recall that each of the functions $\lambda_k$ is smooth and
  therefore bounded, and also that $g_0$ is a bounded, measurable
  semimetric.  Therefore, for each fixed $k \in \N$, $\psi(\lambda_k)$
  is also a bounded, measurable semimetric, and so by Proposition
  \ref{prop:27} we can find a sequence $\{ g^k_l \}$ in $\M$ that
  $\omega$-converges to $\psi(\lambda_k)$ for $l \rightarrow \infty$.
  By a standard diagonal argument, it is then possible to select $l_k
  \in \N$ for each $k \in \N$ such that the sequence $\{ g^k_{l_k} \}$
  $\omega$-converges to $\tilde{g}$ for $k \rightarrow \infty$.  Thus
  we have found the desired sequence.

  It still remains to prove \eqref{eq:51}.  The following lemma does
  this and thus completes the proof of the theorem.
\end{proof}

\begin{lem}\label{lem:33}
  If $\sigma, \tau \in C^\infty(M)$ satisfy $\sigma, \tau > -
  \frac{4}{n}$, then
  \begin{equation*}
    d(\psi(\sigma), \psi(\tau)) \leq \sqrt{n} \| \tau - \sigma \|_{g_0}.
  \end{equation*}
\end{lem}
\begin{proof}
  Since $g_0$ is bounded, we can find a quasi-amenable subset $\U$
  such that $g_0 \in \U^0$, i.e., such that $g_0$ belongs to the
  completion of $\U$ with respect to $\| \cdot \|_g$.  Using
  Proposition \ref{prop:27}, choose a sequence $\{ g_k \}$ in $\U$
  that both $L^2$- and $\omega$-converges to $g_0$.  For each $k \in
  \N$, define a map $\psi_k$ by $\psi_k(\kappa) := \exp_{g_k}(\kappa
  g_k)$.

  By the triangle inequality, we have
  \begin{equation}\label{eq:59}
    d(\psi(\sigma), \psi(\tau)) \leq d(\psi(\sigma), \psi_k(\sigma)) +
    d(\psi_k(\sigma), \psi_k(\tau)) + d(\psi_k(\tau), \psi(\tau))
  \end{equation}
  for each $k$.  But since $g_k \in \M$, Proposition \ref{prop:9}
  applies to give
  \begin{equation}\label{eq:58}
    d(\psi_k(\sigma), \psi_k(\tau)) \leq \sqrt{n} \| \tau - \sigma
    \|_{g_k} \overarrow{k \rightarrow \infty} \sqrt{n} \| \tau - \sigma \|_{g_0},
  \end{equation}
  where the convergence follows from Lemma \ref{lem:54}.  (Note we
  have an inequality in \eqref{eq:58}, rather than an equality like in
  Proposition \ref{prop:9}.  Proposition \ref{prop:9} is a statement
  about the metric on $\pos \cdot \tilde{g}$ for some $\tilde{g} \in
  \M$.  This is a submanifold of $\M$, and the distance between points
  in a submanifold is always greater than or equal to the distance in
  the ambient space.)  By \eqref{eq:59} and \eqref{eq:58}, if we can
  show that
  \begin{equation*}
    d(\psi(\sigma), \psi_k(\sigma)) \rightarrow 0 \quad
    \textnormal{and} \quad d(\psi_k(\tau), \psi(\tau)) \rightarrow 0,
  \end{equation*}
  then we are finished.  In fact, if it holds for one, then it clearly
  holds for the other, so we prove it only for $\sigma$.

  Since $g_k, g_0 \in \U^0$, it suffices by Proposition \ref{prop:27}
  to show that $\psi_k(\sigma)$ $L^2$-converges to $\psi(\sigma)$.
  But this is simple, for if we set
  \begin{equation*}
    \alpha := \left( 1 + \frac{n}{4} \sigma \right)^{4/n},
  \end{equation*}
  then $\psi_k(\sigma) = \alpha g_k$ and $\psi(\sigma) = \alpha
  g_0$.  Thus
  \begin{equation*}
    \| \psi(\sigma) - \psi_k(\sigma) \|_g = \| \alpha g_0 - \alpha
    g_k \|_g \leq \max |\alpha| \cdot \| g_0 - g_k \|_g \rightarrow 0,
  \end{equation*}
  where the convergence follows from our assumptions on the sequence $g_k$.
\end{proof}

From Theorem \ref{thm:38}, we already know that the map $\Omega :
\overline{\M} \rightarrow \Mfhat$ is an injection.  Theorem
\ref{thm:12} now states that this map is a surjection as well.  Thus,
we have already proved the main result of this thesis, which we state
again here in full detail.

\begin{thm}\label{thm:41}
  There is a natural identification of $\overline{\M}$, the completion
  of $\M$ with respect to the $L^2$ metric, with $\Mfhat$, the set of
  measurable semimetrics with finite volume on $M$ modulo the
  equivalence given in Definition \ref{dfn:7}.

  This identification is given by a bijection $\Omega : \overline{\M}
  \rightarrow \Mfhat$, where we map an equivalence class $[\{g_k\}]$
  of $d$-Cauchy sequences to the unique element of $\Mfhat$ that all
  of its members $\omega$-subconverge to.  This map is an isometry if
  we give $\Mfhat$ the metric $\bar{d}$ defined by
  \begin{equation*}
    \bar{d}([g_0], [g_1]) := \lim_{k \rightarrow \infty} d(g^0_k, g^1_k)
  \end{equation*}
  where $\{g^0_k\}$ and $\{g^1_k\}$ are any sequences in $\M$
  $\omega$-subconverging to $[g_0]$ and $[g_1]$, respectively.
\end{thm}

As an end to this chapter, before we describe our application of this
theorem, we briefly discuss the geometry of elements of $\Mfhat$.  In
fact, an element of $\Mfhat$ does not define a geometry in the usual
sense, since the metric space structure does not agree between
different representatives of one equivalence class.  To illustrate
this, let's again take our favorite example $M = T^2$, and consider
the two equivalent semimetrics
\begin{equation*}
  g_0 =
  \begin{pmatrix}
    0 & 0 \\
    0 & 0
  \end{pmatrix}
  \quad \textnormal{and} \quad
  g_1 =
  \begin{pmatrix}
    1 & 0 \\
    0 & 0
  \end{pmatrix}.
\end{equation*}
As metric spaces, $g_0$ is just a point (the torus is completely
collapsed) and $g_1$ is a round circle (one dimension of the torus has
collapsed).

On the other hand, since representatives of a given equivalence class
in $\Mfhat$ all have equal induced measures, things like $L^p$ spaces
of functions are well-defined for an equivalence class, as they are
the same across all representatives.  But even more is true---$C^k$
and $H^s$ spaces of sections of fiber bundles can again be defined as
in Subsection \ref{sec:manifolds-mappings}, since not only are the
measures induced by two representatives equal, but the representatives
themselves are equal almost everywhere with respect to their common
measure.  Therefore, an equivalence class doesn't induce a
well-defined scalar product on a vector bundle at any individual
point, but the integral of the scalar product does not depend on the
chosen representative.

So while one must be careful about regarding an element of $\Mfhat$ as
defining a geometry, there are indeed many geometric concepts that are
well-defined for elements of $\Mfhat$.


\chapter{Application to Teichmüller theory}\label{cha:appl-teichm-space}

In this chapter, we describe an application of our main theorem to the
theory of Teichmüller space.  Teichmüller space was historically
defined in the context of complex manifolds, but the work of Fischer
and Tromba translates this original picture into the context of
Riemannian geometry, using the manifold of metrics.  (See, in
particular, the papers \cite{fischer84:_purel_rieman_proof_of_struc}
and \cite{fischer84:_weil_peter_metric_teich_space}, as well as the
related \cite{fischer84:_almos_compl_princ_fiber_bundl},
\cite{fischer87:_new_proof_that_teich_space_is_cell},
\cite{tromba86:_natur_algeb_affin_connec_space} and
\cite{tromba87:_energ_funct_for_weil_peter_metric_space}.)

In Section \ref{sec:teichmuller-space}, we describe Teichmüller space
according to Fischer and Tromba's picture.  Along with discussing some
properties of Teichmüller space, we also describe a much-studied
Riemannian metric on it, the so-called Weil-Petersson metric.  The
Weil-Petersson metric arises very naturally in this context, and there
is also a very natural way to generalize it, which we give in Section
\ref{sec:metrics-arising-from}.  It is in this section that our
application is given.

The book \cite{tromba-teichmueller} is an excellent presentation of
Fischer and Tromba's approach to Teichmüller space.  It is
essentially a self-contained work incorporating the references listed
above.  We will use it as the standard reference in this chapter---any
facts that are not directly cited or proved can be found in this book.

\section{Teichmüller space}\label{sec:teichmuller-space}

\begin{cvt}\label{cvt:2}
  For the entirety of this chapter, let our base manifold $M$ be a smooth,
  closed, oriented, two-dimensional manifold of genus $p \geq 2$.
\end{cvt}

\begin{cvt}
  In this chapter, we abandon Convention \ref{cvt:3}.  That is, when
  we write $g$ for a metric in $\M$, we no longer assume that this is
  fixed, but allow $g$ to vary arbitrarily.
\end{cvt}

\subsection{The definition of Teichmüller space}\label{sec:defin-teichm-space}

Since the group $\pos$ of positive $C^\infty$ functions acts on $\M$
by pointwise multiplication, we can define the quotient space by this
action, $\M / \pos$.  It is not hard to see that this action is
smooth, free and proper, from which one can show that the quotient
space $\M / \pos$ is a smooth Fréchet manifold.  This is called the
manifold of conformal classes on $M$.  The name comes from that of a
conformal class $[g]$, which is the set of all metrics of the form
$\rho g$, where $\rho$ is a smooth, positive function.  We cannot use
a conformal class to give a well-defined notion of the length of
vectors in a tangent space, since this varies among representatives of
the class.  However, the \emph{angle} between two vectors is the same
for all representatives, so this notion is well-defined for conformal
classes.  This is analogous to the way that a conformal mapping
preserves angles---in fact, the identity mapping $(M, g) \rightarrow
(M, \rho g)$ is obviously conformal for any $g \in \M$ and $\rho \in
\pos$.

Let $\D$ denote the Fréchet Lie group of smooth,
orientation-preserving diffeomorphisms of $M$ (cf.~Remark
\ref{rmk:20}).  There is an action of $\D$ on $\M$ given by pull-back.
Actually, we can even define the action on $\s$: for all $\phi \in
\D$, $h \in \s$, $x \in M$, and $v, w \in T_x M$, the explicit formula
is
\begin{equation}\label{eq:55}
  \varphi^* h (x) (v, w) = h (\varphi(x)) (D \varphi (x) v, D
  \varphi(x) w).
\end{equation}
This action is compatible with the $\pos$-action on $\M$ in the sense
that if $g_0$ and $g_1$ are equivalent under the $\pos$-action, say
$g_1 = \rho g_0$, then $\varphi^* g_0$ and $\varphi^* g_1$ are also
equivalent under the $\pos$-action, since $\varphi^* g_1 = (\rho \circ
\varphi) \varphi^* g_0$.  In other words, there is a natural action of
$\D$ on $\M / \pos$ that makes the projection $\tilde{\pi} : \M
\rightarrow \M / \pos$ $\D$-equivariant.  Thus, we can define the
quotient space
\begin{equation*}
  \riem := ( \M / \pos ) / \D,
\end{equation*}\label{p:def-R}
which is called the \emph{Riemann moduli space of $M$}, or usually
just \emph{moduli space}.

\begin{rmk}\label{rmk:23}
  As we mentioned above, the Riemann moduli space (and Teichmüller
  space, which we'll meet later) was originally defined in terms of
  complex structures on $M$, not metrics.  It turns out that the
  manifold $\M / \pos$ is, in a sense, diffeomorphic to the manifold
  of complex structures on $M$, which gives the connection to the
  original theory.  Since we don't need this connection for our
  purposes, however, we omit it and instead refer the reader again to
  \cite{tromba-teichmueller} for details.  The approach we take here
  may be less familiar, but is more economic given our previous
  preparations.
\end{rmk}

Moduli space has a somewhat technically challenging structure.  Since
the action of $\D$ on $\M$ has a fixed point at any metric with a
nontrivial isometry group, moduli space has singularities.  It turns
out that these are not very difficult to deal with, as they are only
orbifold singularities---this follows from the fact that the isometry
group of a Riemann surface with genus $p \geq 2$ is necessarily finite
(see Lemma \ref{lem:6} below).  Yet one might still prefer to work
with a smooth manifold.  Teichmüller space is a manifold that can be
seen as a sort of intermediate space between the manifold of conformal
classes and moduli space.  We define this now.

Let $\DO \subset \D$ be the subset of diffeomorphisms that are
homotopic to the identity.  It turns out that the action of $\D_0$ on
$\M$ is free and, though the proof is quite involved, one can show
that the quotient space\label{p:def-T}
\begin{equation*}
  \T := ( \M / \pos ) / \D_0
\end{equation*}
is a smooth manifold.  (This is not done directly, but rather using
the intermediate step of identifying $\M / \pos$ with the space of
hyperbolic metrics on $M$; see below.)  This quotient space is the
\emph{Teichmüller space of $M$}, or simply \emph{Teichmüller space}.

Not only is Teichmüller space a smooth manifold, it is
finite-dimensional.  This allows us to avoid the many difficulties
that arise when dealing with infinite-dimensional spaces like $\M$.

The \emph{mapping class group} of $M$ is defined to be $MCG := \D /
\DO$.\label{p:def-mcg}  This group acts on Teichmüller space, and we have
\begin{equation*}
  \riem = \T / MCG.
\end{equation*}
The general philosophy to keep in mind in this setup is that natural
objects on Teichmüller space should be $MCG$-invariant so that they
descend to moduli space.  This corresponds to ensuring that objects
defined on $\M / \pos$ are $\D$-invariant and not just
$\DO$-invariant.

\subsection{The Weil-Petersson metric on Teichmüller
  space}\label{sec:hyperb-metr-unif}

Before we can define the Weil-Petersson metric, we need to discuss
hyperbolic metrics on $M$.  In particular, the following theorem
allows us to identify the quotient space $\M / \pos$ with the set of
hyperbolic metrics on $M$.  We define a hyperbolic metric as one that
has constant scalar curvature $-1$.  (Other authors may use the
sectional curvature or Gaussian curvature, which differs from the
scalar curvature simply by a constant factor.  We stick here to the
convention of \cite{tromba-teichmueller} for simplicity.)

\begin{thm}[Poincaré uniformization theorem]\label{thm:14}
  Let $g \in \M$ be any Riemannian metric on the closed, oriented,
  smooth surface $M$ of genus $p \geq 2$.  Then there exists a unique
  $\lambda(g) \in \pos$ such that $\lambda(g) g$ is hyperbolic.
\end{thm}

Additionally, it can be shown that the assignment $g \mapsto
\lambda(g)$ is smooth.

Let $\Mhyp \subset \M$\label{p:def-Mhyp} denote the subset of hyperbolic metrics on $M$.
Theorem \ref{thm:14} implies that there is a bijection between $\M /
\pos$ and $\Mhyp$.  It can be shown that in fact, $\Mhyp$ is a smooth
submanifold of $\M$ and this bijection is actually a diffeomorphism.

Furthermore, we can easily show that $\Mhyp$ is $\D$-invariant.
Denote the scalar curvature of a metric $g \in \M$ by $R(g)$---this is
a function on $M$, and $g \in \Mhyp$ if and only if $R(g) \equiv -1$.
But for all $x \in M$ and $g \in \Mhyp$,
\begin{equation*}
  R(\varphi^* g)(x) = R(g)(\varphi(x)) = -1.
\end{equation*}
Therefore $\varphi^* g \in \Mhyp$ as well.

Using the statements of the last two paragraphs, we can
diffeomorphically identify Teichmüller space with the space of
hyperbolic metrics modulo diffeomorphisms homotopic to the identity.
That is,
\begin{equation*}
  \T = (\M / \pos) / \DO \cong \Mhyp / \DO.
\end{equation*}
This is the model of Teichmüller space that we will use from here
on.  We furthermore denote the projection by\label{p:def-pi}
\begin{equation*}
  \pi : \Mhyp \rightarrow \Mhyp / \DO.
\end{equation*}

Since $\Mhyp$ is a submanifold of $\M$, the $L^2$ metric $(\cdot,
\cdot)$ on $\M$ induces a weak Riemannian metric on $\Mhyp$ by
restriction.  We again denote this metric by $(\cdot, \cdot)$, and we
claim that $\D$ acts by isometries on $(\cdot, \cdot)$.  To see this,
we first denote the pull-back action by\label{p:def-A}
\begin{equation}\label{eq:3}
  A : \s \times \D \rightarrow \s,
\end{equation}
and for any $\varphi \in \D$, we define a map\label{p:def-Aphi}
\begin{equation}\label{eq:4}
  \begin{aligned}
    A_\varphi : \s &\rightarrow \s \\
    h &\mapsto A(h, \varphi) = \varphi^* h.
  \end{aligned}
\end{equation}
Note from the definition \eqref{eq:55} of the pull-back that
$A_\varphi$ is a linear map.  Therefore, its differential at each
point is equal to the map itself.  From this, we see that for any
$g \in \M$ and any $h, k \in T_{g} \M \cong \s$,
\begin{equation*}
  (D A_\varphi(g) h, D A_\varphi(g)
  k)_{A_\varphi(g)} = \integral{M}{}{\tr_{\varphi^*
      g}( (\varphi^* h) (\varphi^* k) )}{\mu_{\varphi^* g}}.
\end{equation*}
Let's define $f := \tr_{g}(h k)$, so that $f$ is a function on
$M$.  It's not hard to convince oneself that
\begin{equation*}
  \tr_{\varphi^* g}( (\varphi^* h) (\varphi^* k) ) = f \circ
  \varphi = \varphi^* f,
\end{equation*}
as well as that $\mu_{\varphi^* g} = \varphi^*
\mu_{g}$.  Therefore,
\begin{equation*}
  (D A_\varphi(g) h, D A_\varphi(g)
  k)_{A_\varphi(g)} = \integral{M}{}{(\varphi^* f)}{\varphi^*
    \mu_{g}} = \integral{M}{}{}{\varphi^*(f \mu_{g})} = \integral{M}{}{f}{\mu_{g}} = (h, k)_{g}.
\end{equation*}
This shows that $(\cdot, \cdot)$ is $\D$-invariant.

$\D$-invariance of $(\cdot, \cdot)$ implies that it descends to an
$MCG$-invariant Riemannian metric, also denoted $(\cdot, \cdot)$, on
the quotient $\Mhyp / \DO$.  This metric is called the
\emph{Weil-Petersson metric}.

The Weil-Petersson metric is an extremely interesting and intensely
studied object.  Some of its most important properties are the
following.  There is a natural complex structure on Teichmüller space,
which we won't describe here, and with respect to this structure the
Weil-Petersson metric is Kähler.  It has strictly negative sectional
curvature and strictly negative holomorphic sectional curvature.  The
Weil-Petersson metric is incomplete, and its completion leads to
interesting connections with the so-called Deligne-Mumford
compactification of moduli space.  We will explore this metric some
more in Subsection \ref{sec:nonc-weil-peterss}.

For the moment, though, we leave the Weil-Petersson metric and move on
to some other properties of Teichmüller space that we will need.

\subsection{The fiber bundle structure of Teichmüller
  space}\label{sec:fiber-bundle-struct}

It is clear that Teichmüller space is the base space of a principal
$\DO$-bundle with total space $\pi : \Mhyp \rightarrow \Mhypd$.  If we
put the $L^2$ metric on $\Mhyp$ and the Weil-Petersson metric on $\T$,
then this bundle forms what is called a \emph{weak Riemannian
  principal $\DO$-bundle}.  That is, it is a principal $\DO$-bundle
with a weak Riemannian metric on each of the base and total spaces,
and the differential of the projection is an isometry when restricted
to the horizontal space.  In other words, for all $g \in \Mhyp$,
$D\pi_{g}|_{H_g}$ is an isometry.  Here, $H_g$\label{p:def-hg} is the
horizontal tangent space defined as follows.  Let
$V_{g}$\label{p:def-vg} be the vertical tangent space, i.e., the
tangent space to the orbit $\DO \cdot g$. Then $H_{g} = V_{g}^\perp$.
Of course, the tangent space decomposes as $T_g \Mhyp = H_g \oplus
V_g$.

We can easily determine what the vertical tangent space $V_g$ is.  If
$\varphi_t$, for $t \in (-\epsilon, \epsilon)$, is a one-parameter
family of diffeomorphisms for which $\varphi_0 = \id$, then the
differential of $\varphi_t$ at $t = 0$ is a vector field.  That is, if
we denote the set of $C^\infty$ vector fields on $M$ by
$\mathfrak{X}(M)$, then there is some $X \in \mathfrak{X}(M)$ for which
\begin{equation*}
  \left. \frac{d}{dt} \right|_{t=0} \varphi_t = X.
\end{equation*}
Every $X \in \mathfrak{X}(M)$ arises in this way.  Furthermore, if $g
\in \M$ is any metric, then by definition,
\begin{equation*}
  \left. \frac{d}{dt} \right|_{t=0} \varphi_t^* g = L_X g,
\end{equation*}
where $L_X g$ is the Lie derivative of $g$ with respect to $X$.
Therefore we have
\begin{equation}\label{eq:1}
  V_g = \{ L_X g \mid X \in \mathfrak{X}(M) \}.
\end{equation}

It is also possible to explicitly describe the horizontal tangent
space $H_g$, though we will not derive this description here.  Define
the divergence of an element $h \in \s$ to be the one-form given
locally by
\begin{equation*}
  (\delta_{g} h)_i = \frac{1}{\sqrt{\det g}}
  \frac{\partial}{\partial x^j} \left( g^{j k} h_{k i}
    \sqrt{\det g} \right) - \frac{1}{2} g^{k l}
  g^{j m} h_{m l} \frac{\partial g_{j k}}{\partial x^i}.
\end{equation*}
Then we have\label{p:def-sttg}
\begin{equation*}
  H_g = \stt := \{ h \in \s \mid \tr_g h = 0,\ \delta_g h = 0 \}.
\end{equation*}

Horizontal lifts of paths exist for the bundle $\pi : \Mhyp
\rightarrow \Mhyp / \DO$.  That is, given a path $\gamma : [0,1]
\rightarrow \Mhyp / \DO$ and an element $g \in \pi^{-1}(\gamma(0))$,
there exists a unique path $\tilde{\gamma} : [0,1] \rightarrow \Mhyp$
such that $\pi \circ \tilde{\gamma} = \gamma$, $\tilde{\gamma}(0) =
g$, and $\tilde{\gamma}'(t) \in H_{\gamma(t)}$ for all $t$.  Note that
this fact does not hold in general for weak Riemannian principal
bundles.  There are a number of ways to see that it does hold
here---we will now present a proof that relies on the existence of a
slice for the $\DO$-action and the ability to take any path in $\M$
and construct a horizontal path from it.

The existence of a slice is given by the following theorem.

\begin{thm}[{\cite[Thms.~2.4.2 and 2.4.5]{tromba-teichmueller}}]\label{thm:9}
  Let $g \in \Mhyp$ be arbitrary.  Then there exists a local smooth
  submanifold $\mathcal{X}_g \subset \Mhyp$ passing through $g$ such
  that each point of $\mathcal{X}_g$ corresponds to exactly one orbit
  of $\DO$.  That is, if $\varphi \in \DO$, $\tilde{g} \in {\cal X}_g$
  and $\varphi^* \tilde{g} \in \mathcal{X}_g$, then $\varphi =
  \textnormal{id}$.  Furthermore, the local submanifolds
  $\mathcal{X}_g$ form the (nonlinear) charts of an atlas for
  $\Mhypd$.
\end{thm}

Now, we want to take a given path in $\M$ and construct a horizontal
path from it.  We note that the horizontal space for the $\D$-action
on $\M$, i.e., the vectors tangent to the $\D$-orbits, is given by
\cite[\S 3]{freed89:_basic_geomet_of_manif_of}\label{p:def-htilg}
\begin{equation*}
  \widetilde{H}_g = \{ h \in \s \mid \delta_g h = 0 \}.
\end{equation*}
The vertical tangent space is again given by \eqref{eq:1}, since we
showed \eqref{eq:1} for any $g \in \M$, not just $g \in \Mhyp$.  Again
we have a decomposition $T_g \M = \widetilde{H}_g \oplus V_g$.

Let's denote the projection of $\M$ onto the $\DO$-orbit space by
\begin{equation*}
  \pi_\M : \M \rightarrow \M / \DO.
\end{equation*}
We simply view this as a mapping of sets, since we have not considered
any particular structure on $\M / \DO$ (and don't need to).

The statement we need is the following.

\begin{lem}\label{lem:2}
  Let $g_t$, $t \in [0,1]$, be any piecewise $C^1$ path in $\M$.  Then
  there exists a unique piecewise $C^1$ path $\tilde{g}_t$, $t \in
  [0,1]$, with the properties that $\tilde{g}_0 = g_0$, $\tilde{g}_t$
  is horizontal wherever $g_t$ is differentiable, and $\tilde{g}_t$ is
  equivalent to $g_t$ under the $\DO$-action on $\M$.  That is,
  $\tilde{g}_0 = g_0$, $\tilde{g}'_t \in \widetilde{H}_{\tilde{g}_t}$
  for all $t$ for which $g'_t$ exists, and $\pi_\M(\tilde{g}_t) =
  \pi_\M(g_t)$ for all $t \in [0,1]$.

  Furthermore, $\tilde{g}_t$ is of minimal length among the class of
  all paths equivalent to $g_t$ (though it is of course not the unique
  minimizer).
\end{lem}
\begin{proof}
  Without loss of generality, we assume that $g_t$ is actually $C^1$
  on its entire domain.  (Otherwise just apply the proof to each
  segment on which it is $C^1$.)  By the decomposition shown above,
  for each $t \in [0,1]$, there exist $h_t \in \widetilde{H}_g$ and
  $X_t \in \mathfrak{X}(M)$ such that
  \begin{equation*}
    g'_t = h_t + L_{X_t} g_t.
  \end{equation*}
  
  Now, as is well known (or easily looked up, say in
  \cite[Thm.~17.15]{lee03:_introd_to_smoot_manif}), since $M$ is
  compact, we can integrate the time-dependent vector field $-X_t$ to
  get a one-parameter family of diffeomorphisms $\varphi_t$ for which
  $\varphi_0 = \id$ and
  \begin{equation*}
    \frac{d}{dt} \varphi_t = - X_t \circ \varphi_t
  \end{equation*}
  for all $t \in [0,1]$.

  We then define $\tilde{g}_t := \varphi_t^* g_t$ and claim that this
  is the desired path.  Clearly $\pi_\M(\tilde{g}_t) = \pi_\M(g_t)$.
  To show that $\tilde{g}'_t \in \widetilde{H}_{\tilde{g}_t}$, we
  recall that $A$ denotes the action of $\D$ on $\s$
  (cf.~\eqref{eq:3}) and compute
  \begin{equation}\label{eq:6}
    \tilde{g}'_t = \frac{d}{dt} ( \varphi_t^* g_t ) =
    \frac{d}{dt}  A(g_t, \varphi_t) = DA(g_t, \varphi_t)[g'_t, -X_t
    \circ \varphi_t],
  \end{equation}
  since the $t$-derivative of $\varphi_t$ is $-X_t \circ \varphi_t$.

  Now, denote the partial derivatives of $A$ in the first and second
  arguments by $D_1 A$ and $D_2 A$, respectively.  We have
  \begin{equation}\label{eq:5}
    D_1 A(g_t, \varphi_t) g'_t = DA_{\varphi_t}(g_t) g'_t =
    \varphi_t^* g'_t = \varphi_t^* (h_t + L_{X_t} g_t).
  \end{equation}
  Recall that $A_{\varphi_t} = A(\cdot, \varphi_t)$ by definition
  (cf.~\eqref{eq:4}), and the second equality follows because, as
  mentioned above, $A_{\varphi_t}$ is a linear map.

  Next, we compute
  \begin{equation}\label{eq:2}
    \begin{aligned}
      D_2 A(g_t, \varphi_t)[-X_t \circ \varphi_t] &=
      \left. \frac{d}{ds} \right|_{s=0} (\varphi_{t+s}^* g_t) =
      \varphi_t^* \left( \left. \frac{d}{ds} \right|_{s=0}
        (\varphi_{t+s} \circ \varphi_t^{-1})^* g_t \right) \\
      &= - \varphi_t^* (L_{X_t} g_t).
    \end{aligned}
  \end{equation}

  Since $DA = D_1 A + D_2 A$, inserting \eqref{eq:5} and \eqref{eq:2}
  into \eqref{eq:6} gives
  \begin{equation*}
    \tilde{g}'_t = \varphi_t^* h_t =: \tilde{h}_t.
  \end{equation*}
  But since $\delta_{g_t} h_t = 0$, we also have $\delta_{\tilde{g}_t}
  \tilde{h}_t = \delta_{\varphi_t^* g_t} (\varphi_t^* h_t) = 0$.  Thus
  we have shown that $\tilde{g}'_t \in \widetilde{H}_{\tilde{g}_t}$,
  and so $\tilde{g}_t$ is horizontal as desired.

  Uniqueness of $\tilde{g}_t$ with the desired properties follows from
  the fact that on a Riemann surface of genus $p \geq 2$, there are no
  Killing fields---we prove this in Lemma \ref{lem:6}, immediately
  following the proof of this lemma.  Thus, the family $\varphi_t$
  above is the only one for which $\varphi_0 = \id$ and $\varphi_t^*
  g_t$ is horizontal.

  To show that $\tilde{g}_t$ is of minimal length among all paths
  equivalent to $g_t$, let $\bar{g}_t$ be another path with
  $\pi_\M(\bar{g}_t) = \pi_\M(g_t)$, and let $\psi_t$ be the unique
  one-parameter family of diffeomorphisms from $\DO$ such that
  $\bar{g}_t = \psi_t^* \tilde{g}_t$.  Just as above, we can compute
  that
  \begin{equation*}
    \bar{g}'_t = \frac{d}{dt} (\psi_t^* \tilde{g}_t) = \psi_t^*
    \tilde{h}_t + \psi_t^* (L_{Y_t} \tilde{g}_t),
  \end{equation*}
  where
  \begin{equation*}
    Y_t := \left( \frac{d}{dt} \psi_t \right) \circ \psi_t^{-1} \in \mathfrak{X}(M).
  \end{equation*}
  and we recall that $\tilde{h}_t = \tilde{g}'_t$.  But by the
  orthogonality of horizontal and vertical vectors,
  \begin{align*}
    L(\bar{g}_t) &= \integral{0}{1}{\| \bar{g}'_t \|_{\bar{g}_t}}{dt}
    = \integral{0}{1}{\| \psi_t^* \tilde{h}_t \|_{\psi^*
        \tilde{g}_t}}{dt} + \integral{0}{1}{\| \psi_t^* (L_{Y_t}
      \tilde{g}_t)
      \|_{\psi_t^* \tilde{g}_t}}{dt} \\
    &= \integral{0}{1}{\| \tilde{h}_t \|_{\tilde{g}_t}}{dt} +
    \integral{0}{1}{\| \psi_t^* (L_{Y_t} \tilde{g}_t) \|_{\psi_t^*
        \tilde{g}_t}}{dt} \geq L(\tilde{g}_t),
  \end{align*}
  where we have used the $\D$-invariance of $(\cdot, \cdot)$ in the
  second line.
\end{proof}

\begin{lem}\label{lem:6}
  Let $g \in \M$ be any Riemannian metric on the genus $p$ surface
  $M$.  Then $g$ has finite isometry group.  In particular, $g$ admits
  no Killing fields.
\end{lem}
\begin{proof}
  By the Poincare uniformization theorem (Theorem \ref{thm:14}), there
  exists a function $\rho \in C^\infty(M)$ and a metric $\bar{g} \in
  \Mhyp$ such that $g = \rho \bar{g}$.  Our goal is to show that every
  isometry of $g$ is also an isometry of $\bar{g}$, which then implies
  that the isometry group of $g$ is finite, since by Hurwitz's theorem
  \cite[p.~258]{farkas92:_rieman_surfac} the isometry group of
  $\bar{g}$ is finite.
  
  So let $\varphi \in \D$ be an isometry of $g$.  We then have that
  $\varphi^* g = g$, so
  \begin{equation*}
    (\varphi^* \rho) \varphi^* \bar{g} = \varphi^* (\rho \bar{g}) =
    \varphi^* g = g = \rho \bar{g}.
  \end{equation*}
  Thus $\varphi^* \bar{g}$ and $\bar{g}$ are conformally equivalent.
  Furthermore, since the space of hyperbolic metrics is
  $\D$-invariant, these two metrics are both hyperbolic.  But since
  the Poincaré uniformization theorem says that there is exactly one
  hyperbolic metric in each conformal class of metrics, we must have
  $\varphi^* \bar{g} = \bar{g}$.  Thus $\varphi$ is an isometry of
  $\bar{g}$, as was to be shown.
\end{proof}

\begin{rmk}\label{rmk:25}
  Note the following astounding fact, implied by the proof of Lemma
  \ref{lem:6}.  We have just shown that the unique hyperbolic metric
  in a conformal class is, in a very strong sense, the most symmetric
  metric in that class.  Namely, \emph{any isometry} of \emph{any
    metric} in that class is also an isometry of the hyperbolic
  metric.
\end{rmk}

Using the results above, we can prove the existence of horizontal
lifts.

\begin{thm}\label{thm:16}
  For any $C^1$ path $\gamma : [0,1] \rightarrow \Mhypd$ and any $g
  \in \pi^{-1}(\gamma(0))$, there exists a unique horizontal lift
  $\tilde{\gamma} : [0,1] \rightarrow \Mhypd$ with $\tilde{\gamma}(0)
  = g$.  In particular, $\tilde{\gamma}'(t) \in H_{\tilde{\gamma}(t)}$
  for all $t \in [0,1]$.

  Furthermore, $L(\tilde{\gamma}) = L(\gamma)$ and $\tilde{\gamma}$
  has minimal length among the class of curves whose image projects to
  $\gamma$ under $\pi$.
\end{thm}
\begin{proof}
  Recall that $\mathcal{X}_g$ denotes the slice around $g$ guaranteed
  by Theorem \ref{thm:9}.  By the compactness of the interval $[0,1]$,
  we can choose a finite set $\{g_1, \ldots, g_m\} \subset \Mhyp$ such
  that the collection $\{\pi(\mathcal{X}_{g_1}), \ldots,
  \pi(\mathcal{X}_{g_m})\}$ covers $\gamma$; furthermore, we choose
  this collection such that all the sets in it have nonempty
  intersection with $\gamma$.  Let the collection further be chosen
  such that the intersection $\pi(\mathcal{X}_{g_k}) \cap \gamma$ is
  equal to $\gamma(J_k)$ for some interval $J_k \subseteq [0,1]$.  To
  achieve this condition, we simply shrink the slices if necessary.
  Finally, we assume that the numbering is done such that the initial
  points of the intervals $J_k$ are in increasing order---again, we
  may have to shrink the slices to achieve this.  In particular, this
  assures us that $0 \in J_1$.

  Let $\gamma_k$ denote the lift of $\gamma|_{J_k}$ to
  $\mathcal{X}_{g_k}$, and let $\hat{\gamma}_k$ be the horizontal path
  equivalent to $\gamma_k$ guaranteed by Lemma \ref{lem:2}.  The path
  $\hat{\gamma}_k$ is a horizontal lift of $\gamma|_{J_k}$.

  Let $\varphi_1 \in \DO$ be the unique element such that $\varphi_1^*
  \hat{\gamma}_1(0) = g$, and define $\tilde{\gamma}_1 := \varphi_1^*
  \hat{\gamma}_1$.  Note that $\tilde{\gamma}_1$ is still a horizontal
  lift of $\gamma|_{J_k}$, and that $\tilde{\gamma}_1(0) = g$.  Let
  $a_2 \in J_1 \cap J_2$, and let $\varphi_2 \in \DO$ be the unique
  element such that $\varphi_2^* \hat{\gamma}_2(a_2) =
  \tilde{\gamma}_1(a_2)$.  Define $\tilde{\gamma}_2 := \varphi_2^*
  \hat{\gamma}_2$.  By repeating this procedure, we get a path
  $\tilde{\gamma}_k$ that is a horizontal lift of $\gamma|_{J_k}$,
  such that $\tilde{\gamma}_k$ intersects $\tilde{\gamma}_{k+1}$ in at
  least one point, for each $k = 1, \ldots, m$.

  Using the uniqueness of the horizontal paths of Lemma \ref{lem:2},
  we see that since the paths $\tilde{\gamma}_k$ and
  $\tilde{\gamma}_{k+1}$ intersect in one point, they intersect over
  the entire range where they are equivalent under $\DO$.  Therefore,
  we can glue the paths $\tilde{\gamma}_k$ together to a
  differentiable path $\tilde{\gamma}$ that is a horizontal lift of
  $\gamma$---and since $\tilde{\gamma}_1(0) = g$, we clearly have
  $\tilde{\gamma}(0) = g$, as desired.

  The minimality of $\tilde{\gamma}$ follows from Lemma \ref{lem:2}.
  To show that $L(\tilde{\gamma}) = L(\gamma)$, recall that $\pi :
  \Mhyp \rightarrow \Mhypd$ is a weak Riemannian principal
  $\DO$-bundle, so $D \pi(g) |_{H_g}$ is an isometry for every $g \in
  \Mhypd$.  Therefore
  \begin{equation*}
    L(\tilde{\gamma}) = \integral{0}{1}{\| \tilde{\gamma}'(t)
      \|_{\tilde{\gamma}(t)}}{dt} = \integral{0}{1}{\| D \pi(\tilde{\gamma}(t))
      \tilde{\gamma}'(t) \|_{\pi(\tilde{\gamma}(t))}}{dt} = \integral{0}{1}{\|
      \gamma'(t) \|_{\gamma(t)}}{dt} = L(\gamma).
  \end{equation*}
\end{proof}

The structures we've described in this section will all be put to work
for us in the next section.

\section{Metrics arising from submanifolds of
  $\M$}\label{sec:metrics-arising-from}

In this section, our goal is to define an entire class of metrics that
includes the Weil-Petersson metric, and to use the main result of the
thesis, Theorem \ref{thm:41}, to prove a fact about the completions of
Teichmüller space with respect to such metrics.

Before we do this in Subsection \ref{sec:gener-weil-peterss}, we will
go into some more detail on the completion of the Weil-Petersson
metric.  This discussion will motivate our considerations in
Subsection \ref{sec:gener-weil-peterss}.

\subsection{Completing Teichmüller space with respect to the
  Weil-Petersson metric}\label{sec:nonc-weil-peterss}

It has long been known that the Weil-Petersson metric is
incomplete---this was initially and independently proved in
\cite{wolpert75:_noncom_of_weil_peter_metric_for_space} and
\cite{chu76:_weil_peter_metric_in_modul_space}.  The proof shows that
there are Weil-Petersson geodesics that, in finite time, leave
Teichmüller space.

The limit points of such geodesics can be given a meaning as Riemann
surfaces themselves, which we would like to describe heuristically
here.  We will not justify anything, but instead suggest to the reader
the various references given in this chapter.

First, we note that for a hyperbolic metric on a compact surface,
there is a unique geodesic in each free homotopy class, and this is
the shortest curve in the class
\cite[Lem.~2.4.4]{jost06:_compac_rieman_surfac}.

Let a Weil-Petersson geodesic $\gamma : (0,1] \rightarrow \T$ be such
that it cannot be continuously extended to the domain $[0,1]$, and let
$\tilde{\gamma} : (0,1]$ be a horizontal lift of $\gamma$.  Then there
exist $r$ disjoint, nonhomotopic, noncontractible simple closed curves
$\eta^1, \dots, \eta^r$ on $M$, with $1 \leq r \leq 3 p - 3$, such
that the following holds.  For each $i = 1, \dots, r$, let $\eta^i_t$
denote the unique $\tilde{\gamma}(t)$-geodesic in the free homotopy class
of $\eta^i$.  Then the length of each $\eta^i_t$ with respect to
$\tilde{\gamma}(t)$ converges to zero for $t \rightarrow 0$.  In this
case, we say that the curves $\eta^1_t, \dots, \eta^r_t$ are
\emph{pinched} along $\tilde{\gamma}$, since geometrically the curves
shrink to points.

\begin{figure}[t]
  \centering
  \includegraphics[width=11cm]{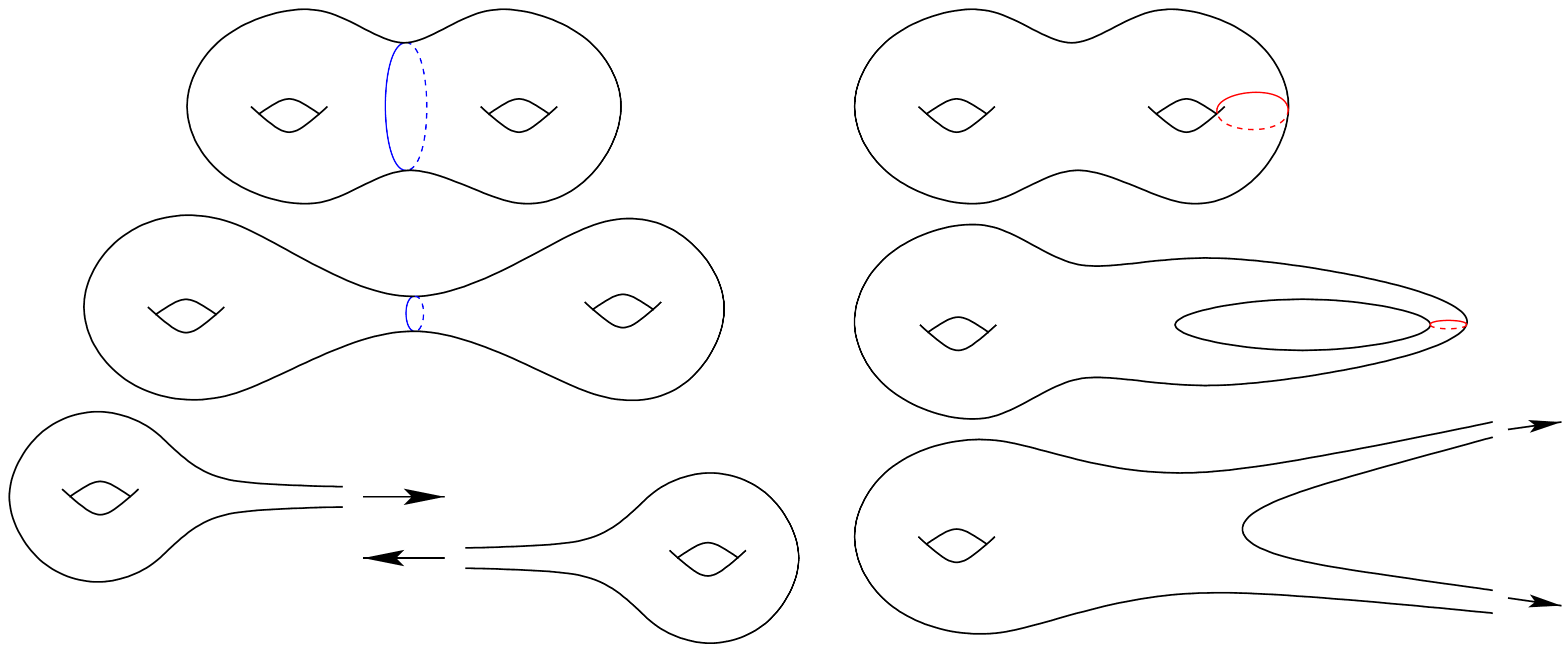}
  \caption{The formation of cusps when pinching a simple closed
    hyperbolic geodesic.  On the left, the blue, homologically
    trivial, curve is pinched and the resulting limit surface is
    disconnected, with two cusps.  On the right, the red,
    homologically nontrivial curve is pinched---the limit surface has
    lower genus than the original one, and has two
    cusps.}\label{fig:cusped-surfaces}
\end{figure}

Thanks to the so-called collar lemma (see, e.g.,
\cite{randol79:_cylin_in_rieman_surfac} or, for surfaces of variable
curvature, \cite{buser78:_collar_theor_and_examp}), around each
geodesic on a hyperbolic surface there exists a neighborhood that is
diffeomorphic to an open-ended cylinder.  Furthermore, the width of
this neighborhood increases to infinity as the length of the
hyperbolic geodesic decreases to zero.  Thus, around each of the
curves $\eta^i_t$, two so-called cusps develop as $t \rightarrow 0$,
meaning that in the limit, two infinitely long cylinders extend out
from the surface, and the width of these cylinders approaches zero at
infinity.  Figure \ref{fig:cusped-surfaces} illustrates this in case
of the two basic possibilities here.  By pinching a homologically
trivial curve, we end up with a disconnected surface of the same total
genus (heuristically, the same number of ``donut holes'').  When
pinching a homologically nontrivial curve, the limit surface stays
connected but is of lower genus.  By combining these two pictures for
all pinched curves, one can imagine a general limit surface.  By
adding in all such limit surfaces, we obtain the completion of
Teichmüller space with respect to the Weil-Petersson metric, which we
will denote by $\overline{\T}$.

Let's translate this discussion into the language that we've been
using throughout the rest of the thesis.  What essentially happens is
that the family of metrics $\tilde{\gamma}(t)$ is equivalent (under
the $\DO$-action on $\Mhyp$) to a family $g_t$ of metrics that
becomes unbounded and deflates along the curves $\eta^1,
\dots, \eta^r$ as $t \rightarrow 0$.  With respect to the metrics
$g_t$, the lengths of vectors tangent to each $\eta^i$ converge to
zero, and the lengths of vectors perpendicular to each $\eta^i$ become
infinite.  Thus, if we define $g_0$ to be equal to the pointwise limit
of $g_t$ off of $\eta^1, \dots, \eta^r$ and, say, zero on $\eta^1,
\dots, \eta^r$, then we get a limit metric on $M$ that is measurable.
Furthermore, by \cite[p.~233]{farkas92:_rieman_surfac}, the volume of
the limit surface is finite, as we would expect from our main result,
Theorem \ref{thm:41}.

Thus, the completion of Teichmüller space fits very nicely into the
setting that we have established in this thesis.  Of course, since we
are dealing only with a special type of metric on a special type of
base manifold, \emph{and} we only consider horizontal paths---i.e.,
there are no limit metrics that arise from families of degenerating
diffeomorphisms---the limit metrics that are possible make up only a
small subset, with very nice properties, of the limit metrics that we
get when considering the completion of all of $\M$.

Before we leave this subsection, let's just mention a couple of the
rich properties of the completion of Teichmüller space with respect to
the Weil-Petersson metric.

As described above, $\overline{\T}$ is closely related to a
compactification of moduli space.  As in the case of $\overline{\M}$
(and the completion of any metric space), the distance function of the
Weil-Petersson metric extends to the completion $\overline{\T}$.
Here, however, more is true.  In a certain sense, the Weil-Petersson
Riemannian metric (the scalar product) also extends to
$\overline{\T}$.  This is proved, and given precise meaning, in
\cite{masur-extension}.

The action of the mapping class group $MCG$ extends to $\overline{\T}$
\cite{abikoff77:_degen_famil_of_rieman_surfac}, and the action of any
individual element of $MCG$ is an isometry of the extended
Weil-Petersson metric on $\overline{\T}$.  Therefore, we can form the
quotient
\begin{equation*}
  \overline{\riem} := \overline{\T} / MCG,
\end{equation*}
and the Weil-Petersson distance function projects to a complete metric
(in the sense of metric spaces) on $\overline{\riem}$.

It turns out that $\overline{\riem}$ is a compactification of moduli
space.  Moreover, this compactification agrees with the famous
\emph{Deligne-Mumford compactification}
\cite{deligne69:_irred_of_space_of_curves}, which arises via very
different considerations in algebraic geometry.  Thus, the
Weil-Petersson metric forms the bridge between two very important
aspects of Riemannian geometry and algebraic geometry.

Hopefully this has provided sufficient motivation to convince the
reader that the Weil-Petersson metric and generalizations thereof are
worthwhile objects of study.

\subsection{Generalizations of the Weil-Petersson metric}\label{sec:gener-weil-peterss}

The natural way to generalize the Weil-Petersson metric within this
context is to take also non-hyperbolic (variable curvature)
representatives for each conformal class in $\M / \pos$, giving us
some submanifold of $\M$ which differs from $\Mhyp$ but still contains
exactly one representative of each conformal class.  The goal of this
subsection is to describe this idea rigorously.

This idea is directly inspired by \cite{hj-riemannian} and
\cite{habermann98:_metric_rieman_surfac_and_geomet}, where metrics on
Teichmüller space were also defined using variable curvature metrics
in place of the hyperbolic metric.  These metrics differ from the ones
considered here, however.  After we define our own generalization, we
remark on the differences.  Unfortunately, completely describing the
concepts necessary to understand the differences is outside the scope
of this thesis, so we must regrettably do this in a way that will be
helpful only to those ``in the know.''

By the Poincaré uniformization theorem, Theorem \ref{thm:14}, the
principal $\pos$-bundle $\M \rightarrow \M / \pos$ is trivial, and
$\Mhyp$ is a section of this bundle.  (Of course, we could have
already deduced from the product structure $\M \cong \pos \times
\M_\mu$ given in Subsection \ref{sec:prod-manif-struct} that the
bundle is trivial.)  The idea now is to select a different section of
$\M \rightarrow \M / \pos$.  In fact, we will simultaneously consider
all smooth sections $\mathcal{N}$ with the property that they are
$\D$-invariant, which we require so that we still have diffeomorphisms
$\T \cong \Ncal / \DO$ and $\riem \cong \Ncal / \D$.

\begin{dfn}\label{dfn:3}
  We call a smooth, $\D$-invariant section of $\M \rightarrow \M /
  \pos$ a \emph{modular section}.  Given a modular section $\Ncal
  \subset \M$, we call the quotients $\Ncal / \DO$ and $\Ncal / \D$
  the \emph{$\Ncal$-model of Teichmüller space} and the
  \emph{$\Ncal$-model of moduli space}, respectively.
\end{dfn}

The proof of the next lemma is obvious from the decomposition $\M
\cong \pos \times \Mhyp$ implied by the Poincaré uniformization
theorem.

\begin{lem}\label{lem:5}
  For all $g \in \Mhyp$, choose $\rho_g \in \pos$ such that
  \begin{enumerate}
  \item the assignment $g \mapsto \rho_g$ is smooth and
  \item $\rho_{\varphi^* g} = \varphi^* \rho_g$ for all $g \in \Mhyp$.
  \end{enumerate}
  Then the set
  \begin{equation*}
    \Ncal := \{ \rho_g g \mid g \in \Mhyp \}
  \end{equation*}
  is a modular section.  Furthermore, every modular section arises in
  this way.
\end{lem}

Modular sections other than $\Mhyp$ of course exist.  Let us mention
just one important example, that of the space of Bergman metrics on
$M$.  It requires a few facts about Riemann surfaces that we won't
prove, and can be safely skipped.

\begin{eg}\label{eg:4}
  As is well-known and proved, for example, in
  \cite{tromba-teichmueller}, in two dimensions complex structures are
  in one-to-one correspondence with conformal structures---so each
  element of $\M / \pos$ determines complex structure on $M$.  So for
  this example, we work with complex instead of conformal structures.

  Let $c$ be a complex structure on $M$.  Then the space of
  holomorphic one-forms on $(M, c)$ has complex dimension $p$, the
  genus of $M$ \cite[Prop.~III.2.7]{farkas92:_rieman_surfac}.  Let
  $\theta_1, \dots, \theta_p$ be an $L^2$-orthonormal basis of this
  space.  That is,
  \begin{equation*}
    \frac{i}{2} \int_M \theta_j \wedge \overline{\theta_k} = \delta_{jk}.\label{correction:bergman}
  \end{equation*}
  The Bergman metric is defined by
  \begin{equation*}
    g_B := \frac{1}{p} \sum_{i=1}^p \theta_i \bar{\theta}_j.
  \end{equation*}
  It is clear from this construction that the set of all Bergman
  metrics is indeed a modular section.

  The Bergman metric can also be seen as the pull-back of the flat
  metric on the Jacobian of $M$ via the Albanese period map.  It
  arises, for example, in arithmetic geometry
  \cite{chinburg86:_introd_to_arakel_inter_theor}.
\end{eg}

Let us now return to our general considerations.

\begin{cvt}\label{cvt:1}
  For the remainder of this chapter, let $\Ncal$ be a fixed but
  arbitrary modular section.
\end{cvt}

The next proposition tells us that we are, from the differential
topological point of view, justified in calling $\Ncal / \DO$ the
$\Ncal$-model of Teichmüller space.

\begin{prop}\label{prop:5}
  The quotient $\Ncal / \DO$ is a smooth, finite-dimensional manifold.
  Furthermore, we have diffeomorphisms $\Psi : \Mhyp \rightarrow
  \Ncal$ and $\Phi: \Mhyp / \DO \rightarrow \Ncal / \DO$.  The
  diffeomorphism $\Psi$ is $\D$-equivariant.
\end{prop}
\begin{proof}
  Let $\rho_g$ be the assignment that gives $\Ncal$ from $\Mhyp$, as
  in Lemma \ref{lem:5}.  Then it is clear that the following map is a
  diffeomorphism:
  \begin{align*}
    \Psi : \Mhyp &\rightarrow \Ncal \\
    g &\mapsto \rho_g g.
  \end{align*}
  The rest of the claims follow from the fact that $\Psi$ is
  $\D$-equivariant:
  \begin{equation*}
    \Psi(\varphi^* g) = \rho_{\varphi^* g} (\varphi^* g) = (\varphi^*
    \rho_g) (\varphi^* g) = \varphi^* (\rho_g g) = \varphi^* \Psi(g)
  \end{equation*}
  by the assumptions on $\rho_g$.  Thus, the manifold structure on
  $\Ncal / \DO$ is given by the bijection with $\Mhyp / \DO$ induced
  by $\Psi$.
\end{proof}

As in the case of the section $\Mhyp$, the $L^2$ metric on $\M$
restricts to $\Ncal$, and its $\D$-invariance implies that it projects
to an $MCG$-invariant metric on $\Ncal / \DO$.  We call this metric,
as well as the metric it induces on $\T \cong \Mhyp$ via the
diffeomorphism of Proposition \ref{prop:5}, the \emph{generalized
  Weil-Petersson metric} on the $\Ncal$-model of Teichmüller space.
As in the case of the bundle $\Mhyp \rightarrow \Mhypd$, these metrics
turn the bundle $\Ncal \rightarrow \Ncal / \DO$ into a weak Riemannian
principal $\DO$-bundle.

\begin{rmk}\label{rmk:4}
  The following remark requires some basic knowledge about Teichmüller
  theory.  For those lacking this, it can be safely skipped.

  For those with this background, we note here the difference between
  the metrics of Habermann and Jost (cf.~\cite{hj-riemannian},
  \cite{habermann98:_metric_rieman_surfac_and_geomet}) and the
  generalized Weil-Petersson metrics we have just introduced.

  Of course, Teichmüller space was historically defined in complex
  analysis as the space of complex structures on $M$ modulo $\DO$.
  (See, for example, \cite{imayoshi-taniguchi}.) The correspondence
  between complex structures and conformal classes is given by the
  existence of local isothermal (or conformal) coordinates on any
  two-dimensional manifold.

  Now, recall that the cotangent space of Teichmüller space, when
  defined in the complex analytic way, is given by the space of
  holomorphic quadratic differentials on $M$ with respect to the given
  complex structure.  The correspondence between these and horizontal
  vectors of $\Mhyp$ is given by the fact that a traceless,
  divergence-free element of $\s$ is the real part of a holomorphic
  quadratic differential.

  Habermann and Jost generalize the Weil-Petersson metric on the
  complex analytic version of Teichmüller space as follows.  From this
  point of view, a point $\tau \in \T$ represents an equivalence class
  of complex structures on $M$.  Let's choose representatives
  $\Sigma_\tau$ of the equivalence classes $\tau \in \T$---thus,
  $\Sigma_\tau$ is a complex manifold with one complex dimension---in
  a smooth manner (we will have to be vague about what this means for
  reasons of space).

  For each fixed $\tau \in \T$, choose complex coordinates $z$ on
  $\Sigma_\tau$ and a Hermitian metric
  \begin{equation*}
    \lambda_\tau^2(z) \,dz d\bar{z}.
  \end{equation*}
  If $\lambda_\tau^2$ is chosen such that it varies smoothly with
  $\tau$, then we get a Riemannian cometric on $\T$ by defining, for
  each $\tau \in \T$ and any two holomorphic quadratic differentials
  on $\Sigma_\tau$ locally given by $\psi_0(z) \, dz d\bar{z}$ and
  $\psi_1(z) \, dz d\bar{z}$,
  \begin{equation}\label{eq:12}
    (\!( \psi_0, \psi_1 )\!)_\tau := \frac{i}{2}
    \integral{\Sigma_\tau}{}{\frac{\psi_0(z)
        \overline{\psi_1(z)}}{\lambda_\tau^2(z)}}{dz \wedge d\bar{z}}.
  \end{equation}
  If $\lambda_\tau^2 \, dz d\bar{z}$ is the hyperbolic metric on
  $\Sigma_\tau$ for each $\tau \in \T$, then $(\!( \cdot , \cdot )\!)$
  is the Weil-Petersson cometric on $\T$.  Otherwise we get some
  generalization of it.

  The difference between these generalizations and the ones we study
  here is that a horizontal tangent vector to $\Ncal$ is
  divergence-free, but need no longer be traceless.  Thus, in contrast
  to the case where $\Ncal = \Mhyp$, a horizontal tangent vector to
  $\Ncal$ need not in general be the real part of a holomorphic
  quadratic differential, and so some extra terms will enter into
  \eqref{eq:12} if we try to view our generalized Weil-Petersson
  metric through the lens of the complex analytic theory of
  Teichmüller space.  In essence, the objects $\psi_0$ and $\psi_1$ on
  which \eqref{eq:12} is evaluated are natural tangent vectors when we
  view $\Sigma_\tau$ as an element of the moduli space of
  one-dimensional complex manifolds, but not when we view $(\Sigma,
  \lambda_\tau^2 \, dz d\bar{z})$ as an element of the section ${\cal
    N}$.
\end{rmk}

Our next goal is to establish the existence of horizontal lifts for
the bundle $\Ncal \rightarrow \Ncal / \DO$.  Thanks to our previous
work on $\Mhyp$, this is not difficult.

Let us define some notation before stating the result.  We denote the
bundle projection by\label{p:def-pincal}
\begin{equation*}
  \pi_\Ncal : \Ncal \rightarrow \Ncal / \DO,
\end{equation*}
and we denote the horizontal tangent space of this bundle at $g \in
\Ncal$ by\label{p:def-hng}
\begin{equation*}
  H^\Ncal_g := V_g^\perp \subset T_g \Ncal.
\end{equation*}
By Proposition \ref{prop:5}, we have a commutative diagram
\begin{equation}\label{eq:7}
  \begin{CD}
    \Mhyp @>{\Psi}>> \Ncal \\
    @V{\pi}VV  @VV{\pi_\Ncal}V \\
    \Mhypd @>{\Phi}>> \Ncal / \DO,
  \end{CD}
\end{equation}
where the horizontal arrows are diffeomorphisms and the vertical
arrows are projections.

\begin{thm}\label{thm:42}
  For any $C^1$ path $\gamma : [0,1] \rightarrow \Ncal / \DO$ and any
  $g \in \pi_\Ncal^{-1}(\gamma(0))$, there exists a unique horizontal
  lift $\tilde{\gamma} : [0,1] \rightarrow \Ncal$ with
  $\tilde{\gamma}(0) = g$.  In particular, $\tilde{\gamma}'(t) \in
  H^\Ncal_{\tilde{\gamma}(t)}$ for all $t \in [0,1]$.

  Furthermore, $L(\tilde{\gamma}) = L(\gamma)$ and $\tilde{\gamma}$
  has minimal length among the class of curves whose image projects to
  $\gamma$ under $\pi_\Ncal$.
\end{thm}
\begin{proof}
  Let $\bar{\gamma}$ be the horizontal lift of $\Phi^{-1} \circ
  \gamma$ to $\Mhyp$ with initial point $\Psi^{-1}(g)$, and let
  $\hat{\gamma} := \Psi \circ \bar{\gamma}$.  This is a path in
  $\Ncal$.  Finally, we let $\tilde{\gamma}$ be the horizontal path
  equivalent to $\hat{\gamma}$ guaranteed by Lemma \ref{lem:2}.

  We claim that $\tilde{\gamma}$ is the desired lift.  It is a path in
  $\Ncal$ by $\D$-invariance of $\Ncal$, and it is clearly horizontal.
  By construction, we see that $\tilde{\gamma}(0) = g$.

  Finally, it is easily seen from commutativity of \eqref{eq:7} and
  the fact that $\pi \circ \bar{\gamma} = \Phi^{-1} \circ \gamma$ that
  $\pi_\Ncal \circ \tilde{\gamma} = \gamma$.  Uniqueness of
  $\tilde{\gamma}$ with the given properties follows from uniqueness
  of the paths of Lemma \ref{lem:2}.

  The remainder of the theorem is proved precisely as in Theorem
  \ref{thm:16}.
\end{proof}

This theorem allows us to prove the application of the thesis' main
result that we have in mind for the generalized Weil-Petersson metric.
In the following, we denote the distance function of $(\Ncal, (\cdot,
\cdot))$ by $d_\Ncal$.

\begin{thm}\label{thm:43}
  Let $\{[g_k]\}$ be a Cauchy sequence in the $\Ncal$-model of
  Teichmüller space, $\Ncal / \DO$, with respect to the generalized
  Weil-Petersson metric.  Then there exist representatives
  $\tilde{g}_k \in [g_k]$ and an element $[g_\infty] \in \Mfhat$ such
  that $\{\tilde{g}_k\}$ is a $d_\Ncal$-Cauchy sequence that
  $\omega$-subconverges to $[g_\infty]$.

  Furthermore, if $\{[g^0_k]\}$ and $\{[g^1_k]\}$ are equivalent
  Cauchy sequences in $\Ncal / \DO$, then there exist representatives
  $\tilde{g}^0_k \in [g^0_k]$ and $\tilde{g}^1_k \in [g^1_k]$, as well as an element
  $[g_\infty] \in \Mfhat$, such that $\{\tilde{g}^0_k\}$ and $\{\tilde{g}^1_k\}$ are
  $d_\Ncal$-Cauchy sequences that both $\omega$-subconverge to
  $[g_\infty]$.

  Finally, if $\{[g^0_k]\}$ and $\{[g^1_k]\}$ are inequivalent Cauchy
  sequences in $\Ncal / \DO$, then there exists no choice of
  representatives $\tilde{g}^0_k \in [g^0_k]$ and $\tilde{g}^1_k \in
  [g^1_k]$ such that $\{\tilde{g}^0_k\}$ and $\{\tilde{g}^1_k\}$
  $\omega$-subconverge to the same element of $\Mfhat$.
\end{thm}
\begin{proof}
  The first claim would follow from Theorem \ref{thm:41} if we could
  show that there are representatives $\tilde{g}_k \in [g_k]$ such that
  $\{\tilde{g}_k\}$ is a $d_\Ncal$-Cauchy sequence, since this implies that it
  is also a $d$-Cauchy sequence.  So this is what we will show.

  Let's denote the distance function induced by the generalized
  Weil-Petersson metric on $\Ncal / \DO$ by $\delta$.  For each $k \in
  \N$, let $\gamma_k : [0,1] \rightarrow \Ncal / \DO$ be any path from
  $[g_k]$ to $[g_{k+1}]$ such that
  \begin{equation*}
    L(\gamma_k) \leq 2 \delta([g_k], [g_{k+1}]).
  \end{equation*}

  For any $\tilde{g}_1 \in \pi_\Ncal^{-1}([g_1])$, let $\tilde{\gamma}_1$ be
  the horizontal lift of $\gamma_1$ to $\Ncal$ with
  $\tilde{\gamma}_1(0) = \tilde{g}_1$ which is guaranteed by Theorem
  \ref{thm:42}.  Then clearly $\tilde{g}_2 := \gamma_2(1) \in
  \pi_\Ncal^{-1}([g_2])$.  Furthermore,
  \begin{equation*}
    d_\Ncal(\tilde{g}_1, \tilde{g}_2)
    \leq L(\tilde{\gamma}_1) = L(\gamma_1) \leq 2 \delta([g_1], [g_2]).
  \end{equation*}

  We repeat this process, i.e., let $\tilde{\gamma}_2$ be the unique
  horizontal lift of $\gamma_2$ with $\tilde{\gamma}_2(0) = \tilde{g}_2$, and
  set $\tilde{g}_3 := \tilde{\gamma}_2(1)$, etc.  We again get
  \begin{equation*}
    d_\Ncal(\tilde{g}_2, \tilde{g}_3) \leq 2 \delta([g_2], [g_3]).
  \end{equation*}

  By continuing, we get a sequence of representatives $\tilde{g}_k \in [g_k]$
  such that for each $k \in \N$,
  \begin{equation*}
    d_\Ncal(\tilde{g}_k, \tilde{g}_{k+1}) \leq 2 \delta([g_k], [g_{k+1}]).
  \end{equation*}
  Thus, since $\{[g_k]\}$ is a Cauchy sequence, $\{\tilde{g}_k\}$ is a
  $d_\Ncal$-Cauchy sequence, as was to be shown.

  To prove the second statement, it suffices by Theorem \ref{thm:41}
  to show that we can find representatives $\tilde{g}^0_k \in [g^0_k]$
  and $\tilde{g}^1_k \in [g^1_k]$ such that $\{\tilde{g}^0_k\}$ and
  $\{\tilde{g}^1_k\}$ are equivalent $d_\Ncal$-Cauchy sequences.

  To do this, select representatives $\tilde{g}^0_k \in [g^0_k]$ as guaranteed
  by the first statement of the proof, so that in particular
  $\{\tilde{g}^0_k\}$ is $d_\Ncal$-Cauchy.  Next, for each $k \in \N$, choose
  a path $\gamma_k$ in $\Ncal / \DO$ from $[g^0_k]$ to $[g^1_k]$ such
  that
  \begin{equation*}
    L(\gamma_k) \leq 2 \delta([g^0_k], [g^1_k]).
  \end{equation*}
  Let $\tilde{\gamma}_k$ be the horizontal lift of $\gamma_k$ to
  $\Ncal$ with $\tilde{\gamma}_k(0) = \tilde{g}^0_k$, and define $\tilde{g}^1_k :=
  \tilde{\gamma}_k(1) \in [g^1_k]$.  Then
  \begin{equation*}
    d_\Ncal(\tilde{g}^0_k, \tilde{g}^1_k) \leq L(\tilde{\gamma}_k) = L(\gamma_k) \leq
    2 \delta([g^0_k], [g^1_k]).
  \end{equation*}
  From the above inequality, the fact that $\{\tilde{g}^0_k\}$ is
  $d_\Ncal$-Cauchy, and the fact that $\{[g^0_k]\}$ and $\{[g^1_k]\}$
  are equivalent Cauchy sequences, it is easy to see that $\{\tilde{g}^1_k\}$
  is $d_\Ncal$-Cauchy and that $\{\tilde{g}^0_k\}$ and $\{\tilde{g}^1_k\}$ are
  equivalent $d_\Ncal$-Cauchy sequences.

  To prove the last statement, note that since $\{[g^0_k]\}$ and
  $\{[g^1_k]\}$ are inequivalent, we have
  \begin{equation*}
    \epsilon := \lim_{k \rightarrow \infty} \delta([g^0_k], [g^1_k]) > 0.
  \end{equation*}
  By definition, we also have
  \begin{equation*}
    \delta([g^0_k], [g^1_k]) = \inf \{ d_\Ncal(\tilde{g}^0_k,
    \tilde{g}^1_k) \mid \tilde{g}^0_k \in
    [g^0_k],\ \tilde{g}^1_k \in [g^1_k] \}.
  \end{equation*}
  Thus, no matter what representatives $\tilde{g}^0_k \in [g^0_k]$ and
  $\tilde{g}^1_k \in [g^1_k]$ we choose,
  \begin{equation*}
    \lim_{k \rightarrow \infty} d_\Ncal(\tilde{g}^0_k, \tilde{g}^1_k)
    \geq \epsilon > 0.
  \end{equation*}
  So Theorem \ref{thm:41} implies the statement immediately.
\end{proof}

We have thus given one interesting application of our main result.  We
close the thesis with some brief comments about the above theorem.  Of
course, this theorem is considerably weaker than the picture for
hyperbolic metrics in two regards.  Firstly, the convergence notion
that one takes for hyperbolic metrics (which we have not given
explicitly here) is stronger than $\omega$-convergence.  Secondly, the
class of limit metrics is very bad---our results do not rule out that
a Cauchy sequence of metrics degenerates anywhere on the surface $M$,
whereas a Cauchy sequence of hyperbolic metrics can degenerate only on
a finite set of simple closed curves, as we saw above.  We hope that
by exploiting knowledge about the conformal structure induced by a
sequence of metrics in $\Ncal$, one should be able to constrain these
degenerations in the limit of a Cauchy sequence---ideally, for
well-behaved $\Ncal$, restricting degeneration to the ``nodes,'' as
the limits of these closed curves are known in Teichmüller theory.
Limitations on degenerations also arise from the $\D$-invariance of
$\Ncal$ and the fact that only horizontal paths in $\Ncal$---and not
vertical paths, coming from families of diffeomorphisms---matter for
the quotient $\Ncal / \DO$.  However, going deeper into these aspects
is beyond the scope of this thesis and must be regarded as a future
direction for study.

Despite the shortcomings of the above result, we see it as quite
useful, as it gives relatively strong information about a new class of
metrics on Teichmüller space---namely, that their completions can
consist only of finite-volume metrics.  Furthermore, it illustrates
the potential for applications of our main theorem and provides a
starting point for further investigations.



\chapter*{Metrics and convergence notions}\label{cha:metr-conv-noti}

\vspace{1em}

\begin{center}
  {\bf Riemannian metrics and the distance functions associated to
    them}

  \vspace{1.5em}

  \begin{tabular}[h]{ccc}
    Manifold & Metric & Distance function
    \\ \bottomrule
    $\M$ & $(\cdot, \cdot)$ & $d$ \\
    $\U^\dagger$ & $(\cdot, \cdot)$ & $d_\U$ \\
    $\Ncal^\dagger$ & $(\cdot, \cdot)$ & $d_\Ncal$ \\
    $\Matx$ & $\langle \cdot , \cdot \rangle$ & $d_x$ \\
    $\Matx$ & $\langle \cdot, \cdot \rangle^0$ & $\theta^g_x$ \\[6pt]
    \multicolumn{3}{p{6cm}}{{\small $^\dagger$ Here, $\U$ represents an
        amenable subset and $\Ncal$ represents a modular section.}}
  \end{tabular}
\end{center}

\vspace{2em}

\begin{center}
  {\bf Relations between various notions of convergence and Cauchy
    sequences}
\end{center}

\vspace{1em}

In the following chart, we illustrate the relationships between the
different notions of Cauchy and convergent sequences on $\M$.  We let
$\{g_k\}$ be a sequence in $\M$ and $\tilde{g} \in \Mf$.  A double
arrow (``$\Longrightarrow$'') between two statements means that the
one implies the other.  A single arrow (``$\longrightarrow$'') means
that one statement implies the other, assuming the condition that is
listed below the chart.

\begin{equation*}
  \xymatrix@R=48pt{
    & & & *+[F-:<3pt>]{\txt{$\{\mu_{g_k}\}$ converges\\weakly to $\mu_{\tilde{g}}$}} \\
    *+[F-:<3pt>]{\txt{$\{g_k\}$ is a.e.\\$\theta^g_x$-Cauchy}} &
    *+[F-:<3pt>]{\txt{$\{g_k\}$ is\\$\Theta_M$-Cauchy}} \ar[l]_-1 &
    *+[F-:<3pt>]{\txt{$\{g_k\}$ is\\$d$-Cauchy}} \ar@{=>}[l]
    \ar@<1ex>[d]^-2 \ar@<1ex>[r]^-4
    & *+[F-:<3pt>]{\txt{$\{g_k\}$ $\omega$-converges\\to $\tilde{g}$}}
    \ar@<1ex>@{=>}[l] \ar@{=>}[d] \ar@{=>}[u] \\
    *\txt{\phantom{some space here}} & & *+[F-:<3pt>]{\txt{$\{g_k\}$
        $L^2$-converges\\to $\tilde{g}$}} \ar@<1ex>[u]^-3
    & *+[F-:<3pt>]{\txt{$\Vol(Y,g_k) \rightarrow
        \Vol(Y,\tilde{g})$\\for $Y$ measurable}}
  }
\end{equation*}

\begin{enumerate}
\item After passing to a subsequence
\item If there exists an amenable subset $\U$ such that $\{g_k\}
  \subset \U$, then there exists some $\tilde{g} \in \U^0$ such that
  the implication holds
\item If there exists a quasi-amenable subset $\U$ such that $\{g_k\}
  \subset \U$
\item After passing to a subsequence, there exists some $\tilde{g}
  \in \Mf$ such that the implication holds
\end{enumerate}

\chapter*{List of frequently used symbols}\label{cha:list-frequently-used}

\renewcommand{\arraystretch}{1.4}

\begin{centering}
  \begin{tabular}{lp{8.2cm}p{4cm}}
    Symbol & Meaning & Location in text \\ \bottomrule
    $A$ & The pull-back action $\s \times \D \rightarrow \s$ &
    p.~\ref{p:def-A} \\
    $A_\varphi$ & The linear map $A(\cdot, \varphi) : \s \rightarrow
    \s$ & p.~\ref{p:def-Aphi} \\
    $\cl(\U)$ & The closure of the subset $\U \subseteq \M$ in the
    $C^\infty$ topology of $\s$ & p.~\pageref{p:cl-U} \\
    $d$ & The Riemannian distance function of $(\M, (\cdot , \cdot))$
    & Definition \ref{dfn:22}, p.~\pageref{dfn:22} \\
    $d_\U$ & The distance function induced by $d$ on the completion of
    an amenable subset $\U$ & Definition \ref{dfn:14},
    p.~\pageref{dfn:14} \\
    $d_x$ & The Riemannian distance function of $(\M_x, \langle \cdot
    , \cdot \rangle)$ & Definition \ref{dfn:11},
    p.~\pageref{dfn:11} \\
    $\textnormal{End}(M)$ & The endomorphism bundle of the manifold
    $M$ &
    p.~\pageref{p:end-m} \\
    $g$ & From Section \ref{sec:conventions} of Chapter
    \ref{cha:preliminaries} onwards, a fixed, smooth
    reference metric & Convention \ref{cvt:3}, p.~\pageref{cvt:3} \\
    $H_g$ & The horizontal tangent space for the bundle $\Mhyp
    \rightarrow \Mhypd$ at $g$ & p.~\pageref{p:def-hg} \\
    $\widetilde{H}_g$ & The horizontal tangent space for the bundle $\M
    \rightarrow \M / \DO$ at $g$ & p.~\pageref{p:def-htilg} \\
    $H_g^\Ncal$ & The horizontal tangent space for the bundle $\Ncal
    \rightarrow \Ncal / \DO$ at $g$ & p.~\ref{p:def-hng} \\
    $i_\mu$ & A diffeomorphism $\M_\mu \times \pos \rightarrow \M$ &
    Equation \ref{eq:111}, p.~\pageref{eq:111} \\
    $L^{\langle \cdot , \cdot \rangle}(a_t)$ & The length of the path
    $a_t$ in $\M_x$ with respect to the Riemannian metric $\langle
    \cdot , \cdot \rangle$ & Definition \ref{dfn:11},
    p.~\pageref{dfn:11} \\
    $L^{\langle \cdot , \cdot \rangle^0}(a_t)$ & The length of the path
    $a_t$ in $\M_x$ with respect to the Riemannian metric $\langle
    \cdot , \cdot \rangle^0$ & Definition \ref{dfn:11},
    p.~\pageref{dfn:11} \\
    $M$ & The base manifold, a smooth, closed, finite-dimensional
    manifold. & Convention \ref{cvt:6}, p.~\pageref{cvt:6} \\
    $\M$ & The Fréchet manifold of smooth Riemannian metrics on $M$
    &
    p.~\pageref{p:manif-metr} \\
    $\M^s$ & The Hilbert manifold of Riemannian metrics on $M$ with $H^s$
    coefficients (for $s > n/2$) &
    p.~\pageref{p:manif-metr} \\
    $\M^0$ & The set of $L^2$-sections of $S^2 T^* M$ that are
    a.e.~positive definite & Definition \ref{dfn:4},
    p.~\pageref{dfn:4} \\
    $\Mf$ & The set of measurable semimetrics on $M$ with finite volume &
    Definition \ref{dfn:27}, p.~\ref{dfn:27} \\
  \end{tabular}

  \begin{tabular}{lp{8.2cm}p{4cm}}
    Symbol & Meaning & Location in text \\ \bottomrule
    $\Mfhat$ & The quotient of $\Mf$ formed by identifying semimetrics
    that differ only on their degenerate sets and a nullset &
    Definition \ref{dfn:9}, p.~\pageref{dfn:9} \\
    $\Mm$ & The set of measurable semimetrics on $M$ & Definition
    \ref{dfn:7}, p.~\pageref{dfn:7} \\
    $\Mmhat$ & The quotient of $\Mm$ formed by identifying semimetrics
    that differ only on their degenerate sets and a nullset &
    Definition \ref{dfn:7}, p.~\pageref{dfn:7} \\
    $\M_\mu$ & The Fréchet manifold of metrics inducing the volume form $\mu$
    & Equation \eqref{eq:115}, p.~\pageref{eq:115} \\
    $\Matx$ & The manifold of positive-definite symmetric
    $(0,2)$-tensors at $x \in M$ & Equation \eqref{eq:120},
    p.~\pageref{eq:120} \\
    $\Mhyp$ & In Chapter \ref{cha:appl-teichm-space}, the manifold of hyperbolic metrics on $M$ &
    p.~\pageref{p:def-Mhyp} \\
    $MCG$ & In Chapter \ref{cha:appl-teichm-space}, the mapping class
    group $\D / \DO$ of $M$ & p.~\ref{p:def-mcg} \\
    $n$ & The dimension of the base manifold $M$ & Convention
    \ref{cvt:6}, p.~\pageref{cvt:6} \\
    $\Ncal$ & A fixed modular section & Convention \ref{cvt:1},
    p.~\pageref{cvt:1} \\
    $p$ & In Chapter \ref{cha:appl-teichm-space}, the genus of the
    Riemann surface $M$ & Convention \ref{cvt:2}, p.~\pageref{cvt:2} \\
    $\pos$ & The Fréchet manifold of smooth, positive functions on $M$ &
    p.~\pageref{sec:manif-posit-funct} \\
    $\mathcal{R}$ & The moduli space of a Riemann surface of genus $p
    \geq 2$ & p.~\pageref{p:def-R}\\
    $\s$ & The Fréchet space of smooth, symmetric $(0,2)$-tensor
    fields on $M$ &
    p.~\pageref{p:def-s} \\ 
    $\s^s$ & The Hilbert space of symmetric $(0,2)$-tensor
    fields on $M$ with $H^s$ coefficients &
    p.~\pageref{p:def-ss} \\
    $\satx$ & The vector space of symmetric $(0,2)$-tensors at $x \in
    M$ &
    p.~\ref{p:satx} \\
    $S_{\{g_k\}}$ & The singular set of a sequence $\{g_k\} \subset
    \M$ & Definition \ref{dfn:25}, p.~\pageref{dfn:25} \\
    $\sconf$ & The set of pure trace tensors (w.r.t.~$g$) &
    p.~\pageref{p:def-scg} \\
    $\st$ & The set of $g$-traceless tensors & p.~\pageref{p:def-stg} \\
    $\stt$ & The set of traceless, divergence-free tensors
    (w.r.t.~$g$) & p.~\pageref{p:def-sttg} \\
    $\T$ & The Teichmüller space of a Riemann surface of genus $p \geq
    2$ & p.~\pageref{p:def-T}\\
    $\tr_{\tilde{g}}$ & The $\tilde{g}$-trace of a tensor or product
    of tensors & Definition \ref{dfn:21}, p.~\pageref{dfn:21} \\
    $\U$ & Usually denotes an amenable or quasi-amenable subset of
    $\M$ & Definition \ref{dfn:2}, p.~\pageref{dfn:2}; Definition
    \ref{dfn:26}, p.~\pageref{dfn:26} \\
    $\U^0$ & The $L^2$-completion of the set $\U \subset \M$ (i.e., the completion with
    respect to $\| \cdot \|_g$) & Definition \ref{dfn:4},
    p.~\pageref{dfn:4} \\
  \end{tabular}
  
  \begin{tabular}{lp{8.2cm}p{4cm}}
    Symbol & Meaning & Location in text \\ \bottomrule
    $V_g$ & The vertical tangent space for the $\DO$-action at $g \in
    \M$ & p.~\pageref{p:def-vg} \\
    $\overline{X}^{\textnormal{pre}}$ & The precompletion of a metric
    space $X$ &
    p.~\pageref{p:precompl} \\
    $\overline{X}$ & The completion of a metric space $X$ & 
    p.~\pageref{p:compl} \\
    $X_{\tilde{g}}$ & The degenerate set of $\tilde{g} \in \M$ &
    Definition \ref{dfn:23}, p.~\pageref{dfn:23} \\
    $X_{\{g_k\}}$ & The degenerate set of a sequence $\{g_k\} \subset
    \M$ & Definition \ref{dfn:25}, p.~\pageref{dfn:25} \\
    $\mathcal{X}_g$ & A local submanifold passing through $g$ forming
    a nonlinear chart for $\Mhypd$ & Theorem \ref{thm:9},
    p.~\pageref{thm:9} \\
    $\left( \frac{\alpha}{\mu} \right)$ & For $\alpha$ an $n$-form and
    $\mu$ a volume form, the unique function with the property that
    $\alpha = \left( \frac{\alpha}{\mu} \right) \mu$ & Equation
    \eqref{eq:118}, p.~\pageref{eq:118} \\
    $\btop$ & The boundary of $\M$ in the $C^\infty$ topology of $\s$
    & p.~\pageref{p:btop} \\
    $\mu_{\tilde{g}}$ & The volume form induced by a metric
    $\tilde{g}$ & Equation \eqref{eq:119}, p.~\pageref{eq:119} \\
    $\pi$ & In Chapter \ref{sec:teichmuller-space}, the projection
    $\pi : \Mhyp \rightarrow \Mhypd$ & p.~\pageref{p:def-pi} \\
    $\pi_\Ncal$ & In Chapter \ref{sec:teichmuller-space}, the projection
    $\pi_\Ncal : \Ncal \rightarrow \Ncal / \DO$ &
    p.~\ref{p:def-pincal} \\
    $\theta^g_x$ & A metric (in the sense of metric spaces) defined on
    $\Matx$ as the distance function induced by $\langle \cdot , \cdot
    \rangle^0$ & Definition \ref{dfn:15}, p.~\pageref{dfn:15} \\
    $\Theta_M$ & A metric (in the sense of metric spaces) on $\M$
    given by integrating $\theta^g_x$ over $M$ & Definition
    \ref{dfn:16}, p.~\pageref{dfn:16} \\
    $\Theta_Y$ & A pseudometric on $\M$ given by integrating
    $\theta^g_x$ over $Y \subseteq M$ & Definition
    \ref{dfn:16}, p.~\pageref{dfn:16} \\
    $\Omega$ & The mapping $\overline{\M} \rightarrow \M_f$ sending an
    equivalence class of Cauchy sequences to the semimetric they
    $\omega$-subconverge to & Theorem \ref{thm:38},
    p.~\pageref{thm:38} \\
    $(\cdot, \cdot)$ & The $L^2$ weak Riemannian metric on $\M$ and
    its submanifolds & Definition \ref{dfn:22}, p.~\pageref{dfn:22} \\
    $(\cdot, \cdot)_{\tilde{g}}$ & The $L^2$ scalar product on
    functions and tensors induced from a metric $\tilde{g}$ &
    Definition \ref{dfn:22}, p.~\pageref{dfn:22}; Equation
    \eqref{eq:121}, p.~\pageref{eq:121} \\
    $\| \cdot \|_{\tilde{g}}$ & The norm induced from the $L^2$ scalar
    product & Definition \ref{dfn:22}, p.~\pageref{dfn:22} \\
    $\langle \cdot , \cdot \rangle$ & The Riemannian metric on $\Matx$
    given by the trace & Lemma \ref{lem:48}, p.~\pageref{lem:48} \\
    $\langle \cdot , \cdot \rangle_{\tilde{g}}$ & The scalar product
    on $\satx$ given by the $\tilde{g}$-trace & Definition
    \ref{dfn:21}, p.~\pageref{dfn:21} \\
    $\langle \cdot , \cdot \rangle^0$ & A Riemannian metric on $\Matx$
    related to $\langle \cdot, \cdot \rangle$ by $\langle h, k
    \rangle^0_{\tilde{g}} = \langle h, k \rangle_{\tilde{g}} \det
    \tilde{G}$ & Definition \ref{dfn:15}, p.~\pageref{dfn:15}
  \end{tabular}
\end{centering}


\bibliography{main} \bibliographystyle{hamsplain}

\newpage

\chapter*{List of Corrections}

The following is a list of the changes that have been made from the
version that was submitted in September 2008 to the Mathematical
Institute of the University of Leipzig.

We have not listed the corrections of minor typos that did not affect
the mathematical consistency of the text.

\vspace{3ex}

\begin{center}
  \begin{tabular}[h]{p{0.15\textwidth}p{0.8\textwidth}}
    \textbf{p.~\pageref{correction:charts}:} & Corrected typo in the
    set notation for the maximal atlas. \\
    \textbf{p.~\pageref{p:amenable-atlas}:} & Added condition that
    $\phi_\alpha = \psi_\alpha | U_\alpha$ to definition of amenable
    atlas; adjusted proof of Lemma \ref{lem:47} to reflect this. \\
    \textbf{p.~\pageref{correction:conv-conditions}} & Added remark
    on dependence of conditions for $\omega$-convergence. \\
    \textbf{p.~\pageref{p:deflated-sets}:} & Corrected typos in
    second paragraph of proof of Theorem \ref{thm:39}---all
    appearances of $X_{g_{k_l}}$ changed to $X_{\{g_{k_l}\}}$. \\
    \textbf{p.~\pageref{sec:first-uniq-result}ff:} & Added Lemma
    \ref{lem:10} and Remark \ref{rmk:26}; improved statement and
    corrected proof of Proposition \ref{prop:21}. \\
    \textbf{p.~\pageref{p:gik}:} & Corrected typo in proof of
    Proposition \ref{prop:21}: in second to last paragraph, $g^k_i$
    changed to $g^i_k$. \\
    \textbf{p.~\pageref{prop:27}:} & Changed statement to reflect that
    an element of $\Mfhat$ may have both bounded and unbounded
    representatives. \\
    \textbf{p.~\pageref{eq:141}:} & Changed $\psi(\lambda_k)$ and
    $\psi(\lambda_{k+1})$ in \eqref{eq:141} to $\lambda_k$ and
    $\lambda_{k+1}$, respectively. \\
    \textbf{p.~\pageref{correction:bergman}} & Corrected definition of
    ``$L^2$-orthonomal'' in Example \ref{eg:4}. \\
  \end{tabular}
\end{center}



\end{document}